\documentclass[a4paper,12pt]{book}

\usepackage{tikz}
\usepackage{amsthm}
\usepackage[all]{xy}
\usepackage{graphicx}
\usepackage{amsfonts}
\usepackage{enumerate}
\usepackage{amsmath}
\usepackage[hang,small,bf]{caption}
\usepackage{amssymb}
\usepackage{verbatim}
\usepackage{longtable}
\usepackage{hyperref}

\newtheorem{defn}{Definition}[section]
\newtheorem{teo}[defn]{Theorem}
\newtheorem*{teo*}{Theorem}
\newtheorem*{prop*}{Proposition}
\newtheorem{prop}[defn]{Proposition}
\newtheorem{oss}[defn]{Remark}
\newtheorem{lemma}[defn]{Lemma}
\newtheorem{cor}[defn]{Corollary}

\renewenvironment{proof}{\textit{Proof.}}{\hfill $\Box$}
\renewcommand{\theta}{\vartheta}
\newcommand{\T}{{\mathbb T}}
\newcommand{\Z}{{\mathbb Z}}
\newcommand{\R}{{\mathbb R}}
\newcommand{\N}{{\mathbb N}}

\newcommand{\A}{{\mathbb A}}
\newcommand{\HH}{{\mathbb H}}

\newcommand{\dt}{\frac{d}{dt}}

\renewcommand{\phi}{\varphi}
\newcommand{\bigslant}[2]{{\raisebox{.2em}{$#1$}\left/\raisebox{-.2em}{$#2$}\right.}}

\begin{document}


\frontmatter 

\begin{titlepage}
\begin{center}
\large
\textsc{\Large Ruhr Universit\"at Bochum\\}
\vspace{0.2cm}
\includegraphics[width=3.8cm]{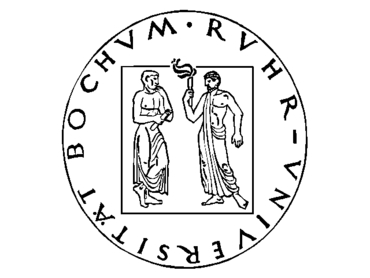}

\vspace{0.2cm}
\textsc{\large Fakult\"at f\"ur Mathematik}\\ \vspace{0.1cm}

\vspace{3cm}

{\LARGE \textbf{On the existence of orbits satisfying periodic or conormal boundary conditions for Euler-Lagrange flows}}\\[0.2cm]
{\Large 29.06.2015}\\[0.2cm]
\textsc{\LARGE Ph.D. Thesis}\\
\vspace{2.5cm}

\begin{center}
\textbf{{\Large Luca Asselle}}\\ \vspace{-0.0cm} {\normalsize
\ttfamily luca.asselle@rub.de}
\end{center}
\vspace{0.5cm}

\centering {\Large Advisor}\\ \vspace{-0.0cm} \textbf{{\Large Prof. Dr. Alberto Abbondandolo}} 
\end{center}

\vspace{1cm}
\begin{center}
\textsc{\Large Academic Year 2014/2015}
\end{center}
\newpage
\thispagestyle{empty} \phantom{}
\end{titlepage}


$$ $$

\vspace{70mm}

\begin{verse}
\textit{\hspace{10mm}This dissertation is the result of my own work and includes\\
\hspace{10mm}nothing that is the outcome of work done in collaboration \\
\hspace{10mm}except where specifically indicated in the text.\\
\hspace{10mm}This dissertation is not substantially the same as any \\
\hspace{10mm}that I have submitted for a degree or diploma or any \\
\hspace{10mm}other qualification at any other university.}\\
\quad \\
\quad  \\
\quad \\
\quad \hspace{100mm}Luca Asselle\\
\hspace{100mm} June 29, 2015.
\end{verse}


\newpage
\cleardoublepage

\begin{center}
\textbf{ACKNOWLEDGMENTS}\\

\vspace{2mm}

I express my deep gratitude to the following people, things and places.
\end{center}

\vspace{-2mm}

To my PhD (and Master as well) advisor Alberto Abbondandolo, for many helpful discussions, for his invaluable support in preparing this thesis, but most of all for all the dinners at his place I was invited to. 

\vspace{2mm}

To all my colleagues at the ``Ruhr Universit\"at Bochum'', for having been great friends during the past three years.

\vspace{2mm}

To Stephan Mescher, for kindly having told me all his best jokes and also for having called me at 3 am completely drunk the day before easter.

\vspace{2mm}

To Verena Gr\"af, for having been coorganizer of the \#game4good-project providing humanitarian help in Ethiopia and allowing me and my girlfriend to fly towards Sardinia for a 4-days VIP trip in occasion 
of the Sardinia World Rally Championship. 

\vspace{2mm}

To the ``Stadt Bochum'' and to the whole ``Nord-Rhein-Westfalen'' region, for having been (despite the not idyllic weather) a great and calm place to be.

\vspace{2mm}

To the ``Fitness Gym'' Bochum, for having let me train very hard almost every day of the week in the last year. In fact, half of this work is there arisen.

\vspace{2mm}

To the ``Universit\`a degli studi di Pisa'', where I really felt home during my graduation period, to all the professors and the dear friends I met there.

\vspace{2mm}

To all my teenage friends, with whom I spent really crazy moments and in particular to Giuseppe Ivaldi, for kindly having organized the go kart race where I won the precious Super Mario Kart trophy.

\vspace{2mm}

To my hometown Imperia, where I was born and grown safely, and to all my teachers and friends of the primary and high school, in particular to the great "Prof. Merlo", who deeply encouraged me to study mathematics.

\vspace{2mm}

To my father, my sister, my brother, my nephews and all my family, for having supported me through all these years and for letting me never stop believe in myself.

\vspace{2mm}

To my girlfriend Claudia Bonanno, for making my days wonderful, for patiently tolerating me and for steadily pushing me work hard and never give up.

\vspace{2mm}

Finally, to my mum, who passed away 4 years ago. You were a real example of purity, kindness, devotion, humility, diligence and, most of all, you believed in me. I swear you, you will always be proud of me. 


\tableofcontents
\mainmatter 

\chapter{Introduction}
\label{chapter0}

\section{Overview of the problem.}

In this chapter we give a brief overview of the problems we are interested in and state the main results of this thesis. 
Throughout the whole work $M$ will be a boundaryless compact manifold. Given a Tonelli Lagrangian, i.e. a smooth function $L:TM\rightarrow \R$ which is $C^2$-strictly convex and superlinear in each fiber, 
we consider the Euler-Lagrange flow $\phi_t:TM\rightarrow TM$, that is the flow defined by the Euler-Lagrange equation, which in local coordinates can be written as 
$$\frac{d}{dt} \frac{\partial L}{\partial v} (q,v) \ +\ \frac{\partial L}{\partial q}(q,v) \ = \ 0\, .$$

Since the Lagrangian is time-independent, the energy $E$ associated to $L$ is a first integral of the motion, meaning that it is constant along solutions of the Euler-Lagrange equation. Therefore, it makes sense to study the dynamics of the Euler-Lagrange 
flow $\phi_t$ restricted to a given energy level set $E^{-1}(k)$, $k\in \R$.

We will be mainly interested in the existence of orbits connecting two given submanifolds $Q_0,Q_1\subseteq M$ and satisfying suitable boundary conditions, known as \textit{conormal boundary conditions}, and of periodic orbits on 
a given energy level. 

The method of attack that will be used is that the desired Euler-Lagrange orbits are in one to one correspondence with the critical points of a suitable action functional (or, more generally, with the zeros of a suitable 1-form).

\begin{oss}
This approach has a nice functional setting only under the additional assumption that the Tonelli Lagrangian is quadratic at infinity in each fiber. However, this is not a problem for our purposes since the energy levels of a Tonelli 
Lagrangian are always compact and hence we can modify $L$ outside a compact set to achieve the desired quadratic growth condition. 
Hereafter all the Lagrangians will be therefore supposed quadratic at infinity. 
\end{oss}

For the ``connecting $Q_0$ with $Q_1$'' problem, this action functional is given by the so called free-time Lagrangian action functional 
$$\A_k: \bigcup_{T>0} \ H^1_Q([0,T],M)\longrightarrow \R\, , \ \ \ \ \A_k(\gamma) := \int_0^T \Big [L(\gamma(t),\dot \gamma(t))+k\Big ]\, dt\, ,$$
where $H^1_Q([0,T],M)$ is the Hilbert manifold of $H^1$-paths in $M$ defined on $[0,T]$ and connecting $Q_0$ to $Q_1$. In fact, variations of $\gamma$ with fixed $T$ yield that a critical point of $\A_k$ is an Euler-Lagrange orbit connecting $Q_0$ to $Q_1$ and 
satisfying the conormal boundary conditions; variations on $T$ yield then the energy $k$ condition.

The domain of definition of $\A_k$ can be endowed with a structure of Hilbert manifold by identifying  it with the product manifold $\mathcal M_Q = H^1_Q([0,1],M)\times (0,+\infty)$. Here $\gamma\in H^1_Q([0,T],M)$ is identified with 
the pair $(x,T)$, where $x:[0,1]\rightarrow M$ is given by $x(s):=\gamma(T\, s)$. Using this identification we can write  
$$\A_k:\mathcal M_Q \longrightarrow \R\, , \ \ \ \ \A_k(x,T) := \ T\, \int_0^1 \Big [L\Big (x(s),\frac{x'(s)}{T}\Big ) + k \Big ]\, ds\, ,$$

A very careful study of the properties of $\A_k$ will be needed, since $\mathcal M_Q$ is infinite dimensional and non-complete; these turn to be 
influenced by the value $k$ of the energy. In particular they change drastically when crossing a special energy value $c(L;Q_0,Q_1)$, which depends on $L$ and 
on the topology of $Q_0$ and $Q_1$ as in $M$ embedded submanifolds. This is actually no surprise, since also the dynamical and geometric properties of the system depend on the energy; see e.g. \cite{Con06} or \cite{Abb13}.
Therefore, we will have to distinguish between ``supercritical'' and ``subcritical'' energies and we will get different existence and multiplicity results accordingly. 

\vspace{3mm}

The existence of periodic orbits on a given energy level set has already been intensively studied in the last decades and many existence results in this direction have already been obtained. The interested reader may find a beautiful overview 
in \cite{Abb13} (and also references therein); other references will be provided later on. 

We will first study the existence of periodic orbits for the flow of the pair $(L,\sigma)$, with $L$ Tonelli-Lagrangian and $\sigma$ a closed 2-form; namely we prove an almost everywhere existence result of 
periodic orbits, which generalizes the well-known Lusternik and Fet theorem \cite{FL51} about the existence of one contractible closed geodesic on every closed Riemannian manifold $M$ with $\pi_l(M)\neq 0$ for some $l\geq 2$. 

We then focus on oscillating magnetic fields on $\T^2$ and show that almost every sufficiently low energy level set carries infinitely many periodic orbits. The result in \cite{AB15a}, where oscillating magnetic fields on surfaces of 
genus larger than one were considered, is therefore extended here to the case of the 2-Torus. Extending this to $S^2$ represents a challenging open problem. 
Both of the results build on ideas contained in \cite{AMMP14}, where the exact case was treated. 


\section{Orbits ``connecting'' $Q_0$ with $Q_1$.}

Consider a Tonelli Hamiltonian $H:T^*M\rightarrow \R$ (i.e. strictly convex and superlinear in each fiber) and let $Q_0,Q_1\subseteq M$ be two closed submanifolds. The question we are interested in is the following: For which $k\in \R$ does  $H^{-1}(k)$ contain
orbits $x:[0,T]\rightarrow T^*M$ of the Hamiltonian flow defined by $H$ and satisfying
\begin{equation}
x(0) \in N^*Q_0\,, \ \ \ \ x(T)\in N^*Q_1\, ?
\label{conormalintroham}
\end{equation}

Here, for a given submanifold $A\subseteq M$, $N^*A\subseteq T^*M$ denotes the subbundle of the cotangent bundle defined by 
$$N^*A := \ \Big \{(q,p)\in T^*M\ \Big | \ q\in A\, , \ \ker p\supseteq T_qA\Big \}$$
and is called the \textit{conormal bundle of A}. The boundary conditions \eqref{conormalintroham} are then called \textit{conormal boundary conditions}. 
For generalities and properties of conormal bundles we refer to the appendix and to \cite{Dui76}, \cite[page 149]{Hör90} or \cite{AS09}. 

This problem admits an equivalent reformulation in the Lagrangian setting. Let $L:TM\rightarrow \R$ be the Tonelli Lagrangian given as the Fenchel dual of $H$. For which $k\in \R$ does $E^{-1}(k)$ carry Euler-Lagrange orbits 
$\gamma:[0,T]\rightarrow M$ connecting $Q_0$ with $Q_1$ and satisfying the conormal boundary conditions 
\begin{equation}
d_v L(\gamma(0),\dot \gamma(0)) \ \Big |_{T_{\gamma(0)}Q_0}\ = \ d_v L(\gamma(T),\dot \gamma(T)) \ \Big |_{T_{\gamma(T)}Q_1} \ = \ 0 \ ?
\label{conormalintro}
\end{equation}

As already pointed out in the introduction to this chapter, this equivalent reformulation allows to put the problem into a nice functional analytical setting, since Euler-Lagrange orbits with energy $k$ satisfying the conormal boundary conditions 
are in correspondence with the critical points of the functional
$$\A_k:\mathcal M_Q \longrightarrow \R\, , \ \ \ \ \A_k(x,T) := \ T\, \int_0^1 \Big [L\Big (x(s),\frac{x'(s)}{T}\Big ) + k \Big ]\, ds\, ,$$
where $Q:=Q_0\times Q_1$ and $\mathcal M_Q$ is the space of $H^1$-paths connecting $Q_0$ with $Q_1$ with arbitrary interval of definition. It is clear that the properties of $\A_k$ have to depend on the topology of the 
space $\mathcal M_Q$. What it is not so clear at this moment is that the properties of $\A_k$ also depend on the value of the energy $k$ and change drastically when crossing a suitable Ma\~n\'e critical value, 
which depends on $L$ and on the topology of $Q_0$ and $Q_1$ as embedded submanifolds. 

It is worth to observe already at this point that the problem we are interested in need not have solutions. In other words, $\A_k$ need not have critical points in general. 

Consider for instance the geodesic flow of a Riemannian metric $g$ on $M$ and suppose $Q_0=M$. 
This flow can be seen as the Euler-Lagrange flow associated to the kinetic energy. The conormal boundary conditions \eqref{conormalintro} are then given by
$$\dot \gamma(0) \ \perp \ T_{\gamma(0)}Q_0\, ,\ \ \ \ \dot \gamma(T)\ \perp \ T_{\gamma(T)}Q_1\, .$$ 

Being $Q_0=M$, we necessarily have $\dot \gamma(0)=0$. This implies that Euler-Lagrange orbits satisfying the conormal boundary conditions exist only at energy $k=0$.

Consider now the geodesic flow on $(\T^2,g_{\text{flat}})$, where $g_{\text{flat}}$ denotes the flat metric and let $Q_0,Q_1$ be as in the figure below

\begin{center}
\includegraphics[height=30mm]{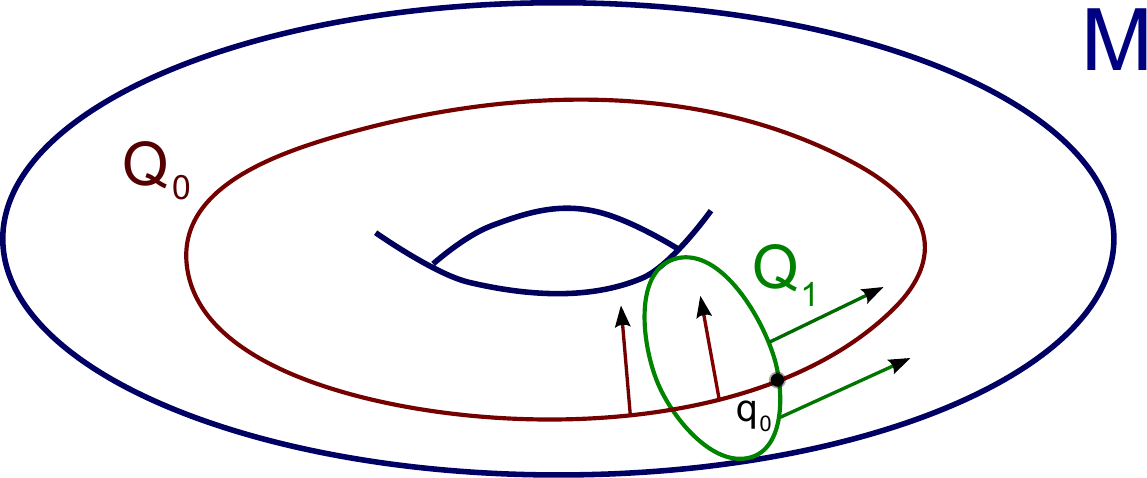}
\end{center}

In this case it is easy to see that the only Euler-Lagrange orbit which satisfies the conormal boundary conditions is the constant orbit through the intersection point $q_0$; in particular
for $k>0$ the energy level $E^{-1}(k)$ carries no Euler-Lagrange orbits satisfying the conormal boundary conditions. 
This counterexample shows that the existence fails also up to small perturbations of $Q_0$ and $Q_1$. 

\begin{center}
\includegraphics[height=30mm]{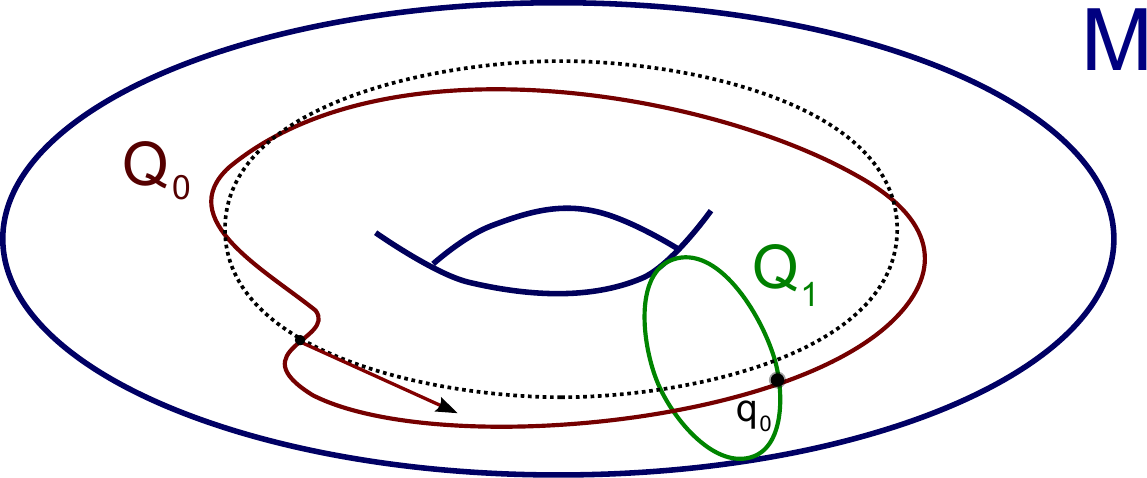}
\end{center}

What goes wrong in these examples is that the space $\mathcal M_Q$ is connected, contractible, contains constant paths and the infimum of $\A_k$ on $\mathcal M_Q$ is zero for every $k$. Therefore, one cannot expect to prove the existence 
of Euler-Lagrange orbits satisfying the conormal boundary conditions by minimizing $\A_k$ on the space $\mathcal M_Q$.  Also, one might not expect to apply minimax 
arguments, being $\mathcal M_Q$ contractible. 

We will show in Theorems \ref{teorema2} and \ref{teorema3} that these are the only cases in which the existence of the desired Euler-Lagrange orbits fails, at least when the energy is sufficiently high in a sense that we now explain. 

First let us recall that, with any cover $N\rightarrow M$ is associated a \textit{Ma\~n\'e critical value} $c(L^N)$. Consider the lift $L^N:TN\rightarrow \R$ of $L$ to the cover $N$ and define 
\begin{equation}
c(L^N) := \ \inf \Big \{k\in \R\ \Big |\ \A_k^N(\gamma)\geq 0\, , \ \forall \ \gamma \text{ loop in} \ N\Big \}\, ,
\label{maneintro}
\end{equation}
where $\A_k^N$ is the action functional associated with $L^N$. When $N=\widetilde M$, $N=\widehat M$ universal cover, resp. Abelian cover of $M$ one denotes the corresponding Ma\~n\'e critical value with $c_u(L)$, resp. $c_0(L)$
and calls it the \textit{Ma\~n\'e critical value of the universal cover}, resp. \textit{of the Abelian cover}. We will get back to the relations of these two energy values with the dynamical and geometric properties of the Euler-Lagrange flow later on.

Let now $H$ be the smallest normal subgroup in $\pi_1(M)$ containing both $\imath_*(\pi_1(Q_0))$ and $\imath_*(\pi_1(Q_1))$, where $\imath :Q_i\rightarrow M$ is the canonical inclusion. Consider the cover 
$$M_1:= \ \bigslant{\widetilde M}{H}$$
and define the Ma\~n\'e critical value 
$$c(L;Q_0,Q_1) := \ c(L_1)\, ,$$
where $L_1:TM_1\rightarrow \R$ is the lift of $L$ to the cover $M_1$. It is possible to show that, if $k\geq c(L;Q_0,Q_1)$, then $\A_k$ is bounded from below on every connected component of $\mathcal M_Q$ and it is unbounded 
from below on every connected component otherwise. 

Moreover, for every $k>c(L;Q_0,Q_1)$, every Palais-Smale sequence for $\A_k$ with times bounded away from zero (see Sections \ref{palaissmalesequences} 
and \ref{manecriticalvalues} for further details) has converging subsequences. These facts will enable us to prove the following

\begin{teo*}
Let $L:TM\rightarrow \R$ be a Tonelli Lagrangian and let $Q_0,Q_1\subseteq M$ be closed submanifolds of $M$. Then the following hold:

\begin{enumerate}
\item For every $k>c(L;Q_0,Q_1)$, each connected component $\mathcal N$ of $\mathcal M_Q$ not containing constant paths carries an Euler-Lagrange orbit with energy $k$ satisfying the conormal boundary conditions, 
which is a global minimizer of $\A_k$ on $\mathcal N$.

\item Let $\mathcal N$ be a component of $\mathcal M_Q$ containing constant paths and define 
\begin{eqnarray*}
k_{\mathcal N}(L) &:=& \inf \ \Big \{k\in \R \ \Big |\ \inf_{\mathcal N} \ \A_k \geq 0\Big \} \ = \\
                             &=& \sup  \Big \{k\in \R \ \Big | \ \inf_{\mathcal N} \ \A_k < 0 \Big \}\  \geq \ c(L;Q_0,Q_1)\, .
\end{eqnarray*}
For all $k\in (c(L;Q_0,Q_1),k_{\mathcal N}(L))$, there exists an Euler-Lagrange orbit with energy $k$ satisfying the conormal boundary conditions, which is a global minimizer of $\A_k$ on $\mathcal N$. 
Furthermore, $\mathcal N$ carries an Euler-Lagrange orbit with energy $k$ satisfying the conormal boundary conditions also for every $k>k_{\mathcal N}(L)$, provided that $\pi_l(\mathcal N)\neq 0$ for some $l\geq 1$.
\end{enumerate}
\end{teo*}

\vspace{3mm}

Existence results for subcritical energies are harder to achieve than the corresponding ones for supercritical energies and the reasons for that are of various nature. 

\vspace{2mm}

First, the action functional $\A_k$ is unbounded from below on each connected component of $\mathcal M_Q$; therefore, we cannot expect to find solutions by minimizing the free-time action functional $\A_k$. Furthermore, when $k$ is 
subcritical, $\A_k$ might have Palais-Smale sequences $(x_h,T_h)$ with $T_h\rightarrow +\infty$. The convergence issues for Palais-Smale sequences for $\A_k$ are ultimately responsible of the fact that one is able to prove existence results
 only on dense subsets of subcritical energies, using for instance an argument due to Struwe \cite{Str90}, called the \textit{Struwe monotonicity argument}, to overcome the lack of the Palais-Smale 
condition for $\A_k$. This method has been already intensively applied to the existence of periodic orbits; see for instance \cite{Con06,Abb13,AMP13,AMMP14,AB14,AB15a}. 
We will see in Section \ref{subcriticalorbits} how to apply this method in our context.  
A possible way to overcome the lack of the Palais-Smale condition for $\A_k$ would be to prove that subcritical energy levels are \textit{stable} \cite[Page 122]{HZ94}, at least for a certain range of energies; this would allow to extend the known 
results about almost every energy to results which hold for all energies (see \cite[Corollary 8.2]{Abb13} for further details). However, only partial answers to this question and in very particular cases are known so far: for instance,
very low energy levels of symplectic magnetic flows on surfaces different from $\T^2$ are of contact type (in particular, stable) and a clear geometric description of their dynamics has been recently given by Benedetti in \cite{Ben14a, Ben14b}.
What makes the stability condition more difficult to study than the contact condition is what actually makes it more flexible and general. In the Tonelli setting, it is not difficult to
characterise contact energy levels in terms of the Lagrangian action: for instance, McDuff's criterion from \cite{McD87} implies that the energy level $E^{-1}(k)$ is of contact type if and only if every invariant measure on it with 
vanishing asymptotic cycle has positive $k$-action. The characterisations of stability coming from Wadsley's and Sullivan's works (see \cite[Theorem 2.1 and 2.2]{CFP10}) are more difficult to use in this context.

\vspace{2mm}

Second, low energy levels of the Hamiltonian $H$ associated to $L$ could be in general disjoint from the conormal bundle of a given submanifold, so one can hope to find solutions only above a certain value of the energy.
We explain this problematic with an example: Suppose $L$ is a magnetic Lagrangian, i.e. of the form
\begin{equation}
L(q,v) \ = \ \frac12 \, \|v\|_q^2 \ + \ \theta_q(v)\, ,
\label{magneticintro}
\end{equation}
where $\|\cdot\|_q$ is the norm induced by a Riemannian metric $g$ on $M$ and $\theta$ is a smooth 1-form on $M$. In this case the energy is given by 
$$E(q,v) \ =\ \frac12 \, \|v\|_q^2\, ,$$ 
so that  Euler-Lagrange orbits are parametrized proportional to arc-length. The conormal boundary conditions (\ref{conormalintro}) can be rewritten as 
\begin{equation}
g_{\gamma(i)}( \dot \gamma(i), \cdot )  \ + \ \theta_{\gamma(i)}(\cdot) \Big |_{T_{\gamma(i)}Q_i} = \ 0\, , \ \ \ \ i = 0,1\, .
\label{conormalmagnetic1}
\end{equation}
For $i=0,1$ denote by $w_i\in T_{\gamma(i)}M$ the unique vector representing $\theta_i$, that is 
$$g_{\gamma(i)}( w_i, \cdot ) \ = \ \theta_i(\cdot)\, ,$$
and assume for sake of simplicity that $g_{\gamma(i)}( \cdot,\cdot)$ is the Euclidean scalar product on $T_{\gamma(i)}M$. Then (\ref{conormalmagnetic1}) is equivalent to 
$$g_{\gamma(i)}(\dot \gamma(i) + w_i, \cdot ) \Big |_{T_{\gamma(i)}Q_i} = \ 0\, , \ \ \ \ i=0,1\, ,$$
which necessarily implies $\|\dot \gamma(i)\|\geq \|\mathcal P_i w_i\|$ for $i=0,1$, where $\mathcal P_i:TM|_{Q_i}\rightarrow TQ_i$ denotes the orthogonal projection. It follows that Euler-Lagrange orbits satisfying the 
conormal boundary conditions (\ref{conormalmagnetic1}) migth exist only for energies 
\begin{equation}
k \ \geq\ \max \ \left \{ \min \, \left \{ \left. \frac 12 \, \|\mathcal P_0 w_{q_0}\|^2 \ \right | \ q_0\in Q_0\right \}\!,\, \min \, \left \{ \left. \frac 12 \, \|\mathcal P_1 w_{q_1}\|^2 \ \right | \ q_1\in Q_1\right \}\right \}
\label{obstruction}
\end{equation}
where $w_{q_i}\in T_{q_i}M$ is the unique tangent vector representing $\theta_{q_i}$. In the Hamiltonian setting, the right-hand side of (\ref{obstruction}) is the lowest energy value for which the energy level set $H^{-1}(k)$ intersects
both the conormal bundles of $Q_0$ and $Q_1$. If it is positive, then there are no Euler-Lagrange orbits satisfying the conormal boundary conditions with energy less than it, even if the submanifolds intersect or if $Q_0=Q_1$.

\vspace{2mm}

Finally, the problem becomes even harder if $Q_0\cap Q_1 = \emptyset$, since it contains as a very special case the famous open problem of finding the energy levels for which any pair of points in $M$ can be joined by an Euler-Lagrange orbit.
This question has an easy answer in the case of mechanical Lagrangians, that is functions of the form
$$L(q,v) \ = \ \frac12 \, \|v\|_q^2\, -\, V(q)\, ,$$
but is made extremely hard by the presence of a magnetic potential $\theta$  (see e.g. \cite[Chapter I.3 and Appendix F]{Gli97}). In this sense a very claryfing example is provided by the magnetic flow of the standard area form 
$\sigma$ on $(S^2,g_{\text{std}})$, even though this is an Euler-Lagrange flow only locally, that is the Hamiltonian flow defined by the kinetic energy $E(q,v)= \|v\|_q^2/2$ and by the twisted symplectic form 
$$\omega_\sigma \ =\ \omega \ + \ \pi^*\sigma\, ,$$ 
where $\omega$ is the pull-back of $dp\wedge dq$ on $TS^2$ via the Riemannian metric. For every $k>0$ the flow on $E^{-1}(k)$ is periodic and projected orbits are circles on $S^2$ which can be 
seen as the intersection of $S^2$ with suitable affine planes in $\R^3$. One can also prove that they converge to great circles for $k\rightarrow +\infty$; it follows that there is no energy level for which the south pole can be joined with the north pole. 

In the case of ``global'' Euler-Lagrange flows on $M$ it has been  proven by Ma\~n\'e in \cite{Man97} that, for every $k>c_0(L)$, every pair of points $q_0,q_1\in M$ can be joined by an Euler-Lagrange orbit with energy $k$. This result has been then strengthen 
by Contreras in \cite{Con06} to every $k>c_u(L)$. We will show in Section \ref{counterexamples} that Contreras' result is sharp exhibiting, for every $\epsilon >0$, examples of magnetic Lagrangians $L_\epsilon$ on compact connected orientable surfaces 
and points $q_0,q_1$ that cannot be joined by Euler-Lagrange orbits with energy less than $c_u(L_\epsilon) -\epsilon$.

\vspace{2mm}

We therefore assume $Q_0\cap Q_1\neq \emptyset$ (for the moment say also connected) and show that in a (possibly empty, but in general not) certain energy range, which depends only on $L$ and on the intersection
$Q_0\cap Q_1$, the free-time action functional $\A_k$ has a mountain-pass geometry on the connected component of $\mathcal M_Q$ containing constant paths. 
Here the two valleys are represented by the set of constant paths and by the set of paths with negative $k$-action. 
We get therefore a minimax class just by considering paths starting from a constant path and going to paths with negative action, and a relative minimax function, which depends monotonically on $k$. 
An analogue of the Struwe monotonicity argument (cf. Lemma \ref{struwe}) will allow us to show the existence of compact Palais-Smale sequences for almost every energy in this energy range. 

In the statement of the following theorem we suppose, for sake of simplicity, that $L$ is of the form \eqref{magneticintro}, though the result holds more generally for every autonomous Tonelli Lagrangian.
This assumption allows at this moment an easier definition of the energy value $k_{Q_0\cap Q_1}^-$; the general one will be given in Chapter \ref{chapter3}.

\begin{teo*}
Let $L:TM\rightarrow \R$ be as in \eqref{magneticintro}. Suppose $Q_0\cap Q_1\neq \emptyset$ connected and let $\mathcal N$ be the connected component of $\mathcal M_Q$ containing the constant paths. Define 
$$k_{Q_0\cap Q_1}^- := \ \max_{q\in Q_0\cap Q_1} \ \frac12 \, \|\theta_q\|^2\ \leq\ c(L;Q_0,Q_1)\, .$$
Then, for almost every $k\in (k_{Q_0\cap Q_1}^-,c(L;Q_0,Q_1))$, there exists an Euler-Lagrange orbit $\gamma \in \mathcal N$ with energy $k$ satisfying the conormal boundary conditions.
\end{teo*}

A very special case of intersecting submanifolds is given by the choice $Q_0=Q_1$, which corresponds to (a particular case of) the \textit{Arnold chord conjecture} about the existence of a Reeb orbit starting and ending at a 
given Legendrian submanifold of a contact manifold, see \cite{Arn86, Moh01}, but in a possibly non-contact situation.
As a trivial corollary of the theorem above we get existence results of \textit{Arnold chords} for subcritical energies (cf. Corollary \ref{arnoldchord}).

In Section \ref{counterexamples} we complement the theorems above with some explicit counterexamples, which show that all the results are optimal.


\section{A generalization of the Lusternik-Fet theorem} 

Let $(M,g)$ be a closed connected Riemannian manifold, $L:TM\rightarrow \R$ be a Tonelli Lagrangian and $\sigma \in \Omega^2(M)$ be a closed 2-form. 
Associated with the pair $(L,\sigma)$ is a flow on $TM$, for which the energy $E$ defined by $L$ is a prime integral; it is defined
by gluing together all the local Euler-Lagrange flows of the Lagrangians $L+\theta_i$, where $\theta_i$ are local primitives of $\sigma$. This flow is conjugated via the Legendre transform 
to the Hamiltonian flow on $T^*M$ defined by $H$, the Fenchel dual of $L$, and by the twisted symplectic form 
$$\omega_\sigma := \ dp\wedge dq + \pi^*\sigma\, .$$

This class of flows contains the class of (possibly non-exact) \textit{magnetic flows} on $TM$; these are given as flow of the pair $(L,\sigma)$ by choosing 
$$L(q,v) \ =\ E_{kin}(q,v) \ =\ \frac12 \, \|v\|_q^2$$
kinetic energy associated with a Riemannian metric on $M$. The reason for this terminology is that this flow can be thought of as modelling the motion of a particle of unit mass and charge under the effect of a magnetic field represented by the 2-form $\sigma$. 
Periodic orbits of the flow of $(E_{kin},\sigma)$ are then called  \textit{closed magnetic geodesics}.

In Chapter \ref{chapter5} we prove a generalization of the celebrated Lusternik and Fet theorem \cite{FL51} about the existence of a contractible closed geodesic on every closed Riemannian manifold $M$ with $\pi_l(M)\neq 0$ for some $l\geq 2$. In the statement of 
the following theorem we set 
$$e_0(L):= \ \max_{q\in M} \ E(q,0)\,.$$
Observe that, in case of magnetic flows, $e_0(L)=0$.

\begin{teo*}[Generalized Lusternik-Fet theorem]
Let $(M,g)$ be a closed connected Riemannian manifold, $L:TM\rightarrow \R$ be a Tonelli Lagrangian and $\sigma$ be a closed 2-form. If $\pi_l(M)\neq0$ for some $l\geq 2$, then for 
almost every $k>e_0(L)$ there exists a contractible periodic orbit for the flow of the pair $(L,\sigma)$ with energy $k$.
\end{teo*}

This result generalizes the corresponding statements in \cite{Con06} (see also \cite[theorem 8.2]{Abb13}) and in \cite{Mer10} (see also the forthcoming corrigendum \cite{Mer15}), where respectively the cases $\sigma$ exact, $\sigma$ weakly-exact are treated. 
This theorem is the outcome of joint work with Gabriele Benedetti and is contained in the preprint \cite{AB14}. There a slightly different proof is given, since the cases $l=2$ and $l\geq 2$ are considered separately; here we use a
construction which allows to treat both cases at once.

A result of this kind for simply connected manifolds and for $l=2$ appears for the first time in \cite[Theorem 7]{Koz85}, where it is claimed to hold for \textit{every} $k>e_0(L)$. The author gives only 
a sketch of the proof and does not take into account some crucial convergence problems, which are today only partially solved and are also ultimately responsible for the fact that with our method we do not 
get a contractible periodic orbit for every energy. In the case of magnetic flows, the existence of a periodic orbit was already proven, for $\sigma \neq 0$,  by Schlenk \cite{Sch06} for almost every $k$ in the energy range $(0,d_1(g,\sigma))$, where
\begin{equation*}
d_1(g,\sigma):=\sup\left\{k>0\ \left|\ \Big \{(q,p)\in T^*M\ \Big | \ \frac12\, \|p\|_q^2\leq k\Big \} \mbox{ stably displaceable}\right\}\right.\,.
\end{equation*}

\noindent It follows from results in \cite{LS94} and in \cite{Pol95} that $d_1(g,\sigma)$ is positive. Finally, this result for $M=S^2$ and $\sigma$ non-exact has concrete applications  to the motion of rigid bodies (see \cite[Theorem 8]{Koz85} and \cite{Nov82}). 

\vspace{3mm}

In general, our methods yields existence results only for almost every $k$ but, when a particular energy level set is \textit{stable} \cite[Page 122]{HZ94} we can upgrade such almost existence results to show that there is a contractible periodic orbit of energy $k$
(we refer to \cite[Corollary 8.2]{Abb13} for the details). This is for instance the case for low energy levels of symplectic magnetic flows on surfaces (i.e. with $\sigma$ a symplectic form). 
 
\vspace{3mm}

We now give an account of the tools we use to prove the aforementioned theorem. We denote by $\mathcal M:= H^1(\T,M)\times (0,+\infty)$ the space of $H^1$-loops in $M$ with arbitrary period and with $\mathcal M_0$ the 
connected component given by contractible loops. 

Notice that a free-period Lagrangian action functional is not available in this generality, since the 2-form $\sigma$ is by assumption only closed. 
However, its differential $\eta_k$ is still well-defined and its zeros are in one to one correspondence with the periodic orbits of the flow defined by $(L,\sigma)$ contained in $E^{-1}(k)$. We call $\eta_k\in \Omega^1(\mathcal M)$ the \textit{action 1-form}; it is given by
$$\eta_k(x,T) := \ d\A_k^L(x,T)  \ + \ \int_0^1 \sigma_{x(s)}(x'(s),\cdot ) \, ds\, ,$$
where $\A_k^L$ is the free-period action functional associated with $L$. The action 1-form turns out to be locally Lipschitz continuous and (in a suitable sense) closed. Moreover, it satisfies a crucial compactness property 
for critical sequences (namely, sequences $(x_h,T_h)$ such that $\|\eta_k(x_h,T_h)\|\rightarrow 0$). More precisely, every critical sequence with periods bounded and bounded away from zero admits a converging subsequence. 

The assumption that $\pi_l(M)\neq 0$ for some $l\geq 2$ will be used to define a suitable minimax class $\mathfrak U$ of maps 
$$(B^{l-1},S^{l-2})\ \longrightarrow \ (\mathcal M_0,M_0)\, ,$$ 
where $M_0$ is the submanifold of $\mathcal M_0$ of constant loops, and an associated minimax function $k\longmapsto c^{\mathfrak u}(k)$. The monotonicity of $c^{\mathfrak u}$  
allows to prove the existence of critical sequences for $\eta_k$ with periods bounded and bounded away from zero for almost every $k>e_0(L)$ by generalizing the \textit{Struwe monotonicity argument} to this setting.


\section{Oscillating magnetic fields on $\T^2$}

In this section we restrict our attention to the class of magnetic flows on $T\T^2$ defined by oscillating forms. Recall that a closed 2-form $\sigma$ is said to be \textit{oscillating} if its density\footnote{The density of $\sigma$ with respect to $\mu_g$ is the (unique) function $f:M\rightarrow \R$ such that $\sigma = f\, \mu_g$.} with respect to the area form takes both positive and negative values. 
Notice that oscillating forms are the natural generalization of exact forms, since we can think of exact forms as ``balanced'' oscillating forms, being their integral over $M$ zero. 

The aim of chapter \ref{chapter6} will be to generalize the main theorem of \cite{AMMP14} (for $M=\T^2$) to the non-exact case, thus proving the following

\begin{teo*}
Let $\sigma$ be a non-exact oscillating 2-form on $(\T^2,g)$. Then there exists a constant $\tau_+(g,\sigma)>0$ such that for almost every $k\in (0,\tau_+(g,\sigma))$ 
the energy level $E^{-1}(k)$ carries infinitely many geometrically distinct closed magnetic geodesics.
\end{teo*}

By ``geometrically distinct'' we mean that the closed magnetic geodesics are not iterates of each other. This theorem is the result of joint work with Gabriele Benedetti 
and complements our previous result in \cite{AB15a}, where we consider the case of surfaces with genus larger than one. 
The high genus case is actually much easier than the case $M=\T^2$, since the action 1-form $\eta_k$  is exact on the whole $\mathcal M$ and a primitive $S_k$ can be explicitly written down
(cf. \cite{Mer10}); the proof follows then roughly from the one in \cite{AMMP14} replacing the free-period Lagrangian action functional by $S_k$.

The case $M=\T^2$ is harder and requires methods similar to the ones used in the proof of the generalized Lusternik-Fet theorem. 
The proof will therefore consist  in showing the existence of infinitely many zeros of  $\eta_k$ via a minimax method. 

\vspace{3mm}

Now we briefly explain the main ideas involved in the proof of the theorem above.
Since the action 1-form $\eta_k$ is locally exact (in particular near a critical point) and local primitives of $\eta_k$ have the same structure as a Lagrangian action functional 
(with a primitive $\theta$ of $\sigma$ not defined on the whole $\T^2$), the local theory is the same as in the exact case: iterates of (strict) local minimizers are still (strict) local minimizers (cf. Proposition \ref{persistenceoflocalminimizers}) and the Morse 
index of the critical points satisfies the same iteration properties as described in \cite[Section 1]{AMP13} and in \cite{AMMP14}. In particular, as shown in \cite{AMMP14} for the exact case, a sufficiently high iterate of a periodic orbit cannot be a 
mountain pass critical point (see Proposition \ref{iterationofmountainpasses} for further details).

Also, it follows from results by Taimanov \cite{Tai92b,Tai92a,Tai93} and indipendently by Contreras, Macarini and Paternain \cite{CMP04} that there is $\tau_+(g,\sigma)>0$ such that for all $k\in (0,\tau_+(g,\sigma))$ there exists a closed magnetic geodesic 
$\alpha_k$ which is a local minimizer of the action. Now one has two cases: either $\alpha_k$ is contractible or it is not contractible.
If $\alpha_k$ is contractible, then one can run the same proof as in \cite{AB15a} and the theorem follows. 

Therefore, we may assume $\alpha_k$ to be non contractible. Being $\eta_k$ exact only on $\mathcal M_0$, for every other connected 
component $\mathcal M'$ of $\mathcal M$ there exists a generator of $\pi_1(\mathcal M')$, say $\beta$, on which $\eta_k$ is non-zero.
One now gets minimax classes by considering, for every $n\in \N$, the class of loops in $\mathcal M$ based at $\alpha_k^n$ which are homotopic to $\beta$. 

However, this natural choice might yield non-monotone minimax functions, since the Taimanov's local minimizer might not depend continuously on $k$. We therefore modify the minimax classes in order to achieve the desired monotonicity, 
which will be crucial to prove the existence of infinitely many zeros for $\eta_k$ for almost every $k\in (0,\tau_+(g,\sigma))$ by generalizing the Struwe monotonicity argument to this setting. 

Excluding that these zeros are iterates of only finitely many zeros  with the argument given by Proposition \ref{iterationofmountainpasses} will yield infinitely many geometrically distinct closed magnetic geodesics for almost every $k\in (0,\tau_+(g,\sigma))$.

\vspace{3mm}

The same proof would a priori run also for $M=S^2$; the problem is that in this case one has no tools to show that the infinitely many zeros of the action are not iterates of each other. More precisely, one can exclude (with the same argument 
used for the case $M\neq S^2$) that the zeros are ``large'' iterates of each other, but one cannot exclude that they are ``low'' iterates of each other (or even that they are all equal). 

This difficulty can be easily overcome in case $M=\T^2$, since the action 1-form $\eta_k$ is exact on $\mathcal M_0$ (cf. \cite{Mer10}), the connected component of $\mathcal M$ given by the contractible loops. Therefore, if the mountain-pass 
critical points are contractible, then one can use the action to show that they cannot be ``low'' iterates of each other proving that the action tends to $-\infty$; to do this one uses the so-called Bangert's technique \cite{Ban80} of 
\textit{pulling one loop at a time}. In case of non-contractible mountain-passes, this can be excluded by a simple topological argument.

In the case $M=S^2$, combining Taimanov's result \cite{Tai92a} with Theorem \ref{generalizedlyusternikfettheorem} of Chapter \ref{chapter5}, we get the following

\begin{prop*}
Consider a non-exact oscillating form $\sigma$ on $(S^2,g)$. Then there exists a constant $\tau_+(g,\sigma)>0$ such that for almost every $k\in (0,\tau_+(g,\sigma))$ 
the energy level $E^{-1}(k)$ carries at least two geometrically distinct closed magnetic geodesics.
\end{prop*}

This discrepance between $S^2$ and genus $\geq 1$ is actually not a huge surpise, since also in the case in which $\sigma$ is a symplectic form the strongest known result is that 
\textit{every} sufficiently low energy level carries either two or infinitely many closed magnetic geodesics \cite{Ben14b}. In this setting, Benedetti recently showed an example of ``low'' energy level with exactly two closed magnetic geodesics.


\chapter{Preliminaries}
\label{Preliminaries}

In this chapter we recall the basic tools that will be needed in the rest of the thesis. Throughout the whole work we will assume $(M,g)$ to be a closed connected Riemannian manifold and $Q\subseteq M\times M$ to be a connected 
boundaryless submanifold of $M\times M$. We will be mainly interested in the cases $Q=\Delta$ diagonal in $M\times M$ and $Q=Q_0\times Q_1$ product of two closed submanifolds of $M$.

\vspace{2mm}

In Section \ref{lagrangianandhamiltoniandynamics} we introduce the so-called \textit{Lagrangian} and \textit{Hamiltonian formalisms}: we define Tonelli Lagrangians, the Euler-Lagrange 
equation, the Euler-Lagrange flow, the energy function and the Hamiltonian associated to a Tonelli Lagrangian, the Hamiltonian flow and briefly discuss their properties and relations.

In Section \ref{ahilbertmanifoldofloops} we define the Hilbert manifold 
$$H^1_Q([0,1],M)$$
of paths  ``starting at'' and ``ending in'' $Q$ and study its topology, with particular attention to its connected components, in the two aforementioned cases. 
In the first one we readily see that the connected components of 
$$H^1_\Delta([0,1],M)\ =\ H^1(\T,M)$$
correspond to the conjugacy classes in $\pi_1(M)$, whilst in the latter one we show that there exists an equivalence relation $\sim_{Q_0,Q_1}$ 
on $\pi_1(M)$, which depends only on the topology of $Q_0$ and $Q_1$ as embedded submanifolds of $M$, such that the connected components of $H^1_Q([0,1],M)$ are in one to one correspondence with the set of equivalence classes
$$\bigslant{\pi_1(M)}{\sim_{Q_0,Q_1}}\, .$$

In Sections \ref{thelagrangianactionfunctional} and \ref{conormalboundaryconditions} we move to the study of the Lagrangian action functional. We first define it and discuss its regularity properties: 
we show that $\A_L$ is continuously differentiable with locally Lipschitz and Gateaux-differentiable differential. We then show that the critical points of 
$\A_L$ restricted to $H^1_Q([0,1],M)$ correspond to the Euler-Lagrange orbits that satisfy the conormal boundary conditions (\ref{lagrangianformulation1}). 

Finally, in Section \ref{theminimaxprinciple}, we recall the celebrated \textit{minimax theorem} \ref{minimaxtheorem}, which provides a very powerful tool to detect critical points (of functionals on Hilbert manifolds) that are not necessarily global or local minimizers. 


\section{Lagrangian and Hamiltonian dynamics.}
\label{lagrangianandhamiltoniandynamics}

\begin{defn}
A (\textit{autonomous}) $\mathsf{Tonelli\ Lagrangian}$ on $M$ is a smooth function $L:TM\rightarrow \R$ satisfying the following conditions:
\begin{enumerate}
\item $L$ is fiberwise $C^2$-$\mathsf{strictly\ convex}$, i.e. 
$$d_{vv} L(q,v)\ >\ 0 \, , \ \ \ \ \forall \ (q,v)\in TM\, ,$$
where $d_{vv}L$ denotes the fiberwise second differential of $L$.
\item $L$ has $\mathsf{superlinear\ growth}$ on each fiber, meaning
$$\lim_{\|v\|_q\rightarrow +\infty} \frac{L(q,v)}{\|v\|_q}\ =\ +\infty\, .$$  
\end{enumerate}
\label{definitiontonellilagrangian}
\end{defn}

As usual $\pi:TM\rightarrow M$ denotes the tangent bundle of $M$. From the superlinearity condition it readily follows that Tonelli Lagrangians are bounded from below. We will use equivalently the notations
$$\frac{\partial L}{\partial q}\ =\ L_q\, ,\ \ \ \ \frac{\partial L}{\partial v}\ =\ L_v$$ 
for the partial derivatives of $L$ with respect to $q$ and $v$ in local coordinates. The \textit{Euler-Lagrange equation} associated to $L$ is, in local coordinates, given by
$$\dt \frac{\partial L}{\partial v}(q,\dot{q})\ =\ \frac{\partial L}{\partial q}(q,\dot{q})\, .$$

The convexity hypothesis ($L_{vv}$ invertible) implies that the Euler-Lagrange equation can also be seen as a first order differential equation on $TM$
$$\left\{ \begin{array}{l} \dot{q}=v; \\ \\ \dot{v}= (L_{vv})^{-1}(L_q-L_{vq}\cdot v); \end{array}\right.$$

Hence, the convexity hypothesis allows to define a vector field $X_L$ on $TM$, called the \textit{Euler-Lagrange vector field}, such that the solutions of 
$$\dot{u}(t)=X_L(u(t))\, , \ \ \ \ \ \ u(t)=(q(t),\dot{q}(t))$$ 
are precisely the curves satisfying the Euler-Lagrange equation. The flow of $X_L$ is called the \textit{Euler-Lagrange flow}. To any Tonelli Lagrangian $L$ we can associate an energy function defined by 
\begin{equation}
E:TM\longrightarrow \R\, , \ \ \ \ E(q,v):=\ \frac{\partial L}{\partial v}(q,v)\cdot v - L(q,v) \, ,
\label{energyfunction}
\end{equation}
which is an integral of the motion, i.e. an invariant function for the Euler-Lagrange flow. Indeed, if $q(t)$ satisfies the Euler-Lagrange equation, then 
\begin{equation*}
\dt \ E(q(t),\dot{q}(t))\ = \ \left [\dt \left (\frac{\partial L}{\partial v}\right ) - \frac{\partial L}{\partial q}\right ] (q(t),\dot{q}(t))\ =\ 0\, .
\end{equation*}

Therefore, the energy level sets of $E$ are invariant under the Euler-Lagrange flow. Furthermore, the function $E$ satisfies the following properties:
\begin{itemize}
\item $E(q,v)$ is fiberwise $C^2$-strictly convex and superlinear; in particular, since $M$ is compact, the energy level sets $E^{-1}(k)$ are compact.
\item For any $q\in M$, the restriction of $E$ to $T_qM$ achieves its minimum at $v=0$.
\item The point $(\bar{q},0)$ is singular for the Euler-Lagrange flow  if and only if $(\bar{q},0)$ is a critical point of $E$.
\end{itemize}

Since the energy level sets are compact, the Euler-Lagrange flow is complete, meaning that every maximal integral curve for $X_L$  has $\R$ as domain of definition.

The main example of Tonelli Lagrangians is given by the so called \textit{electro-magnetic Lagrangians}, that is functions of the form
\begin{equation}
L(q,v)\ =\ \frac{1}{2}\|v\|_q^2 + \theta_q(v) - V(q)\, ,
\label{magneticlagrangian}
\end{equation}
with $\theta$ smooth $1$-form on $M$ and $V$ smooth function. The reason of this name is that it models the motion of a unity mass and charge particle under the effect of the magnetic field $\sigma=d\theta$ and the potential energy $V(q)$. 
When $V=0$, the Euler-Lagrange flow associated to the Lagrangian $L$ in (\ref{magneticlagrangian}) is called the \textit{magnetic flow} of the pair $(g,\sigma)$. This modern dynamical approach to magnetic flows was first introduced by Arnold (cf. \cite{Arn61}); 
magnetic flows present many interesting phenomena that have been intensively studied by various mathematicians, as for instance Novikov and Taimanov (cf. \cite{Nov82,Tai83,Tai92a,Tai92b,Tai93}), and are still today object of ongoing research 
(see \cite{CMP04,Mer10,Schn11,Schn12a,Schn12b,AMP13,AMMP14,GGM14,AB14,AB15a} for recent developments in this context).

It is easy to see that for electro-magnetic Lagrangians the energy is given by
\begin{equation}
E(q,v) \ = \ \frac12 \|v\|_q^2 + V(q)\, .
\label{magneticenergy}
\end{equation}

Given a Tonelli Lagrangian $L:TM\rightarrow \R$, we define the corresponding Hamiltonian $H:T^*M\rightarrow \R$ as the Fenchel transform of $L$, that is
\begin{equation}
H(q,p):=\ \max_{v\in T_qM} \ \Big [ \langle p,v\rangle_q- \ L(q,v) \Big ]
\label{definitiontonellihamiltonian}
\end{equation}
where $\langle\cdot,\cdot\rangle_q$ denotes the duality pairing between the tangent and the cotangent space. One can prove (cf. \cite[Section 31]{Roc70}) that the Hamiltonian defined above  is 
a smooth function, finite everywhere, superlinear and $C^2$-strictly convex in each fiber; we call such a function a \textit{Tonelli Hamiltonian}.
Recall that the cotangent bundle $T^*M$ is naturally equipped with a structure of symplectic manifold given by the canonical symplectic form $\omega:=d\lambda$, where $\lambda$ is the \textit{Liouville form} on $T^*M$ defined by
$$\lambda_p(\zeta)=p\, \big [d\pi^* (\zeta,p)[\zeta]\big ]\,    \ \ \ \ \forall \ \zeta\in T_p(T^*M)\, .$$

Here $\pi^*:T^*M\rightarrow M$ is the canonical projection; a local chart $q=(q_1,...,q_n)$ of $M$ induces a local chart 
$(q,p)=(q_1,..,q_n,p_1,..,p_n)$  of $T^*M$ writing $p\in T^*M$ as $p=\sum_i p_idq_i$. In these coordinates the forms $\lambda$ and $\omega$ are given by 
$$\lambda \ = \  p\, dq \ =\ \sum_{i=1}^n p_i\, dq_i\, ,  \ \ \ \ \omega \ = \ dp\wedge dq \ =\  \sum_{i=1}^n dp_i \wedge dq_i\, .$$
The \textit{Hamiltonian vector field} $X_H$ associated to $H$ is defined by
\begin{equation}
\imath_\omega X_H=\omega(X_H,\cdot )=-dH\, ,
\label{eq1.2.1}
\end{equation}
where as usual $\imath_\omega X$ denotes the contraction of the form $\omega$ along the vector field $X$. In local charts, $X_H$ defines the system of differential equations
\begin{equation*}
\left \{\begin{array}{l} \dot{q}=\ \  H_p; \\ \\ \dot{p}=- H_q; \end{array}\right .
\end{equation*}
where $H_p$ and $H_q$ are the partial derivatives of $H$ with respect to $q$ and $p$. The Hamiltonian flow (that is the flow of the vector field $X_H$) preserves $H$, since
$$\dt H \ = \  H_q \, \dot{q} + H_p \, \dot{p}\ =\ 0\, . $$

One can also show that the Hamiltonian flow preserves the symplectic form $\omega$ and it is therefore a flow of symplectomorphisms (see for instance \cite{HZ94}). It is clear from the definition of $H$  that the Fenchel inequality
$$\langle p,v\rangle_q \ \leq \ L(q,v) + H(q,p) \, , \ \ \ \ \forall\, (q,v)\in TM, \ \ \forall \, (q,p)\in T^*M$$
holds. This inequality plays a crucial role in the study of Lagrangian and Hamiltonian dynamics; in particular, equality holds if and only if $p=L_v(q,v)$. Therefore one can define the \textit{Legendre transform} as
$$\mathcal L:TM \longrightarrow T^*M, \ \ \ (q,v) \longmapsto \big ( q, L_v(q,v)\big )$$
which is a diffeomorphism between the tangent and the cotangent bundle (cf. \cite{Roc70}). A simple computation using the Legendre transform shows that the Hamiltonian associated to the electro-magnetic Lagrangian in (\ref{magneticlagrangian}) is given by 
\begin{equation}
H(q,p)\ =\ \frac{1}{2}\, \|p-\theta_q\|^2 \ + \ V(q)\, . 
\label{magnetichamiltonian}
\end{equation}
The importance of the Legendre transform is explained by the following

\begin{lemma}
The Euler-Lagrange flow on $TM$ associated to $L$ and the Hamiltonian flow on $T^*M$ associated to $H$ are congiugated via the Legendre transform. 
\end{lemma}

\noindent By the very definition of the Legendre transform $\mathcal L$ and (\ref{definitiontonellihamiltonian}) we also have 
$$H\circ \mathcal L(q,v) \ = \ \langle L_v(q,v),v\rangle _q -\ L(q,v) \ =\ E(q,v)\, .$$

Therefore one can equivalently study the Euler-Lagrange flow or the Hamiltonian flow, obtaining in both cases information on the dynamics of the system. 

Each of these equivalent approaches will provide different tools and advantages, which may be very useful to understand the dynamical properties of the system. 

For instance, the tangent space is the natural setting for the classical calculus of variations (see \cite{Con06} and \cite{Abb13}) and for Mather's and Ma$\tilde{\text{n}}$\'e's theories
(see \cite{Mat91,Mat04,FM94,Man92,Man96,Man97,Car95} and also the book \cite{CI99} and the beautiful overview paper \cite{Sor10}). 

On the other hand, the cotangent bundle is equipped with a canonical symplectic structure which allows one to use several symplectic topological tools, coming from the study of Lagrangian graphs, Hofer's geometry, Floer homology, etc. 
A particularly fruitful approach is the so called Hamilton-Jacobi method or \textit{Weak KAM theory}, which is concerned with the study of existence of (sub)solutions of the Hamilton-Jacobi equation (see for instance 
\cite{Fat97a,Fat97b,Fat98,Fat09} and \cite[chapter 6]{Sor10}) and represents, in a certain sense, the functional analytical counterpart of the aforementioned variational approach. 

\vspace{3mm}

In this thesis we will be interested into proving the existence of periodic orbits of the Euler-Lagrange flow and, more generally, of orbits connecting two given submanifolds of $M$ and satisfying suitable boundary conditions, called \textit{conormal boundary conditions},
on a given energy level $E^{-1}(k)$. The first step in this direction is to introduce the tools we need, namely the Hilbert manifold of paths $H^1_Q([0,1],M)$, the Lagrangian action functional $\A_L$ and the minimax principle.


\section{A Hilbert manifold of paths.}
\label{ahilbertmanifoldofloops}

In this section we introduce the Hilbert manifold of paths we will need in the following chapters and study its properties, with particular attention to its connected components. 
Let us denote by $H^1([0,1],M)$ the set of absolutely continuous curves $x:[0,1]\rightarrow M$ with square-integrable weak derivative 
$$H^1([0,1],M) := \ \left \{ x:[0,1]\rightarrow M \, \Big |\, x \ \text{abs. continuous}\ ,\ \int_0^1 \|x'(s)\|^2 ds \ < \ +\infty\right \}\, .$$

It is a well-known fact that this set has a natural structure of Hilbert manifold modelled over the Hilbert space $H^1([0,1],\R^n)$; for further reference it is useful to 
recall here the construction of this structure (see [AS09] for the details). Let 
$$\phi : [0,1]\times U\longrightarrow M$$
be a time-depending local coordinates system, that is a smooth function $\phi$ defined on $[0,1]\times U$, where $U\subseteq \R^n$ is an open subset, such that for any $t\in [0,1]$ 
the map $\phi(t,\cdot)$ is a diffeomorphism on the open subset $\phi(\{t\}\times U)$ of $M$. It is often useful to assume the element $(U,\phi)$ to be \textit{bi-bounded}, meaning that $U$ is bounded and all the derivatives of 
$\phi$ and of the map $(t,q)\longmapsto \phi(t,\cdot)^{-1}(q)$ are bounded. 

Observe that the continuity of the inclusion $H^1\hookrightarrow C^0$ implies that the set $H^1([0,1],U)$ (that is the set of curves $x\in H^1([0,1],\R^n)$ whose image is contained in
$U$) is open in $H^1([0,1],\R^n)$. Hereafter we assume all local coordinate systems $(U,\phi)$ to be bi-bounded and time-depending; any such $(U,\phi)$ induces an injective map 
$$\phi_* : H^1([0,1],U)\longrightarrow H^1([0,1],M)\, ,\ \ \ \ \phi_*(x):= \phi(\cdot,x(\cdot))\, .$$ 

The Hilbert manifold structure on $H^1([0,1],M)$ is defined by declaring the family of maps $\phi_*$ to be an atlas.
The tangent space of $H^1([0,1],M)$ at $x$ is naturally identified with the space of $H^1$-sections of $x^*(TM)$; therefore, we can define a Riemannian metric on $H^1([0,1],M)$ by setting 
\begin{equation}
(g_{H^1})_x(\zeta,\eta) := \int_0^1 \Big [ g(\zeta,\eta)+g(\nabla_t \zeta,\nabla_t \eta) \Big ]\, dt
\label{riemannianmetric}
\end{equation}
for all $\zeta,\eta\in T_xH^1([0,1],M)$, where $\nabla_t$ denotes the Levi-Civita covariant derivative along $x$. The distance induced by this Riemannian metric is compatible with the topology of $H^1([0,1],M)$ and 
 $H^1([0,1],M)$ is complete with respect to it. 

If $Q\subseteq M\times M$ is a smooth submanifold then the set 
$$H^1_Q([0,1],M) := \Big \{x\in H^1([0,1],M)\ \Big |\  (x(0),x(1))\in Q\Big \}$$
is a smooth submanifold, being the inverse image of $Q$ by the smooth submersion
$$H^1([0,1],M)\longrightarrow M\times M\, ,\ \ \ \ x\longmapsto (x(0),x(1))\, .$$ 

Actually, a smooth atlas for $H^1_Q([0,1],M)$ can be build by fixing a linear subspace $W$ of $\R^n\times \R^n$ with $\dim W=\dim Q$, by considering time-depending local coordinate systems $(U,\phi)$ such that $0\in U$ and 
$(\phi(0,q),\phi(1,q))\in Q$ for every $q\in U\cap W$, and by restricting the map $\phi_*$ to the intersection of the open set $H^1([0,1],U)$ with the closed linear subspace 
$$H^1_W([0,1],\R^n) = \Big \{x\in H^1([0,1],\R^n)\ \Big |\ (x(0),x(1))\in W\Big \}\, .$$

In the present work we are interested mainly in the particular case $Q=Q_0\times Q_1$, with $Q_0,Q_1\subseteq M$ closed submanifolds; in this case the Hilbert manifold $H^1_Q([0,1],M)$ is nothing else but the space 
of $H^1$-paths in $M$ connecting $Q_0$ to $Q_1$. Later on we will also deal with $Q=\Delta$ diagonal in $M\times M$; in this case we clearly have that
$$H^1_\Delta ([0,1],M) \ = \ H^1(\T,M)$$ 
is the space of 1-periodic $H^1$-loops on $M$. For our purposes we need to know more about the topology of the Hilbert manifold $H^1_Q([0,1],M)$, in particular about its connected components. It is a well known fact that the inclusions 
$$C^\infty_Q ([0,1],M) \ \hookrightarrow \ H^1_Q([0,1],M) \ \hookrightarrow \ C^0_Q([0,1],M)$$ 
are dense homotopy equivalences (cf. \cite{Abb13}). Here the indices $Q$ mean that we are only considering paths ``starting at'' and ``ending in'' $Q$. 
This implies that the connected components of $H^1_\Delta([0,1],M)= H^1(\T,M)$ are in one to one correspondence with the conjugacy classes in $\pi_1(M)$. 

Let us now consider two closed submanifolds $Q_0,Q_1\subseteq M$. Without loss of generality we may suppose $Q_0,Q_1$ connected, as otherwise we just repeat the complete procedure componentwise.
To underline the particular nature of the submanifolds $Q\subseteq M\times M$ we are looking at, let us denote $C^0_Q([0,1],M)$ with

\begin{equation}
\Omega_{Q_0,Q_1}(M):= \ \Big \{x\in C^0([0,1],M) \ \Big |\ x(0)\in Q_0\, ,\ x(1)\in Q_1\Big \}\, .
\label{Q0Q1}
\end{equation}

For any pair of points $(q_0,q_1)\in Q_0\times Q_1$ we also define $\Omega_{q_0,q_1}(M)$ to be the subspace of $\Omega_{Q_0,Q_1}(M)$ 
given by paths in $M$ which start at $q_0$ and end in $q_1$.

\begin{center}
\includegraphics[height=35mm]{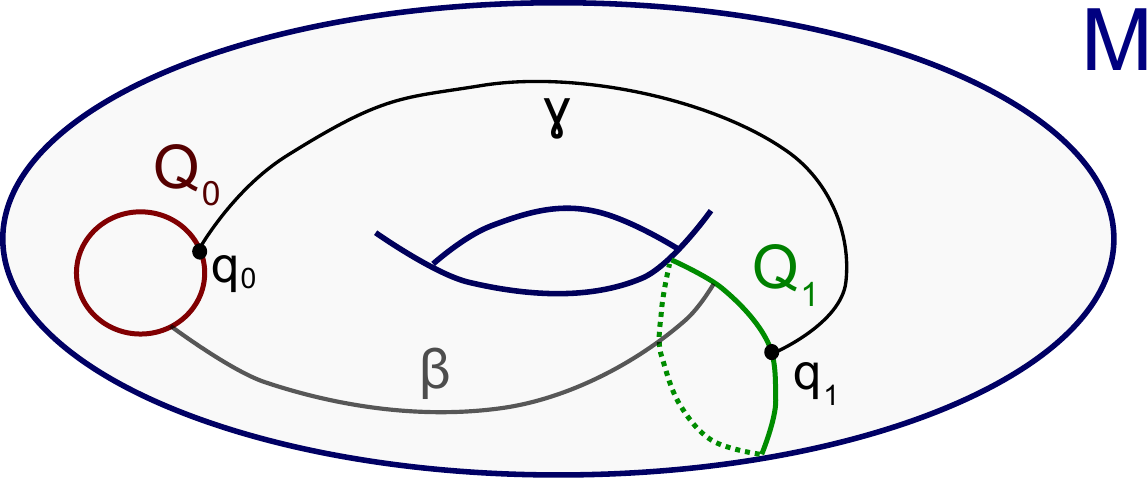}
\end{center}

The space $\Omega_{q_0,q_1}(M)$ is homotopy equivalent to the space of continuous loops based at $q_0$; in fact, given any path $\gamma:[0,1]\rightarrow M$ connecting $q_0$ to $q_1$, the map 
$$\Omega_{q_0}(M)\longrightarrow \Omega_{q_0,q_1}(M)\, ,\ \ \ \ \alpha \longmapsto \gamma\# \alpha$$ 
is a homotopy equivalence with homotopy inverse given by 
$$\Omega_{q_0,q_1}(M) \longrightarrow \Omega_{q_0}(M)\, ,\ \ \ \ \beta \longmapsto \gamma^{-1}\# \beta\, .$$
In particular we have
\begin{equation}
\pi_0(\Omega_{q_0,q_1}(M))\ \cong \ \pi_0(\Omega_{q_0}(M))\ \cong \ \pi_1(M,q_0)\, .
\label{legameconp1}
\end{equation}

Now we want to study the homotopy type of $\Omega_{Q_0,Q_1}(M)$. Hereafter we suppose $M$ fixed and write simply $\Omega_{q_0,q_1}$, $\Omega_{Q_0,Q_1}$ instead of $\Omega_{q_0,q_1}(M)$, $\Omega_{Q_0,Q_1}(M)$ respectively.

To this purpose we define the following equivalence relation on $\Omega_{Q_0,Q_1}$:  
$$\sigma \sim_{Q_0,Q_1} \sigma' \ \ \Longleftrightarrow \ \ \exists \ \alpha \in \imath_*(\pi_1(Q_0,q_0))\, ,\ \beta \in \imath_*(\pi_1(Q_1,q_1)) \ \ \text{s.t.} \ \ \sigma' \ \sim \ \beta \# \sigma \# \alpha\, ,$$
where $\imath:Q_0\rightarrow M$, $\imath:Q_1\rightarrow M$ denote the canonical inclusions.

\begin{lemma}
Let $M$ be a closed connected manifold, $Q_0,Q_1\subseteq M$ be two closed connected submanifolds, $q_0\in Q_0$ and $q_1\in Q_1$. Then
\begin{equation}
\pi_0(\Omega_{Q_0,Q_1}) \ \cong \ \bigslant{\pi_0(\Omega_{q_0,q_1})}{\sim_{Q_0,Q_1}}\, .
\label{primaformulazione}
\end{equation}
\label{componenticonnesseomegaqoq1}
\end{lemma}

\vspace{-6mm}

\begin{proof}
Fix any path  $\delta\in \Omega_{q_0,q_1}$. Associated to the pair $(\Omega_{Q_0,Q_1},\Omega_{q_0,q_1})$  we have an exact sequence in relative homotopy (we refer to the appendix \ref{homotopytheory} for a quick reminder on the general facts 
about homotopy theory needed here)
$$
\begin{array}{c} ... \rightarrow \pi_n(\Omega_{q_0,q_1},\delta) \stackrel{i_*}{\longrightarrow} \pi_n(\Omega_{Q_0,Q_1},\delta) \stackrel{j_*}{\longrightarrow} \pi_n(\Omega_{Q_0,Q_1},\Omega_{q_0,q_1},\delta) \stackrel{\partial}{\longrightarrow} 
\pi_{n-1}(\Omega_{q_0,q_1},\delta ) \rightarrow ...\\ \\
... \rightarrow \pi_1(\Omega_{Q_0,Q_1},\Omega_{q_0,q_1},\delta) \stackrel{\partial}{\longrightarrow} \pi_0(\Omega_{q_0,q_1}) \stackrel{i_*}{\longrightarrow} \pi_0(\Omega_{Q_0,Q_1})\ \ \ \ \ \ \ 
\end{array}$$
where $i_*,j_*$ are the maps induced respectively by the natural inclusions 
$$i: (\Omega_{q_0,q_1},\delta) \longrightarrow (\Omega_{Q_0,Q_1},\delta)\, ,\ \ \ j:(\Omega_{Q_0,Q_1},\{\delta\},\delta) \longrightarrow (\Omega_{Q_0,Q_1},\Omega_{q_0,q_1},\delta)$$
while $\partial$ comes from restricting maps 
$$(I^n,\partial I^n,J^{n-1})\longrightarrow (\Omega_{Q_0,Q_1},\Omega_{q_0,q_1},\delta)$$
to $I^{n-1}$. Here $I^n=[0,1]^n$ denotes the $n$-dimensional cube, $\partial I^n$ its boundary, $I^{n-1}$ the face of $I^n$ with last coordinate equal to zero and $J^{n-1}$ is as in (\ref{jn-1}) the closure of the remaining 
faces of $I^n$. Moreover, the function 
$$p:\Omega_{Q_0,Q_1} \longrightarrow Q_0\times Q_1\, ,\ \ \ \ x \longmapsto (x(0),x(1))\, ,$$
which maps any path $\gamma\in \Omega_{Q_0,Q_1}$ into the pair $(x(0),x(1))$ given by its starting and ending points, is a fibration with fiber $p^{-1}\big ( (q_0,q_1)\big ) = \Omega_{q_0,q_1}$
$$
\xymatrix{
\Omega_{q_0,q_1} \ \ar@{^{(}->}[r] & \ \Omega_{Q_0,Q_1}\ar[d]^p & \  \\ 
& Q_0\times Q_1 & \ } 
$$

Therefore, fixed a base point $(q_0,q_1)\in Q_0\times Q_1$ and a path $\delta\in \Omega_{q_0,q_1}=p^{-1}((q_0,q_1))$  in the corresponding fiber, we have an exact sequence 
$$\begin{array}{c}
... \rightarrow \pi_{n}(\Omega_{q_0,q_1},\delta)\stackrel{i_*}{\longrightarrow} \pi_{n}(\Omega_{Q_0,Q_1},\delta) \stackrel{\tilde{\jmath}_*}{\longrightarrow} \pi_{n}(Q_0\times Q_1,(q_0,q_1))\stackrel{\tilde{\partial}}{\longrightarrow} \pi_{n-1}(\Omega_{q_0,q_1},\delta) 
\rightarrow ...  \\ \\
... \rightarrow \pi_1(Q_0\times Q_1,(q_0,q_1)) \stackrel{\tilde{\partial}}{\longrightarrow} \pi_0(\Omega_{q_0,q_1},\delta) \stackrel{i_*}{\longrightarrow} \pi_0(\Omega_{Q_0,Q_1},\delta)\rightarrow 0
\end{array}$$
induced by the exact sequence in relative homotopy of the pair $(\Omega_{q_0,q_1},\Omega_{Q_0,Q_1})$. The zero at the end comes from the fact that the base space $Q_0\times Q_1$ is path-connected.
To obtain this new exact sequence we have used the fact that 
$$p_*:\pi_n(\Omega_{Q_0,Q_1},\Omega_{q_0,q_1},\delta)\longrightarrow \pi_n(Q_0\times Q_1,(q_0,q_1))\, ,\ \ \ [f]\longmapsto [p\circ f]$$ 
is an isomorphism for any $n\geq 1$. In other words, for any $f\in \pi_n(Q_0\times Q_1,(q_0,q_1))$ there exists a unique $\tilde{f}$ in the relative homotopy group 
$\pi_n(\Omega_{Q_0,Q_1},\Omega_{q_0,q_1},\delta)$ such that $p_*(\tilde{f})=f$; hence $\tilde{\partial}$ is defined by restricting $\tilde{f}$ to $I^{n-1}$
$$\tilde{\partial} [f] := \ \partial [\tilde{f}] : (I^{n-1},\partial I^{n-1})\longrightarrow (\Omega_{q_0,q_1},\delta)\, .$$ 
In the particular case $n=1$ an element $f\in \pi_1(Q_0\times Q_1,(q_0,q_1))$ is represented by 
$$f:(I,\partial I)\longrightarrow (Q_0\times Q_1,(q_0,q_1))$$
and there exists a unique $\tilde{f}\in \pi_1(\Omega_{Q_0,Q_1},\Omega_{q_0,q_1},\delta)$, represented by 
$$\tilde{f}:(I,\partial I,0)\longrightarrow (\Omega_{Q_0,Q_1},\Omega_{q_0,q_1},\delta)\, ,$$
such that $p_*(\tilde{f})=f$. In this case 
$$\tilde{\partial}[f] \ =\ \partial[\tilde{f}]: \{1\}\longrightarrow \Omega_{q_0,q_1}$$ 
is an element $\gamma\in \Omega_{q_0,q_1}$, that is a path from $q_0$ to $q_1$; the map $\tilde{f}$ can be seen as 
$$\tilde{f}:[0,1]\times [0,1] \longrightarrow M$$
such that $\tilde{f}(0,t)=\delta(t)$, $\tilde{f}(1,t)=\gamma(t)$, while 
$$\alpha(s):= \ \tilde{f}(s,0) \ \subseteq \ Q_0\, ,\ \ \ \beta(s):= \tilde{f}(s,1)\ \subseteq \ Q_1\, .$$

\begin{center}
\includegraphics[height=35mm]{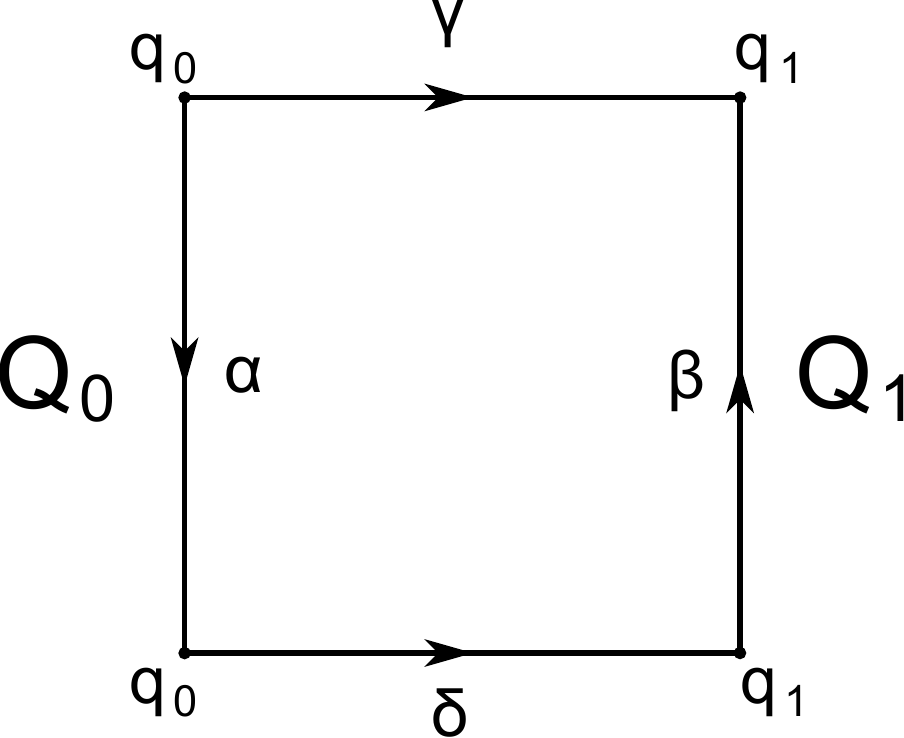}
\end{center}

Therefore we get that the path $\gamma$ is homotopic to the path $\beta\# \delta \# \alpha$, through a homotopy with values in $M$; in particular we get 
$$\text{Im}\, \tilde{\partial}\ =\ \Big \{\beta \# \delta \# \alpha \ \Big | \ \alpha \in \imath_*(\pi_1(Q_0,q_0))\, ,\ \beta \in \imath_*(\pi_1(Q_1,q_1))\Big \}$$ 
where $\imath:Q_0\hookrightarrow M$, $\imath:Q_1\hookrightarrow M$ denote the natural injections. Since the sequence is exact, the zero at  the end implies that the map
$$i_*:\pi_0(\Omega_{q_0,q_1},\delta)\longrightarrow \pi_0 (\Omega_{Q_0,Q_1},\delta)$$
is surjective; hence, again by the exactness of the sequence, we get 
$$\pi_0(\Omega_{Q_0,Q_1},\delta) \ \cong \ \bigslant{\pi_0(\Omega_{q_0,q_1},\delta)}{\sim_{Q_0,Q_1}}\, ,$$
where the equivalence relation is defined by $\, \sigma \sim_{Q_0,Q_1} \sigma'\, $ if and only if 
$$\exists \ \alpha \in \imath_*(\pi_1(Q_0,q_0))\, ,\ \beta \in \imath_*(\pi_1(Q_1,q_1)) \ \ \ \ \text{s.t.}\ \ \ \ \sigma' \ \sim \ \beta \# \sigma \# \alpha\, ,$$
exactly as we wished to show.
\end{proof}

\vspace{5mm}

We already know from (\ref{legameconp1}) that $\pi_0(\Omega_{q_0,q_1},\delta)$ coincides with $\pi_1(M,q_0)$; therefore we would like to investigate how (\ref{primaformulazione})
can be expressed in terms of the fundamental group of $M$. In order to do that we have to write any loop in $i_*(\pi_1(Q_1,q_1))$ as a loop with base point $q_0$; thus, let $\beta \in i_*(\pi_1(Q_1,q_1))$ and let 
$\gamma$ be any path connecting $q_0$ to $q_1$. The loop $\gamma^{-1}\#\beta \#\gamma$ represents then $\beta$ as a closed loop based at $q_0$. In particular 
\begin{equation}
\pi_0 (\Omega_{Q_0,Q_1},\delta) \ \cong \ \bigslant{\pi_1(M,q_0)}{\sim_{Q_0,Q_1}}\, ,
\label{secondaformulazione}
\end{equation}
where in this case the relation $\sim_{Q_0,Q_1}$ on $\pi_1(M,q_0)$  is defined by $\sigma \sim_{Q_0,Q_1} \sigma'$ if and only if there exist $\alpha \in i_*(\pi_1(Q_0,q_0))$ and $\beta\in i_*(\pi_1(Q_1,q_1))$ such that 
$$\sigma' \ \sim \ (\gamma^{-1}\# \beta \#\gamma) \# \sigma \# \alpha\, .$$

We end this section with some easy example, which may help the reader to understand the general picture explained above. We  suppose $M=\T^2$ and we consider submanifolds  $Q_0 \cong Q_1 \cong S^1$. 
Since in this case $\pi_1(M)$ is abelian, the subgroups  $H_0:=i_*(\pi_1(Q_0))$, $H_1:=i_*(\pi_1(Q_1))$ are normal and we may rewrite (\ref{secondaformulazione}) as 
$$\pi_0(\Omega_{Q_0,Q_1},\delta) \ \cong \ \bigslant{\pi_1(M,q_0)}{\langle H_0,H_1\rangle}\, ,$$
where $\langle H_0,H_1\rangle $ denotes the subgroup generated by $H_0,H_1$. We shall keep the same notation (i.e. $H_0,H_1$) also later on in the general setting to denote the smallest normal subgroups which contain respectively $i_*(\pi_1(Q_0))$, $i_*(\pi_1(Q_1))$.

Let now $\sigma_0, \, \sigma_1$ be the standard generators of $\pi_1(M)=\Z\times \Z$. Consider first the case $H_0=\langle \sigma_0\rangle$, $H_1=\langle \sigma_1\rangle$, which is
represented by Figure (\ref{primo}) below. In this case we clearly have $\langle H_0,H_1\rangle = \pi_1(M)$; hence, the space $\Omega_{Q_0,Q_1}$ is connected. In fact, $Q_0$ and $Q_1$ necessarily intersect, so any path in $\Omega_{Q_0,Q_1}$ 
is homotopic (in $\Omega_{Q_0,Q_1}$) to a constant path, i.e. the space $\Omega_{Q_0,Q_1}$ is contractible, while $\pi_0(\Omega_{q_0,q_1})\cong \Z\times \Z$. 

\begin{figure}[h]
\begin{center}
\includegraphics[height=32mm]{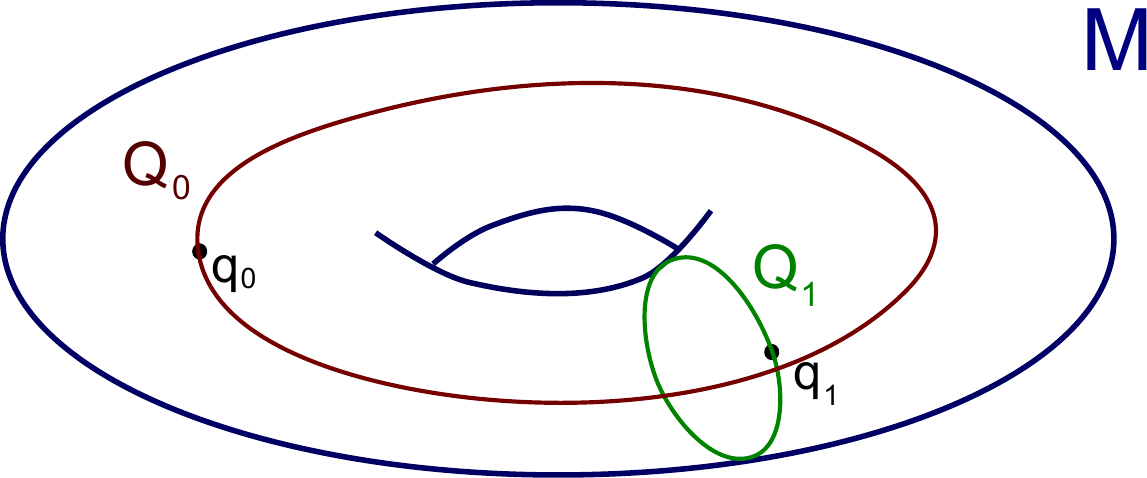}
\end{center}
\caption{An example where $H_0=\langle \sigma_0 \rangle$, $H_1=\langle \sigma_1\rangle$.}
\label{primo}
\end{figure}

Consider now the case $H_0=H_1=\langle \sigma_1\rangle$ as represented in Figure (\ref{secondo}). Here we have $\langle H_0,H_1\rangle = \langle \sigma_1 \rangle$; therefore $\pi_0(\Omega_{Q_0,Q_1}) \cong \Z$ and any connected component is uniquely determined by the winding number around any meridian. 

\begin{figure}[h]
\begin{center}
\includegraphics[height=35mm]{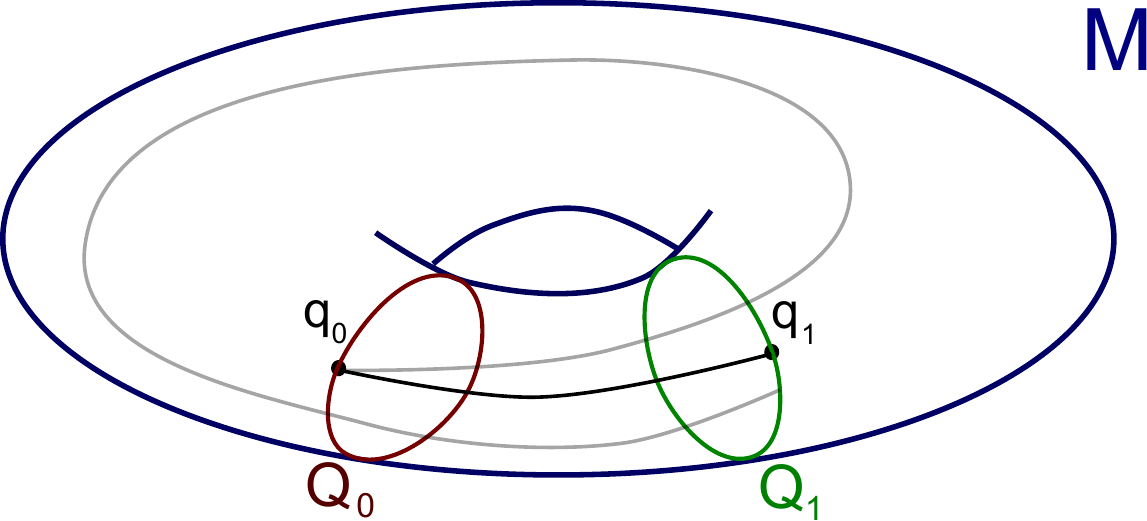}
\end{center}
\caption{An example where $H_0=H_1=\langle \sigma_1\rangle$.}
\label{secondo}
\end{figure}

In other words, the connected components of $\Omega_{Q_0,Q_1}(M)$ are in one to one correspondence with the powers of $\sigma_0$. In the figure above the black path and the grey path are in different connected components, being their winding numbers around any meridian different. Finally, observe that in the case $H_0=H_1=0$ we have 
$$\pi_0(\Omega_{Q_0,Q_1},\delta )\ \cong \ \pi_1(M,q_0)\ \cong\  \Z\times \Z\, .$$

\begin{figure}[h]
\begin{center}
\includegraphics[height=35mm]{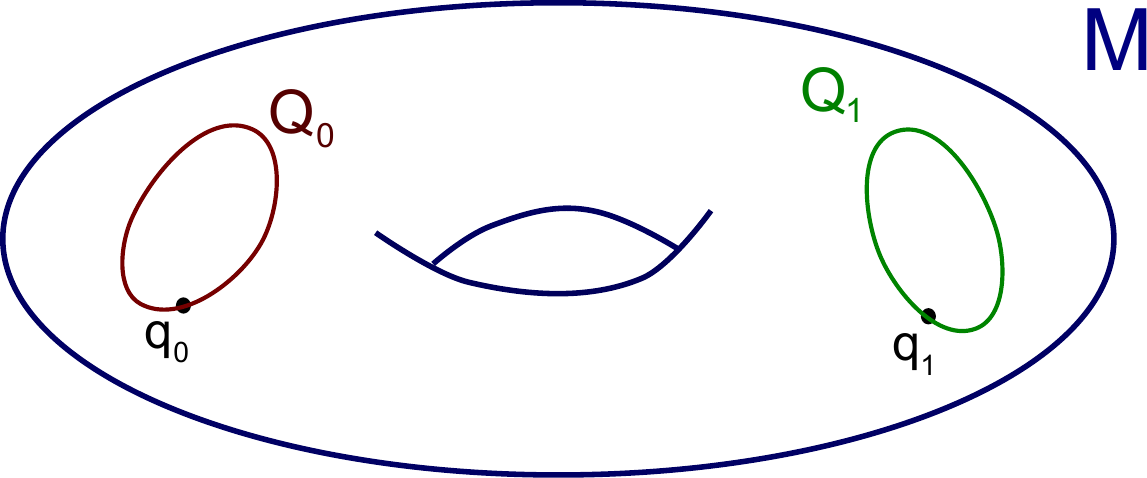}
\end{center}
\caption{An example where $H_0=H_1=\{0\}$.}
\label{terzo}
\end{figure}


\section{The Lagrangian action functional}
\label{thelagrangianactionfunctional}

In this section, following \cite{AS09}, we define the Lagrangian action functional $\A_L$ and check its regularity properties.
The metric $g$ on $M$ induces a metric on the tangent bundle $TM$, covariant derivatives  on $M$ and on $TM$, the horizontal subbundle $T^h TM$ of $TTM$ and isomorphisms 
$$T_{(q,v)}TM \ =\ T^h_{(q,v)} TM \oplus T^v_{(q,v)} TM\ \cong \ T_q M\oplus T_q M$$
where $T^v_{(q,v)} TM$ is the vertical subspace. We denote with $\nabla_q$, $\nabla_v$ respectively the horizontal, vertical components of the gradient of a function defined on $TM$; we use similar notations for higher 
derivatives. Even though in this work we will be interested only in the autonomous case, in order to prove the required regularity properties of the Lagrangian action functional it will be convenient to work in the 
more general setting of non-autonomous (i.e. time-depending) Lagrangians. To get a well-defined functional we will need however some additional growth-condition on the Lagrangian. Namely, throughout this and the next section we consider 
smooth Lagrangians $L:[0,1]\times TM \longrightarrow \R$ satisfying the following growth conditions:

\begin{description}
\item $(L1)\ $ there exists a constant $l_1\in \R$ such that  
$$\begin{array}{l}
\|\nabla_{vv} L(t,q,v)\| \ \leq \ l_1\, ,\\
\|\nabla_{vq} L(t,q,v)\| \ \leq \ l_1\, (1+ \|v\|_q),\\
\|\nabla_{qq} L(t,q,v)\| \ \leq \ l_1\, (1+\|v\|_q^2),
\end{array}$$
for any $(t,q,v)\in [0,1]\times TM$.

\item $(L2)\ $ there exists a positive constant $l_2\in \R$ such that 
$$\nabla_{vv} L(t,q,v)\ \geq \ l_2\cdot  Id$$
for any $(t,q,v)\in [0,1]\times TM$.
\end{description}

\vspace{3mm}

Condition $(L1)$ implies that $L$ grows at most quadratically on each fiber, while condition $(L2)$ implies that $L$ grows at least quadratically on each fiber; thus all the Lagrangians that we consider 
are supposed to be quadratic at infinity on each fiber. These conditions are independent on the choice of the metric $g$, meaning that if $L$ satisfies these conditions with respect to a suitable metric $g$ 
then $L$ satisfies the same conditions (with different constants $l_1,l_2$) with respect to any other metric.

\begin{oss}
A (autonomous) Tonelli Lagrangian is not necessarily quadratic at infinity. However, this is not much a problem for our purposes. In fact, when looking for periodic orbits or for orbits connecting two submanifolds satisfying conormal boundary conditions
on a given (compact) energy level $E^{-1}(k)$, we can always modify the Lagrangian outside a compact set to achieve the desired growth-conditions.
\end{oss}

It is interesting to see how conditions $(L1)$, $(L2)$ can be expressed in local charts. Observe that a bi-bounded time-depending local coordinate system $(U,\phi)$ for $M$ (cf. Section \ref{ahilbertmanifoldofloops}) induces a time-depending coordinate system on $TM$
$$[0,1]\times U\times \R^n \longrightarrow TM\, ,\ \ \ (t,q,v)\longmapsto \big (\phi(t,q), D_q\phi(t,q)[v]\big )\, .$$
The pull-back of $L$ by such a coordinate system is the function 
$$(\phi^* L) (t,q,v) \ =\ L(t,\phi(t,q),D_q\phi(t,q)[v])\, , \ \ \ \ \forall \ (t,q,v)\in [0,1]\times U\times \R^n\, .$$

When no confusion is possible, we denote $(\phi^* L)$ simply by $L$. Conditions $(L1)$, $(L2)$ can be then restated by saying that for every $(U,\phi)$ as above 

\begin{description}
\item $(L1')\ $ there exists a positive number $l_1$ such that 
$$\left |\frac{\partial^2 L}{\partial v^2} (t,q,v)\right | \leq\ l_1\, ,\ \ \ \left |\frac{\partial^2 L}{\partial q\partial v}(t,q,v)\right | \leq\ l_1(1+|v|)\, ,\ \ \ \left |\frac{\partial^2 L}{\partial q^2}(t,q,v)\right |\leq\ l_1 (1+|v|^2)$$ 
for any $(t,q,v)\in [0,1]\times U\times \R^n$.

\item $(L2')\ $ there exists a positive number $l_2$ such that 
$$\frac{\partial^2 L}{\partial v^2} (t,q,v)\ \geq \ l_2 \cdot Id$$
for any $(t,q,v)\in [0,1]\times U\times \R^n$.
\end{description}

\noindent If integrated along the fiber, condition $(L1')$ implies the growth conditions
\begin{eqnarray}
\left |\frac{\partial L}{\partial q}(t,q,v)\right | &\leq& l_3 \, (1+|v|^2)\, ,\ \ \ \ \left |\frac{\partial L}{\partial v}(t,q,v)\right | \ \ \leq \ \ l_3 \, (1+|v|) \label{boundsconvergenza} \\
L(t,q,v) &\leq& l_4 \, (1+|v|^2)
\end{eqnarray}
for suitable constants $l_3,l_4\in \R$. Let now $L:[0,1]\times TM\rightarrow \R$ be a Lagrangian which satisfies the condition $(L1)$, then the \textit{Lagrangian action functional}
\begin{equation}
\A_L(x) := \int_0^1 L(t,x(t),x'(t))\, dt
\label{actionfunctional}
\end{equation}
is well-defined on $H^1([0,1],M)$. Observe that for every $(U,\phi)$ as above we have
$$\A_L(\phi_*(x))\ =\ \A_{\phi^*L} (x)\, , \ \ \ \ \forall \ x\in H^1([0,1],U)\, .$$

Therefore, the study of the local properties of $\A_L$ is reduced to the study of the functional $\A_{\phi^* L}$, which is defined on an open subset of a Hilbert space.

\begin{teo} 
Suppose that $L:[0,1]\times TM \rightarrow \R$ satisfies the condition $(L1)$; then the Lagrangian action functional $\A_L$ is continuously differentiable on $H^1([0,1],M)$. Also, its differential $D\A_L$ is locally Lipschitz continuous and Gateaux-differentiable.
\label{regularity}
\end{teo}

\begin{proof}
Since the statement is of local nature, by using the diffoemorphism $\phi_*$ induced by $(U,\phi)$ we may assume that $L$ is defined on $[0,1]\times U\times \R^n$, 
with $U$ open subset of $\R^n$, and satisfies $(L1')$. Thus, let $x\in H^1([0,1],U)$, $\xi\in H^1([0,1],\R^n)$ and $\delta\in \R\setminus \{0\}$ small; then by the dominated convergence theorem the quantity 
\begin{eqnarray*}
& & \frac{1}{\delta} \Big (\A_L(x+\delta\xi)-\A_L(x) \Big )\ = \\
&=& \frac{1}{\delta} \int_0^1 \Big [L(t,x+\delta\xi,x'+\delta\xi')-L(t,x,x')\Big ]\, dt \ =\\
&=& \int_0^1 dt \int_0^1 \Big [ \frac{\partial L}{\partial q}(t,x+s\delta\xi,x'+s\delta\xi')\cdot \xi + \frac{\partial L}{\partial v} (t,x+s\delta\xi,x'+s\delta\xi')\cdot \xi' \Big ]\, ds
\end{eqnarray*}
converges as $\delta\longrightarrow 0$  to 
\begin{equation}
D\A_L(x)[\xi] := \int_0^1 \left [\frac{\partial L}{\partial q}(t,x(t),x'(t)) \cdot \xi + \frac{\partial L}{\partial v}(t,x(t),x'(t))\cdot \xi' \right ]\, dt\, .
\label{differential}
\end{equation}
Indeed, the bounds in (\ref{boundsconvergenza}) imply that 
$$\begin{array}{l}
\displaystyle \frac{\partial L}{\partial q} (\cdot,x+s\delta\xi,x'+s\delta\xi') \ \ \stackrel{\delta\rightarrow 0}{\longrightarrow} \ \ \frac{\partial L}{\partial q}(\cdot,x,x') \ \ \ \ \text{in}\ \ L^1([0,1]) \\ \\
\displaystyle \frac{\partial L}{\partial v} (\cdot,x+s\delta\xi,x'+s\delta\xi') \ \ \stackrel{\delta\rightarrow 0}{\longrightarrow} \ \ \frac{\partial L}{\partial v}(\cdot,x,x') \ \ \ \ \text{in}\ \ L^2([0,1])
\end{array}$$
and hence in particular
\begin{eqnarray*}
& & \int_0^1 \left |\left [ \frac{\partial L}{\partial q}(t,x+s\delta\xi,x'+s\delta\xi')-\frac{\partial L}{\partial q}(t,x,x') \right ]\cdot \xi \right |\, ds \ \leq \\
&\leq& \|\xi\|_\infty \cdot \left \|\frac{\partial L}{\partial q}(\cdot,x+s\delta\xi,x'+s\delta\xi')-\frac{\partial L}{\partial q}(\cdot,x,x')\right \|_1 \ \ \longrightarrow \ \ 0
\end{eqnarray*}

\begin{eqnarray*}
& & \int_0^1 \left |\left [ \frac{\partial L}{\partial v}(t,x+s\delta\xi,x'+s\delta\xi')-\frac{\partial L}{\partial v}(t,x,x') \right ]\cdot \xi' \right |\, ds \ \leq \\
&\leq& \|\xi'\|_2 \cdot \left \|\frac{\partial L}{\partial v}(\cdot,x+s\delta\xi,x'+s\delta\xi')-\frac{\partial L}{\partial v}(\cdot,x,x')\right \|_2 \ \ \longrightarrow \ \ 0\, .
\end{eqnarray*}

Since $D\A_L(x)$ is a bounded linear functional on $H^1([0,1],\R^n)$, $\A_L$ is Gateaux differentiable and $D\A_L(x)$ is its Gateaux differential at $x$. In order to prove that 
$D\A_L(\cdot)$ is continuous at $x$, we must show that 
$$x_h \ \stackrel{H^1}{\longrightarrow}\ x \ \ \ \Longrightarrow \ \ \ D\A_L(x_h)\ \stackrel{(H^1)^*}{\longrightarrow}\ D\A_L(x)\, .$$

So let us assume that $x_h$ converges to $x$ in $H^1([0,1],U)$; in particular, by the continuity of the inclusion $H^1\hookrightarrow L^\infty$ it follows that 
$x_h$ converges to $x$ uniformly. Moreover, there is a function $f\in L^2([0,1])$ such that $|x_h'|\leq f$ almost everywhere for every $h\in \N$ (here we have used the fact that 
a sequence of real-valued functions which converges in $L^1$ is dominated almost everywhere by an $L^1$ function). By a standard argument involving subsequences\footnote{A sequence $\{x_h\}$
in a metric space converges to $x$ if and only if every subsequence of $\{x_h\}$ has a subsequence which converges to $x$.}, we may also assume that $x_h'\rightarrow x'$ a.e.
The bounds in (\ref{boundsconvergenza}) and the dominated convergence theorem imply then that 
$$\begin{array}{l}
\displaystyle \frac{\partial L}{\partial q}(\cdot,x_h,x_h')\ \longrightarrow \ \frac{\partial L}{\partial q}(\cdot,x,x')\ \ \ \ \text{in}\ L^1([0,1])\\ \\
\displaystyle \frac{\partial L}{\partial v}(\cdot,x_h,x_h')\ \longrightarrow \ \frac{\partial L}{\partial v}(\cdot,x,x')\ \ \ \ \text{in}\ L^2([0,1])
\end{array}$$ 
Thus the convergence of $D\A_L(x_h)$ to $D\A_L(x)$ in $(H^1)^*$ follows. In fact, 
\begin{eqnarray*}
& & \big \| D\A_L(x_h)-D\A_L(x) \big \| \ =\ \sup_{\|\xi\|=1} \big | [D\A_L(x_h)-D\A_L(x)]\cdot \xi \big | \ \leq \\ 
&\leq& \int_0^1 \left (\left | \Big [\frac{\partial L}{\partial q}(t,x_h,x_h')-\frac{\partial L}{\partial q}(t,x,x') \Big ] \xi \right | + 
\left | \Big [\frac{\partial L}{\partial v}(t,x_h,x_h')-\frac{\partial L}{\partial v}(t,x,x') \Big ] \xi' \right |\right )\, dt \ \leq\\
&\leq& \|\xi\|_\infty \cdot \left \| \frac{\partial L}{\partial q}(\cdot,x_h,x_h')-\frac{\partial L}{\partial q}(\cdot,x,x') \right \|_1  +
\ \|\xi'\|_2 \cdot  \left \|\frac{\partial L}{\partial v}(\cdot,x_h,x_h')-\frac{\partial L}{\partial v}(\cdot,x,x')\right \|_1
\end{eqnarray*}
goes to zero as $h\rightarrow +\infty$. Here we have used Cauchy-Schwarz inequality and, again, the continuity of the immersion $H^1\hookrightarrow C^0$.
The Gateaux-differential $D\A_L(\cdot)$ depends therefore continuously on $x$ and this implies that $\A_L$ is Fr\'ech\'et-differentiable and $D\A_L(x)$ 
is its Fr\'ech\'et differential at $x$. Now we want to prove that the differential $D\A_L$ is Lipschitz continuous and Gateaux-differentiable; in order to do that 
let us consider $x$, $\xi$ as above and let $\eta\in H^1([0,1],\R^n)$; the property $(L1')$ and the dominated convergence theorem imply that the quantity 
\begin{eqnarray*}
& & \frac{1}{\delta}\, \Big (D\A_L(x+\delta \eta)[\xi]-D\A_L(x)[\xi] \Big) \ =\\
&=& \int_0^1  \int_0^1 \left [\frac{\partial^2 L}{\partial v^2} (t,x+s\delta\eta,x'+s\delta\eta')\, \xi'\cdot \eta ' + \frac{\partial^2 L}{\partial q\partial v}(t,x+s\delta\eta,x'+s\delta\eta')\, \xi\cdot \eta' + \right. \\
&+& \left. \frac{\partial^2 L}{\partial v\partial q} (t,x+s\delta\eta,x'+s\delta\eta')\, \xi'\cdot \eta + \frac{\partial^2 L}{\partial q^2}(t,x+s\delta\eta,x'\!+\!s\delta\eta')\, \xi\cdot \eta \right ]\, ds\, dt
\end{eqnarray*}
converges, as $\delta\rightarrow 0$, to
\begin{eqnarray}
d^2 \A_L(x)[\xi,\eta] &:=& \int_0^1 \Big [ \frac{\partial^2 L}{\partial v^2}(t,x,x')\, \xi'\cdot \eta' + \frac{\partial^2 L}{\partial q \partial v} (t,x,x')\, \xi\cdot \eta' + \nonumber \\
&+& \frac{\partial^2 L}{\partial v \partial q}(t,x,x')\, \xi'\cdot \eta + \frac{\partial^2 L}{\partial q^2}(t,x,x')\, \xi\cdot \eta\Big ]\, dt\, .
\label{seconddifferential}
\end{eqnarray}

Since $d^2\A_L(x)$ is a bounded symmetric bilinear form on $H^1([0,1],\R^n)$, $D\A_L$ is Gateaux-differentiable at $x$ and its Gateux differential at $x$ is the bounded linear operator 
$D^2\A_L(x):H^1([0,1],\R^n)\longrightarrow H^1([0,1],\R^n)^*$ defined by
$$\big ( D^2\A_L(x)\cdot \xi\big )[\eta] := d^2\A_L(x)[\xi,\eta] \ \ \ \ \forall \xi,\eta \in H^1([0,1],\R^n)\, .$$

By $(L1')$ the map $x\mapsto D^2\A_L(x)$ is bounded with respect to the norm-topology on the space of bounded self-adjoint operators, so the mean value theorem implies that $D\A_L$ is Lipschitz 
on convex subsets of $H^1([0,1],U)$.
\end{proof}

\vspace{5mm}

The Lagrangian action functional is not of class $C^2$, unless $L$ is a polynomial of degree at most two on each fiber of $TM$; in this case, $\A_L$
is actually smooth on $H^1([0,1],M)$. In particular, electro-magnetic Lagrangians are the only Lagrangians which satisfy the condition $(L2)$ and induce a smooth action functional.
In general, the action functional $\A_L$ even fails to be twice differentiable, as the following proposition states (cf. \cite[Proposition 3.2]{AS09}).

\begin{prop}
Assume that the Lagrangian $L$ satisfies $(L1)$; if the functional $\A_L$ is twice differentiable at $x\in H^1([0,1],M)$, then for every $t\in [0,1]$ the function 
$$T_{x(t)} M \longrightarrow \R\, ,\ \ \ \ v\longmapsto L(t,x(t),v)$$ 
is a polynomial of degree at most two.
\label{nonregularity}
\end{prop}

\section{Conormal boundary conditions.}
\label{conormalboundaryconditions}

In this section we introduce the boundary conditions that we are going to consider in the next chapters. To do this it will be again convenient to consider the more general class of non-autonomous Tonelli Lagrangians.
Throughout this section we shall furthermore assume that the Lagrangian $L:[0,1]\times TM\rightarrow \R$ satisfies both the growth-conditions $(L1)$ and $(L2)$ as in Section \ref{thelagrangianactionfunctional}. 
By $(L2)$ the Euler-Lagrange equation associated to $L$, which in local coordinates can be written as 
\begin{equation}
\frac{d}{dt}\left [\frac{\partial L}{\partial v}(t,x(t),x'(t))\right ] \ =\ \frac{\partial L}{\partial q}(t,x(t),x'(t))\, ,
\label{eulerlagrange1}
\end{equation}
defines a locally well-posed second order Cauchy problem. We treat different boundary conditions  in a unified way by considering a non-empty, boundaryless smooth submanifold 
$Q\subseteq M\times M$ and by imposing \textit{conormal boundary conditions}
\begin{equation}
\left \{\begin{array}{l}
(x(0),x(1)) \in Q\, ;\\ \\ 
d_vL(0,x(0),x'(0)) [\zeta_0] \! =\! d_v L(1,x(1),x'(1))[\zeta_1]\, , \forall \ (\zeta_0,\zeta_1)\in T_{(x(0),x(1))}Q\, ;
\end{array}\right.
\label{nonlocalboundaryconditions}
\end{equation}
where $d_v L$ denotes the fiberwise differential of $L$. If $H:[0,1]\times T^*M\rightarrow \R$ 
is the Fenchel dual of $L$, then the boundary value problem (\ref{eulerlagrange1}), (\ref{nonlocalboundaryconditions}) is equivalent to the problem of finding Hamiltonian orbits $\eta:[0,1]\rightarrow T^*M$ such that 
\begin{equation}
(\eta(0),-\eta(1))\in N^*Q\, ,
\label{hamiltoniannonlocalboundaryconditions}
\end{equation}
where $N^*Q$ denotes as usual the conormal bundle of $Q$ (for generalities about conormal bundles see Appendix \ref{conormalbundles}). 
In fact, (\ref{hamiltoniannonlocalboundaryconditions}) is equivalent to 
\begin{equation*}
\left \{\begin{array}{l}
\big (\pi^*(\eta(0)),\pi^*(\eta(1))\big )\in Q\, ;\\ \\ 
\eta(0)\, [\zeta_0] \ =\ \eta(1)\, [\zeta_1] \ \ \ \ \forall\ (\zeta_0,\zeta_1)\in T_{(\eta(0),\eta(1))}Q\, ;
\end{array}\right.
\end{equation*}
which is exactly the reformulation of (\ref{nonlocalboundaryconditions}) through the Legendre transform. We will be interested in the two particular cases  $Q=\Delta$ diagonal in $M\times M$ and $Q=Q_0\times Q_1$, with $Q_0,Q_1\subseteq M$ smooth closed connected 
submanifolds. In the first case (\ref{nonlocalboundaryconditions}) means that we are looking for periodic orbits of the Euler-Lagrange flow, while in the latter one (\ref{nonlocalboundaryconditions})  can be rewritten as 
\begin{equation}
\left \{\begin{array}{l}
x(0)\in Q_0\, ,\ x(1)\in Q_1\, ;\\ \\
d_vL(0,x(0),x'(0))\big |_{T_{x(0)}Q_0} \equiv  \ \ d_vL(1,x(1),x'(1))\big |_{T_{x(1)}Q_1} \equiv \ 0\, ;
\end{array}\right.
\label{lagrangianformulation1}
\end{equation}

We will get back to this in the next chapters. Recall that, in Section \ref{ahilbertmanifoldofloops}, for any smooth submanifold $Q\subseteq M\times M$ we defined the space
$$H^1_Q([0,1],M) := \Big \{x\in H^1([0,1],M)\ \Big |\  (x(0),x(1))\in Q\Big \}\, ,$$
which is a smooth submanifold of $H^1([0,1],M)$. We denote by $\A_L^Q$ the restriction of the Lagrangian action functional $\A_L$ defined in (\ref{actionfunctional}) to $H^1_Q([0,1],M)$. 
Theorem \ref{criticalpoints} below states that critical points of  $\A_L^Q$ correspond to the (smooth) solutions of the Euler-Lagrange equation (\ref{eulerlagrange1}) that satisfies the boundary 
conditions (\ref{nonlocalboundaryconditions}).

\begin{teo}
Let $L:[0,1]\times TM\rightarrow \R$ be a Lagrangian that satisfies the conditions $(L1)$, $(L2)$ and let $Q\subseteq M\times M$ be a smooth submanifold. Then:

\begin{enumerate}
\item The critical points of $\A_L^Q$ are precisely the (smooth) solutions of (\ref{eulerlagrange1}), (\ref{nonlocalboundaryconditions}).
\item For every critical point $x$ of $\A_L^Q$, the second Gateaux differential $d^2\A_L^Q(x)$ of $\A_L^Q$ at $x$ is Fredholm and has finite Morse index.
\end{enumerate}
\label{criticalpoints}
\end{teo}

\begin{proof}
The statements above are both of local nature, so by using a diffeomorphism $\phi_*$ induced by a local coordinate system $(U,\phi)$ for $M$ as in Section \ref{ahilbertmanifoldofloops} 
we may assume that $L$ is defined on $[0,1]\times U\times \R^n$, with $U\subseteq \R^n$ open, and satisfies $(L2')$. We already know by Theorem \ref{regularity}
that $\A_L$ is Fr\'ech\'et-differentiable with locally Lipschitz continuous and Gateaux-differentiable differential $D\A_L$.

\vspace{5mm}

$1.\ $ Let $x$ be a critical point for $\A_L^Q$; we want to prove that $x$ is actually a smooth curve and a solution of the Euler-Lagrange equation (\ref{eulerlagrange1}) 
satisfying the boundary conditions (\ref{nonlocalboundaryconditions}). By choosing a local chart $(U,\phi)$ as in the definition of the atlas of $H^1_Q([0,1],M)$ (cf. Section \ref{ahilbertmanifoldofloops}), 
we may assume that $x$ is a critical point of $\A_L^W$, the restriction of $\A_L:H^1([0,1],U)\rightarrow \R$ to the intersection of $H^1([0,1],U)$ with the closed linear subspace 
$$H^1_W([0,1],\R^n) \ =\ \Big \{x\in H^1([0,1],\R^n) \ \Big |\ (x(0),x(1))\in W \Big \}\, ,$$
where $W\subseteq \R^n\times \R^n$ is a suitable linear subspace. The first condition in (\ref{nonlocalboundaryconditions}) can be rewritten as 
$$(x(0),x(1))\in W$$
and it is satisfied by any element in $H^1_W([0,1],\R^n)$. Differentiating the condition $(\phi(0,q),\phi(1,q))\in Q$ for every $q\in U\cap W$, we obtain that the linear map 
$$(\xi_0,\xi_1)\ \longmapsto \ \big (D_q\phi(0,q_0)[\xi_0],D_q\phi (1,q_1)[\xi_1]\big )$$
maps $W$ isomorphically onto the tangent space of $Q$ at $(\phi(0,q_0),\phi(1,q_1))$. Therefore, the second condition in (\ref{nonlocalboundaryconditions}) is equivalent to 
\begin{equation}
\frac{\partial L}{\partial v}\big (0,x(0),x'(0)\big ) [\xi_0] \ =\ \frac{\partial L}{\partial v}\big (1,x(1),x'(1)\big ) [\xi_1]\, ,  \ \ \ \ \forall \ (\xi_0,\xi_1)\in W\, .
\label{equivalentboundarycondition}
\end{equation}

Identity (\ref{differential}) and an integration by parts produce for every smooth curve $\xi:[0,1]\rightarrow \R^n$ with compact support in $(0,1)$ the identity 
\begin{eqnarray}
0 &=&  D\A_L^W(x)[\xi] \ =\ \int_0^1 \left [\frac{\partial L}{\partial q}(t,x,x')\cdot \xi + \frac{\partial L}{\partial v}(t,x,x')\cdot \xi'\right ]\, dt \nonumber \\
&=& \int_0^1 \! \left [\frac{\partial L}{\partial v}(t,x,x')\cdot \xi' + \Big [\xi \cdot \int_0^t \frac{\partial L}{\partial q}(s,x,x') ds \Big ]_0^1 - \int_0^t \frac{\partial L}{\partial q}(s,x,x')\cdot \xi'  ds  \right ] dt \nonumber \\
&=& \int_0^1 \! \left [\frac{\partial L}{\partial v}(t,x(t),x'(t)) - \int_0^t \frac{\partial L}{\partial q}(s,x(s),x'(s))\, ds \right ]\cdot \xi' \, dt\, .
\end{eqnarray}
Then the Du Bois-Reymond Lemma implies that there is a vector $u\in \R^n$ such that 
\begin{equation}
\frac{\partial L}{\partial v}(t,x(t),x'(t)) -\int_0^t \frac{\partial L}{\partial q}(s,x(s),x'(s))\, ds \ =\  u\, , \ \ \ \ \text{a.e. in } [0,1]\, .
\label{Dubois}
\end{equation}
Observe that the function
\begin{equation}
f_x(t):=\int_0^t \frac{\partial L}{\partial q} (s,x(s),x'(s))\, ds
\label{function}
\end{equation}
is continuous on $[0,1]$, indeed if $t_n\rightarrow t$ then the bound in (\ref{boundsconvergenza}) implies 
\begin{eqnarray*}
|f_x(t_n)-f_x(t)| &=&  \Big |\int_t^{t_n} \frac{\partial L}{\partial q}(s,x(s),x'(s))\, ds \Big | \ \leq \\
                                                 &\leq& \int_t^{t_n} \Big |\frac{\partial L}{\partial q}(s,x(s),x'(s)) \Big |\, ds\ \leq \\
                                                 &\leq&  l_3 \int_t^{t_n} \big (1+|x'(s)|^2 )\, ds \ = \  l_3 \, (t_n-t) + l_3\, \|x'\|_{L^2([t,t_n])}^2\, .
\end{eqnarray*}
At the same time, condition $(L2')$ implies that the map 
\begin{equation}
[0,1]\times U\times \R^n\longrightarrow [0,1]\times U\times \R^n\, ,\ \ \ (t,q,v)\longmapsto \left (t,q,\frac{\partial L}{\partial v}(t,q,v)\right )
\label{map}
\end{equation}
is a surjective smooth diffeomorphism. If we denote by $(t,q,p)\longrightarrow (t,q,\psi(t,q,p))$ its inverse, we have that 
$$x'(t) \ =\ \psi\left (t,x(t), \frac{\partial L}{\partial v}(t,x(t),x'(t))\right )$$ 
and hence (\ref{Dubois}) implies that 
\begin{equation}
x'(t)\ =\ \psi \big (t,x(t), u +f_x(t) \big ) \ \ \ \ \ \text{a.e. in } [0,1]\, .
\label{boot}
\end{equation}

In particular, $x'$ coincides almost everywhere with a continuous function; therefore $x\in C^1$ and now a boot-strap argument shows that $x$ is actually smooth. Therefore, we can apply 
a different integration by parts to the identity $D\A_L^W(x)[\xi]=0$ obtaining 
\begin{eqnarray}
 0 &=& D\A_L^W(x)[\xi] \ = \int_0^1 \Big [\frac{\partial L}{\partial q}(t,x,x') - \frac{d}{dt} \Big (\frac{\partial L}{\partial v}(t,x,x')\Big )\Big ]\cdot \xi \, dt +\nonumber \\
&+& \frac{\partial L}{\partial v}(1,x(1),x'(1))\cdot \xi(1) - \frac{\partial L}{\partial v}(0,x(0),x'(0))\cdot \xi(0)\, ,
\end{eqnarray}
where $\xi:[0,1]\rightarrow \R^n$ is any regular curve with $(\xi(0),\xi(1))\in W$. By taking curves $\xi$ with compact support in $(0,1)$ we get that $x$ satisfies the Euler-Lagrange equation (\ref{eulerlagrange1}); 
then, letting $\xi$ vary among all the smooth curves such that $(\xi(0),\xi(1))\in W$ we find that also (\ref{equivalentboundarycondition}) holds. This shows that every critical point of $\A_L^Q$ is a smooth solution of (\ref{eulerlagrange1}) 
satisfying the boundary conditions (\ref{nonlocalboundaryconditions}). 

Conversely, the fact that 
$$\frac{\partial^2 L}{\partial v^2} (t,q,v)$$ 
is invertible for every $(t,q,v)\in [0,1]\times U\times \R^n$ and the differentiable dependence of solutions of ordinary differential equations on the coefficients imply that every solution of (\ref{eulerlagrange1}) is smooth. If the 
boundary conditions (\ref{nonlocalboundaryconditions}) are also satisfied, then by integrating by parts the identity $D\A_L^Q(x)[\xi]=0$ as done above, one immediately sees that $x$ is a critical 
point of $\A_L^Q$ and this concludes the proof.

\vspace{5mm}

$2.\ $ Let $x$ be a critical point for $\A_L^Q$. By Theorem \ref{regularity}, $\A_L^Q$ is twice Gateaux-differentiable and its second Gateaux differential 
$$d^2\A_L(x): T_x H^1_Q([0,1],M) \times T_x H^1_Q([0,1],M) \longrightarrow \R$$
is a symmetric continuous bilinear form. Using the above localization argument, we may identify $x$ with a critical point of $\A_L^W$ in $H^1_W([0,1],\R^n)$ and 
$d^2\A_L^Q(x)$ with $d^2\A_L^W(x)$,  the restriction of the simmetric bilinear form (\ref{seconddifferential}) to  $H^1_W([0,1],\R^n)$. By $(L2')$, the self-adjoint operator $A$ on $H^1_W([0,1],\R^n)$ representing 
$$\alpha(x)[\xi,\eta] := \int_0^1 \frac{\partial^2 L}{\partial v^2}(t,x,x')\, \xi'\cdot \eta' \, dt\ =\ \langle A\xi,\eta\rangle_{H^1}$$ 
with respect to the Hilbert product is Fredholm and non-negative. In fact, if we consider the orthogonal decomposition of $H^1_W$ in 
$$H^1_W([0,1],\R^n)\ =\ \{\text{constants}\}\oplus \{\text{null average}\}\ =\ E_1\oplus E_2$$
we get that $\ker A =E_1$ while $A|_{E_2}$ is positive; in particular, $A$ is Fredholm with zero Morse index. The remaining three terms in (\ref{seconddifferential}) are continuous bilinear forms  
respectively on $L^2\times H^1$, $H^1\times L^2$ and $L^2\times L^2$. Therefore, the compactness of the embedding $H^1\hookrightarrow L^2$ implies that the self-adjoint operator representing 
$d^2\A_L^Q(x)-\alpha(x)$ is compact. More precisely, the bilinear form 
$$\int_0^1 \frac{\partial^2 L}{\partial q^2}(t,x,x')\, \xi\cdot \eta\, dt$$
is continuous on $L^2\times L^2$ and hence compact on $H^1\times H^1$. The bilinear form 
$$\beta(x)[\xi,\eta] := \int_0^1 \frac{\partial^2 L}{\partial q\partial v}(t,x,x')\, \xi\cdot \eta'\, dt \ = \ \langle B\xi,\eta\rangle_{H^1}$$ 
is instead continuous on $L^2\times H^1$. Thus $\beta$ is compact on $H^1\times H^1$ if and only if $B\xi_n \rightarrow 0$ whenever $\xi_n \rightharpoonup 0$; the latter fact is implied by
\begin{equation}
\xi_n \rightharpoonup 0 \ \ \ \Longrightarrow \ \ \ \langle B\xi_n,\eta_n\rangle_{H^1} \longrightarrow 0\, , \ \ \ \forall \eta_n\rightharpoonup 0\, .
\label{property}
\end{equation}

Indeed, if (\ref{property}) holds, then we can choose $\eta_n=B\xi_n$ obtaining $\|B\xi_n\|_{H^1}\rightarrow 0$. Observe that if $\xi_n\rightharpoonup 0$ then $\xi_n$ converges strongly to zero in $L^2$, because of 
the compactness of the embedding $H^1\hookrightarrow L^2$; therefore 
$$\langle B\xi_n,\eta_n\rangle_{H^1}\ = \int_0^1 \left [\frac{\partial^2 L}{\partial q\partial v} (t,x,x') \, \xi_n\right ] \cdot \eta_n'\, dt \ \longrightarrow \ 0$$ 
since $\eta_n\rightharpoonup 0$ and the other term in the integral tends to zero in $L^2$. This proves that the bilinear form $d^2\A_L^Q(x)$ can be written as 
$$d^2\A_L^Q(x) \ =\ \alpha(x) + \big [d^2\A_L^Q(x)-\alpha(x) \big ]$$ 
a compact perturbation of a Fredholm non negative operator; therefore, the second Gateaux differential $d^2\A_L^Q(x)$ is itself Fredholm with finite Morse index, since compact perturbations of a given operator modify the Spectrum only by adding 
a finite number of negative eingenvalues, each of which of finite mulipilicity.
\end{proof}


\section{The minimax principle.}
\label{theminimaxprinciple}

In the previous section we showed that the critical points of the Lagrangian action functional on $H^1_Q([0,1],M)$ correspond to the solutions of the Euler-Lagrange 
equation (\ref{eulerlagrange1}) that satisfy the boundary conditions (\ref{nonlocalboundaryconditions}). 

The goal of the next chapters will be therefore to prove the existence of critical points of the Lagrangian action functional, more precisely of the free-time Lagrangian action functional $\A_k$ (see Section \ref{thefreetimeactionfunctional} for the definition
and for more details), which detects the solutions of the Euler-Lagrange equation that satisfy the conormal boundary conditions and are contained on the energy level $E^{-1}(k)$. Clearly the easiest thing to try would be to look for global/local minimizers; this is however 
not possible in general, since the free-time Lagrangian action functional might be unbounded from below and, even if bounded from below, might not attain its infimum. 

Thus we will need a method to detect critical points which are not necessarily global or local minimizer. This will be provided from the so-called \textit{minimax principle}. 

\begin{defn}
Let $(\mathcal H,g)$ be a Riemannian Hilbert manifold and let $f\in C^1(\mathcal H)$. A sequence $\{x_n\}_{n\in\N}\subseteq \mathcal H$ is said to be a $\mathsf{Palais}$-$\mathsf{Smale\ sequence}$ at level $c$ if 
$$\lim_{n\rightarrow +\infty} f(x_n)\ =\ c\, ,\ \ \ \ \lim_{n\rightarrow +\infty} \|df(x_n)\| \ = \ 0$$
where $\|\cdot \|$ denotes the dual norm induced by $g$. 
\end{defn}

One would like to know if Palais-Smale sequences for $f$ admit converging subsequences, since limiting points are automatically critical points of the functional $f$. However, this is unfortunately not always the case
as simple counterexamples for $\mathcal H=\R^2$ already show (cf. [Abb13]). Therefore, we will need the following

\begin{defn}
Let $(\mathcal H,g)$ be a Riemannian Hilbert manifold. The functional $f\in C^1(\mathcal H)$ is said to satisfy the $\mathsf{Palais}$-$\mathsf{Smale\ condition\ at\ level\ c}$ if any Palais-Smale sequence at level $c$ is compact, 
meaning that it admits converging subsequences. 

More generally, $f$ is said to satisfy the $\mathsf{Palais}$-$\mathsf{Smale\ condition}$ if it satisfies the Palais-Smale condition at level $c$, for every $c\in \R$.
\end{defn}

Notice that the Palais-Smale condition and the completeness of $g$ are somehow antagonist requirements: one may achieve completeness multiplying $g$ by a positive function which diverges at infinity 
(thus reducing the set of Cauchy sequences), while the Palais-Smale condition could be achieved multiplying $g$ by a positive function which is infinitesimal at infinity (since the dual norm is multiplied by the inverse of this function, 
this operation reduces the set of Palais-Smale sequences). 

Here we do not assume any completeness for $g$, since in the following chapters we will have to deal with non-complete Hilbert manifolds. The completeness will be replaced by the weaker condition  that the sublevel sets
of $f$ are complete. 

\vspace{3mm}

Now, let us assume that $f\in C^{1,1}(\mathcal H)$, where $C^{1,1}(\mathcal H)$ denotes the space of $C^1$-functionals on $\mathcal H$ with locally Lipschitz differential. Assume furthermore that the sublevel sets of $f$ are complete.
Denote with $\nabla f$ the gradient of $f$ with respect to $g$. Since $-\nabla f$ is only locally Lipschitz, it need not define a positively complete flow. 

To avoid this problem we consider the conformally equivalent bounded vector field 
$$X_f:= -\frac{\nabla f}{\sqrt{1+\|\nabla f\|^2}}\, .$$
With the vector field $X_f$ is associated the flow $\phi$, given by the solutions of 
$$\left \{\begin{array}{l}
\displaystyle \frac{\partial}{\partial t} \, \phi_t (u) = X_f(\phi_t(u));
\\ \\
\phi_0(u)=u;
\end{array}\right.$$

Since for every $u\in \mathcal H$ the maximal solution to the Cauchy-problem above is defined on the whole $[0,+\infty)$, we say that $\phi$ is positively complete on $\mathcal H$ and refer to it
as the \textit{negative gradient flow of f}.

\begin{teo}[General minimax principle]
Let $f$ be a $C^{1,1}$-functional on a Riemannian Hilbert manifold $(\mathcal H,g)$ such that the sublevel sets $\{f\leq c\}$ are complete and let $\Gamma$ be a set of subsets of $\mathcal H$ which is positively invariant with 
respect to the negative gradient flow of $f$. If the number 
\begin{equation}
c:=\ \inf_{\gamma\in \Gamma}\  \sup_{x\in \gamma} \ f(x)
\label{definizionedic}
\end{equation}
is finite, then $f$ admits a Palais-Smale sequence at level $c$. In particular, if $f$ satisfies the Palais-Smale condition at level $c$, then $c$ is a critical value for $f$.
\label{minimaxtheorem}
\end{teo}

\begin{proof} By contradiction, suppose that there exists $\epsilon >0$ such that 
$$\|X_f\|\geq \epsilon\, , \ \ \ \ \text{on} \ \ \ \ \Big \{|f-c|\leq \epsilon\Big \}\, .$$
Notice that 
\begin{equation}
\frac{d}{dt} f(\phi_t(u))\ =\ df(\phi_t(u))\big[X_f(\phi_t(u))\big ]\ =-\frac{\|df(\phi_t(u))\|^2}{\sqrt{1+\|df(\phi_t(u))\|^2}}
\label{decreasing}
\end{equation}
so the function $t\longmapsto f(\phi_t(u))$ is decreasing. Suppose that 
$$|f(\phi_t(u))-c|\ \leq\ \epsilon\, , \ \ \ \ \forall \ t\in [0,T]\, ,$$
then we have 
\begin{eqnarray*}
2\epsilon &\geq& f(u)-f(\phi_T(u))\ = -\int_0^T \frac{d}{dt}\, f(\phi_t(u))\, dt \ = \\ 
                &=& \int_0^T \frac{\|df(\phi_t(u))\|^2}{\sqrt{1+\|df(\phi_t(u))\|^2}}\, dt \ \geq\ \int_0^T \|X_k(\phi_t(u))\|^2 \, dt \ \geq \ \epsilon^2 T
\end{eqnarray*}
from which we deduce that $T\leq 2/\epsilon$. Now choose $\gamma\in \Gamma$ such that 
$$\max_{x\in \gamma} \ f(x) \ \leq  \ c+\epsilon$$
(such a $\gamma$ exists because of the definition of $c$) and set $\tilde{\gamma} := \phi_T (\gamma)$, for some $T> 2/\epsilon$.
Observe that $\tilde{\gamma}\in \Gamma$, since by assumption $\Gamma$ is positively invariant under $\phi$. Moreover, since $f\leq c+\epsilon$ on $\gamma$, any $x\in \gamma$ satisfies exactly one of the following properties:
\begin{enumerate}
\item $|f(x)-c|\leq \epsilon$.
\item $f(x)<c-\epsilon$.
\end{enumerate}
If $x\in \gamma$ satisfies 1, then the choice of $T>2/\epsilon$ implies
\begin{equation}
f(\phi_T(x)) \ <\ c-\epsilon\, .
\label{dec}
\end{equation}
Clearly (\ref{dec}) holds also if $x$ satisfies 2, since $f$ decreases along the orbits of $\phi$. It follows that $\tilde{\gamma} \subseteq \{f<c-\epsilon\},$ which contradicts the definition of $c$.
\end{proof}

\vspace{3mm}

\begin{oss}
\label{flussotroncato}
It is sometimes useful to replace the negative gradient flow by a flow which fixes a certain sublevel of $f$. Let $\rho:\R\rightarrow \R^+$ be smooth, bounded and such that
$$\rho\ \equiv \ 0\  \ \ \text{on} \ (-\infty,b]\, , \ \ \ \ \rho\ >\ 0 \ \  \ \text{on} \ \ (b,+\infty)\, .$$

Consider the vector field $\rho(f)\, X_f$ and denote its flow with $\phi$. It is a negative gradient flow truncated below level $b$. The function  $t\longmapsto f(\phi_t(u))$ is constant if 
$$u\ \in \ \text{Crit}\, f \, \cup\, \big \{f\leq b\big \}$$
and it is strictly decreasing otherwise. If $\Gamma$ is positively invariant with respect to this negative gradient flow truncated below level $b$ 
and the minimax value $c$ is strictly larger than $b$, then $f$ has a Palais-Smale sequence at level $c$.
\end{oss}

We end this section discussing some interesting particular cases of the theorem above. First assume $f\in C^{1,1}(\mathcal H)$ is such that $\{f<a\}$ is not connected, say 
$\{f<a\} = A\cup B$ with $A,B$ disjoint non-empty open sets. We may think of $A$ and $B$ as two valleys, consider the set of paths going from one valley to the other
$$\Gamma:=\  \Big \{\text{curves in}\ H\ \text{with one end in}\ A\ \text{and the other in}\ B\Big \}\, ,$$
and define the minimax value $c$ of $f$ on $\Gamma$ as in (\ref{definizionedic}).  Observe that necessarily $c\in [a,+\infty)$, since $\Gamma$ is non empty and each of its elements intersects 
$$H\setminus (A\cup B)\ =\ \{f\geq a\}\, ,$$
so that $c$ is finite. Moreover, $\Gamma$ is positively invariant under the negative gradient flow, since $f$ is decreasing along the orbits of $\phi$. As a particular case of the above theorem we then 
get the celebrated mountain pass theorem of Ambrosetti and Rabinowitz.

\begin{teo}[Mountain pass theorem]
Let $f$ be a $C^{1,1}$-functional on a Riemannian Hilbert manifold $(\mathcal H,g)$ such that the sublevel sets are complete. Suppose that the sublevel $\{f<a\}$ is not connected and define $c$ as in 
(\ref{definizionedic}); then $f$ admits a Palais-Smale sequence at level $c$. In particular, if $f$ satisfies the Palais-Smale condition at level $c$, then $c$ is a critical value for $f$.
\label{mountainpasstheorem}
\end{teo}

If we choose for $\Gamma$ the class of all one-point sets in $\mathcal H$, then $c$ as in (\ref{definizionedic}) is nothing but the infimum of $f$ on $\mathcal H$. 
Therefore, the general minimax principle has as a particular case the following

\begin{cor}
Assume that $f$ is a $C^{1,1}$-functional on a Riemannian Hilbert manifold $(\mathcal H,g)$ such that the sublevel sets are complete. If $f$ is bounded from below and satisfies the Palais-Smale condition at the level $c = \inf f$, then $f$ has a minimizer.
\end{cor}


\chapter{The variational setting}
\label{chapter2}

Let $M$ be a closed connected Riemannian manifold, $Q_0,Q_1\subseteq M$ be closed connected submanifolds and $L:TM\rightarrow \R$ be a Tonelli Lagrangian.
Being $L$ autonomous, the energy $E$ in (\ref{energyfunction}) is constant along the solutions of the Euler-Lagrange equation (\ref{eulerlagrange1}). Therefore, it makes sense to look at 
Euler-Lagrange orbits that satisfy the conormal boundary conditions (\ref{lagrangianformulation1}) and are contained in a given energy level $E^{-1}(k)$. Goal of this chapter will be to 
provide the tools needed to attack this problem.

As already explained in Sections \ref{thelagrangianactionfunctional} and \ref{conormalboundaryconditions}, this can be interpreted as the problem of finding critical points of a suitable functional defined 
on the Hilbert manifold $H^1_Q([0,1],M)$ of $H^1$-paths connecting the submanifolds $Q_0$ and $Q_1$. The ``energy $k$'' condition can be then achieved by considering a slightly different functional, namely the free-time action functional $\A_k$,
on the product manifold 
$$\mathcal M_Q := \ H^1_Q([0,1],M) \times (0,+\infty)\, .$$

This approach brings however several complications, since the manifold $\mathcal M_Q$ is not complete anymore. In this sense, a very careful study of the Palais-Smale sequences for $\A_k$ will be needed. 
Furthermore, the properties of $\A_k$ (as well as those of the Euler-Lagrange flow associated to $L$) depend essentially on $k$ and change drastically when crossing some special energy values, called the \textit{Ma\~n\'e critical values}.

\vspace{4mm}

In Section \ref{thefreetimeactionfunctional} we define the free-time Lagrangian action functional $\A_k$ and discuss its regularity properties. We show that $\A_k\in C^{1,1}(\mathcal M_Q)$ and that $\A_k$ is twice differentiable 
if and only if $L$ is electro-magnetic as in (\ref{magneticlagrangian}). We prove then that critical points of $\A_k$ on $\mathcal M_Q$ correspond to the Euler-Lagrange orbits that satisfy the conormal boundary 
conditions (\ref{lagrangianformulation1}) and are contained in the energy level $E^{-1}(k)$. 

\vspace{1mm}

In Section \ref{palaissmalesequences} we proceed to the study of the Palais-Smale sequences for $\A_k$. We show that Palais-Smale sequences $(x_h,T_h)$ with $T_h\rightarrow 0$ may occur only on connected components of $\mathcal M_Q$ that 
contain constant paths and only at level zero, meaning that necessarily $\A_k(x_h,T_h)\rightarrow 0$. We then prove that Palais-Smale sequences with times bounded and bounded away from zero always admit converging subsequences. The two results 
combined imply that the only Palais-Smale sequences for $\A_k$ that might cause difficulties are the ones for which the times are unbounded. 

\vspace{1mm}

In Section \ref{manecriticalvalues} we recall the definition of the Ma\~n\'e critical values $c(L),c_0(L),c_u(L)$ and briefly discuss their relations with the dynamical and geometric properties of the Euler-Lagrange flow. 
We then move to the definition of the critical value $c(L;Q_0,Q_1)$ which is relevant for our purposes. We show that for all $k<c(L;Q_0,Q_1)$ the action functional $\A_k$ is unbounded from below on any connected 
component of $\mathcal M_Q$, whilst it will turn to be bounded from below on each connected component of $\mathcal M_Q$ for every $k\geq c(L;Q_0,Q_1)$. The latter fact implies that, for all $k>c(L;Q_0,Q_1)$,
the free-time action functional $\A_k$ satisfies the Palais-Smale condition on Palais-Smale sequences with times bounded away from zero; in particular, $\A_k$ satisfies the Palais-Smale condition on the connected components of 
$\mathcal M_Q$ not containing constant paths.


\section{The free-time action functional.}
\label{thefreetimeactionfunctional}

For any given absolutely continuous curve $\gamma:[0,T]\rightarrow M$ we define $x:[0,1]\rightarrow M$ as $x(s):=\gamma(s\, T)$. Throughout the whole work we will identify $\gamma$ with the pair $(x,T)$.

To avoid confusion we will always denote with a \textit{dot} the derivative with respect to $t$ and with a \textit{prime} the derivative with respect to $s$.

Fix a real number $k$, the value of the energy for which we would like to find solutions of the Euler-Lagrange equation (\ref{eulerlagrange1}) satisfying the conormal boundary conditions (\ref{lagrangianformulation1}).
Recall that, since the energy level $E^{-1}(k)$ is compact, up to the modification of $L$ outside it, we may assume the Tonelli Lagrangian $L$ to be electro-magnetic for $\|v\|_q$ large enough. In particular
\begin{eqnarray}
L(q,v) &\geq& \ a \, \|v\|_q^2 - b\, , \ \ \ \ \forall (q,v)\in TM\, , \label{firstinequality} \\ && \nonumber \\ 
d_{vv}L(q,v)[u,u] &\geq& 2a \, \|u\|_q^2 \, \ \ \ \forall (q,v)\in TM\, ,\ \forall u\in T_qM \label{secondinequality}
\end{eqnarray}
for suitable numbers $a >0$, $b\in \R$ and
\begin{equation}
\A_k(x,T) := \ \int_0^T \Big[L(\gamma(t),\dot{\gamma}(t))+k \Big ]\, dt \ = \ T\int_0^1 \Big[ L\Big(x(s),\frac{x'(s)}{T} \Big ) + k\Big ]\, ds \label{freetimeactionfunctional}
\end{equation}
is well-defined for every $x\in H^1([0,1],M)$. Hence, we get a well-defined functional
$$\A_k : H^1([0,1],M) \times (0,+\infty) \longrightarrow \R\, ,$$
called the \textit{free-time action functional}. We denote the space  $H^1([0,1],M)\times (0,+\infty)$ simply with $\mathcal M$. Clearly $\mathcal M$ can be interpreted as the space of Sobolev paths in $M$ 
with arbitrary interval of definition through the identification $\gamma=(x,T)$ above. Furthermore, $\mathcal M$ is a product Hilbert manifold; we endow $\mathcal M$ with the product metric 
\begin{equation}
g_{\mathcal M} := \ g_{H^1}\ + \ dT^2\, ,
\label{productmetric}
\end{equation}
where $g_{H^1}$ is, as in (\ref{riemannianmetric}), the standard metric on $H^1([0,1],M)$ induced by the given Riemannian metric $g$ on $M$. Obviously, $(\mathcal M,g_{\mathcal M})$ is not complete as the factor $(0,+\infty)$ is 
not complete with respect to the Euclidean metric. The following proposition is about the regularity of the free-time action functional $\A_k$.

\begin{prop}
The following hold:
\begin{enumerate}
\item $\A_k \in C^{1,1}(\mathcal M)$ and it has second Gateaux differential at every point.
\item $\A_k$ is twice Fr\'ech\'et differentiable at every point if and only if $L$ is electromagnetic on the whole $TM$; in this case, $\A_k$ is actually smooth.
\end{enumerate}
\label{propregularity}
\end{prop}

\begin{proof}
Statement 2 is an obvious consequence of Proposition  \ref{nonregularity}; observe that condition $(L1)$ is satisfied by $L$ since $L$ is electromagnetic outside a compact set. We also already know from 
Theorem \ref{regularity} that the fixed-time action functional is continuously differentiable on $H^1([0,1],M)$ and its differential is locally Lipschitz continuous and Gateaux-differentiable.
Thus, $\A_k$ is Fr\'ech\'et-differentiable in the $x$-direction, that is there exists the partial differential $d_x\A_k(x,T)$ and 
\begin{equation}
d_x\A_k(x,T)[(\zeta,0)]\ = \int_0^1 \Big [d_qL\Big (x(s),\frac{x'(s)}{T}\Big )[\zeta] + d_vL\Big (x(s),\frac{x'(s)}{T}\Big ) [\zeta'] \Big ]\, ds
\label{differenzialelungox}
\end{equation}
is Lipschitz-continuous and Gateaux-differentiable. On the other hand we have
\begin{eqnarray*}
&& \frac{\A_k(x,T+h)-\A_k(x,T)}{h}\ =\\
&=& \frac{T+h}{h} \int_0^1 \Big [L\Big (x(s),\frac{x'(s)}{T+h}\Big ) + k \Big ] ds \ - \ \frac{T}{h} \int_0^1 \Big [L\Big (x(s),\frac{x'(s)}{T}\Big ) + k \Big ] ds \ =\\
&=& \!\!\! \int_0^1 \! \Big [L\Big (x(s),\frac{x'(s)}{T+h}\Big )+k\Big ]ds \ + \ \frac{T}{h}\int_0^1 \! \Big [ L \Big (x(s),\frac{x'(s)}{T+h}\Big ) - L\Big (x(s),\frac{x'(s)}{T}\Big )\Big ]ds 
\end{eqnarray*}
and then, taking the limit for $h\rightarrow 0$, we get 
\begin{eqnarray}
\frac{\partial \A_k}{\partial T}(x,T) &=& \int_0^1 \Big [L\Big (x(s),\frac{x'(s)}{T}\Big ) + k + T\cdot d_vL\Big (x(s),\frac{x'(s)}{T}\Big ) \Big [-\frac{x'(s)}{T^2}\Big ]\Big )ds\ =\nonumber \\
&=&  \int_0^1 \Big [L\Big (x(s),\frac{x'(s)}{T}\Big ) + k - d_vL\Big (x(s),\frac{x'(s)}{T}\Big )\Big [\frac{x'(s)}{T}\Big ] \Big ] ds  \ =\nonumber \\
&=&  \int_0^1 \Big [k-E\Big (x(s),\frac{x'(s)}{T}\Big ) \Big ] ds\, .  \label{derivatainT}
\end{eqnarray}

Therefore $\A_k$ is Fr\'ech\'et-differentiable in both the $x$ and $T$-direction with continuous partial differentials and hence, by the total differential theorem,
it is continuously Fr\'ech\'et-differentiable at $(x,T)$ with
$$d\A_k(x,T)[(\zeta,H)]\ =\ d_x\A_k(x,T)[(\zeta,0)] \ +\ \frac{\partial \A_k}{\partial T}(x,T)\, H\, .$$

Now we want to prove that the Fr\'ech\'et-differential is locally Lipschitz-continuous and Gateaux-differentiable; thus, we consider the quantity

\begin{eqnarray}
&&  \frac{1}{\delta} \Big [ d\A_k(x+\delta \eta,T+\delta W)-d\A_k(x,T)\Big ]\big [(\zeta,H)\big ]\ =\nonumber \\
&=& \frac{1}{\delta} \Big [d_x\A_k(x+\delta \eta,T+\delta W) - d_x\A_k(x,T)\Big ]\big [(\zeta,0)\big ] + \nonumber \\
&+& \frac{1}{\delta}\Big [\frac{\partial \A_k}{\partial T}(x+\delta\eta,T+\delta W)-\frac{\partial \A_k}{\partial T}(x,T)\Big ]\, H\ =\nonumber \\
&=& \frac{1}{\delta} \Big [d_x\A_k(x+\delta\eta,T+\delta W) -d_x\A_k(x,T+\delta W) \Big ]\big [(\zeta,0)\big ]\ +\label{primaquantita}\\
&+& \frac{1}{\delta} \Big [d_x\A_k(x,T+\delta W)-d_x\A_k(x,T)\Big ]\big [(\zeta,0)\big ]\ + \label{secondaquantita} \\
&+& \frac{1}{\delta} \Big [\frac{\partial \A_k}{\partial T}(x+\delta\eta,T+\delta W)-\frac{\partial \A_k}{\partial T}(x,T+\delta W)\Big ]\, H \ +\label{terzaquantita} \\
&+& \frac{1}{\delta} \Big [\frac{\partial \A_k}{\partial T}(x,T+\delta W)-\frac{\partial \A_k}{\partial T}(x,T)\Big ]\, H \label{quartaquantita}
\end{eqnarray}
The expression in (\ref{primaquantita}) converges for $\delta\rightarrow 0$ to 
$$d^2_{xx}\A_k (x,T)\big [(\zeta,0),(\eta,0)\big ]$$
which we already know from Theorem \ref{regularity} to be a bounded symmetric bilinear form on $H^1([0,1],M)$ (thus on $\mathcal M$). The quantity in (\ref{secondaquantita}) is instead equal to 
$$\frac{1}{\delta} \int_0^1 \!\!\left [ \! \left (\! d_qL\Big (x,\frac{x'}{T\!+\!\delta W}\Big )\! - d_qL\Big (x,\frac{x'}{T}\Big )\! \right )\!\zeta + \! \left (\! d_vL\Big (x,\frac{x'}{T\!+\!\delta W}\Big )\! - d_vL\Big (x,\frac{x'}{T}\Big )\! \right )\!\zeta'\right ]\, ds$$
and  converges for $\delta \rightarrow 0$  by the Lebesgue dominated convergence theorem to 
$$\int_0^1 \left [d^2_{vq} L\Big (x(s),\frac{x'(s)}{T}\Big )\left [\zeta,-\frac{x'(s)}{T^2}\right ] + d^2_{vv}L\Big (x(s),\frac{x'(s)}{T}\Big ) \left [\zeta',-\frac{x'(s)}{T^2}\right ]\right]\, W\, ds$$
which is a bounded bilinear operator on $\mathcal M$. Analogously one can prove that the quantities in (\ref{terzaquantita}) and in (\ref{quartaquantita}) converge for $\delta \rightarrow 0$ to bounded bilinear operators on $\mathcal M$. Therefore,
$d\A_k$ is Gateaux-differentiable at every point $(x,T)$; the locally Lipschitz-continuity follows now from the mean value theorem.
\end{proof}

\vspace{5mm}

Since we want to get solutions of the Euler-Lagrange equation satisfying the conormal boundary conditions (\ref{nonlocalboundaryconditions}), we shall consider the restriction of the free-time action functional to the smooth submanifold 
$$\mathcal M_Q:= \ H^1_Q([0,1],M) \times (0,+\infty)\, ,$$
with $Q=Q_0\times Q_1$ or $Q=\Delta$ diagonal in $M\times M$. In the latter case we call $\A_k|_{\mathcal M_\Delta}$ the \textit{free-period Lagrangian action functional}.
For the sake of simplicity we denote the restriction of $\A_k$ to $\mathcal M_Q$ again with $\A_k$. Observe that $\mathcal M_Q$ is homotopy equivalent to $\Omega_{Q_0,Q_1}(M)$ (cf. Section \ref{ahilbertmanifoldofloops}); 
in particular its connected components are as explained in Lemma \ref{componenticonnesseomegaqoq1}. Also, Proposition \ref{regularity} implies the following

\begin{teo}
A curve $\gamma=(x,T)$ is a (smooth) solution of (\ref{eulerlagrange1}) satisfying the conormal boundary conditions (\ref{nonlocalboundaryconditions}) and with energy $\, E(\gamma,\dot{\gamma})\equiv k\, $ 
if and only if $(x,T)$ is a critical point of the free-time action functional $\A_k$ restricted to $\mathcal M_Q$.
\label{equivalence}
\end{teo}

\begin{proof}
The pair $(x,T)$ is a critical point for $\A_k$ if and only if 
$$d\A_k(x,T)[(\zeta,H)]\ =\ 0$$ 
for any choice of $(\zeta,H)$. From Theorem \ref{criticalpoints} it follows that the condition 
$$d_x\A_k(x,T)[(\zeta,0)]\ =\ 0$$ 
is equivalent to $\gamma(t):=x(t/T)\in H^1_Q([0,T],M)$ to be a solution of (\ref{eulerlagrange1}) satisfying the conormal boundary conditions (\ref{nonlocalboundaryconditions}). Furthermore, using (\ref{derivatainT}), we get that 
$$0\ =\ \frac{\partial \A_k}{\partial T}(x,T)\ =\ \frac{1}{T}\int_0^T \Big [k-E(\gamma(t),\dot{\gamma}(t))\Big ]\, dt$$
and hence $E(\gamma,\dot{\gamma})\equiv k$, since the energy is constant along $\gamma$.
\end{proof}

\vspace{5mm} 

The Hilbert manifold $\mathcal M_Q$ is clearly not complete. Therefore it is useful to know whether sublevel sets of the free-time action functional $\A_k$ are complete or not. It turns out hat the completeness of sublevel sets of $\A_k$
neither depends on the value $k$ of the energy nor on the topological property of $Q_0$ and $Q_1$, but only on the fact that the submanifolds intersect or not. This is in strong contrast with what happens for 
the geometric and analytical properties of $\A_k$, as we will see later on.

\begin{lemma}
If $\, Q_0\cap Q_1=\emptyset\, $ then the sublevels 
$$\Big \{ (x,T) \in \mathcal M_Q \ \Big | \ \A_k(x,T) \leq c\Big \}$$
are complete.
\label{completeness0}
\end{lemma}

\begin{proof}
By (\ref{firstinequality}) we have the chain of inequalities
\begin{eqnarray}
\A_k(x,T) &=& T \int_0^1 \Big  [L\Big (x(s),\frac{x'(s)}{T}\Big ) + k\Big ] ds  \ \geq \  T\int_0^1 \Big [ a \, \frac{\|x'(s)\|^2}{T^2} - b + k \Big ] ds  \nonumber \\
                &=& \frac{a}{T} \int_0^1 \|x'(s)\|^2 \, ds +T(k-b) \ \geq \ \frac{a}{T} \, l(x)^2 + T(k-b) \label{stimaazione}
\end{eqnarray}
where $l(x)$ denotes the length of the path $x$. Since $Q_0\cap Q_1=\emptyset$, the length of any path connecting $Q_0$ to $Q_1$ is bounded away from zero 
by a suitable positive constant. Therefore, $T$ is bounded away from zero on
$$\Big \{(x,T) \in \mathcal M_Q\  \Big | \ \A_k(x,T) \leq c\Big \}$$
for any $c\in \R$, proving the statement.
\end{proof}

\vspace{5mm}

If $\, Q_0\cap Q_1\neq \emptyset\, $ then the length of paths connecting $Q_0$ to $Q_1$ is not any more bounded away from zero. In this case we have to distinguish 
between connected components of $\mathcal M_Q$ that do, respectively do not, contain constant paths.

\begin{lemma}
If $Q_0\cap Q_1\neq \emptyset$ then:

\begin{enumerate}
\item The sublevel sets of $\A_k$  in each connected component of $\mathcal M_Q$ that does not contain any constant path are complete.
\item If $(x_h,T_h)$ is such that $T_h\rightarrow 0$, then
\begin{equation}
\liminf_{h\rightarrow +\infty}\ \A_k(x_h,T_h)\ \geq \ 0\, .
\label{liminf}
\end{equation}
\end{enumerate}
\label{completeness1}
\end{lemma}

\begin{proof}
The proof of the first statement is analogous to that of Lemma \ref{completeness0} above. Inequality (\ref{stimaazione}) also proves the second statement. In fact, if $\, T_h\rightarrow 0\, $ then 
$$\, T_h(k-b)\ \longrightarrow \ 0\, , \ \ \ \ \text{for} \ \  h\longrightarrow +\infty$$
and hence the action $\A_k(x_h,T_h)$ is eventually bigger than $-\epsilon$, for arbitrary $\epsilon>0$.
\end{proof}

\vspace{4mm}

We end this section studying the possible sources of non-completeness of the negative gradient flow of $\A_k$ on $\mathcal M_Q$.  Up to replacing $-\nabla \A_k$ by
$$- \frac{\nabla \A_k}{\sqrt{1+\|\nabla \A_k\|^2}}\, $$
we may assume the negative gradient flow to be complete on every connected component of $\mathcal M_Q$ that does not contain constant paths. A similar statement when looking for periodic orbits (in a even more
general setting) can be found in Section \ref{ageneralizedpseudogradient}. Also, on the connected components of $\mathcal M_Q$ that do contain constant paths, 
the only source of incompleteness is represented by flow-lines for which $T(\cdot) \rightarrow 0$ in finite time. The next lemma ensures that, for such flow lines, $\A_k$ necessarily goes to zero.

\vspace{1mm}

\begin{lemma}
Let $\big (x(\cdot),T(\cdot)\big ):[0,\sigma^*)\rightarrow \mathcal M_Q$ be a negative gradient flow-line with 
$$\liminf_{\sigma\rightarrow \sigma^*}\ T(\sigma)\ = \ 0\, .$$
\noindent Then 
$$\lim_{\sigma \rightarrow \sigma^*} \ \A_k\big (x(\sigma),T(\sigma)\big )\ =\ 0\, .$$
\label{convergenza}
\end{lemma}

\begin{proof}
The proof is analogous to \cite[Lemma 3.3]{Abb13}, where the case of periodic orbits is considered.
Since both $E$ and $L$ are quadratic in $v$ for $\|v\|_q$ large, we have
$$E(q,v)\ \geq \ c_0 \, L(q,v)-c_1$$
for some $c_0>0$ and $c_1\in \R$. Therefore from (\ref{derivatainT}) it follows that
\begin{eqnarray*}
\frac{\partial \A_k}{\partial T} (x,T) &=& \frac{1}{T}\int_0^T \Big [k-E(\gamma(t),\dot{\gamma}(t))\Big ]\, dt \ \leq \\
&\leq& \frac{1}{T} \int_0^T \Big [k-c_0\, L(\gamma(t),\dot{\gamma}(t))+c_1\Big ]\, dt\ =\\
&=& (c_0+1) k + c_1 - \frac{c_0}{T} \, \A_k(x,T)
\end{eqnarray*}
and hence 
\begin{equation}
\A_k(x,T) \ \leq \ \frac{T}{c_0}\, \Big [(c_0+1) k + c_1 - \frac{\partial \A_k}{\partial T}(x,T)\Big ]\ =\ \frac{T}{c_0} \, \Big [C-\frac{\partial \A_k}{\partial T}(x,T)\Big ]\, ,
\label{stima}
\end{equation}
where $C$ is a suitable constant. By assumption, there is a sequence $\sigma_h\uparrow \sigma^*$ with 
$$T'(\sigma_h)\ \leq \ 0\, , \ \ \ \  T(\sigma_h)\ \longrightarrow \ 0 \, .$$
Since $\sigma\mapsto (x(\sigma),T(\sigma))$ is a negative gradient flow-line, we have
$$0\ \geq \ T'(\sigma_h)\ =-\frac{\partial \A_k}{\partial T}\big (x(\sigma_h),T(\sigma_h)\big )$$ 
and hence 
$$\A_k\big (x(\sigma_h),T(\sigma_h)\big ) \ \leq \ \frac{T(\sigma_h)}{c_0}\, \Big [C-\frac{\partial \A_k}{\partial T}\big ( x(\sigma_h),T(\sigma_h)\big )\Big ] \ \leq \ \frac{C}{c_0}\, T(\sigma_h)\, .$$
Since $T(\sigma_h)\rightarrow 0$, from the inequality above we deduce that
$$\limsup_{h\rightarrow +\infty} \ \A_k\big (x(\sigma_h),T(\sigma_h)\big )\ \leq \ 0\, .$$ 

The assertion follows now from statement 2 of Lemma \ref{completeness1} and from the monotonicity of the function $\sigma \mapsto \A_k(x(\sigma),T(\sigma))$.
\end{proof}


\section{Palais-Smale sequences.}
\label{palaissmalesequences}

When looking for critical points of a given functional defined on a Hilbert manifold, it is natural to consider Palais-Smale sequences as a ``source of critical points'', being their limit points critical points of the considered functional. However, as already observed in Section \ref{theminimaxprinciple}, it is in general not true that Palais-Smale sequences admit converging subsequences. 
Therefore, it is worth to look for necessary and sufficient conditions for a Palais-Smale sequence to admit converging subsequences. 

In this section we investigate this problematic in the case we are interested in, namely when the Hilbert manifold is the space $\mathcal M_Q$ of paths connecting the submanifolds $Q_0$ and $Q_1$ with arbitrary interval of 
definition and the functional is the free-time action functional $\A_k$. Palais-Smale sequences with times going to zero are a possible source of non-completeness but, as Lemma \ref{completeness0} states, they might occur only in 
connected components of $\mathcal M_Q$ that contain constant paths. The next lemma ensures also that such Palais-Smale sequences may appear only at level zero. 

This property turns out to be particularly fruitful when looking for global minimizers or for minimax critical points. More precisely, if one is able to prove that the infimum of $\A_k$ or a certain minimax value for $\A_k$ is finite and not zero, 
then the related Palais-Smale sequences have automatically times bounded away from zero. 

Finally, Lemma \ref{yesmodification} combined with Lemma \ref{lemmalimitatezza} shows that the only Palais-Smale sequences which may cause troubles are those for which the times diverge.

\begin{lemma}
Let $\gamma_h=(x_h,T_h)$ be a Palais-Smale sequence at level $c\in \R$ for $\A_k$ such that $T_h\rightarrow 0$. Then necessarily $c=0$.
\label{yesmodification}
\end{lemma}

\begin{proof}
First  we prove that
\begin{equation}
\int_0^{T_h} \|\dot \gamma_h(t)\|^2 \, dt \ =\ \mathcal O(T_h)\, ,\ \ \ \ \text{for} \ h\rightarrow +\infty\, .
\label{stimacont}
\end{equation}
Being $(x_h,T_h)$ a Palais-Smale sequence for $\A_k$, we have 
$$\|d\A_k(x_h,T_h)\|\ =\ o(1)\, ,\ \ \ \ \text{for}\ h\rightarrow +\infty$$
where $\|\cdot\|$ denotes the dual norm. In particular, using (\ref{derivatainT}) we get that 
$$\left | d\A_k(x_h,T_h)\Big [\frac{\partial}{\partial T}\Big ] \right | \ =\ \left | \frac{\partial \A_k}{\partial T}(x_h,T_h) \right |\ =\ \left | \frac{1}{T_h} \int_0^{T_h} \Big [E\big (\gamma_h(t),\dot \gamma_h(t)\big ) - k \Big ]\, dt \right |\ = \ o(1)$$
and hence 
\begin{equation}
\alpha_h:= \ \frac{1}{T_h}\int_0^{T_h} \Big [E\big (\gamma_h (t),\dot \gamma_h(t)\big ) -k\Big ]\, dt \ \longrightarrow \ 0\, .
\label{convenergia}
\end{equation}

Now by assumption $E$ is quadratic for $\|v\|_q$ large and hence $E(q,v)\geq a' \|v\|_q^2 - b',$ for some $a'>0$ and $b'\in \R$. Using this in (\ref{convenergia}) we get that 
$$\alpha_h \ =\ \frac{1}{T_h}\int_0^{T_h} \Big [E\big (\gamma_h(t),\dot \gamma_h(t)\big )-k\Big ]\, dt\ \geq \ \frac{1}{T_h} \int_0^{T_h} \Big [a'\, \|\dot \gamma_h(t)\|^2 - b'-k\Big ]\, dt$$
and hence 
$$\int_0^{T_h} \|\dot \gamma_h(t)\|^2 \, dt \ \leq \ \frac{T_h}{a'} \, \big [\alpha_h + b'+k\big ]$$
which implies (\ref{stimacont}). Since also $L$ is quadratic for $\|v\|_q$ large we have 
$$a\, \|v\|_q^2 - b \ \leq \ L(q,v)\ \leq \ \tilde{a}\, \|v\|_q^2 + \tilde{b}$$ 
for some constants $a,\tilde{a} >0$ and $b,\tilde{b}\in \R$. The first inequality implies 
\begin{eqnarray*}
\A_k(x_h,T_h) &=&  \int_0^{T_h} \Big [ L(\gamma_h(t),\dot \gamma_h(t))+k \Big ]\, dt \ \geq \\ 
                         &\geq& a \int_0^{T_h} \|\dot \gamma_h(t)\|^2\, dt + T_h(k-b) \ =\ O(T_h)
\end{eqnarray*}
while the second
\begin{eqnarray*}
\A_k(x_h,T_h) &=&  \int_0^{T_h} \Big [ L(\gamma_h(t),\dot \gamma_h(t))+k \Big ]\, dt \ \leq \\ 
                        &\leq& \tilde{a} \int_0^{T_h} \|\dot \gamma_h(t)\|^2\, dt + T_h(k+\tilde{b}) \ =\ O(T_h)
\end{eqnarray*}
and hence obviously $\A_k(x_h,T_h) \rightarrow 0$.
\end{proof}

\vspace{5mm}

The following lemma ensures the existence of converging subsequences for any Palais-Smale sequence with times bounded and bounded away from zero. The proof is analogous (with some minor adjustments) to the one of 
\cite[Proposition 3.12]{Con06} (see also \cite[Lemma 5.3]{Abb13}), where the case of periodic orbits is considered. 

\begin{lemma}
Let $(x_h,T_h)$ be a Palais-Smale sequence at level $c\in \R$ for $\A_k$ in some connected component of $\mathcal M_Q$ with $\, 0<T_*\leq T_h\leq T^*<+\infty$. Then, $(x_h,T_h)$ is compact in $\mathcal M_Q$, meaning that it admits a converging subsequence. 
\label{lemmalimitatezza}
\end{lemma}

\begin{proof}
From (\ref{firstinequality}) it follows that 
\begin{eqnarray*}
c+o(1) &\geq& \A_k(x_h,T_h)\ \geq \  T_h \int_0^1 \Big [a \, \frac{\|x_h'(s)\|^2}{T_h^2} - (b-k)\Big ]\, ds \ = \\
&=& \frac{a}{T_h} \int_0^1 \|x_h'(s)\|^2 \, ds - T_h\, (b-k)\ \geq \\
&\geq& \frac{a}{T^*} \, \|x'_h\|_2^2 - T^*\, |b-k|
\end{eqnarray*}
where $\|\cdot \|_2$ denotes the $L^2$-norm with respect to the fixed Riemannian metric $g$ on $M$. Therefore $\|x'_h\|_2$ is uniformly bounded 
$$\|x'_h\|_2^2 \ \leq \ \frac{T^*}{a} \, \Big ( c+o(1)+T^*\, |b-k|\Big )$$ 
and hence $\{x_h\}$ is $1/2$-equi-H\"older-continuous 
$$\text{dist}\, \big (x_h(s),x_h(s')\big ) \ \leq \ \int_s^{s'} |x'_h(r)|\, dr \ \leq \ |s-s'|^{1/2} \|x'_h\|_2\, .$$

By the Ascoli-Arzel\'a theorem, up to subsequences $x_h$ converges uniformly to some $x\in C_Q([0,1],M)$; in particular, $x_h$ eventually belongs to the image of the parametrization $\phi_*$ induced by a smooth time-depending local coordinate system 
as in the definition of the atlas for $H^1_Q([0,1],M)$ (recall that the image of this parametrization is $C^0$-open). Then we have $x_h=\phi_*(\zeta_h)$ for all $h\in \N$, where $\{\zeta_h\} \subseteq H^1_W([0,1],B_r)$ is a Palais-Smale sequence for the functional
$$\widehat{\A}(\zeta,T) := \ T\int_0^1 \widehat{L} \Big ( s, \zeta(s), \frac{\zeta'(s)}{T}\Big )\, ds$$ 
with respect to the standard Hilbert product on $H^1_W([0,1],\R^n)$, where the Lagrangian 
$$\widehat{L}(s,q,v)\ \in C^\infty\big ([0,1]\times B_r\times \R^n\big )$$
is obtained by pulling back $L+k$ using $\phi$. Moreover, $\zeta_h$ converges uniformly and, since $\|\zeta_h'\|_2$ is bounded, weakly in $H^1$ to some $\zeta\in H^1_W([0,1],M)$; 
we must prove that this convergence is actually strong in $H^1$. Since $\widehat{L}$ is electromagnetic for $|v|$ large,
\begin{eqnarray}
\big |d_q\widehat{L}(t,q,v)\big | &\leq & C\, \big (1+|v|^2\big )\label{primastima1}\\
\big |d_v\widehat{L}(t,q,v)\big | &\leq & C\, \big  (1+|v|\, \big )\label{secondastima1}
\end{eqnarray}
for a suitable constant $C>0$. Since $(\zeta_h,T_h)$ is a Palais-Smale sequence with $T_h$ bounded away from zero and $\zeta_h$ is bounded in $H^1$, we also have 

\begin{eqnarray*}
o(1) &=& d\widehat{\A}(\zeta_h,T_h)[(\zeta_h-\zeta,0)]\ =\\
&=& T_h \int_0^1 d_q\widehat{L}\Big (s,\zeta_h,\frac{\zeta_h'}{T_h}\Big ) \big [\zeta_h - \zeta \big ]\, ds \ + \ T_h \int_0^1 d_v\widehat{L}\Big (s,\zeta_h,\frac{\zeta'_h}{T_h}\Big ) \Big [\frac{\zeta_h'-\zeta'}{T_h}\Big ]\, ds\, .
\end{eqnarray*}

By the bound in (\ref{primastima1}), the boundedness of $T_h$ and $\|\zeta_h'\|_2$ and the uniform convergence of $\zeta_h$ to $\zeta$, the first integral in the last expression is infinitesimal; indeed 
\begin{eqnarray*}
\left  | T_h \int_0^1 d_q \widehat{L} \Big (s,\zeta_h,\frac{\zeta'_h}{T_h}\Big ) \big [\zeta_h-\zeta\big ]\, ds\right | &\leq&  CT_h \int_0^1 \left (1 + \frac{|\zeta_h'|^2}{T_h^2}\right) \big \|\zeta_h-\zeta\big \|_\infty \, ds \\
&=&  \big \|\zeta_h-\zeta\big \|_\infty \, \left ( CT_h + \frac{C}{T_h}\big \|\zeta_h'\big \|_2^2\right )\, .
\end{eqnarray*}
It follows that 
\begin{equation}
\int_0^1  d_v\widehat{L} \Big (s,\zeta_h,\frac{\zeta_h'}{T_h}\Big ) \Big [\frac{\zeta_h'-\zeta'}{T_h}\Big ]\, ds \ = \ o(1)\, .
\label{integraleinfinitesimo}
\end{equation}
Moreover, by the fiberwise convexity of $\widehat{L}$, we have that 
$$d_{vv}\widehat{L}(s,q,v) \big [u,u\big ] \ \geq \ \delta \, |u|^2\, ,  \ \ \ \ \forall \ (s,q,v)\in [0,1]\times B_r \times \R^n\, ,\ \ \forall \ u\in \R^n\, ,$$
where $\delta>0$ is a suitable number. It follows that 
\begin{eqnarray*}
d_v\widehat{L} \Big ( s, \!\!\!\! &\zeta_h&\!\!\!\!\! ,\frac{\zeta_h'}{T_h}\Big ) \Big [\frac{\zeta_h'-\zeta'}{T_h}\Big ] - d_v \widehat{L}\Big (s,\zeta_h,\frac{\zeta'}{T_h}\Big )\Big [\frac{\zeta_h'-\zeta'}{T_h}\Big ]\ =\\
&=& \int_0^1 d_{vv}\widehat{L} \Big (s,\zeta_h, \frac{\zeta'}{T_h} + \sigma \cdot \frac{\zeta_h'-\zeta'}{T_h}\Big ) \Big [\frac{\zeta_h'-\zeta'}{T_h}, \frac{\zeta_h'-\zeta'}{T_h}\Big ]\, d\sigma \ \geq \ \frac{\delta}{T_h^2}\cdot \big |\zeta_h'-\zeta'\big |^2\, .
\end{eqnarray*}
Now integrating this inequality over $s\in [0,1]$ and using (\ref{integraleinfinitesimo}) we obtain
$$o(1) - \int_0^1 d_v\widehat{L}\Big (s,\zeta_h,\frac{\zeta'}{T_h}\Big ) \Big[\frac{\zeta_h'-\zeta'}{T_h}\Big ]\, ds \ \geq \ \frac{\delta}{T_h^2}\cdot \big \|\zeta_h'-\zeta'\big \|_2^2\, .$$
Since $\, \zeta_h'-\zeta' \rightharpoonup 0$ and since by the bound in (\ref{secondastima1}) the sequence 
$$d_v\widehat{L}\Big (s,\zeta_h,\frac{\zeta'}{T_h}\Big )$$ 
converges strongly in $L^2$, the integral on the left-hand side of the above inequality is infinitesimal and hence we conclude that $\zeta_h$ converges to $\zeta$ strongly.
\end{proof}


\section{Ma\~{n}\'e critical values.}
\label{manecriticalvalues}

The following numbers should be interpreted as energy levels and mark important dynamical and geometric changes for the Euler-Lagrange flow induced by the Tonelli Lagrangian $L$. The reader may take a look at the 
expository article \cite{Abb13} for a survey on the relevance of these energy values and on their relation with the geometric and dynamical properties of the Euler-Lagrange flow. 
First, let us define the \textit{Ma\~n\'e critical value} associated to $L$ as 
\begin{equation}
c(L) := \ \inf \Big \{ k\in \R \ \Big |\ \A_k(\gamma)\geq 0 \, ,\ \forall \ \gamma \text{ closed loop} \Big \}\, .
\label{c(L)}
\end{equation}
      
We refer to \cite{CI99} and references therein for its relation with the geometric and dynamical properties of the system and for other equivalent definitions. Second, we recall the definition of the \textit{Ma\~n\'e critical value of the Abelian cover} 
\begin{equation}
c_0(L) :=\ \inf \Big \{ k\in \R \ \Big |\ \A_k(\gamma)\geq 0 \, ,\ \forall \ \gamma \text{ closed loop homologous to zero} \Big \}\, .
\label{c0(L)}
\end{equation}

This is the relevant energy value, for instance, when trying to use methods coming from Finsler theory. In fact, for every 
$k>c_0(L)$ the Euler-Lagrange flow restricted to the energy level $E^{-1}(k)$ is conjugated with the geodesic flow defined by a suitable Finsler metric. We refer to \cite{Abb13} for the details. 
The value $c_0(L)$ is also related to the existence of periodic orbits for exact magnetic flows (i.e. Euler-Lagrange flows associated to a Lagrangian $L$ as in (\ref{magneticlagrangian}) with $V\equiv 0$)
on surfaces which are local minimizers of the free-period Lagrangian action functional, as explained in \cite{CMP04} and in \cite{AMMP14}.
When looking for periodic orbits, the energy value which turns out to be relevant for the properties of the free-period action functional (cf. \cite{Con06} and \cite{Abb13}) is however the so-called \textit{Ma\~n\'e critical value of the universal cover}
\begin{equation}
c_u(L) :=\ \inf \Big \{ k\in \R \ \Big |\ \A_k(\gamma)\geq 0 \, ,\ \forall \ \gamma \text{ closed contractible loop} \Big \}\, .
\label{cu(L)}
\end{equation}
We also define 
\begin{equation}
e_0(L):= \ \max_{q\in M} \ E(q,0)
\label{e0(L)}
\end{equation}
to be the maximum of the energy on the zero section of $TM$. The topology of the energy level sets changes when crossing the value $e_0(L)$. In fact, for any $k>e_0(L)$, the energy level sets $E^{-1}(k)$ have all the same topology,
namely of a sphere bundle over $M$. This is instead false for $k<e_0(L)$, being the projection $E^{-1}(k) \rightarrow M$  not surjective any more. We will get back to these critical values in Chapters \ref{chapter5} and \ref{chapter6} 
when we will focus on the existence of periodic orbits. Notice that 
\begin{equation}
\min \ E \ \leq \ e_0(L)\ \leq\ c_u(L)\ \leq \ c_0(L)\ \leq \ c(L)\, .
\label{relazionemane}
\end{equation}

When $L$ is electro-magnetic as in (\ref{magneticlagrangian}), $\min E$ is the minimum of the scalar potential $V$ and $e_0(L)$  its maximum. When the magnetic potential $\theta$ vanishes we have
$$e_0(L)\ =\ c_u(L)\ =\ c_0(L) \ = \ c(L)\, ,$$ 
but in general the inequalities in (\ref{relazionemane}) are strict. See for instance \cite{Man97} (or Section \ref{counterexamples}) for an example where $e_0(L)< c_0(L)<c(L)$ and \cite{PP97} for an example where $c_u(L)<c_0(L)$. 
The values $c_u(L)$ and $c_0(L)$ clearly coincide when $\pi_1(M)$ is abelian; more generally, they coincide whenever $\pi_1(M)$ is ameanable (cf. \cite{FM07}).

When the fundamental group of $M$ is rich, there are other Ma\~{n}\'e critical values, which are associated to the different coverings of $M$. We now show which one is relevant for our purposes. 
Given a covering $p:M_1\rightarrow M$, consider the lifted Lagrangian
$$L_1 := dp\circ L:TM_1 \longrightarrow \R$$ 
and the associated critical value $c(L_1)$ as defined in (\ref{c(L)}). 

\begin{lemma}
$c(L_1)\leq c(L)$. If the covering is finite, then $\, c(L_1)=c(L)$.
\label{coveringandmane}
\end{lemma}

\begin{proof}
The first part of the statement is obvious because closed curves on $M_1$ project to closed curves in $M$. Now suppose by contradiction 
that the strict inequality holds and pick $k\in \R$ such that $c(L_1)< k < c(L)$. By definition there exists a closed curve $\gamma$ in $M$ with negative 
$(L+k)$-action and since $M_1$ is a finite covering of $M$ some iterate of $\gamma$ lifts to a closed curve on $M_1$ with negative $(L_1+k)$-action.
\end{proof}

\vspace{5mm}

It is well known that coverings correspond to normal subgroups of $\pi_1(M)$, i.e. for any covering $p:M_1\rightarrow M$ there is a unique normal subgroup $H < \pi_1(M)$ with
$$M_1 \ \cong \ \bigslant{\widetilde{M}}{H}\, ,$$
where $\widetilde M$ denotes the universal cover of $M$. We denote the Ma\~{n}\'e critical value $c(L_1)$ of the lifted Lagrangian by
$$ c(L;H):= \ c(L_1)\, .$$

\begin{lemma}
Let $H,K< \pi_1(M)$ be two normal subgroups; then 
$$c(L; \langle H,K\rangle ) \ =\ \max \big \{ c(L;H),c(L;K)\big \}\, ,$$
where $\langle H,K \rangle$ denotes the normal subgroup generated by $H$ and $K$.
\end{lemma}

\begin{proof}
Since $H < \langle H,K\rangle$ is a normal subgroup, we have a covering 
$$p: \bigslant{\widetilde{M}}{H} \longrightarrow \bigslant{\widetilde{M}}{\langle H,K\rangle}$$ 
and hence by the Lemma above $c(L;H)\leq c(L;\langle H,K\rangle)$. The same holds clearly also when considering $K$ instead of $H$ and hence we get 
$$\max \big \{ c(L;H),c(L;K)\big \} \ \leq \ c(L;\langle H,K\rangle )\, .$$
Conversely,  let $k<c(L;\langle H,K\rangle )$. By definition there exists 
$$\gamma \ =\ \alpha_1 \# \beta_1 \# ... \# \alpha_n \# \beta_n$$ 
with $\alpha_i\in H$, $\beta_i\in K$ for all $i=1,...,n$, such that $\A_k(\gamma)<0$. It follows
$$\A_k(\gamma)\ =\ \A_k(\alpha_1)+\A_k(\beta_1) + ... +\A_k(\alpha_n) + \A_k(\beta_n) \ < \ 0\, .$$ 

In particular there is one loop, say $\alpha_1$, such that $\A_k(\alpha_1)<0$; hence, by definition we have  $k<c(L;H)$. This implies the opposite inequality.
\end{proof}

\vspace{5mm}

Let $Q_0,Q_1\subseteq M$ be two closed connected submanifolds and set $Q:=Q_0\times Q_1$. Clearly $\mathcal M_Q$ is homotopically equivalent to $H^1_Q([0,1],M)$  and hence to $\Omega_{Q_0,Q_1}(M)$. Lemma 
\ref{componenticonnesseomegaqoq1} implies then that the connected components of $\mathcal M_Q$ are given by
$$\pi_0(\mathcal M_Q)\ \cong \ \pi_0(\Omega_{Q_0,Q_1}(M))\ \cong \ \bigslant{\pi_1(M,q_0)}{\sim_{Q_0,Q_1}}$$
where $\sigma \sim_{Q_0,Q_1} \sigma'$ iff there exist $\alpha \in i_*(\pi_1(Q_0,q_0))$, $\beta\in i_*(\pi_1(Q_1,q_1))$ such that 
$$\sigma' \ \sim \ (\eta^{-1}\# \beta \# \eta) \# \sigma \# \alpha\, .$$

Here $i:Q_0\hookrightarrow M$, $i:Q_1\rightarrow M$ denote the inclusion maps, while $\eta:[0,1]\rightarrow M$ is any path connecting $q_0\in Q_0$ to $q_1\in Q_1$. We denote by
\begin{equation}
H_0 :=\ N \big \langle i_*(\pi_1(Q_0,q_0))\big \rangle\, ,\ \ \ \ H_1 :=\ N \big \langle i_*(\pi_1(Q_1,q_1))\big \rangle 
\end{equation}
the smallest normal subgroups in $\pi_1(M)$ which contain $i_*(\pi_1(Q_0))$, $i_*(\pi_1(Q_1))$ respectively. Now we want to understand when the action functional $\A_k$ is bounded from below on each connected component of 
$\mathcal M_Q$. Suppose that there exists a loop $\delta$ freely-homotopic (i.e. homotopic via a homotopy in which 
the base point of each loop is free to vary) to an element in $i_*(\pi_1(Q_0,q_0))$ and with negative $(L+k)$-action. Without loss of generality we may suppose that the loop $\delta$ is based at $q_0$, as otherwise we can choose 
a path $\eta$ from $q_0$ to $\delta(0)$ and since $\A_k(\delta)<0$ there is $n\in \N$ such that 
$$\A_k( \eta^{-1}\# \delta^n \# \eta) \ <\ 0\, ,$$
with $\eta^{-1}\# \delta^n \# \eta$ free-homotopic to an element of $i_*(\pi_1(Q_0,q_0))$. Thus we suppose there exists $\delta\in \pi_1(M,q_0)$ with negative $(L+k)$-action and freely-homotopic to an element of $i_*(\pi_1(Q_0,q_0))$; 
this condition may be better restated by saying that 
$$\delta \in \ H_0:= N\big \langle i_*(\pi_1(Q_0,q_0))\big \rangle$$ 
and $\A_k(\delta) < 0$. Since $\delta \in H_0$ there exists $\eta\in \pi_1(M,q_0)$ such that 
$$\eta^{-1} \# \delta \#\eta  \ \in i_*(\pi_1(Q_0,q_0))$$ 
and hence we get that for any $\sigma \in \mathcal M_Q$ the path $\sigma \# (\eta^{-1}\# \delta^n \# \eta)$ lies in the same connected component as $\sigma$ and 
$$\lim_{n\rightarrow +\infty} \ \A_k\big (\sigma \# \eta^{-1}\# \delta^n \# \eta \big )\ = \ -\infty\, .$$ 

Therefore if such a loop $\delta$ exists, that is if $k<c(L;H_0)$, then the free-time action functional $\A_k$ is unbounded from below on each connected component of $\mathcal M_Q$. Clearly the same 
holds when considering $Q_1$ instead of $Q_0$. Therefore we define the \textit{Ma\~n\'e critical value of the pair} $Q_0,Q_1$ as 
\begin{equation}
c(L;Q_0,Q_1) := \ c(L;\langle H_0,H_1\rangle)\ = \ \max \big \{c(L;H_0),c(L;H_1)\big \} \, .
\label{clh0h1}
\end{equation}
We can sum up the discussion above in the following

\begin{lemma}
For every $k< c(L;Q_0,Q_1)$, the free-time action functional $\A_k$ is unbounded from below on each connected component of $\mathcal M_Q$.
\label{unboundedness}
\end{lemma}

The next lemma ensures instead that, if we consider $k\geq c(L;Q_0,Q_1)$, then the free-time action functional $\A_k$ is bounded from below on each connected component of $\mathcal M_Q$.
The proof is analogous to the one of \cite[Lemma 4.1]{Abb13}, where the case of periodic orbits is treated and $c(L;Q_0,Q_1)$ is replaced by $c_u(L)$.

\begin{lemma}
For every $k\geq c(L;Q_0,Q_1)$ the free-time action functional $\A_k$ is bounded from below on every connected component of $\mathcal M_Q$.
\label{boundedness}
\end{lemma}

\begin{proof}
Consider $\sigma:[0,T]\rightarrow M$ in some connected component of $\mathcal M_Q$ and
\begin{equation}
M_1:=\bigslant{\widetilde{M}}{\langle H_0,H_1\rangle}\  \stackrel{p}{\longrightarrow}\  M\, ,
\label{rivestimento}
\end{equation}
where $\widetilde{M}$ is the universal cover. Denote by $\sigma_1$ the lift of $\sigma$ to $M_1$; we lift the metric of $M$ to $M_1$ 
and notice that, having fixed the connected component of $\mathcal M_Q$, 
$$\text{dist}\, (\sigma_1(0),\sigma_1(T))$$
is uniformly bounded. Therefore, there exists a path $\eta_1:[0,1]\rightarrow M_1$ which joins $\sigma_1(T)$ with $\sigma_1(0)$ and has uniformly bounded action 
$$\tilde{\A}_k(\eta_1)\ =\ \int_0^1 \Big [L_1(\eta_1(t),\dot \eta_1(t))+k \Big ]\, dt \ \leq \ C\, ,$$
where $L_1$ denotes the lifted Lagrangian on $M_1$. If $\eta:=p\, \circ \, \eta_1$, then the juxtaposition $\sigma \# \eta\in \langle H_0,H_1\rangle$ and, since 
by assumption $k\geq c(L;Q_0,Q_1)$, we get 
$$0\ \leq \ \A_k(\sigma \# \eta) \ =\ \A_k(\sigma) + \A_k(\eta)\ =\ \A_k(\sigma) + \tilde{\A}_k(\eta_1)\ \leq \ \A_k(\sigma) + C$$
from which we deduce that $\A_k(\sigma) \geq -C$.
\end{proof}

\begin{cor}
If $k>c(L;Q_0,Q_1)$, then any Palais-Smale sequence $(x_h,T_h)$ for $\A_k$  in a given connected component of $\mathcal M_Q$ with $\, T_h\geq T_*>0$ is compact. As a corollary, the free-time Lagrangian action functional $\A_k$ 
satisfies the Palais-Smale condition on every connected component of $\mathcal M_Q$ that does not contain constant paths.
\label{compattezza}
\end{cor}

\begin{proof}
In virtue of Lemma \ref{lemmalimitatezza} it is enough to show that the $(T_h)$'s are uniformly bounded from above. Since 
$$\A_k(x,T)\ =\ \A_{c(L;Q_0,Q_1)} (x,T)+\Big (k-c(L;Q_0,Q_1)\Big )T$$
for any $(x,T)\in \mathcal M_Q$, the period 
$$T_h\ =\ \frac{1}{k-c(L;Q_0,Q_1 )} \, \Big [\A_k(x_h,T_h)-\A_{c(L;Q_0,Q_1)}(x_h,T_h)\Big ]$$ 
is clearly uniformly bounded from above, being $\A_k$ bounded on the Palais-Smale sequence $(x_h,T_h)$ and being $\A_{c(L;Q_0,Q_1)}(x_h,T_h)$ bounded from below 
by Lemma \ref{boundedness}. 

The second assertion follows trivially from the first one, since on connected components of $\mathcal M_Q$ that do not contain constant paths
the times on a Palais-Smale sequence are automatically bounded away from zero.
\end{proof}

\vspace{3mm}

When looking for Euler-Lagrange orbits satisfying conormal boundary conditions in case $Q_0,Q_1$ intersect, there is another relevant energy value which we now define. 

We saw that the only Palais-Smale sequences for $\A_k$ that may cause troubles are either the ones for which the times are not bounded or the ones for which the times tend to zero; 
in the latter case the Palais-Smale sequence lies in a connected component of $\mathcal M_Q$ that contains constant paths and the action necessarily goes to zero (cf. Lemma \ref{yesmodification}). 
This fact will enable us to prove in Theorem \ref{teorema2} that, for all $k>c(L;Q_0,Q_1)$, there is an Euler-Lagrange orbit with energy $k$ satisfying the conormal boundary 
conditions (\ref{lagrangianformulation1}) in every connected component of $\mathcal M_Q$ that does not contain constant paths, even if the submanifolds $Q_0$ and $Q_1$ intersect.

Thus, the following question arises naturally: what happens for the connected components of $\mathcal M_Q$ containing constant paths? The situation is clearly more complicated and indeed in general we may not expect 
that they contain solutions of the problem. However, positive existence results are possible also for these connected components. 
More precisely, for every connected component $\mathcal N$ of $\mathcal M_Q$ containing constant paths we introduce the following energy value
\begin{equation}
k_{\mathcal N}(L) :=\  \inf \ \Big \{k\in \R \ \Big |\ \A_k(\gamma)\geq 0\, ,\ \ \forall \ \gamma \in \mathcal N\Big \}\, .
\label{komega(L)}
\end{equation}

By definition we readily see that $c(L;Q_0,Q_1 )\leq k_{\mathcal N}(L)$. In Section \ref{supercriticalorbits} we show that in the (possibly, but in general not, empty) interval 
$(c(L;Q_0,Q_1 ),k_{\mathcal N}(L))$ we find Euler-Lagrange orbits in $\mathcal N$ satisfying the conormal boundary conditions (\ref{lagrangianformulation1}). Existence results above $k_{\mathcal N}(L)$ are in general achievable
only under the additional assumption that $\pi_l(\mathcal N)\neq 0$ for some $l\geq 1$ (cf. Theorem \ref{teorema3}). Another ``natural'' energy value is given by 
$$k_0(L) := \  \inf \ \Big \{k\in \R \ \Big |\ \A_k(\gamma)\geq 0\, ,\ \ \forall \ \gamma \in \mathcal M_Q\Big \}\, .$$

It is interesting to study the relation of $k_0(L)$ with the critical value $c(L;Q_0,Q_1)$ and more generally with the other critical values we introduced in this session; this will also give us an estimate on how much the 
various critical values can differ. Clearly there holds $c(L;Q_0,Q_1) \leq k_0(L)$, since the action functional $\A_{k_0(L)}$ is bounded from below on every connected component of $\mathcal M_Q$. 

We claim that actually $c(L) \leq  k_0(L).$ Thus, assume that $k<c(L)$; by definition there exists a loop $\delta$ such that $\A_k(\delta)<0$. It is now easy to construct a path from $Q_0$ to $Q_1$ with negative action 
by going around $\delta$ a sufficiently large amount of time. More precisely, we construct the desired path connecting $Q_0$ and $Q_1$ as follows: 
pick a path $\eta$ from a point $q_0\in Q_0$ to the base point $\delta(0)$, then wind $n$-times around $\delta$ and finally join $\delta(0)$ with a point $q_1\in Q_1$ by a path $\mu$. If $n$ is large enough then 
$$\A_k(\mu\# \delta^n \# \eta) \ =\ \A_k (\mu) + n\, \A_k(\delta) + \A_k(\eta)\ <\ 0\, ,$$
which implies  $k<k_0(L)$ and the claim follows. Therefore, we have 

$$e_0(L) \ \leq \ c_u(L)\ \leq \ c(L;Q_0,Q_1) \ \leq \ c(L)\ \leq \ k_0(L)\, ,$$
where the second and third inequalities follow from Lemma \ref{coveringandmane}. In general there is no relation between $c_0(L)$ and $c(L;Q_0,Q_1)$, as the examples in Section \ref{ahilbertmanifoldofloops} show. 
In fact, in the situation represented by Figure \ref{primo} at page 22 we have $c(L; Q_0,Q_1) = c(L)$, whilst in the case illustrated by Figure \ref{terzo} at page 23 we have $c(L;Q_0,Q_1) = c_u(L)$. 

In order to estimate how much the various Ma\~n\'e critical values can differ, one can measure the difference $k_0(L)-e_0(L)$. Consider the smooth one-form 
$$\theta(q)[v]:= \ d_vL(q,0)[v]\, ;$$
by taking a Taylor expansion and by using (\ref{secondinequality}), we get that
\begin{eqnarray*}
L(q,v)&=& L(q,0)+d_vL(q,0)[v] + \frac{1}{2}\, d_{vv}L(q,sv)[v,v]\ \geq \\ 
          &\geq& - E(q,0)+ \theta(q)[v] + a\, \big \|v\big \|^2\, .
\end{eqnarray*}
If we set as usual $\gamma(t):=x(t/T)$, then for every $k>e_0(L)$ we obtain
\begin{eqnarray*}
\A_k(x,T) &=& \A_k(\gamma)\ \geq \int_0^T \Big [-E(\gamma(t),0)+\theta(\gamma(t))[\dot \gamma(t)] + a\, \big \|\dot \gamma(t)\big \|^2 + k \Big ]\, dt \ =\\ 
   &=& \int_0^T \Big [k-E(\gamma(t),0)\Big ]\, dt \ + \int_0^T \gamma^*\theta\ + \ a\, \int_0^T \big \|\dot \gamma(t)\big \|^2 \, dt\ \geq \\
   &\geq& \Big [k-e_0(L)\Big ]\, T  \ + \ \frac{a}{T}\, l(\gamma)^2 \ -\  \|\theta\|_\infty \,  l(\gamma) \, .
\end{eqnarray*}
For $T$ fixed, the latter expression is a parabola in $l(\gamma)$ with minimum
$$(k-e_0(L)) \, T \ -\  \frac{\|\theta\|_\infty^2}{4a}\, T\ =\ \left (k-e_0(L) - \frac{\|\theta\|_\infty^2}{4a}\right )\, T\, .$$
In particular, if 
$$k \ >\ e_0(L) + \frac{\|\theta\|_\infty^2}{4a}\, ,$$ 
then $\A_k(\gamma)\geq 0$ for any path $\gamma$ connecting $Q_0$ with $Q_1$ and this implies, by the definition of $k_0(L)$, that $k>k_0(L)$. Therefore we get
$$k_0(L) \ \leq \ e_0(L) + \frac{\|\theta\|_\infty^2}{4a}\, .$$ 
We sum up the discussion above in the following

\begin{prop}
Let $L:TM\rightarrow \R$ be a Tonelli Lagrangian, $Q_0,Q_1\subseteq M$ closed submanifolds. Then the following chain of inequalities holds 
\begin{equation}
e_0(L)\ \leq \ c_u(L)\ \leq \ c(L;Q_0,Q_1) \ \leq \ c(L) \ \leq \ k_0(L)\ \leq \ e_0(L) + \frac{\|\theta\|_\infty^2}{4a}\, .
\label{chainofcriticalvalues}
\end{equation}
\end{prop}

Observe that, when $\theta\equiv 0$ (in particular when $L$ is a mechanical Lagrangian, i.e. of the form (\ref{magneticlagrangian}) with vanishing magnetic potential) we retrieve 
$$e_0(L)\ =\  c_u(L)\ =\ c(L;Q_0,Q_1)\ =\ c(L)\ =\ k_0(L)\, .$$


\chapter{Orbits with conormal boundary conditions}
\label{chapter3}

In this chapter, building on the analytical background given in the previous chapters, we prove the main results about the existence of Euler-Lagrange orbits connecting to given submanifolds $Q_0,Q_1\subseteq M$ and
satisfying the conormal boundary conditions.

In Section \ref{supercriticalorbits} we deal with supercritical energies, whilst in Section \ref{subcriticalorbits} we move to the study of existence for subcritical energies. 

Finally, in Section \ref{counterexamples} we provide counterexamples showing that all the results obtained are sharp.


\section{Existence for supercritical energies.}
\label{supercriticalorbits}

Throughout this section we suppose $k>c(L;Q_0,Q_1)$ and prove existence results of Euler-Lagrange orbits satisfying the conormal boundary conditions (\ref{lagrangianformulation1}) with energy $k$. 
We first suppose that $Q_0\cap Q_1=\emptyset$ and show that, in every connected component $\mathcal N$ of $\mathcal M_Q$, there is an Euler-Lagrange orbit with energy $k$ which satisfies the conormal boundary conditions 
and which is a global minimizer of $\A_k$ on $\mathcal N$. 

\begin{teo}
Suppose $Q_0\cap Q_1 =\emptyset$ and let $k>c(L;Q_0,Q_1)$. Then, there is an Euler-Lagrange orbit with energy $k$ satisfying the conormal boundary conditions (\ref{lagrangianformulation1}) in every connected component 
$\mathcal N$ of $\mathcal M_Q$, which is furthermore a global minimizers of $\A_k$ among the connected component $\mathcal N$.
\label{teorema1}
\end{teo}

\begin{proof}
Since $Q_0\cap Q_1=\emptyset$, Lemma \ref{completeness0} implies that the sublevels of $\A_k$ in $\mathcal N$
$$\big \{(x,T)\in \mathcal N\ \big | \A_k(x,T)\leq c\big \}$$
are complete. Moreover, Corollary \ref{compattezza} implies that $\A_k$ satisfies the Palais-Smale condition on $\mathcal N$ for every $k>c(L;Q_0,Q_1)$. 
We may then conclude that $\A_k$ has a global minimizer on $\mathcal N$, as we wished to prove.
\end{proof}

\vspace{5mm}

We move now to study the case of non-empty intersection; for the sake of simplicity we first suppose the intersection to be connected. Before stating and proving the main result in this context we shall introduce the concept of \textit{degenerate orbit}.

\begin{defn}
Let $Q_0, Q_1\subseteq M$ be two closed connected submanifolds such that $Q_0\cap Q_1$ is non-empty. An Euler-Lagrange orbit $\gamma:(-\epsilon,\epsilon)\rightarrow M$ is called a $\mathsf{degenerate}$ 
Euler-Lagrange orbit satisfying the conormal boundary conditions if there hold
$$\gamma(0)\ \in \ Q_0\cap Q_1 \, ,\ \ \ \  d_v L(\gamma(0),\dot \gamma(0))\Big |_{T_{\gamma(0)}Q_0 \cup  T_{\gamma(0)}Q_1} \ \equiv \ 0 \, .$$
\end{defn}

In order to have an explicit picture of what a degenerate orbit is, let us again consider the geodesic flow of a Riemannian manifold $(M,g)$. Let for instance $Q_0,Q_1$ be two in $M$ embedded circles which intersect in exactly one point, say $q$; 
a degenerate solution is then an Euler-Lagrange orbit $\gamma:(-\epsilon,\epsilon) \rightarrow M$ (in this case, a geodesic) through the point $q$ which is orthogonal to both $Q_0$ and $Q_1$.

\begin{center}
\includegraphics[height=40mm]{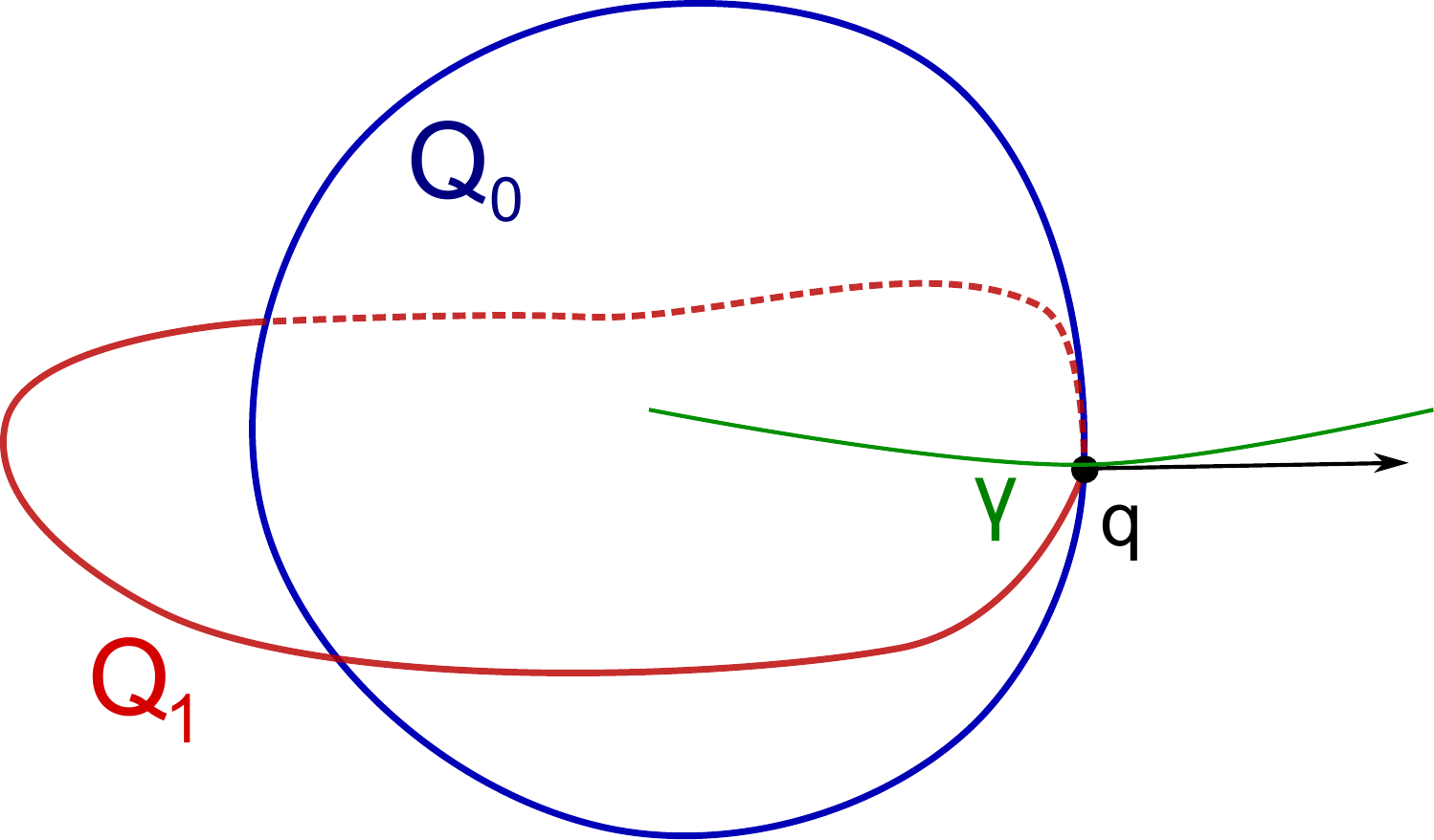}
\end{center}

\begin{teo}
Suppose $\, Q_0\cap Q_1 \neq \emptyset$ connected; then the following hold:

\begin{enumerate}
\item For every $k>c(L;Q_0,Q_1)$ and for every connected component of $\mathcal M_Q$ that does not contain constant paths there exists an Euler-Lagrange orbit with energy 
$k$ satisfying the conormal boundary conditions (\ref{lagrangianformulation1}) which is a global minimizer of $\A_k$ on this connected component.

\item Let $\mathcal N$ be the connected component of $\mathcal M_Q$ containing constant paths and let $k_{\mathcal N}(L)$ as in \eqref{komega(L)}. Then, for every $k\in (c(L;Q_0,Q_1), k_{\mathcal N}(L))$ there exists an Euler-Lagrange orbit with energy 
$k$ satisfying the conormal boundary conditions (\ref{lagrangianformulation1}) which is a global minimizer of $\A_k$ on $\mathcal N$. 

\item If $\pi_l(\mathcal N)\neq 0$ for some $l\geq 1$, then for every $k> k_{\mathcal N}(L)$ there exists an Euler-Lagrange orbit in $\mathcal N$ with energy $k$ satisfying the conormal boundary conditions (\ref{lagrangianformulation1}). 
If in addition there holds $k_{\mathcal N}(L)>c(L;Q_0,Q_1)$, then there is such an orbit (possibly degenerate) also with energy $k_{\mathcal N}(L)$. 
\end{enumerate}
\label{teorema2}
\end{teo}

\begin{proof}
From what concerning statement 1, since $\, Q_0\cap Q_1\, $ is connected, there is only one component of $\mathcal M_Q$ containing constant paths. For any other connected component,
the same argument as in the proof of Theorem \ref{teorema1} goes through and gives us the desired Euler-Lagrange orbit.

Now we prove statement 2. We can assume that the interval $(c(L;Q_0,Q_1), k_{\mathcal N}(L))$ is non-empty and fix $k$ in it; since $c:= \inf \A_k <0\, $, the sublevels
$$\Big\{(x,T)\in \mathcal M_Q \ \Big |\ \A_k(x,T) \leq c+\epsilon \Big \}$$ 
are complete for every $\epsilon >0$ small. Indeed, from (\ref{firstinequality}) we get that
\begin{equation*}
0 \ > \ c + \epsilon \ \geq \  \A_k(x,T) \ \geq \ \frac{a}{T} \, \|x'\|_2^2 + T(k-b)\ \geq \ T(k-b)
\end{equation*}
with the last quantity going to zero for $T\rightarrow 0$. Moreover, Lemma \ref{yesmodification} implies that all the Palais-Smale sequences at level $c$ have $T_h$'s bounded away from zero and hence
$\A_k$ satisfies the Palais-Smale condition at level $c$ by Lemma \ref{boundedness}. We now retrieve the existence of a  global minimizer for $\A_k$ exactly as in the proof of Theorem \ref{teorema1}.

We finally focus on the more interesting statement 3. We suppose for the moment $k>k_{\mathcal N}(L)$; in this case, using the existence of at least one non-trivial $\pi_l(\mathcal N)$ for $l\geq 1$, we retrieve the Euler-Lagrange orbit using a minimax argument 
analogous to that used by Lusternik and Fet \cite{FL51} in their proof of the existence of one closed geodesic on a simply connected manifold.
By assumption there exists a non-trivial element $\mathcal H \in  \big [S^{l},\mathcal N\big ]$ and therefore we can consider the minimax value 
$$c:=\ \inf_{h\in \mathcal H} \ \max_{\zeta \in S^{l}} \ \A_k(h(\zeta))\, .$$ 

Let us show that $c>0$; since $\mathcal H$ is non-trivial, there exists a positive number $\lambda$ such that for every map $h=(x,T):S^{l}\rightarrow \mathcal N$ belonging to the class $\mathcal H$ there holds 
$$\max_{\zeta\in S^{l}} \ l(x(\zeta))\ \geq \ \lambda\, ,$$ 
where as usual $l(x(\zeta))$ denotes the length of the path $x(\zeta)$ (cf. \cite[Theorem 2.1.8]{Kli78}). If $(x,T)\in \mathcal N$ has length $l(x)\geq \lambda$, then (\ref{firstinequality}) implies that 
\begin{eqnarray*}
\A_k(x,T) &=& T\int_0^1 \Big [L\Big ( x(s),\frac{x'(s)}{T}\Big ) + k \Big ]\, ds \ \geq\\ 
                &\geq& \frac{a}{T}\int_0^1 \|x'(s)\|^2\, ds \ +\  T(k-b)\ \geq \\
                &\geq& \frac{a}{T} \, l(x)^2\ +\ T(k-b) \ \geq \\ \
                &\geq& \frac{a}{T} \, \lambda^2\ +\ T(k-b)\, .
\end{eqnarray*}

Since $\lambda >0$, the above inequality implies that if $(x,T)\in \mathcal N$ has length $l(x)\geq \lambda$ and action $\A_k(x,T)\leq c+1$ then 
$$c+1 \ \geq \ \frac{a}{T}\, \lambda^2 \ + \ T(k-b)$$ 
and hence $T\geq T_0$ for some $T_0>0$, because the quantity on the righthand-side goes to infinity as $T\rightarrow 0$. Now let $h\in \mathcal H$ be such that 
$$\max_{\zeta \in S^{l}}\ \A_k(h(\zeta))\ \leq\ c+1\, ;$$
then by the above considerations there exists $(x,T)\in h(S^{l})$ with $T\geq T_0$ and 
$$\A_k(x,T) \ = \  \A_{k_{\mathcal N}(L)} (x,T) + \big (k-k_{\mathcal N}(L)\big )T \ \geq \ \big (k-k_{\mathcal N}(L) \big ) T_0\ >\ 0\, .$$

The argument above shows that the minimax value $c$ is strictly positive. Combining Lemma \ref{convergenza} with Remark \ref{flussotroncato} we get the existence of a Palais-Smale sequence at level $c$. 
Since $c>0$ we also get from Lemma \ref{yesmodification} that the $T_h$'s are bounded away from zero, so that by Corollary \ref{compattezza} the Palais-Smale sequence has a limiting point in $\mathcal N$, which gives us the 
required Euler-Lagrange orbit.

We are left now with the case $k=k_{\mathcal N}(L)$; by assumption $k_{\mathcal N}(L)>c(L;Q_0,Q_1)$. Consider a sequence $k_n \downarrow k_{\mathcal N}(L)$ and the corresponding $\{(x_n,T_n)\}$, where 
$(x_n,T_n)$ is the Euler-Lagrange orbit with energy $k_n$ satisfying the conormal boundary conditions, whose existence is guaranteed by the discussion above. Notice that
$$c(k_n) := \inf_{h\in \mathcal H}\ \max_{\zeta \in S^{l}}\ \A_{k_n}(h(\zeta))\ =\ \A_{k_n}(x_n,T_n)$$ 
is a decreasing sequence bounded from below by zero and therefore converging to a value $c(k_{\mathcal N}(L))\in [0,+\infty)$. At the same time 
$$\frac{\partial \A_{k_{\mathcal N}(L)}}{\partial T} (x_n,T_n)\ =\ \int_0^1 \Big [k_{\mathcal N}(L)-E\Big (x_n(s),\frac{ x_n'(s)}{T_n}\Big )\Big ]\, ds \ =\ k_{\mathcal N}(L)-k_n$$
and
$$d_x\A_{k_{\mathcal N}(L)}(x_n,T_n)\big [(\cdot,0)]\ =\ \int_0^{T_n} \Big [ d_q L(\gamma_n(t),\dot \gamma_n(t))[\cdot] + d_v L (\gamma_n(t),\dot \gamma_n(t))[\cdot]\Big ]\, dt \ =\ 0\, ,$$
where as usual we use the identification $\gamma_n=(x_n,T_n)$. It follows that $(x_n,T_n)$ is a Palais-Smale sequence for $\A_{k_{\mathcal N}(L)}$ at level $c(k_{\mathcal N}(L))$. Therefore, if the 
$T_n$'s are known to be bounded away from zero, then using Lemma \ref{lemmalimitatezza} and Corollary \ref{compattezza} we get the existence of a subsequence converging in $H^1$ to an element $(x,T)$, 
which gives us the required Euler-Lagrange orbit with conormal boundary conditions at level $k_{\mathcal N}(L)$. 

Thus, we may suppose after passing to a subsequence if necessary that $T_n\rightarrow 0$. Lemma \ref{yesmodification} ensures then that $c(k_{\mathcal N}(L))=0$ and that the sequence $x_n$, after passing to another
subsequence if necessary, converges in $H^1$ to a point $q\in Q_0\cap Q_1$. Observe that $(q,0)$ cannot be a constant Euler-Lagrange orbit, since by assumption $k_{\mathcal N}(L)>e_0(L)$. 
We choose now a strictly decreasing sequence $\{\epsilon_m\}_{m\in \N}\subseteq (0,+\infty)$ such that $\epsilon_m\downarrow 0$ as $m\rightarrow \infty$ and such that, for every $m\in \N$, there holds:

\begin{enumerate}
\item Each $\gamma_n$ can be extended to an orbit $\gamma_n^{(m)}:[-\epsilon_m, T_n + \epsilon_m]\rightarrow M$ with
$$\left \{ \begin{array}{l}
d_v L(\gamma_n^{(m)}(0),\dot {\gamma}^{(m)}_n (0)) \Big |_{T_{\gamma^{(m)}_n(0)}Q_0} \! \! = \ \ 0 \, ,\\ \\
d_v L(\gamma_n^{(m)}(T_n),\dot {\gamma}^{(m)}_n (T_n)) \Big |_{T_{\gamma^{(m)}_n(T_n)}Q_1} \! \! = \ \ 0 \,.
\end{array}\right.$$
\item There exist two closed connected submanifolds $Q_0^{(m)},Q_1^{(m)}\subseteq M$ such that
$$\left \{\begin{array}{l}
d_v L(\gamma^{(m)}_n(-\epsilon_m),\dot {\gamma}^{(m)}_n (-\epsilon_m)) \Big |_{T_{\gamma^{(m)}_n(-\epsilon_m)}Q_0^{(m)}} \! \! = \ \ 0 \, , \\ \\
d_v L(\gamma^{(m)}_n(T_n +\epsilon_m),\dot {\gamma}^{(m)}_n (T_n + \epsilon_m)) \Big |_{T_{\gamma^{(m)}_n(T_n + \epsilon_m)}Q_1^{(m)}} \! \! = \ \ 0 \, .
\end{array}\right.$$
\end{enumerate}

Define $Q^{(m)} =Q_0^{(m)} \times Q^{(m)}_1$ and  $T^{(m)}_n=T_n+2\epsilon_m$; then  $(x^{(m)}_n,T^{(m)}_n)$ is a Palais-Smale sequence for $\A_{k_{\mathcal N}(L)}$ on $\mathcal M_{Q^{(m)}}$ with times bounded away from zero.

\begin{center}
\includegraphics[height=47mm]{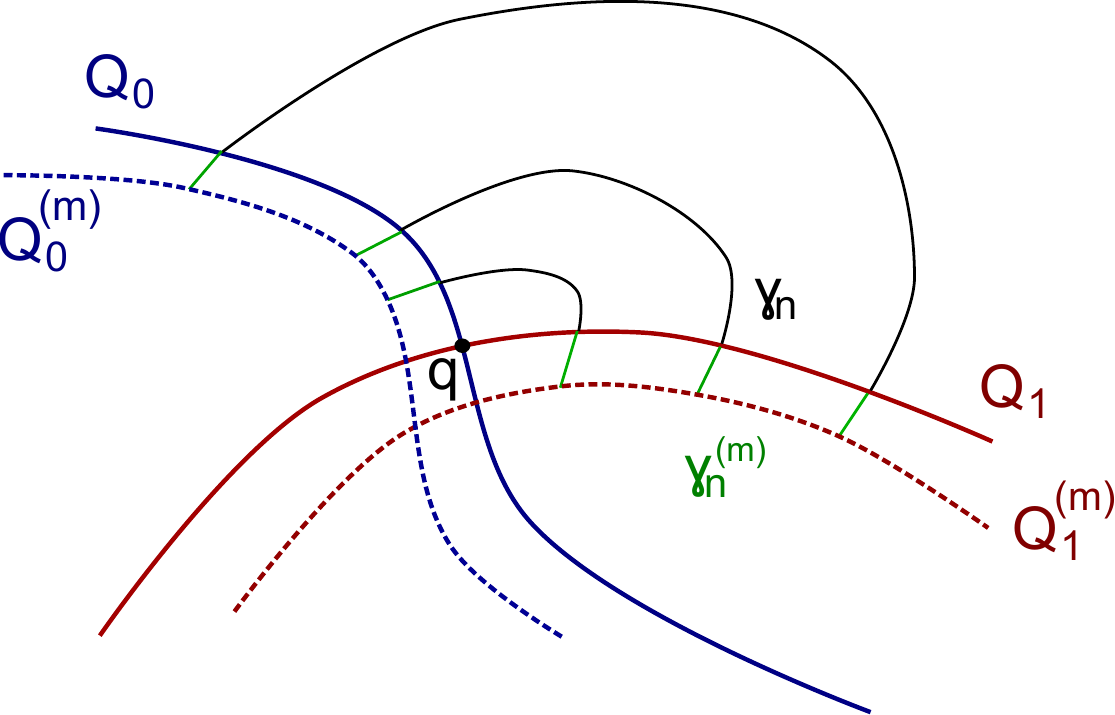}
\end{center}

Therefore, by Corollary \ref{compattezza}, $(x_n^{(m)},T_n^{(m)})$ has a subsequence converging to some $\gamma^{(m)} = (x^{(m)},T^{(m)})$, with $T^{(m)}=2\epsilon_m$, 
which is an Euler-Lagrange orbit with energy $k_{\mathcal N}(L)$ satisfying the conormal boundary conditions for $Q^{(m)}$, namely 
\begin{equation}
\left \{\begin{array}{l}
d_v L(\gamma^{(m)}(-\epsilon_m),\dot {\gamma}^{(m)}(-\epsilon_m))\Big |_{T_{\gamma^{(m)}(-\epsilon_m)}Q_0^{(m)}} \!\!\!\!\!\! = \ \  0 \,  \, ,\\ \\
d_v L(\gamma^{(m)}(\epsilon_m),\dot {\gamma}^{(m)}(\epsilon_m))\Big |_{T_{\gamma^{(m)}(\epsilon_m)}Q_1^{(m)}} \!\!\!\!\!\! = \ \  0 \, .
\end{array}\right.
\label{gru}
\end{equation}

Furthermore, the fact that the $\gamma^{(m)}_n$'s were obtained by extending the $\gamma_n$'s implies that $\gamma^{(m)}(0) = q$, for every $m\in \N$. Now we would like to let $m$ go to infinity; indeed, by \eqref{gru}, the ``limit curve'' $\gamma^\infty$ would satisfy 
$\gamma^\infty(0)=q$ and 
$$d_v L(\gamma^\infty(0),\dot {\gamma}^\infty (0))\Big |_{T_{q}Q_0 \, \cup \, T_{q}Q_1}  \!\!\!\!\!\!\!\! = \ \ 0\, ,$$ 
which is exactly what we want. However, since the intervals of definition of $\gamma^{(m)}$ shrink to a point, this limiting process would not produce a curve. Therefore, before taking the limit we have to extend the $\gamma^{(m)}$'s further.

\begin{center}
\includegraphics[height=40mm]{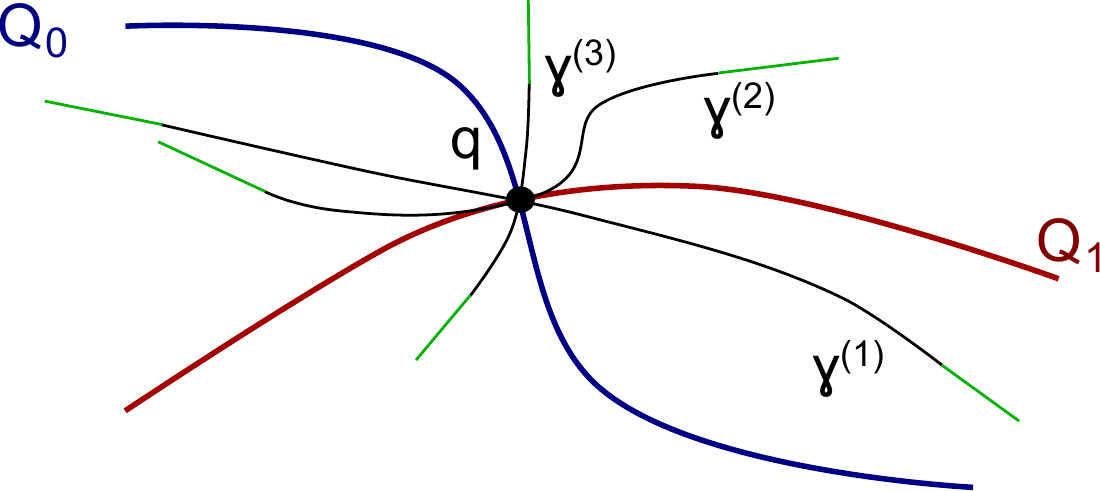}
\end{center}

Thus, consider $\epsilon>0$ and extend the $\gamma^{(m)}$'s to get Euler-Lagrange orbits defined on $[-(\epsilon+\epsilon_m),\epsilon+\epsilon_m]$ still satisfying \eqref{gru}; with a slight abuse of notation we denote the extended orbits again with 
$\gamma^{(m)}: [-(\epsilon+\epsilon_m),\epsilon+\epsilon_m]\rightarrow M.$ The fact that $\gamma^{(m)}=(x^{(m)},2(\epsilon+\epsilon_m))$ is an Euler-Lagrange orbit with energy $k_{\mathcal N}(L)$ combined with the fact that the energy is quadratic at infinity implies that 
\begin{eqnarray*}
0 &=& -\frac{\partial \A_{k_{\mathcal N}(L)}}{\partial T}(x^{(m)},2(\epsilon+\epsilon_m))\ = \ \int_0^1 E\Big (x^{(m)}(s), \frac{x^{(m)}{'}(s)}{2(\epsilon+\epsilon_m)}\Big ) \, ds \ - \ k_{\mathcal N}(L)\ \geq \\ 
   &\geq& \frac{a}{4(\epsilon+\epsilon_m)^2} \int_0^1 \|x^{(m)}{'}(s)\|^2\, ds \ - \ \big (b+k_{\mathcal N}(L)\big )
\end{eqnarray*}
for suitable constants $a>0$ and $b\in \R$. In particular  
$$\int_0^1 \|x^{(m)}{'}(s)\|^2\, ds \ \leq \ \frac{4(\epsilon+\epsilon_m)^2}{a} \big (b+k_{\mathcal N}(L)\big )$$
and hence the family $\{x^{(m)}\}$ is $1/2$-equi-H\"older-continuous. By an argument anologous to the one used in Lemma \ref{lemmalimitatezza} one can now prove that $\{x^{(m)}\}$ has a subsequence converging strongly in $H^1$
to $x^\infty$. Hence, $\gamma^\infty =(x^\infty,2\epsilon)$ is an Euler-Lagrange orbit with energy $k_{\mathcal N}(L)$ such that $\gamma^\infty (0)=q$ and 
$$d_v L(\gamma^\infty(0),\dot {\gamma}^\infty (0))\Big |_{T_{q}Q_0 \, \cup \, T_qQ_1}  \!\!\!\!\!\!\!\! = \ \ 0\, ,$$ 
that is a degenerate solution.
\end{proof}

\vspace{5mm}

Notice that the interval $(c(L;Q_0,Q_1), k_{\mathcal N}(L))$ in statement 2 of Theorem \ref{teorema2} might be empty, but in general is not; we will see an example of that in Section \ref{counterexamples}.

A very particular case of non-empty connected intersection is given by the choice $Q_0=Q_1$, which corresponds to the Arnold chord conjecture. In this case, we call an Euler-Lagrange orbit satisfying the conormal boundary conditions 
an \textit{Arnold chord} for $Q_0$. Notice that in this context it does not make sense to talk about degenerate solutions, since they may correspond to trivial chords one is not interested in. 

Theorem \ref{teorema2} implies  immediately the following

\begin{cor}
Let $Q_0\subseteq M$ be a closed connected submanifold and define $c(L;Q_0)$ as in (\ref{clh0h1}) just by setting $Q_0=Q_1$. Then the following hold:

\begin{enumerate}
\item For every $k>c(L;Q_0)$ and for every connected component of $\mathcal M_Q$ that does not contain constant paths there exists an Arnold chord for $Q_0$
with energy $k$ which is a global minimizer of $\A_k$ in its connected component.

\item Let $\mathcal N$ be the connected component of $\mathcal M_Q$ containing the constant paths. 
For every $k\in (c(L;Q_0), k_{\mathcal N} (L))$ there is an Arnold chord for $Q_0$ with energy $k$ which is a global minimizer of $\A_k$ on $\mathcal N$. 
If in addition $\pi_l(\mathcal N)\neq 0$ for some $l\geq 1$, then for every $k> k_{\mathcal N}(L)$ there exists an Arnold chord with energy $k$.
\end{enumerate}
\label{arnoldchordsupercritical}
\end{cor}

It is well-known (cf. \cite[Theorem 4.2]{Abb13}) that energy level sets above $c_0(L)$ are of (restricted) contact type. It is also known that, if $M\neq \T^2$, every energy level set $E^{-1}(k)$ with $c_u(L)< k\leq c_0(L)$ 
is not of contact type (cf. \cite[Proposition B.1]{Con06}). 

Therefore, Corollary \ref{arnoldchordsupercritical} provides the existence of Arnold chords in a possibly non-contact situation, since $c(L;Q_0)$ might be strictly smaller than $c_0(L)$.

\vspace{2mm}

We end this section noticing that Theorem \ref{teorema2} can be trivially generalized to the case $Q_0\cap Q_1$ not connected. 

\begin{teo}
Suppose $Q_0\cap Q_1 \neq \emptyset$. Then the following hold:

\begin{enumerate}
\item For every $k>c(L;Q_0,Q_1)$ and for every connected component of $\mathcal M_Q$ that does not contain constant paths there exists an Euler-Lagrange orbit with energy $k$ satisfying the 
conormal boundary conditions (\ref{lagrangianformulation1}), which is a global minimizer of $\A_k$ in this connected component. 

\item For every component $\mathcal N$ of $\mathcal M_Q$ containing constant paths and for every $k\in (c(L;Q_0,Q_1), k_{\mathcal N}(L))$ there exists an Euler-Lagrange orbit $\gamma\in \mathcal N$
with energy $k$ satisfying the conormal boundary conditions (\ref{lagrangianformulation1}), which is a global minimizer of $\A_k$ on $\mathcal N$.
If in addition $\pi_l(\mathcal N)\neq 0$ for some $l\geq 1$, then for every $k> k_{\mathcal N}(L)$ there exists an Euler-Lagrange orbit 
$\gamma\in \mathcal N$ with energy $k$ satisfying the conormal boundary conditions (\ref{lagrangianformulation1}).
Finally, if  $k_{\mathcal N}(L)>c(L;Q_0,Q_1)$, then there is such an orbit (possibly degenerate) also with energy $k_{\mathcal N}(L)$. 
\end{enumerate}
\label{teorema3}
\end{teo}


\section{Subcritical energies}
\label{subcriticalorbits}

In this section we study the existence of Euler-Lagrange orbits satisfying the conormal boundary conditions (\ref{lagrangianformulation1}) for subcritical energies $k<c(L;Q_0,Q_1)$. 
As already explained in Chapter \ref{chapter0}, this problem is way harder than the corresponding one for supercritical energies. 
Throughout this section we assume that the submanifolds $Q_0$ and $Q_1$ intersect; we will get back to the case of empty intersection in Section \ref{counterexamples} showing that Theorem \ref{teorema1} is optimal.

We start by considering the following particular case, which should help the reader to understand the general situation later on. Let $Q_0$ and $Q_1$ be two closed connected submanifolds which intersect in one point, say $p$.
We show now that, under the assumption (\ref{condizionetheta}) on the Lagrangian $L$, for every $k\in (e_0(L),c(L;Q_0,Q_1))$ the action functional $\A_k$ exhibits a mountain-pass geometry on the connected component of $\mathcal M_Q$ 
that contains the constant path in $p$, which we denote hereafter with $\mathcal N$.

Since for any $k\in (e_0(L),c(L;Q_0,Q_1))$ the free-time action functional $\A_k$ is unbounded from below, it makes sense to define the following class of paths in $\mathcal N$
$$\Gamma := \Big \{ u :[0,1]\rightarrow \mathcal N \ \Big |\ u(0)=(p,T) \, ,\ \A_k(u(1))<0 \Big \}\,  .$$
Notice that for any $T>0$ we have
$$\A_k(p,T)\ =\ T\, \Big [k-E(p,0)\Big ] \ >\ 0\, ;$$
moreover, $\A_k(p,T)$ goes to zero as $T\rightarrow 0$. Define now 
\begin{equation}
\theta_q(\cdot) := \ d_v L (q,0)[\cdot]\, , \ \ \ \ \forall \ q\in M\, ,
\label{definizionetheta}
\end{equation}
and assume that there exists an open neighborhood $\mathcal U$ of $p$ such that
\begin{equation}
\theta_q \ \equiv \ 0\, , \ \ \ \ \forall \ q \in \mathcal U\, .
\label{condizionetheta}
\end{equation}

Notice that, up to replacing $\mathcal U$ by a smaller neighborhood we might assume $\mathcal U=B_r$ to be an Euclidean Ball with radius $r>0$ and $p$ to be the origin in $\R^n$. Under 
the additional assumption (\ref{condizionetheta}) we show the desired mountain-pass geometry for the action functional $\A_k$. Namely, we prove that there is $\alpha>0$ such that 
$$\max_{s\in[0,1]} \ \A_k(u(s)) \ \geq \ \alpha \, , \ \ \ \ \forall \ u\in \Gamma\, .$$

Here is the scheme of the proof: we first show that if the length of a path $\gamma$ connecting $Q_0$ and $Q_1$ is sufficiently small then the action of $\gamma$ needs to be non-negative. 
Therefore, for every element $u\in \Gamma$ there must be an $s\in [0,1]$ such that $l(u(s))=\epsilon$ for a suitable $\epsilon>0$. Now we get the assertion showing that every path with length 
$\epsilon$  has $\A_k$-action bounded away from zero by a positive constant $\alpha$. 

Since $Q_0$ and $Q_1$ intersect only in $p$, for every sufficiently small $\lambda>0$ there exists $\delta>0$ such that  
$$d(Q_0\setminus B_\delta, Q_1\setminus B_\delta) \ \geq \ \lambda\, ,$$
where $B_\delta$ denotes the ball with radius $\delta$ around $p$. In other words, every path connecting $Q_0$ to $Q_1$ with starting and ending point outside $B_\delta$ has length larger than $\lambda$. It is clear now that, if $\epsilon>0$ is sufficiently small,
then every path $\gamma$ connecting $Q_0$ to $Q_1$ with length $l(\gamma)\leq \epsilon$ is entirely contained in $\mathcal U=B_r$. In fact, at least one between its starting and ending point is contained in 
$B_\delta$ for some $\delta>0$ sufficiently small, say $\gamma(0)\in B_\delta$, and hence by the triangle inequality
\begin{equation}
d(\gamma(t),p) \ < \ d(\gamma(t),\gamma(0)) \ + \ d(\gamma(0),p) \ < \ \epsilon + \delta \ < \ r\, , \ \ \ \ \forall \ t\, .
\label{gammainu}
\end{equation}
A Taylor expansion together with the bound (\ref{secondinequality}) implies
\begin{eqnarray}
L(q,v) &=& L(q,0) + d_vL(q,0)[v] + \frac12 d_{vv} L(q,sv)[v,v] \ \geq \nonumber \\ 
           &\geq& - E(q,0) + \theta_q(v) + a\, |v|^2\, .
\label{stimacontheta}
\end{eqnarray}
Using (\ref{condizionetheta}), (\ref{gammainu}) and (\ref{stimacontheta}) we now compute for every $\gamma=(x,T)$ with length $l(x)\leq \epsilon$ 
\begin{eqnarray*}
\A_k(x,T) &\geq& \int_0^T \Big [-E(\gamma(t),0) + \theta_{\gamma(t)}(\gamma'(t)) + a \, |\gamma'(t)|^2 + k \Big ]\, dt \ \geq \\ 
                &\geq& T\, \big (k-e_0(L)\big ) + \frac{a}{T} \, l(x)^2 
\end{eqnarray*}
which is a non-negative quantity. It follows that for every $u\in \Gamma$ there is $s\in[0,1]$ such that $l(u(s))=\epsilon$; for such $s$ we obtain 
\begin{eqnarray*}
\A_k(u(s)) &\geq& T\, \big (k-e_0(L)\big ) + \frac{a}{T} \, \epsilon^2 \ \geq \ 2\epsilon \sqrt{a(k-e_0(L))} \ =: \alpha
\end{eqnarray*}
as we wished to prove.

\begin{oss}
In the computation above the hypothesis $k>e_0(L)$ can be relaxed assuming only that $k>E(p,0)$. In fact, up to choosing a smaller neighborhood $\mathcal U$ of $p$ (thus, a smaller $\epsilon$), the continuity of the energy implies that 
$$k \ > \ \sup_{q\, \in\, \mathcal U} \ E(q , 0 ) \, .$$
All the estimates above go through just replacing $e_0(L)$ by the supremum above.
\end{oss}

We are now ready to deal with the general case. Let $Q_0,Q_1\subseteq M$ be closed and connected submanifolds with non-empty and connected intersection and define 
\begin{equation}
k_{Q_0\cap Q_1} := \ \max_{q\in Q_0\cap Q_1} E(q,0) \ \ + \max_{q\in Q_0\cap Q_1} \frac{\|\theta_q\|^2}{4a}\, ,
\label{kq0q1}
\end{equation}
where $\theta_q$ is as in (\ref{definizionetheta}), $\|\cdot\|$ is the dual norm on $T^*M$ induced by the Riemannian metric on $M$ and  $a>0$ is as in (\ref{secondinequality}). Lemma \ref{lemmaminimax1} states that, for every 
$k \in (k_{Q_0\cap Q_1},c(L;Q_0,Q_1))$, the action functional $\A_k$ has a mountain-pass geometry on $\mathcal N$. 
Notice that the interval above could be empty; this happens, for instance, when 
$$\max_{q\in Q_0\cap Q_1} E(q,0)\ = \ e_0(L)\, , \ \ \ \ \max_{q\in Q_0\cap Q_1} \frac{\|\theta_q\|^2}{4a} \ = \ \frac{\|\theta\|_\infty^2}{4a}\, ,$$
as the chain of inequalities (\ref{chainofcriticalvalues}) shows. However, this is not always the case as we will show in the next section. 
Finally, notice that in the case $Q_0\cap Q_1=\{p\}$ with $\theta_p=0$ the energy value $k_{Q_0\cap Q_1}$ reduces to the above considered $E(p,0)$.
 
\begin{lemma}
Let $Q_0,Q_1\subseteq M$ be two closed connected submanifolds with non-empty and connected intersection and let $k_{Q_0\cap Q_1}$ be as in (\ref{kq0q1}). Then, for every $k\in (k_{Q_0\cap Q_1},c(L;Q_0,Q_1))$ there exists 
$\alpha>0$ such that 
$$\inf_{u\in \Gamma} \ \max_{s\in [0,1]} \ \A_k(u(s)) \ \geq \ \alpha\, .$$
\label{lemmaminimax1}
\end{lemma}

\vspace{-4mm}

\begin{proof}
The proof follows from the one in the particular case with minor adjustments. First, observe that for all $q\in Q_0\cap Q_1$ we have 
$$\A_k(q,T) \ = \ \Big [k - E(q,0)\Big ] \, T \ > \ 0$$ 
and going to zero as $T\rightarrow 0$. Now consider a neighborhood $\mathcal U$ of $Q_0\cap Q_1$ such that 
\begin{equation}
k \ >\ \sup_{q\in \mathcal U} \ E(q,0) \ + \ \sup_{q\in \mathcal U} \ \frac{\|\theta_q\|^2}{4a}\, .
\label{kmaggiorecectheta}
\end{equation}

As in the particular case one shows now that, if $\epsilon >0$ is sufficiently small, then all the paths joining $Q_0$ to $Q_1$ with length less than or equal to $\epsilon$ are entirely contained in $\mathcal U$. 
Pick such an $\epsilon$; using (\ref{stimacontheta}) we compute for every $\gamma=(x,T)$ with $l(x)\leq \epsilon$
\begin{eqnarray*}
\A_k(x,T) &\geq& \int_0^T \Big [-E(\gamma(t),0) + \theta_{\gamma(t)}(\gamma'(t)) + a \, \|\gamma'(t)\|^2 + k \Big ]\, dt \ \geq \\ 
                 &\geq& \left (k - \sup_{q\in \mathcal U} \ E(q,0) \right ) \, T \ + \ \frac{a}{T}\, l(x)^2 \ + \int_0^T \theta_{\gamma(t)}(\gamma'(t))\, dt \ \geq \\ 
                 &\geq& \left (k - \sup_{q\in \mathcal U} \ E(q,0) \right ) \, T \ + \ \frac{a}{T}\, l(x)^2 \ - \ \left (\sup_{q\in \mathcal U} \ \|\theta_q\|\right )\, l(x)\, .
\end{eqnarray*}
To ease the notation let us define 
$$c_E := \  \sup_{q\in \mathcal U} \ E(q,0)\, ,  \ \ \ \ c_\theta := \ \sup_{q\in \mathcal U} \ \|\theta_q\|$$ 
and consider the function of two variables 
$$f:(0,+\infty)\times [0,\epsilon] \rightarrow \R\, ,\ \ \ \ f(T, l) := \ \big (k - c_E\big )\, T \ + \ \frac{a}{T} \, l^2 \ - \ c_\theta \, l\, .$$
For every $l$ fixed the function $f$ attains its minimum in the unique point 
$$T_l\ = \ l \, \frac{\sqrt{a}}{\sqrt{k-c_E}}$$
and there holds 
$$f(T_l,l) \ = \ \left ( 2\sqrt{a(k-c_E)} - c_\theta\right ) \, l\, .$$
This value is non-negative for all $l\in [0,\epsilon]$, provided 
$$2\sqrt{a(k-c_E)} - c_\theta \ >\ 0\, ,$$
which is equivalent to (\ref{kmaggiorecectheta}). Therefore, arguing as in the particular case, we get that for every $u\in \Gamma$ there exists $s\in [0,1]$ such that $l(u(s))=\epsilon$. For this $s$ we readily have
$$\A_k(u(s)) \ \geq \ \left ( 2\sqrt{a(k-c_E)} - c_\theta\right ) \, \epsilon \ =: \alpha\ >\ 0\, ,$$
exactly as we wished to prove.
\end{proof}

\vspace{5mm}

When the intersection $Q_0\cap Q_1$ consists of more than one point one would be tempted to replace the maximum with the minimum of the energy on $Q_0\cap Q_1$ in the definition of $k_{Q_0\cap Q_1}$, hence defining 
\begin{equation}
k_{Q_0\cap Q_1}^- := \ \min_{q\in Q_0\cap Q_1} E(q,0) \ \ + \max_{q\in Q_0\cap Q_1} \frac{\|\theta_q\|^2}{4a}\, ,
\label{ultimovalore}
\end{equation}
and show that the conclusion of Lemma \ref{lemmaminimax1} holds even considering the a priori larger interval $(k_{Q_0\cap Q_1}^-,c(L;Q_0,Q_1))$. This is however not the case, since under these assumptions
there are constant loops with negative $\A_k$-action. Therefore, in the energy range $(k_{Q_0\cap Q_1}^-,k_{Q_0\cap Q_1})$  instead of the class $\Gamma$ one has to consider the class of deformations 
$u=(x,T):[0,1]\times (Q_0\cap Q_1) \rightarrow \mathcal N$ of the space of constant paths into the space of paths with negative $\A_k$-action
$$\Gamma_{Q_0\cap Q_1} := \Big \{u=(x,T) \ \Big | \ x(0,q) = q\, , \ \A_k(u(1,q))< 0 \,, \ \forall q\in Q_0\cap Q_1\Big \}\, .$$

This argument is analogous to the one in \cite{Abb13}, where the case of periodic orbits is considered and $k_{Q_0\cap Q_1}^-$, $k_{Q_0\cap Q_1}$ are replaced by $\min E$, $e_0(L)$ respectively. 

The proof of \cite[lemma 7.2]{Abb13} goes through withouth any change and shows that the class $\Gamma_{Q_0\cap Q_1}$ is non-empty for every $k\in (k_{Q_0\cap Q_1}^-,c(L;Q_1,Q_0))$.
The proof of the following lemma is analogous to the one of Lemma \ref{lemmaminimax1}.

\begin{lemma}
For every $k\in (k_{Q_0\cap Q_1}^-,k_{Q_0\cap Q_1})$ there exists $\alpha>0$ such that for every $u\in \Gamma_{Q_0\cap Q_1}$ there holds 
$$\max_{(s,q)\in [0,1]\times (Q_0\cap Q_1)}\!\!\!\! \A_k(u(s,q)) \ \ \geq \ \alpha\, .$$
\label{lemmaminimax2}
\end{lemma}

\vspace{-5mm}
\noindent We can now define the minimax functions 
\begin{equation}
c:\big (k_{Q_0\cap Q_1}, c(L;Q_0,Q_1 )\big )\ \longrightarrow \ \R\, , \ \ \ \ c(k) := \ \inf_{u\in \Gamma} \ \max_{[0,1]} \ \A_k\circ u
\label{minimaxfunction1}
\end{equation}
\noindent and 
\begin{equation}
c:\big (k_{Q_0\cap Q_1}^-,k_{Q_0\cap Q_1} \big )\ \longrightarrow \ \R\, , \ \ \ \ c(k) := \ \inf_{u\in \Gamma_{Q_0\cap Q_1}} \ \max_{[0,1]} \ \A_k\circ u \, .
\label{minimaxfunction2}
\end{equation}

Lemmas \ref{lemmaminimax1} and \ref{lemmaminimax2} above imply that $c(k)>0$ for all $k$; furthermore, the monotonicity of $\A_k$ in $k$ implies that the minimax functions $c(\cdot)$ are monotonically 
increasing and hence almost everywhere differentiable. In Lemma \ref{struwe} below is a version of an argument of Struwe  (cf. \cite{Str90}), which allows to 
prove the existence of bounded Palais-Smale sequences for every value of $k$ at which the minimax functions $c(\cdot)$ are differentiable, thus overcoming the lack of the Palais-Smale 
condition for $\A_k$ for subcritical energies. The price to pay is that one is able to get compact Palais-Smale sequences only for almost every energy instead for 
every energy. The proof  is analogous to the one in the periodic case; see \cite{Con06} and \cite{Abb13} for further details.

\begin{lemma}
Suppose that $\bar{k}$ is a point at which the minimax function $c(\cdot)$ in (\ref{minimaxfunction1}) or (\ref{minimaxfunction2}) is differentiable. Then $\A_{\bar k}$ admits a bounded Palais-Smale sequence at level $c(\bar{k})$, 
which consists of paths in $\mathcal N$.
\label{struwe} 
\end{lemma}

\begin{proof}
Since $\bar k$ is a point of differentiability for $c(\cdot)$ we have 
\begin{equation}
|c(k)-c(\bar k)| \ \leq \ M \, |k-\bar k|
\label{modulocontinuita}
\end{equation}
for all $k$ sufficiently close to $\bar k$, where $M>0$ is a suitable constant. Let $\{k_h\}$ be a strictly decreasing sequence which converges to 
$\bar{k}$ and set $\epsilon_h:=k_h -\bar{k} \downarrow 0$. For every $h\in \N$ choose $u_h\in \Gamma$ (or $\Gamma_{Q_0\cap Q_1}$) such that 
$$\max_{u_h}\ \A_{k_h}\ \leq \ c(k_h) + \epsilon_h\, .$$

Up to ignoring a finite numbers of $k_h$'s we may suppose that Equation (\ref{modulocontinuita}) is satisfied by every $k=k_h$. If $z=(x,T)\in u_h$ is such that $\A_{\bar{k}}(z)>c(\bar{k})-\epsilon_h$, then 
$$T\ =\ \frac{\A_{k_h}(z)-\A_{\bar{k}}(z)}{k_h-\bar{k}}\ \leq \ \frac{c(k_h)+\epsilon_h - c(\bar{k}) + \epsilon_h}{\epsilon_h}\ \leq \ M+2\, .$$
Moreover, 
$$\A_{\bar{k}}(z)\ \leq \ \A_{k_h}(z)\ \leq \ c(k_h)+\epsilon_h \ \leq \ c(\bar{k}) + (M+1)\epsilon_h$$
and hence
$$u_h \ \subseteq \ \mathcal A_h \ \cup \ \Big \{\A_{\bar{k}} \ \leq \ c(\bar{k}) - \epsilon_h\Big \}\, ,$$
where 
$$\mathcal A_h\ =\ \Big \{ (x,T) \in \mathcal N \ \Big |\ T \leq  M+2\, ,\ \A_{\bar{k}}(x,T)\leq c(\bar{k}) + (M+1)\epsilon_h\Big \}\, .$$ 
Observe that, if $(x,T)\in \mathcal A_h$, then by (\ref{firstinequality}) we have
$$\A_{\bar{k}}(x,T) \ \geq \ \frac{a}{M+2} \, \big \|x'\big \|_2^2 - (M+2)\big |b-\bar{k}\big |$$
and hence 
\begin{eqnarray*}
\big \|x'\big \|_2^2 &\leq& \frac{M+2}{a} \, \Big ( c(\bar{k}) + (M+1)\epsilon_h + (M+2) \big |b - \bar{k}\big |\Big )\, ,
\end{eqnarray*}
which shows that $\mathcal A_h$ is bounded in $\mathcal N$, uniformly in $h$. Let $\Phi$ be the flow of the vector field obtained by multiplying $-\nabla \A_{\bar{k}}$ by a suitable non-negative function, whose role 
is to make the vector field bounded on $\mathcal N$ and vanishing on the sublevel $\big \{\A_{\bar{k}}\leq c(\bar{k})/4 \big\}$, while keeping the uniform decrease condition 
\begin{equation}
\frac{d}{d\sigma}\ \A_{\bar{k}}(\Phi_\sigma (z))\ \leq \ - \frac{1}{2} \, \min \Big \{ \big \|d\A_{\bar{k}}(\Phi_\sigma(z))\big \|^2, 1\Big \}\, ,\ \ \ \text{if}\ \ \A_{\bar{k}} (\Phi_\sigma (z))\ \geq \ \frac{c(\bar{k})}{2}\, .
\label{decrescita}
\end{equation}

Lemma \ref{convergenza} implies that $\Phi$ is well-defined on $[0,+\infty)\times \mathcal N$ and that $\Gamma$ (or $\Gamma_{Q_0\cap Q_1}$) is positively invariant with respect to $\Phi$. Since $\Phi$ maps bounded sets into bounded sets, we have that
\begin{equation}
\Phi\big ([0,1]\times u_h\big )\ \subseteq \ \mathcal B_h \ \cup \ \Big \{\A_{\bar{k}}\ \leq \ c(\bar{k})-\epsilon_h\Big \}
\label{fi}
\end{equation}
for some uniformly bounded set 
\begin{equation}
\mathcal B_h \ \subseteq \ \Big \{ \A_{\bar{k}} \ \leq \ c(\bar{k}) + (M+1)\epsilon_h\Big \}\, .
\label{Bh}
\end{equation}
We claim that there exists a sequence $\{z_h\}\subseteq \mathcal N$ with 
$$z_h \ \in \ \mathcal B_h \ \cap \ \Big \{ \A_{\bar{k}} \ \geq \ c(\bar{k}) - \epsilon_h \Big \}$$
and $\big \|d\A_{\bar{k}}(z_h)\big \|$ infinitesimal. Such a sequence is clearly a bounded Palais-Smale sequence at level $c(\bar{k})$. Assume by contradiction that there exists $\delta \in (0,1)$  such that
$$\big \|d\A_{\bar{k}}\big \| \ \geq \ \delta\, ,\ \ \ \text{on}\ \ \mathcal B_h \ \cap \ \Big \{ \A_{\bar{k}} \ \geq \ c(\bar{k}) - \epsilon_h \Big \}$$
for every $h$ large enough. Together with (\ref{decrescita}), (\ref{fi}) and (\ref{Bh}), this implies that, for $h$ large enough, for any $z\in u_h$ such that 
$$\Phi \big ([0,1]\times \{z\} \big ) \ \subseteq \ \Big \{\A_{\bar{k}}\ \geq \ c(\bar{k})-\epsilon_h\Big \}$$
there holds 
$$\A_{\bar{k}}(\Phi_1(z)) \ \leq \ \A_{\bar{k}}(z) - \frac{1}{2}\, \delta^2 \ \leq \ c(\bar{k}) + (M+1)\epsilon_h - \frac{1}{2}\, \delta^2\, .$$ 
It follows that 
$$\max_{\Phi_1 (u_h)}\ \A_{\bar{k}}\ \leq \ c(\bar{k}) - \epsilon_h$$ 
\noindent for $h$ large enough. Since $\Phi_1(u_h)\in \Gamma$, this contradicts the definition of $c(\bar{k})$.
\end{proof}

\begin{teo}
Let $Q_0,Q_1$ be two closed connected submanifolds such that the intersection $Q_0\cap Q_1$ is non-empty and connected. Then, for almost every $k \in (k_{Q_0\cap Q_1}^-, c(L;Q_0,Q_1))$ 
there is an Euler-Lagrange orbit $\gamma\in \mathcal N$ with energy $k$ which satisfies the conormal boundary conditions (\ref{lagrangianformulation1}) and has action $\A_k(\gamma)=c(k)$. 
\label{teobasseenergie1}
\end{teo}

\begin{proof}
Let $k$ be a point of differentiability for $c(\cdot )$. By Lemma \ref{struwe} above $\A_k$ admits a Palais-Smale sequence $(x_h,T_h)\subseteq \mathcal N$ at level $c(k)$ with uniformly bounded times $\{T_h\}$. At the same time, Lemmas 
\ref{lemmaminimax1} and \ref{lemmaminimax2} imply that $c(k)> 0$, so that by Lemma \ref{yesmodification} we have also that $\{T_h\}$ is bounded away from zero. Therefore, Lemma \ref{lemmalimitatezza} implies that the sequence 
$(x_h,T_h)$ has a limiting point in $\mathcal N$, which gives us the required Euler-Lagrange orbit. The assertion follows noticing 
that the set of points of differentiability for $c(\cdot)$ is a full measure set in $(k_{Q_0\cap Q_1}^-, c(L;Q_0,Q_1))$.
\end{proof}

\vspace{2mm}

\begin{cor}
Let $Q_0\subseteq M$ be a closed connected submanifold. Then for almost every $k\in (k_{Q_0}^-, c(L;Q_0))$ there is an Arnold chord $\gamma\in \mathcal N$ for $Q_0$
with energy $k$. Here $k_{Q_0}^-$ and $c(L;Q_0)$ are obtained from (\ref{ultimovalore}) and (\ref{clh0h1}) by setting $Q_0=Q_1$.
\label{arnoldchord}
\end{cor}

Notice that Theorem \ref{teobasseenergie1} can be trivially generalized to the case of non-connected intersection just considering separately every connected component of $\mathcal M_Q$ containing constant paths and repeating the same argument 
for any such component. 

More precisely, let $Q_0,Q_1\subseteq M$ be such that $Q_0\cap Q_1\neq \emptyset$ and let $\mathcal N$ be a connected component of $\mathcal M_Q$ containing constant paths. 
Denote by $\Omega$ the set of constant paths contained in $\mathcal N$; observe that $\Omega$ need not be connected. Therefore, let $\Omega_1,\Omega_2, ...$ be the connected components of $\Omega$ and for every
$j=1,2,...$ define 
$$k_{\Omega_j} := \ \max _{q\in \Omega_j} \ E(q,0) \ + \ \max_{q\in \Omega_j} \ \frac{\|\theta_q\|^2}{4a}\,, \ \ \ \ k_{\Omega_j}^- := \ \min_{q\in \Omega_j} \ E(q,0) \ + \ \max_{q\in \Omega_j} \ \frac{\|\theta_q\|^2}{4a} \, .$$

For every $k\in (k_{\Omega_j}^-,k_{\Omega_j})$, resp. $k \in (k_{\Omega_j},c(L;Q_0,Q_1))$, we define the minimax classes  $\Gamma_{\Omega_j}, \Gamma_j$  just replacing the whole intersection
$Q_0\cap Q_1$ with $\Omega_j$.  We then define the minimax functions $c_j$ as in (\ref{minimaxfunction1}) and (\ref{minimaxfunction2}) replacing $\Gamma,\Gamma_{Q_0\cap Q_1}$ with $\Gamma_j,\Gamma_{\Omega_j}$ respectively.
Theorem \ref{teobasseenergie1} therefore generalizes to the following

\begin{teo}
Let $\mathcal N$ be a connected component of $\mathcal M_Q$ which contains constant paths. Then, for almost every
$$k\in \left (\min_j\, k_{\Omega_j}^-\, , \, c(L;Q_0,Q_1)\right )$$ 
there is an Euler-Lagrange orbit $\gamma\in \mathcal N$ with energy $k$ satisfying the conormal boundary conditions and with action $\A_k(\gamma)=c(k)$. 
\label{teobasseenergie2}
\end{teo}

\vspace{2mm}

The minimax functions $c_j$ do not provide in general different critical values of $\A_k$. The only advantage in picking one different minimax class for 
each connected component of $\Omega$ is that one can push the existence result down to $\min k_{\Omega_j}^-$ instead of stopping at $k_\Omega^-$, since $\, \min k_{\Omega_j}^- \leq k_{\Omega}^-\, $
and the inequality might be strict.


\section{Counterexamples}
\label{counterexamples}

In this section we provide counterexamples to the existence of Euler-Lagrange orbits satisfying the conormal boundary conditions for subcritical energies. 

We have already seen in the introduction an example of Hamiltonian flow on the standard 2-sphere (which is locally an Euler-Lagrange flow) for which no energy level set contains an orbit 
joining the south pole with the north pole. Here we show examples of ``global'' Euler-Lagrange flows for which a similar phenomenonon happens for subcritical energies. Namely, for every $\epsilon>0$ 
we construct an example of magnetic Lagrangian $L_\epsilon$ on any surface with $c_u(L_\epsilon)\in [\frac12 -\epsilon,\frac12]$ such that there are two points which cannot be connected by Euler-Lagrange 
orbits with energy less than $\frac12 -\epsilon$. This result shows that Contreras' result \cite{Con06} is sharp. 

We then exhibit an example of magnetic flow on $T\T^2$ and closed disjoint submanifolds $Q_0,Q_1$ such that for all $k<c(L;Q_0,Q_1)$ there are no Euler-Lagrange orbits with energy $k$ 
satisfying the conormal boundary conditions, thus proving the sharpness of Theorem \ref{teorema1}. This example also shows that, in case of disjoint submanifolds, we may not expect any existence of 
orbits satisfying the conormal boundary conditions for $k<c(L;Q_0,Q_1)$ even if the submanifolds are ``close'' to each other. Moreover, it also shows that, 
in contrast with the case of periodic orbits (see e.g. \cite{Tai92a} or \cite{CMP04}), we may not expect the existence of orbits which are local minimizer of the free-time Lagrangian action functional, even if $M$ is a surface.

Finally, we provide examples of intersecting submanifolds for which there are no Euler-Lagrange orbits satisfying the conormal boundary conditions for energies below $k_{Q_0\cap Q_1}^-$, 
thus showing that the results obtained in Section \ref{subcriticalorbits} are optimal. 

\vspace{3mm}
Throughout this section $\Sigma$ will be a closed connected orientable surface and $\widetilde \Sigma$ will be its universal cover.  We start by considering the hyperbolic plane 
$$\HH := \Big \{(x_1,x_2) \in \R^2 \ \Big | \ x_2>0\Big\}$$
endowed with the Riemannian metric 
\begin{equation}
g_{(x_1,x_2)} := \ \frac{1}{x_2^2} \, \big (dx_1^2 + dx_2^2\big )\, .
\label{hyperbolicmetric}
\end{equation}
We refer to \cite{BKS91} for generalities and properties of  $(\HH,g)$. We define
\begin{equation}
L:T\HH\longrightarrow \R\, , \ \ \ \ L(q,v) \ = \ \frac12\, \|v\|_q^2 + \theta_q(v)\, ;
\label{hyperboliclagrangian}
\end{equation}
where $\theta_{(x_1,x_2)} =dx_1/x_2$ is the ``canonical primitive'' of the standard area form 
$$\sigma \ = \ \frac{1}{x_2^2}\, dx_1\wedge dx_2\, .$$ 

It is a well-known fact that $c(L)=\frac12$. In fact, the Hamiltonian characterization of the Ma\~n\'e critical value (cf. \cite{CI99} and \cite{BP02}) implies that 
$$c(L) \ = \ \inf_{u\in C^\infty(\HH)} \ \sup_{q\in \HH} \ \frac12 \, \|d_qu-\theta_q\|^2 \ \leq \ \frac12\, ,$$
being $\|\theta_q\|\equiv 1$. On the other hand, consider the clockwise arc-length parametrization $\gamma_r$ of a (hyperbolic) circle with radius $r$ and denote by $D_r$ the disc bounded by $\gamma_r$. 

Using the standard facts 
$$l(\gamma_r) \ = \ 2\pi\, \sinh r\, , \ \ \ \ \operatorname{Area} (D_r) \ = \ 2\pi \big (\cosh r - 1\big )\, ,$$
we readily compute for the action of $\gamma_r$
\begin{eqnarray*}
\A_k(\gamma_r) &=& \int_0^{l(\gamma_r)} \Big [\frac12\, \|\dot \gamma_r(t)\|^2 + k \Big ]\, dt \ + \int_{\gamma_r} \theta \\
                           &=&  \Big (\frac 12 + k \Big )\, l(\gamma_r) - \operatorname{Area} (D_r) \\ 
                           &=& \Big (\frac12 + k \Big )\, 2\pi \, \sinh r -  2\pi\, \big (\cosh r - 1\big )  \\
                           &=& \pi \, \Big (k-\frac12 \Big)\, e^r + f(r)\, ,
\end{eqnarray*}
with $f(r)$ uniformly bounded function of $r$. It follows that,  for every $k<\frac12$ 
$$\A_k(\gamma_r)\longrightarrow -\infty$$
as $r$ goes to infinity, thus showing that $c(L)\geq \frac12$. The restriction of the Euler-Lagrange flow of $L$ as in (\ref{hyperboliclagrangian}) to the energy level $\frac12$ is the celebrated \textit{horocycle flow} of Hedlund (cf. \cite{Hed32} and \cite{BKS91}). 
Its peculiarity relies on the fact that, once projected to a compact quotient of $\HH$, it becomes minimal, meaning that every orbit is dense. For $k<\frac12$, the Euler-Lagrange flow on $E^{-1}(k)$ is periodic 
and the orbits describe circles on $\HH$ with hyperbolic (thus, euclidean) radius going to zero as $k\rightarrow 0$. 

\paragraph{Orbits connecting two points}
Since for every $k<\frac12$ the Euler-Lagrange flow of $L$ as in (\ref{hyperboliclagrangian}) on the energy level $E^{-1}(k)$ is periodic with (projection of the) orbits given by hyperbolic (hence euclidean) circles with radius going to zero as $k\rightarrow 0$, it follows that  
for any pair of points $q_0\neq q_1$ there exists $k_0$ such that, for all $k< k_0$, there are no Euler-Lagrange orbits with energy $k$ connecting them. 

We explain now in details how to embed this example in any compact surface $\Sigma$; this will be used also later on for the other examples. 
Suppose for instance 
$$q_0\ =\ \Big (-\frac12,4\Big )\, ,\ \ \  \ q_1\ =\ \Big (\frac12,4\Big )$$
and let $B_1\subseteq B_2 \subseteq B_3$ be open (hyperbolic) balls around the point $(0,4)$ containing $q_0$ and $q_1$. Without loss of generality we may suppose that all 
Euler-Lagrange orbits with energy less than $k_0$ starting from $q_0$ or $q_1$ are entirely contained in $B_1$. Moreover, we may assume that $B_1$ contains a closed loop with negative $k$-action for all $k< k_0$.
We extend now the 1-form $\theta|_{B_1}$ to be constantly equal to zero outside $B_2$ using a suitable cut-off function and embed $B_3$ on $\Sigma$. In fact, the embedding induces a Riemannian metric on a subset 
$\mathcal U$ of $\Sigma$ which can be extended to a metric on the whole $\Sigma$ and also a 1-form on $\Sigma$ obtained simply by setting the pull-back of $\theta$ to be zero outside $\mathcal U$.
We denote the metric, the 1-form and the points on $\Sigma$ given by the embedding again with $g,\theta,q_0,q_1$ respectively and define the magnetic Lagrangian
\begin{equation}
L:T\Sigma \rightarrow \R\, , \ \ \ \ L(q,v) \ = \ \frac12 \, \|v\|_q^2 + \theta_q(v)\, .
\label{lagrangianasigma}
\end{equation}

\begin{center}
\includegraphics[height=40mm]{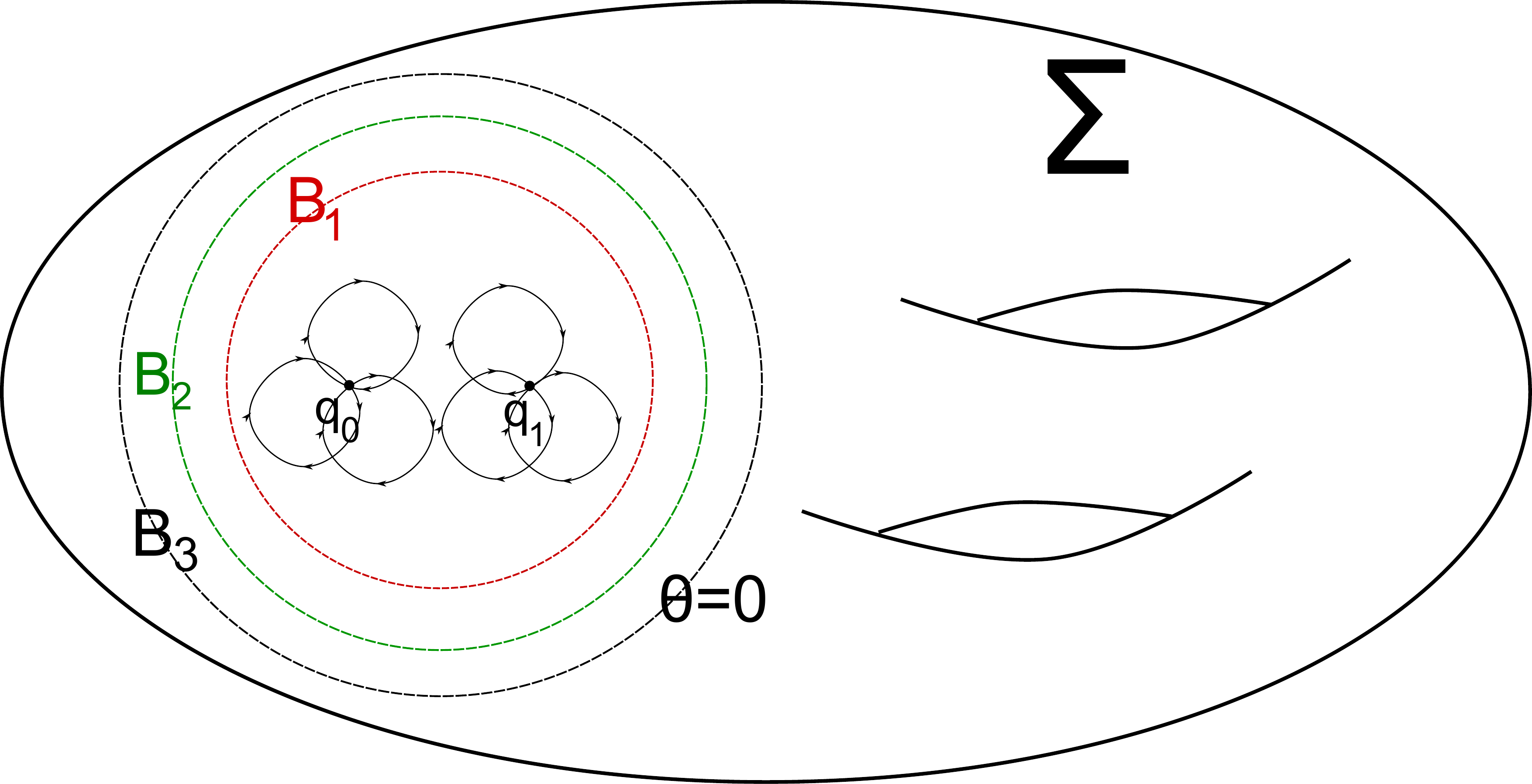}
\end{center} 

By construction there are no Euler-Lagrange orbits with energy $k$ connecting $q_0$ with $q_1$ for every $k<k_0$ and there holds $c_u(L)\geq k_0$; at the same time the Hamiltonian characterization of the Ma\~n\'e critical values
 implies that 
$$c_u(L) \ = \ \inf_{u \in C^\infty(\widetilde \Sigma)} \sup_{\tilde q\in \widetilde \Sigma} \ \frac12\, \|d_{\tilde q} u - \tilde \theta_{\tilde q}\|^2 \ \leq \ \frac12\, ,$$
where $\tilde \theta$ denotes the lift of $\theta$ to $\widetilde \Sigma$, since $\|\theta_q\|\leq 1$ for every $q\in \Sigma$.

\begin{lemma}
Let $\Sigma$ be a closed connected orientable surface. For any $\epsilon >0$ there exist a magnetic Lagrangian $L_\epsilon:T\Sigma \rightarrow \R$, with $c_u(L_\epsilon)\in [\frac12-\epsilon, \frac12]$, and points $q_0\neq q_1\in \Sigma$ such that, 
for all $k< \frac12 -\epsilon$, the energy level $E^{-1}(k)$ carries no Euler-Lagrange orbits joining $q_0$ to $q_1$.
\label{lemmanoq0q1}
\end{lemma}

\begin{proof}
The proof follows readily from the example above. Consider the Euler-Lagrange flow on $T\HH$ associated to the Lagrangian in (\ref{hyperboliclagrangian}) and fix $\epsilon>0$, $q_0\in \HH$. We know that, for every $k<\frac12$, the 
restriction of the Euler-Lagrange flow to $E^{-1}(k)$ is periodic and orbits describe hyperbolic (hence, euclidean) circles with the same hyperbolic radius. If we denote by $\rho(q_0,v)$ the euclidean radius of the (projection of the unique) Euler-Lagrange 
orbit through $q_0$ with speed $v$, then we readily have 
$$\rho := \ \max_{k\, \leq\,  \frac12 -\epsilon} \ \max_{\|v\|=k} \ \rho(q_0,v) \ < \ \infty\, .$$

Choose now $q_1\in \HH$ with euclidean distance from $q_0$ larger than $2\rho$ and, repeating the embedding procedure above with appropriate open sets $B_1,B_2,B_3$, 
we end up with a Lagrangian $L_\epsilon:T\Sigma \rightarrow \R$ as in (\ref{lagrangianasigma}) that satisfies the desired properties. 
\end{proof}

\paragraph{Sharpness of Theorem \ref{teorema1}} Lemma \ref{lemmanoq0q1} can be trivially generalized to

\begin{cor}
For every $\epsilon >0$ there exist a magnetic Lagrangian $L_\epsilon:T\Sigma \rightarrow \R$, with $c_u(L_\epsilon)\in [\frac12 -\epsilon,\frac12 ]$, and disjoint closed connected submanifolds 
$Q_0,Q_1\subseteq \Sigma$ such that, for all $k< \frac12-\epsilon$, the energy level $E^{-1}(k)$ carries no Euler-Lagrange orbits connecting $Q_0$ with $Q_1$. In particular, for every $k<\frac12-\epsilon$ there are no Euler-Lagrange orbits 
satisfying the conormal boundary conditions (\ref{lagrangianformulation1}).
\label{corollarynoq0q1}
\end{cor}

Corollary \ref{corollarynoq0q1} states that below $c_u(L)$ we might not  find Euler-Lagrange orbits connecting two given disjoint submanifolds. The natural question is now to study what happens in the interval $(c_u(L),c(L;Q_0,Q_1))$;
in fact, for every energy in this range every point of $Q_0$ can be joined with every point of $Q_1$ (cf. \cite{Con06}), but it is not clear if such orbits also satisfy the conormal boundary conditions. 

In this sense we show now that Theorem \ref{teorema1} is optimal; namely, we exhibit an example of a Tonelli Lagrangian and disjoint submanifolds $Q_0,Q_1$ such that $c_u(L)<c(L;Q_0,Q_1)$ and for every 
$k<c(L;Q_0,Q_1)$ there are no Euler-Lagrange orbits satisfying the conormal boundary conditions (\ref{lagrangianformulation1}). 

Following \cite{Man96} we first produce a situation where 
$$e_0(L)<c_u(L)<c(L)\, ;$$
think of $\T^2$ as the square $[0,1]^2$ in $\R^2$ with identified sides and equipped with the euclidean metric and consider the magnetic Lagrangian 
\begin{equation}
L:T\T^2 \longrightarrow \R\, , \ \ \ \ L (q,v) \ = \ \frac12 \, \|v\|_q^2 + \psi(y)\, v_x\, ,
\label{teorema1ottimale}
\end{equation}
where $q=(x,y)$, $v=(v_x,v_y)$ and $\psi:[0,1]\rightarrow [0,1]$ is a smooth cut-off function compactly supported in $(0,1)$ and constant in a neighborhood of $\frac12$.

\begin{center}
\includegraphics[height=40mm]{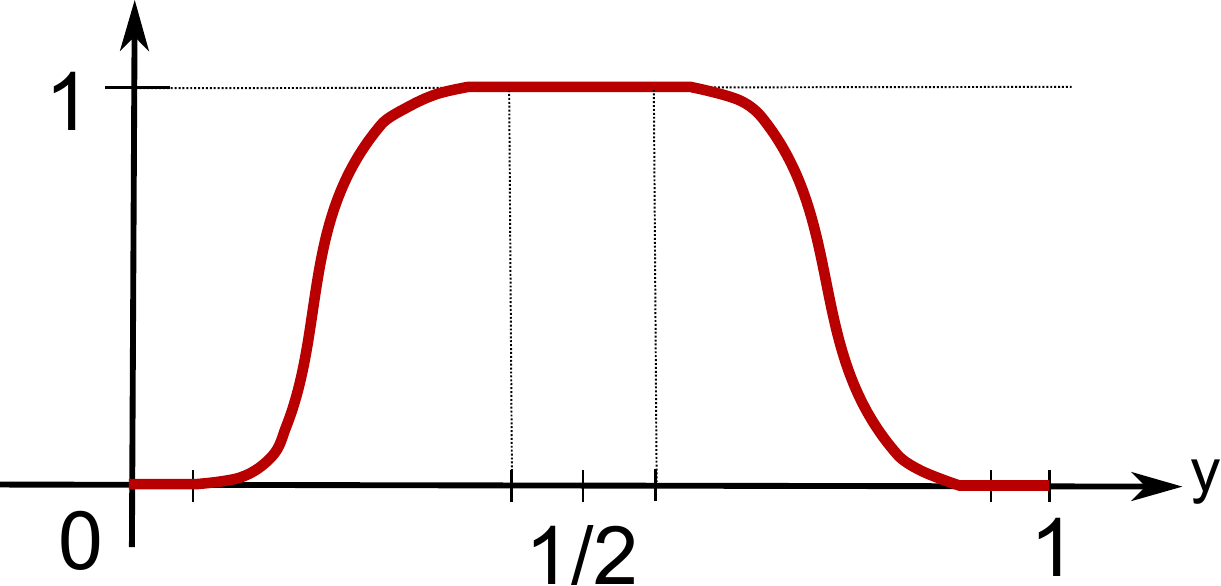}
\end{center}

The Lagrangian in (\ref{teorema1ottimale}) is a magnetic Lagrangian with magnetic 1-form $\theta_q(\cdot) = \psi(y)dx$. It follows that $|\theta_q|=|\psi(y)|$ for every $q=(x,y)$ and hence 
$$c(L) \ = \inf_{u\in C^\infty(\T^2)} \max_{q\in \T^2} \ \frac12 \, \|d_qu-\theta_q\|^2 \ \leq \ \frac12\, .$$

Conversely, consider the path $a:[0,1]\rightarrow \R^2,\, a(t) \ = \ (1-t\, ,\, \frac12)$; it is clear that $a$ is closed as a path in $\T^2$. We now readily compute for $k>0$ 
\begin{eqnarray*}
\A_k(a) &=&  \int_0^1 \left (\frac12 \, \|\dot a(t)\|^2 + \psi(a(t))\dot a_x(t) +k \right )\, dt \ = \\ 
             &=& \int_0^1  \left (\frac12 \, \|(-1,0)\|^2 - \psi\left (1-t\, ,\, \frac 12\, \right ) +k \right )\, dt  \ = \ k \, -\,  \frac12\, ,
\end{eqnarray*}
which is negative for every $k<\frac12$. We may then conclude that $c(L)=\frac12$.

\begin{center}
\includegraphics[height=50mm]{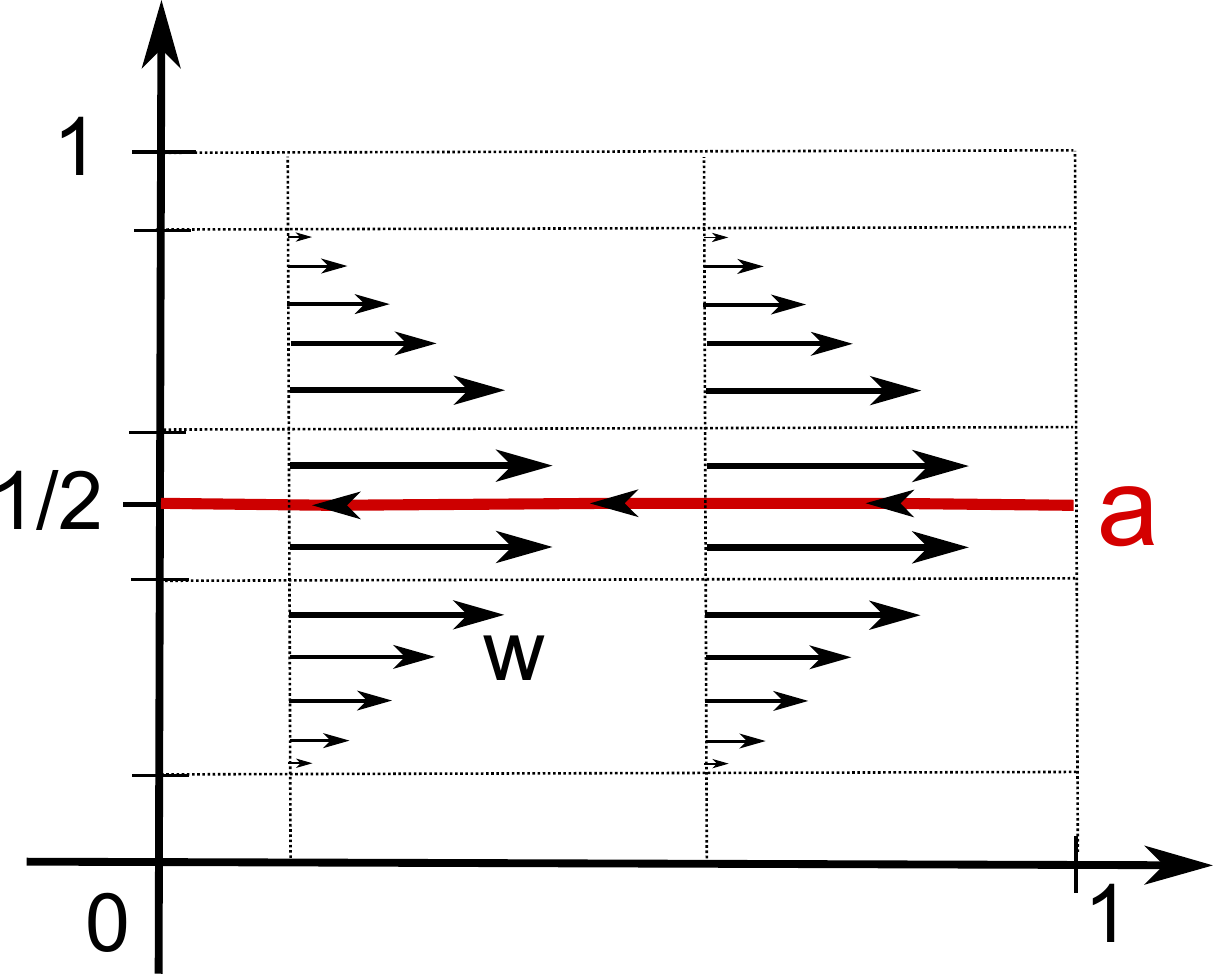}
\end{center}

\noindent Again, by the Hamiltonian characterization of the Ma\~n\'e critical value we have 
$$c_u(L) \ = \inf_{u\in C^\infty(\R^2)} \sup_{q\in \R^2} \ \frac12 \, \|d_qu-\theta_q\|^2\ \leq \ \frac18$$
as one gets by choosing $u:\R^2\rightarrow \R$, $u(x,y) = \frac{x}{2}$
$$\frac12 \, \left\|\frac12 dx - \psi(y)dx\right \|^2\ = \ \frac12 \, \left |\frac12 - \psi(y) \right|^2\ \leq \ \frac18\, , \ \ \ \ \forall (x,y)\in \R^2\, ,$$
since $0 \leq \psi(y)\leq 1$ for every $y$. On the other hand, $c_u(L)>0$ by the following general 

\vspace{2mm}

\begin{lemma}
Consider the magnetic Lagrangian 
$$L:T\Sigma \longrightarrow \R\, , \ \ \ \ L(q,v) \ = \ \frac12 \, \|v\|_q^2 + \eta_q(v)\, ,$$
with $\eta\in \Omega^1(\Sigma)$ non-closed. Then $c_u(L)>0$.
\label{lemmanoexact}
\end{lemma}

\begin{proof}
Being $\eta$ non-closed, there exists a closed contractible loop $\gamma$ in $\Sigma$ over which the integral of $\eta$ is non-zero. Without loss of generality we suppose $\gamma$ to be the $\sqrt{2k}$-arclength parametrization 
of a circle. Denote with $D$ the disc bounded by $\gamma$ and by $\sigma$ the differential of $\eta$; if we choose the right orientation for $\gamma$ we get
\begin{eqnarray*}
\A_k(\gamma) &=& \int_0^{\frac{l(\gamma)}{\sqrt{2k}}} \Big [\frac12\, \|\dot \gamma(t)\|^2 + k \Big ]\, dt  \ + \ \int_\gamma \eta\ = \ \sqrt{2k}\, l(\gamma)  \ - \ \int_D \sigma\, ,
\end{eqnarray*}
which is a negative quantity for $k$ small enough.
\end{proof}

\vspace{5mm}

\noindent We then conclude that for the Lagrangian $L$ in (\ref{teorema1ottimale}) there holds 
$$0 \ = \ e_0(L) \ <\ c_u(L) \ < \ c(L)\, ,$$
as we wished to prove. We pick now $Q_0$ to be any point in $\T^2$, for instance $(\frac12,0)$ and $Q_1$ to be the circle $\{y=\frac12\}$; by construction we have 
$$c(L;Q_0,Q_1)\ = \ c(L)\ =\  \frac12\, .$$

Since $\pi_0(\mathcal M_Q)\cong \Z$, Theorem \ref{teorema1} implies then that for every $k>\frac12$ there are infinitely many Euler-Lagrange orbits with energy $k$ satisfying the conormal 
boundary conditions (\ref{lagrangianformulation1}). Namely, there is one such solution for every connected component of $\mathcal M_Q$, which is moreover a global minimizer of the free-time 
action functional $\A_k$ in its connected component. Furthermore, for every $k\in \big (c_u(L),c(L;Q_0,Q_1)\big )$ and any point $q_1\in Q_1$ there are infinitely many Euler-Lagrange orbits
with energy $k$ joining $q_0$ with $q_1$; however, none of these can satisfy the conormal boundary conditions for $Q_1$, since this is possible only above energy $\frac12$. In fact, in this case we have 
$\frac12 \, \|\mathcal Pw_{q_1}\|^2 = \frac12$ for every $q_1\in Q_1$ and hence the assertion follows from the obstruction (\ref{obstruction}).
The same counterexample holds clearly for every point of the form $q_0=(\frac12,\epsilon)$ for every $\epsilon>0$, in particular showing that we might not expect to find Euler-Lagrange orbits 
with energy less than $c(L;Q_0,Q_1)$ satisfying the conormal boundary conditions, even if the two submanifolds are ``close'' to each other.
Notice that this example for $q_0=(\frac12,\frac12)$ is not in contradiction with theorem \ref{teobasseenergie2}, since in this case we have 
$$k_{Q_0\cap Q_1}^- \ = \ k_{Q_0\cap Q_1} \ = \ c(L;Q_0,Q_1)\, .$$
The Lagrangian in (\ref{teorema1ottimale}) can be used also to construct an example in which the interval $(c(L;Q_0,Q_1),k_{\mathcal N}(L))$ as in the second statement of theorem \ref{teorema2} is non-empty. 
Just consider as $Q_0$ and $Q_1$ two small (contractible) intersecting circles; then it is clear from the construction that 
$$c(L;Q_0,Q_1) \ = \ c_u(L) \ < \ c(L) \ = \ k_{\mathcal N}(L)\ = \ \frac12 \, .$$ 

\paragraph{Sharpness of theorem \ref{teobasseenergie2}}
We preliminarly show that the condition (\ref{obstruction}) is not sufficient to guarantee the existence of Euler-Lagrange orbits satisfying the conormal boundary conditions, even if one assumes $Q_0\cap Q_1$ non-empty.  
Namely, we construct an example for which the right-hand side of (\ref{obstruction}) is zero and, at the same time, sufficiently low energy level sets carry no Euler-Lagrange orbits satisfying the conormal boundary conditions.
Thus, consider the Euler-Lagrange flow on $T\HH$ defined by the Lagrangian in (\ref{hyperboliclagrangian}) and observe that for any circle $Q$ in $\HH$ we have 

$$\min \ \left \{ \left. \frac 12 \, \|\mathcal P w_{q}\|^2 \ \right | \ q\in Q \right  \}\ = \ 0\, ,$$
where $w_q\in T_q\HH$ is the unique vector representing $\theta_q$ and $\mathcal P$ denotes the orthogonal projection onto $TQ$. A simple computation shows that $w_q=(y_q,0)$ for all $q\in \HH$, where 
$y_q$ is the second coordinate of $q$. Let now $Q_0,Q_1$ be as in the figure below 

\begin{center}
\includegraphics[height=45mm]{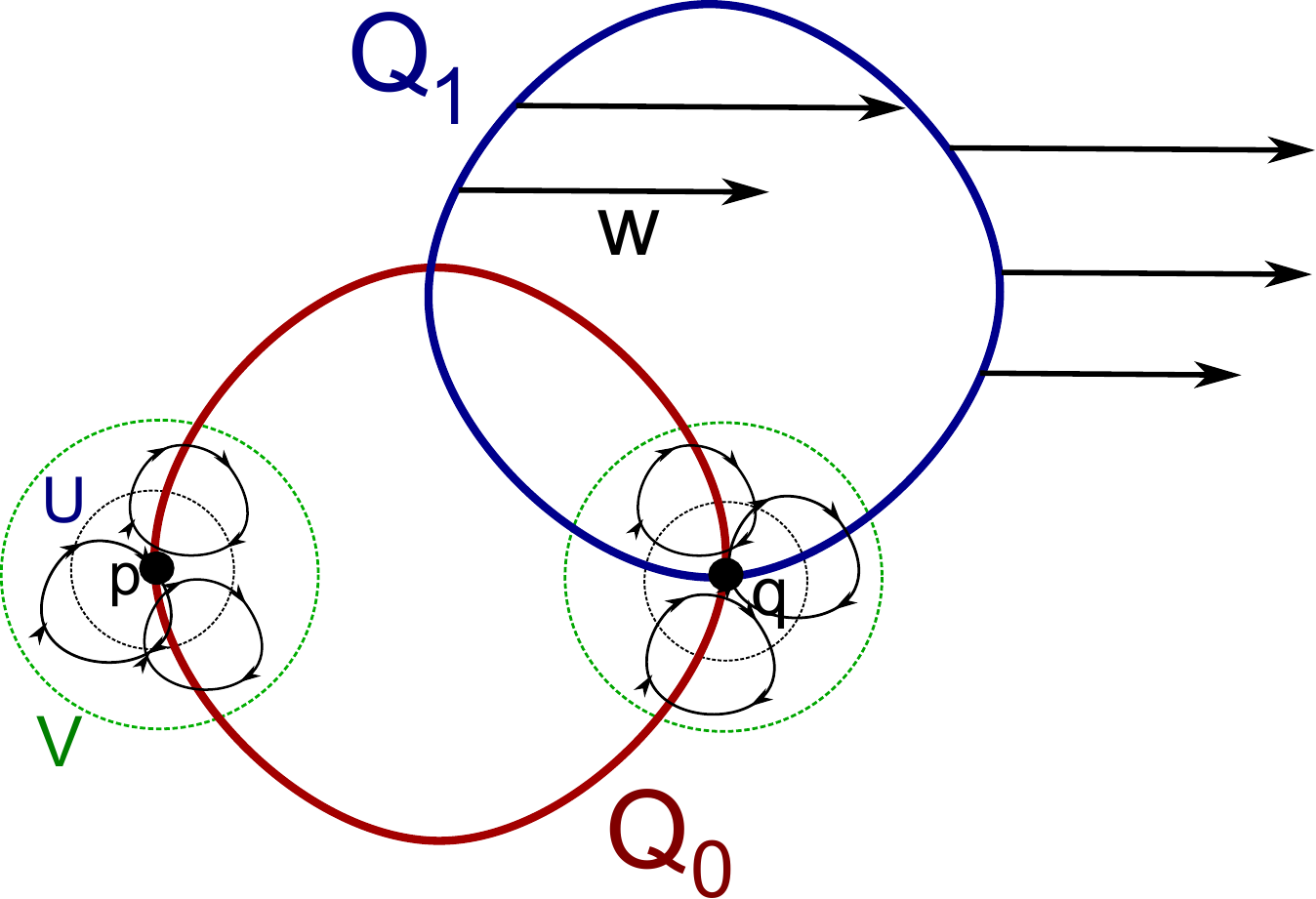}
\end{center}

Here $p$ and $q$ are the (only) points in $Q_0$ where $\theta |_{TQ_0}=0$. Choose now $V$ to be an open neighborhood of $p$ and $q$ as in the picture such that
\begin{equation}
\inf  \ \left \{ \left. \frac 12 \, \|\mathcal P w_{q_1}\|^2 \ \right | \ q_1\in Q_1\cap V \right  \} \ > \ \lambda\, .
\label{no}
\end{equation}

It is clear that, if $\epsilon < \lambda$ is sufficiently small, then all the orbits with energy $k\leq \epsilon$ starting from $Q_0$ and satisfying the conormal boundary conditions for $Q_0$ necessarily 
have starting point in a small neighborhood $U$ of $p$ and $q$. Thus, up to choose a smaller $\epsilon$ (hence, a smaller $U$), we might suppose that the orbits starting from a point in $U\cap Q_0$ with 
energy $k\leq \epsilon$ and satisfying the conormal boundary conditions are entirely contained in $V$. By (\ref{no}) all these orbits cannot satisfy the conormal boundary conditions for $Q_1$.
Hence, for all $k$ sufficiently small, there are no Euler-Lagrange orbits with energy $k$ satisfying the conormal boundary conditions (\ref{lagrangianformulation1}). 
Embedding this local model we get a counterexample for each $\Sigma$.

We extend now the example above to all energies $k<\frac12 -\epsilon$ for every $\epsilon >0$. Thus, consider the Euler-Lagrange flow of the Lagrangian in (\ref{hyperboliclagrangian}) and fix $\epsilon >0$. Consider a smooth 
``circle'' $Q_0$ as in the figure below

\begin{center}
\includegraphics[height=50mm]{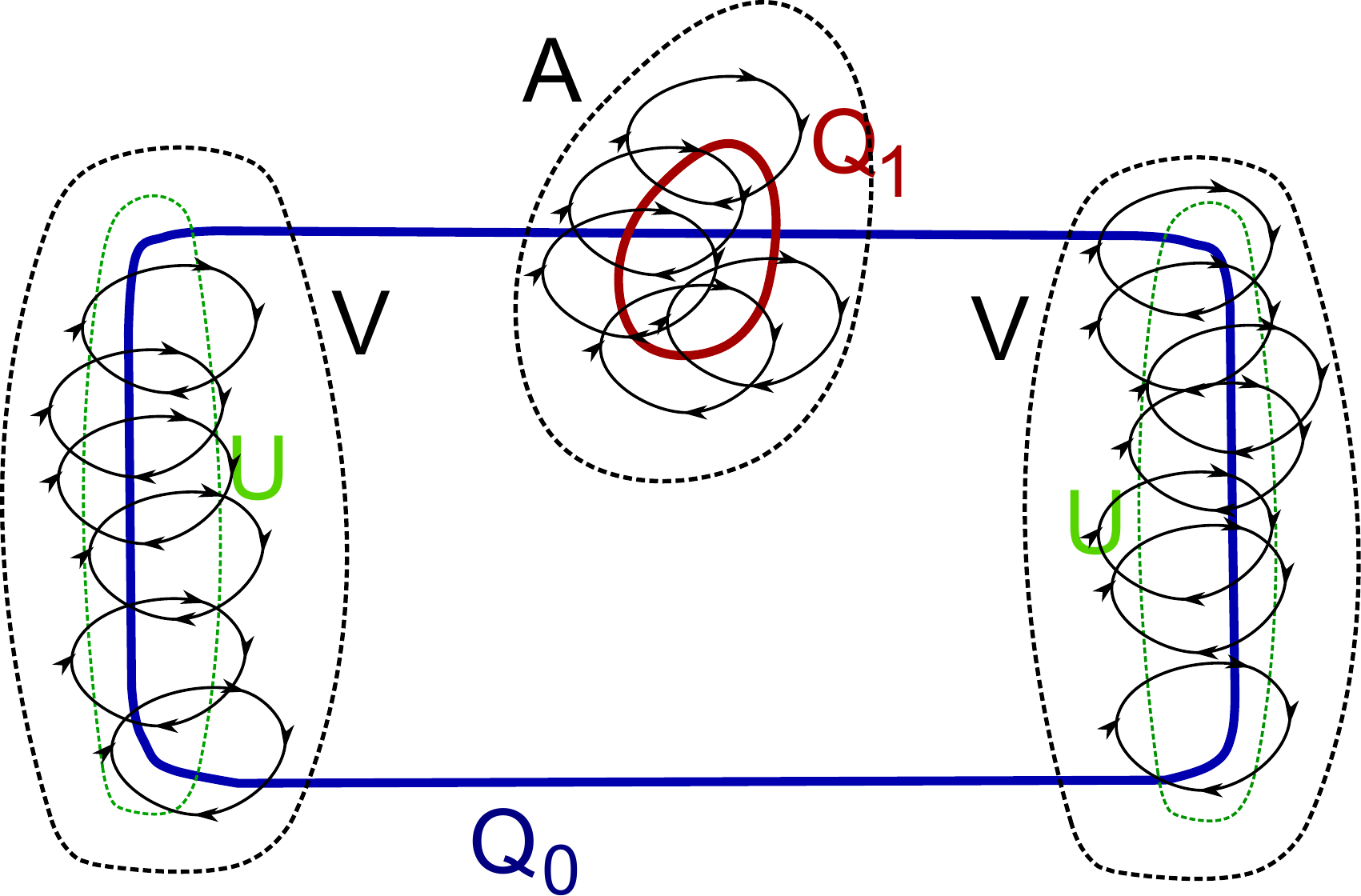}
\end{center}

The ``horizontal edges'' of $Q_0$ are straight segments ``long enough'' such that every Euler-Lagrange orbit with energy $k< \frac12 -\epsilon$ through any point in $Q_0\cap U$ is entirely contained in $V$, where $U,V$ are open sets as in the picture. It follows that all the Euler-Lagrange orbits connecting $Q_0$ to $Q_1$ with energy less than $\frac12-\epsilon$ and satisfying the conormal boundary conditions necessarily start from a point on one of the horizontal sides of $Q_0$.
It is now clear that such orbits do not exist, since $\|\mathcal P w_q\|=1$ for any $q$ on the horizontal edges of $Q_0$ and hence solutions through $q$ might exists only above energy $\frac12$. 
Again, embedding this local example in any surface $\Sigma$ we get the following

\begin{prop}
Let $\Sigma$ be a closed connected orientable surface. For any $\epsilon >0$ there exist $L_\epsilon:T\Sigma \rightarrow \R$, with $c_u(L_\epsilon)\in [\frac12 -\epsilon,\frac12]$, and closed connected submanifolds $Q_0,Q_1\subseteq \Sigma$ 
with $Q_0\cap Q_1\neq \emptyset$ such that for all $k< \frac12-\epsilon$ the energy level $E^{-1}(k)$ carries no Euler-Lagrange orbits satisfying the conormal boundary conditions (\ref{lagrangianformulation1}).
\label{propnoq0q1intersection}
\end{prop}

\vspace{-5mm}

The proposition above implies that Theorem \ref{teobasseenergie2} is optimal, meaning that in general we cannot find Euler-Lagrange orbits satisfying the conormal boundary conditions below $k_{Q_0\cap Q_1}^-$. 
Indeed, fix $\epsilon>0$, consider the last  local example and embed it in an open set $U$ of $\Sigma$. Now consider the Lagrangian $L:T\HH\rightarrow \R$ as in (\ref{hyperboliclagrangian}) replacing $\theta$ with $2 \theta$ and observe that $c(L)=2$. Indeed, we have
$$c(L) \ = \  \inf_{u\in C^\infty(\HH)} \sup_{q\in \HH} \ \frac12 \, \|d_qu-2\theta_q\|^2 \ = \ 4 \inf_{v\in C^\infty(\HH)} \sup_{q\in \HH} \ \frac12 \, \|d_qv-\theta_q\|^2 \ = \ 2\, .$$

Alternatively, one can compute the action of the 2-arclength clockwise parametrization $\gamma_r$ of a (hyperbolic) circle with radius $r$ and show, exactly as done at the beginning of this section, 
that $\A_k(\gamma_r)\rightarrow -\infty$ as $r$ goes to infinity for every $k<2$.

Now, pick $C\subseteq C' \subseteq \HH$ compact sets, with $C$ containing closed loops with negative $k$-action for all $k<2-\epsilon$. Using a suitable smooth cut-off function, define
$$L_2:T\HH\longrightarrow \R\, , \ \ \ \ L_2(q,v) \ = \ \frac12 \, \|v\|_q^2 + \tilde \theta_q(v)$$ 
with $\tilde \theta \equiv 2 \theta$ on $C$ and $\tilde \theta  \equiv  0$ outside $C'$. Clearly there holds $c(L_2)\in [2-\epsilon,2]$. Now, embed this local example in an open set $V$ 
of $\Sigma$ disjoint from $U$. The Riemannian metrics on $U,V$ can be clearly extended to a metric on $\Sigma$. Also, the forms $\theta,\tilde \theta$ on $U,V$ 
can be extended to a 1-form $\mu$ on $\Sigma$, just by setting $\mu$ to be zero outside $U$ and $V$, thus defining a magnetic Lagrangian $L_\epsilon:T\Sigma\rightarrow \R$. By construction there holds 
$$c_u(L_\epsilon)\ = \ c(L_\epsilon;Q_0,Q_1 )\ \in \ [2-\epsilon,2]\, , \ \ \ \  k_{Q_0\cap Q_1} =\ k_{Q_0\cap Q_1}^- =\ \frac12\, .$$

\begin{center}
\includegraphics[height=40mm]{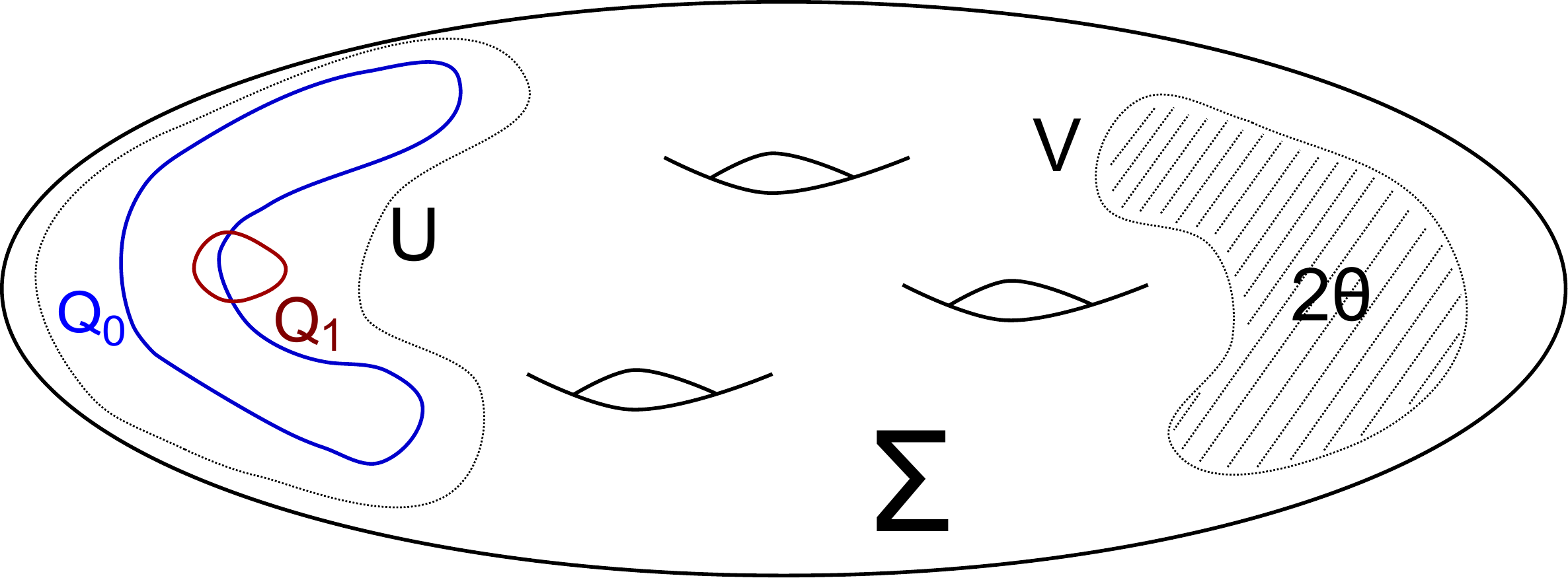}
\end{center}

Theorem \ref{teobasseenergie2} implies then that for almost every $k\in (\frac12,2-\epsilon)$ there is an Euler-Lagrange orbit with energy $k$ satisfying the conormal boundary conditions (\ref{lagrangianformulation1}). At the same time, such orbits
do not exist for $k< \frac12 - \epsilon$.


\chapter{Generalizations}
\label{chapter4}

In this chapter we extend the results of Chapter \ref{chapter3} to a more general setting, in a sense that we are now going to clarify. 
First we shall explain what we mean by the flow of the pair $(L,\sigma)$, with $L$ Tonelli Lagrangian and $\sigma$ closed 2-form on $M$.

Thus, let $(M,g)$ be a closed $n$-dimensional Riemannian manifold, $\sigma\in \Omega^2(M)$ be a closed 2-form on $M$ and $L:TM\rightarrow\R$ be a Tonelli Lagrangian. 
As usual, we denote with $E:TM\rightarrow \R$ the energy function associated to $L$ and with $H:T^*M\rightarrow\R$ the Tonelli Hamiltonian given as the Fenchel dual of $L$.
 
The pair $(L,\sigma)$ defines a flow on $TM$ in the following way. Take an open cover $\{U_i\}$ of $M$ such that $\sigma=d\theta_i$ on $U_i\subset M$. The Lagrangian functions 
$$L+\theta_i:TU_i\longrightarrow\R$$
yield flows on each $TU_i$ via the Euler-Lagrange equation (\ref{eulerlagrange1}). Such flows glue together since on $U_i\cap U_j$, 
$$L+\theta_j-(L+\theta_i)\ =\ \theta_j-\theta_i$$
is a closed form. We then associate to the pair $(L,\sigma)$ the global flow that we get on $TM$ by the gluing procedure. There is also a flow on $T^*M$ associated to $H$ and $\sigma$, that is the Hamiltonian 
flow of $H$ with respect to the \textit{twisted symplectic form}
\begin{equation}
\omega_\sigma:=\ dp \wedge dq + \pi^*\sigma\, .
\label{twistedsymplecticform}
\end{equation}

It is easy to see that the Hamiltonian flow defined by the pair $(H,\omega_\sigma)$ is conjugated to the flow of $(L,\sigma)$ via the Legendre transform $\mathcal L$.
In fact, by the gluing procedure above, it is enough to prove it when $\sigma$ is exact. Thus, let $\sigma=d\theta\in\Omega^1(M)$; if
$$\bar H (q,p)\ =\ H(q,p-\theta_q)\, ,$$
we readily see that the Hamiltonian flow of $H$ with respect to $\omega_{d\theta}$ is conjugated to the Hamiltonian flow of $\bar H$ with respect to $dp\wedge dq$ by the translation map $(q,p)\mapsto (q,p+\theta_q)$.
If $\bar{\mathcal L}:TM\rightarrow T^*M$ is the Legendre transform associated with $\bar H$ and $\bar L:TM\rightarrow \R$  is the Fenchel dual of $\bar H$, it suffices to show that
\begin{equation}\label{lemma1}
\mathcal L(q,p)=\bar{\mathcal L}(q,p+\theta_q)\quad\mbox{and}\quad\bar L(q,v) \ = \ L(q,v) + \theta_q(v)\, . 
\end{equation}
Indeed, we have 
\[v \ = \ d_p \bar H(q,p) \ = \ d_pH(q,p-\theta_q) \ \ \ \Longrightarrow \ \ \ p \ = \ \theta_q + \mathcal L^{-1} (v)\,\]
and the first identity in \eqref{lemma1} follows. For the second identity we observe that
\begin{align*} 
\bar L(q,v) &=\ \langle \theta_q + \mathcal L^{-1} (v),v\rangle - \bar H(q,\theta_q + \mathcal L^{-1}(v)) \ = \\ 
            &=\ \theta_q(v) + \langle \mathcal L^{-1}(v),v\rangle - H(q,\mathcal L^{-1}(v)) \ = \  \theta_q(v) +  L(q,v)\,.
\end{align*}

In particular, trajectories of the flow associated with $(H,\omega_\sigma)$ contained in $H^{-1}(k)$ correspond to trajectories of the flow defined by $(L,\sigma)$ contained in $E^{-1}(k)$. 

\vspace{3mm}

As usual let $\widetilde M$ be the universal cover of $M$. Consider now two closed connected submanifolds $Q_0,Q_1\subseteq M$ and denote by $M_1$ the cover
\begin{equation}
M_1 := \ \bigslant{\widetilde M}{\langle H_0,H_1\rangle}\, ,
\label{m1capitolo5}
\end{equation}
where $H_0$, $H_1$ are, as in Section \ref{manecriticalvalues}, the smallest normal subgroup of $\pi_1(M)$ containing $i_*(\pi_1(Q_0))$, $i_*(\pi_1(Q_1))$ respectively. 
We say that a closed 2-form $\sigma\in \Omega^2(M)$ is $M_1$-\textit{weakly-exact} if its lift $\sigma_1$ to the cover $M_1$ is exact; observe that this definition coincides with the usual definition of 
weak exactness when $M_1=\widetilde M$. We thus consider the flow of the pair $(L,\sigma)$, with $L$ Tonelli Lagrangian and $\sigma$ $M_1$-weakly-exact 2-form, and study the existence of orbits for the flow of $(L,\sigma)$ 
satisfying the ``conormal boundary conditions'' on a given energy level $E^{-1}(k)$. Notice that, since energy level sets are compact, we can assume $L$ to be electromagnetic outside a compact set,
hence satisfying the bounds (\ref{firstinequality}) and (\ref{secondinequality}).

The first problem, when trying to generalize the results of Chapter \ref{chapter3} to this setting, is to make sense of the conormal boundary conditions (\ref{lagrangianformulation1}) 
in this context. Notice that the fact that $\sigma$ is $M_1$-weakly-exact implies that $\sigma$ is exact on $Q_0,Q_1$ (see Section \ref{thefunctionalsk} for further details). Thus, fix two primitives $\theta_0,\theta_1$ of $\sigma$ on $Q_0,Q_1$ respectively. 
We say that an orbit $\gamma:[0,1]\rightarrow M$ of the flow of $(L,\sigma)$ satisfies the \textit{conormal boundary conditions with respect to} $\theta_0,\theta_1$ if
\begin{equation}
\left \{ \begin{array}{l} 
\gamma(0)\in Q_0\, , \ \gamma(1)\in Q_1\, ; \\ \\ 
d_v L(\gamma(0),\dot \gamma(0)) + \theta_0(\gamma(0)) \Big |_{T_{\gamma(0)}Q_0} \!\!\!\!\!\! = \ \ d_v L(\gamma(1),\dot \gamma(1)) + \theta_1(\gamma(1)) \Big |_{T_{\gamma(1)}Q_1} \!\!\!\!\!\! = \ \ 0\, ;
\end{array}\right.
\label{lagrangianformulationtheta0theta1}
\end{equation}

Second, the action functional $\A_k$ is in general not available anymore. We shall then replace $\A_k$ by another functional $S_k:\mathcal M_Q\rightarrow \R$ that detects the orbits of 
the flow of $(L,\sigma)$ with energy $k$ satisfying the boundary conditions (\ref{lagrangianformulationtheta0theta1}). This will be done in the first section under an additional technical assumption on $\sigma$.
In Section \ref{thefunctionalsk} we define $S_k$ and discuss its regularity properties; we then introduce the Ma\~n\'e critical value $c(L,\sigma;Q_0,Q_1)$ which is relevant in this context. 
It is worth to point out that, in contrast with what happens for the critical value in Chapter \ref{chapter3}, $c(L,\sigma;Q_0,Q_1)$ can be infinite. This is for instance the case whenever $\pi_1(M)$ is amenable (cf. \cite{Pat06}).

In Section \ref{thepalaissmalecondition} we study the Palais-Smale condition for $S_k$, showing in particular that $S_k$ satisfies the Palais-Smale condition on subsets of $\mathcal M_Q$ with times bounded and bounded 
away from zero, thus generalizing Lemma \ref{lemmalimitatezza} to this setting. The strength of this result is that it is independent on the finiteness of $c(L,\sigma;Q_0,Q_1)$ (see \cite[section 3]{AB14} 
for the analogous result in the case of periodic orbits). In Section \ref{existenceresults} we then show how the results in Chapter \ref{chapter3} extend to this setting.

Finally, in Section \ref{whensigmaisonlyclosed} we drop the $M_1$-weak-exactness assumption and prove that, under the additional hypothesis that $\pi_l(\mathcal M_Q)\neq 0$ for some $l\geq 1$,
an analogue of the second statement of Theorem \ref{teobasseenergie2weaklyexact} holds also only assuming 
that the pull-back of $\sigma$ to $Q_0$ and $Q_1$ is exact and that we can find primitives $\theta_0$ and $\theta_1$ of $\sigma|_{Q_0}$ and $\sigma|_{Q_1}$ respectively that coincide on the intersection $Q_0\cap Q_1$.

The method used to prove this result is analogous to the one exploited in Chapter \ref{chapter5} to prove the existence of periodic orbits for the flow of the pair $(L,\sigma)$. Roughly speaking, it consists on finding zeros of the 
so-called \textit{action 1-form} $\eta_k$ (namely, the differential of the free-time action functional when $\sigma$ is exact). The importance of $\eta_k$ relies on the fact that it is well-defined, under our assumptions about 
the exactness of $\sigma$ on $Q_0$ and $Q_1$, even if $\sigma$ is only closed and its zeros correspond to the orbits of the flow of $(L,\sigma)$ with energy $k$ that satisfy the boundary conditions (\ref{lagrangianformulationtheta0theta1}). 

We thus show the existence of zeros for $\eta_k$ using the existence of a non-trivial element $\mathfrak U$ in some $\pi_l(\mathcal M_Q)$ to construct a minimax class with associated minimax function $c^{\mathfrak u}$. The monotonicity in $k$
of this minimax function will allow us to prove the existence of critical sequences (the natural replacement of Palais-Smale sequences in this setting) for $\eta_k$ with times bounded and bounded away from zero
for almost every energy by generalizing the \textit{Struwe monotonicity argument} to this setting. We will finally retrieve the existence of the desired zeros of $\eta_k$ using a compactness criterion for critical sequences  
of $\eta_k$ with times bounded and bounded away from zero (cf. Proposition
\ref{prp:pstheta0theta1}), which generalizes the corresponding result for Palais-Smale sequences (cf. Proposition \ref{completezzask}) to this setting.

The rigorous study of the properties of $\eta_k$ will be performed in the next  chapter (cf. Sections \ref{theaction1form} and \ref{ageneralizedpseudogradient}). 


\section{The functional $S_k$.}
\label{thefunctionalsk}

Let $L:TM\rightarrow \R$ be a Tonelli Lagrangian, $Q_0,Q_1\subseteq M$ be two closed submanifolds, $M_1$ be the cover of $M$ as in (\ref{m1capitolo5}) and $\sigma\in \Omega^2(M)$ be a $M_1$-weakly-exact form. 
Inspired by the results of Chapter \ref{chapter3}, we try now to get a well-defined functional $S_k$ whose critical points are exactly the orbits of the flow of $(L,\sigma)$ with energy $k$ satisfying the boundary
conditions (\ref{lagrangianformulationtheta0theta1}). Denote by $\lambda$ a primitive of $\sigma_1$ on $M_1$ and observe that the Euler-Lagrange flow associated to the Lagrangian
\begin{equation}
L_{1,\lambda}:TM_1\longrightarrow \R\, , \ \ \ \ L_{1,\lambda}(q,v) \ = \ L_1(q,v) +\lambda_q(v)\, ,
\label{LagrangianM1}
\end{equation}
where $L_1$ is the lift of $L$ to $M_1$, coincides with the lift of the flow of the pair $(L,\sigma)$. Moreover, the exactness of $\sigma_1$ yields primitives of  $\sigma$ on $Q_0$ and $Q_1$. Indeed, denote
with $Q_0^1$ any lift of $Q_0$ to $M_1$; then $\lambda|_{Q_0^1}$ is a primitive of $\sigma_1|_{Q_0^1}$. Now, from the definition of $M_1$ it follows that there exist neighborhoods $\, \mathcal U, \mathcal U^1$ of $Q_0,Q_0^1$ in 
$M,M_1$ respectively such that the covering map $p:\mathcal U^1\rightarrow \mathcal U$ is a diffeomorphism; therefore the 1-form $\lambda|_{Q_0^1}$ descends to a primitive $\theta_0$ of $\sigma|_{Q_0}$. 
Clearly the same argument applies to $Q_1$, thus giving 1-forms $\theta_0,\theta_1$ such that 
\begin{equation}
\theta_j \ = \ (p^*)^{-1} \big (\lambda|_{Q_j^1}\big ) \, , \ \ \ \  d\theta_j \ = \sigma\big |_{Q_j} \ \ \ \ \text{on} \ \ Q_j\ \text{for} \ j=0,1\, .
\label{forme}
\end{equation}

In particular the boundary conditions can be expressed as in (\ref{lagrangianformulationtheta0theta1}) with respect to the ``canonical'' 1-forms $\theta_0,\theta_1$ on $Q_0,Q_1$ given by (\ref{forme}). 
Following the ideas in \cite{Mer10} and \cite{AB14} we fix now a connected component $\mathcal M_\nu$ of $\mathcal M_Q$ and pick a reference path $x_\nu$ in this component. For any element $(x,T)\in \mathcal M_\nu$ we define 
\begin{equation}
S_k(x,T) := \ T\int_0^1 \Big [L\Big (x(s),\frac{x'(s)}{T}\Big ) + k \Big ]\, ds \ + \int_X \sigma \ - \int_{x_0} \theta_0 \ + \int_{x_1} \theta_1\, ,
\label{functionalcbc}
\end{equation}
where $X:[0,1]\times [0,1]\rightarrow M$ is as in the picture below as homotopy from $x_\nu$ to $x$ with starting points in $Q_0$ and ending points in $Q_1$ and, for $j=0,1$, 
$$x_j:[0,1]\longrightarrow Q_j\, , \ \ \ \ x_j(s):=\ X(s,j)\, .$$

\begin{center}
\includegraphics[height=40mm]{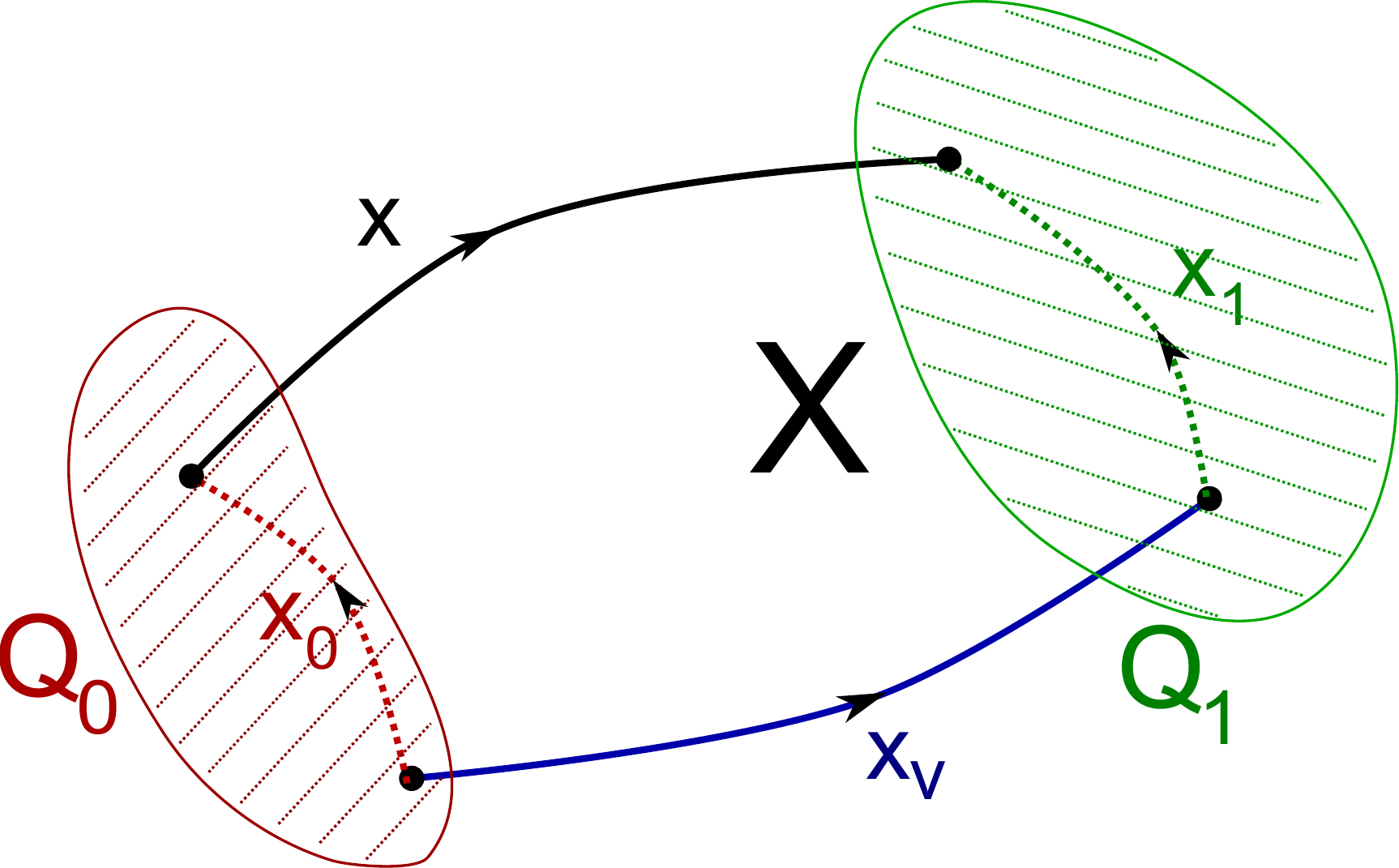}
\end{center}
 
Observe that the first integral on the right-hand side of (\ref{functionalcbc}) is well-defined, since $L$ is assumed to be electro-magnetic outside a compact set. 
At this point it is however not clear that the definition of $S_k$ does not depend on the homotopy $X$.  Actually, $S_k$ will not be well-defined without any additional assumption on the 2-form $\sigma$.

Thus, suppose that $X'$ is another homotopy from $x_\nu$ to $x$ with the same properties of $X$ and consider the cylinder $C:=X\cup \overline{X'}$, where $\overline{X'}$ denotes the homotopy induced by $X'$
connecting $x$ to $x_\nu$. Denote with $\partial C_j$  the part of $\partial C$ contained in $Q_j$. Clearly $\sigma|_C$ is exact; if $\theta$ is a primitive of $\sigma$ on $C$, then by Stokes' theorem 
\begin{eqnarray*} 
 \int_X \sigma \ - \int_{x_0} \theta_0 \ \!\! &+& \!\!\! \int_{x_1} \theta_1 \ -  \ \left (\int_{X'} \sigma  - \int_{x_0'} \theta_0 +  \int_{x_1'} \theta_1 \right) \ = \\
&=& \int_C \sigma \ - \int_{x_0\#{x_0'}^{-1} }\!\!\!\!\!\!\!\!\!\!\!\! \theta_0 \ + \int_{x_1\#x_1'^{-1}} \!\!\!\!\!\!\!\!\!\!\!\! \theta_1 \ = \\ 
&=&  \int_{\partial C}\!\!\! \theta \ - \int_{x_0\#{x_0'}^{-1} }\!\!\!\!\!\!\!\!\!\!\!\! \theta_0 \ + \int_{x_1\#x_1'^{-1}} \!\!\!\!\!\!\!\!\!\!\!\! \theta_1 \ = \\ 
&=& \int_{\partial C_0} \!\!\!\theta \ - \int_{\partial C_1} \!\!\! \theta \ - \int_{x_0\#{x_0'}^{-1} }\!\!\!\!\!\!\!\!\!\!\!\! \theta_0 \ + \int_{x_1\#x_1'^{-1}} \!\!\!\!\!\!\!\!\!\!\!\! \theta_1 \ = \\ 
&=& \int_{x_0\#{x_0'}^{-1} }\!\!\!\!\!\!\!\!\!\!\!\! \big (\theta|_{\partial C_0}- \theta_0\big ) \ +\int_{x_1\#x_1'^{-1}} \!\!\!\!\!\!\!\!\!\!\!\! \big (\theta_1 -  \theta|_{\partial C_1}\big )\ = \ 0\, ,
\end{eqnarray*}
provided that $(\theta-\theta_i)|_{\partial C_i}$ is exact, for $i=0,1$. This is the case, for instance, when 
$$H^1(Q_0,\R)\ =\ H^1(Q_1,\R)\ =\ 0\, .$$
Summarizing, we have proven the following 

\begin{lemma}
Let $L:TM\rightarrow \R$ be a Tonelli Lagrangian, $Q_0,Q_1\subseteq M$ be closed connected submanifolds and $\sigma\in \Omega^2(M)$ be a $M_1$-weakly-exact 2-form on $M$. Furthermore, let $\theta_0,\theta_1$ be as 
in (\ref{forme}) and suppose that the following property holds: for every cylinder $C\subseteq M$ with first part of the boundary $\partial C_0$ contained in $Q_0$ and second part $\partial C_1$ contained in $Q_1$, the pull-back of 
$\sigma$ to $C$ admits a primitive $\theta$ such that 
$$(\theta- \theta_i)|_{\partial C_i}  \ \ \ \text{exact}\, , \ \ \ \ i=0,1\, .$$

Then the functional (\ref{functionalcbc}) is well-defined on every connected component of $\mathcal M_Q$, namely it is independent on the choice of the homotopy $X$ connecting $x$ to $x_\nu$.
\label{buonadefinizionesk}
\end{lemma}

It is clear from the definition that the functional $S_k$ depends on the choice of the reference path; however, if we pick a different reference path instead of $x_\nu$, the functional $S_k$ changes only by the addition of a constant
and hence its geometric properties remain unchanged. Notice that, when $\sigma=d\eta$ is exact, one can choose 
$$\theta_0 \ = \ \eta|_{Q_0} \, , \ \ \ \ \theta_1\ = \ \eta|_{Q_1}\, , \ \ \ \ \theta \ = \ \eta|_C\, .$$
With these choices $S_k$ reduces to 
$$T \int_0^1 \Big [L\Big (x(s),\frac{x'(s)}{T}\Big ) + k \Big ]\, ds \ + \int_x \eta\ - \int_{x_\nu} \eta\, ,$$
which is (up to a constant) the free-time Lagrangian action functional for $L+\eta$. 

\vspace{3mm}

The functional $S_k$ is of class $C^{1,1}(\mathcal M_Q)$, being the sum of a $C^{1,1}$-functional (the free-time Lagrangian action functional associated to $L$) and of a smooth part. 
The interest on $S_k$ relies on the fact that its critical points are exactly the orbits of the flow of $(L,\sigma)$ with energy $k$ satisfyng the boundary conditions (\ref{lagrangianformulationtheta0theta1}), as we now prove. 

We first explicitly compute the differential of $S_k$, starting with the partial derivative $\partial S_k/\partial T$. By the Lebesgue dominated convergence theorem 
\begin{eqnarray*} 
&& \frac{S_k(x,T+\epsilon)-S_k(x,T)}{\epsilon}  \ = \\ 
					      &=&  \frac 1\epsilon \left ( (T+\epsilon) \int_0^1 \Big [ L\Big (x(s),\frac{x'(s) }{T+\epsilon}\Big ) + k \Big ]\, ds \ - \ T \int_0^1 \Big [L\Big(x(s),\frac{x'(s)}{T}\Big ) + k \Big ]\, ds \right )\\
                                                           &=& k \ + \int_0^1 L\Big (x(s),\frac{x'(s) }{T+\epsilon}\Big ) \, ds \ + \ \frac T\epsilon \int_0^1 \Big [L\Big (x(s),\frac{x'(s) }{T+\epsilon}\Big ) - L\Big(x(s),\frac{x'(s)}{T}\Big )\Big ]\, ds
\end{eqnarray*}
converges as $\epsilon\rightarrow 0$ to 
\begin{eqnarray}
\frac{\partial S_k}{\partial T}(x,T) &=&  k + \int_0^1 \Big [L\Big (x(s),\frac{x'(s) }{T}\Big )- d_vL \Big (x(s),\frac{x'(s)}{T}\Big ) \Big [\frac{x'(s)}{T}\Big ]\Big ] \, ds  = \nonumber \\
                                                        &=&  \int_0^1 \Big [ k - E\Big (x(s),\frac{x'(s) }{T}\Big )\Big ]\, ds\, .
\label{derivataskT}
\end{eqnarray}
Now let $x_r$ be a variation of $x$ and denote by 
$$\zeta (s) := \ \frac{\partial}{\partial r}\Big |_{r=0} x_r(s)\, .$$

For every $r$,  let $X_r:[0,1]\times [0,1] \rightarrow M$ be a homotopy connecting  the reference path $x_\nu$ to $x_r$. Observe that, for every $r$,
$$Y_r:[0,1]\times [0,1]\longrightarrow M \, , \ \ \ \ (\epsilon,s) \longmapsto x_{\epsilon r}(s)$$ 
is a homotopy from $x$ to $x_r$. We denote by 
$$x_j^r(s):= \ X_r(s,j) \, , \ \ \ \ y_j^r(s) := \ Y_r (s,j)\, , \ \ \ \ j=0,1$$
and compute
\begin{eqnarray*}
S_k(x,T)-S_k(x_r,T) &=& T\int_0^1 \!\! \Big [L\big (x,\frac{x'}{T}\Big ) - L\Big (x_r,\frac{x'_r}{T}\Big )\Big ] ds \ +\\
                                &+& \int_{X\cup \overline{X_r}} \!\!\!\!\!\! \sigma \ - \int_{x_0\#(x_0^r)^{-1}} \!\!\!\!\!\!\!\!\!\!\!\!\!\! \theta_0 \ + \int_{x_1\#(x_1^r)^{-1}}\!\!\!\!\!\!\!\!\!\!\!\!\!\! \theta_1 \ = \\
                                 &=& T\int_0^1 \!\! \Big [L\big (x,\frac{x'}{T}\Big ) - L\Big (x_r,\frac{x'_r}{T}\Big )\Big ] ds \ +\\
                                 &-& \int_{Y_r} \sigma \ + \int_{y_0^r} \theta_0 \ - \ \int_{y_1^r} \theta_1
\end{eqnarray*}
The second equality follows from the fact that $X\cup \overline{X_r}$ is a homotopy from $x_r$ to $x$, while $Y_r$ is a homotopy from $x$ to $x_r$, and we know by Lemma \ref{buonadefinizionesk} that the value of $S_k$ is independent of the homotopy.

\begin{center}
\includegraphics[height=45mm]{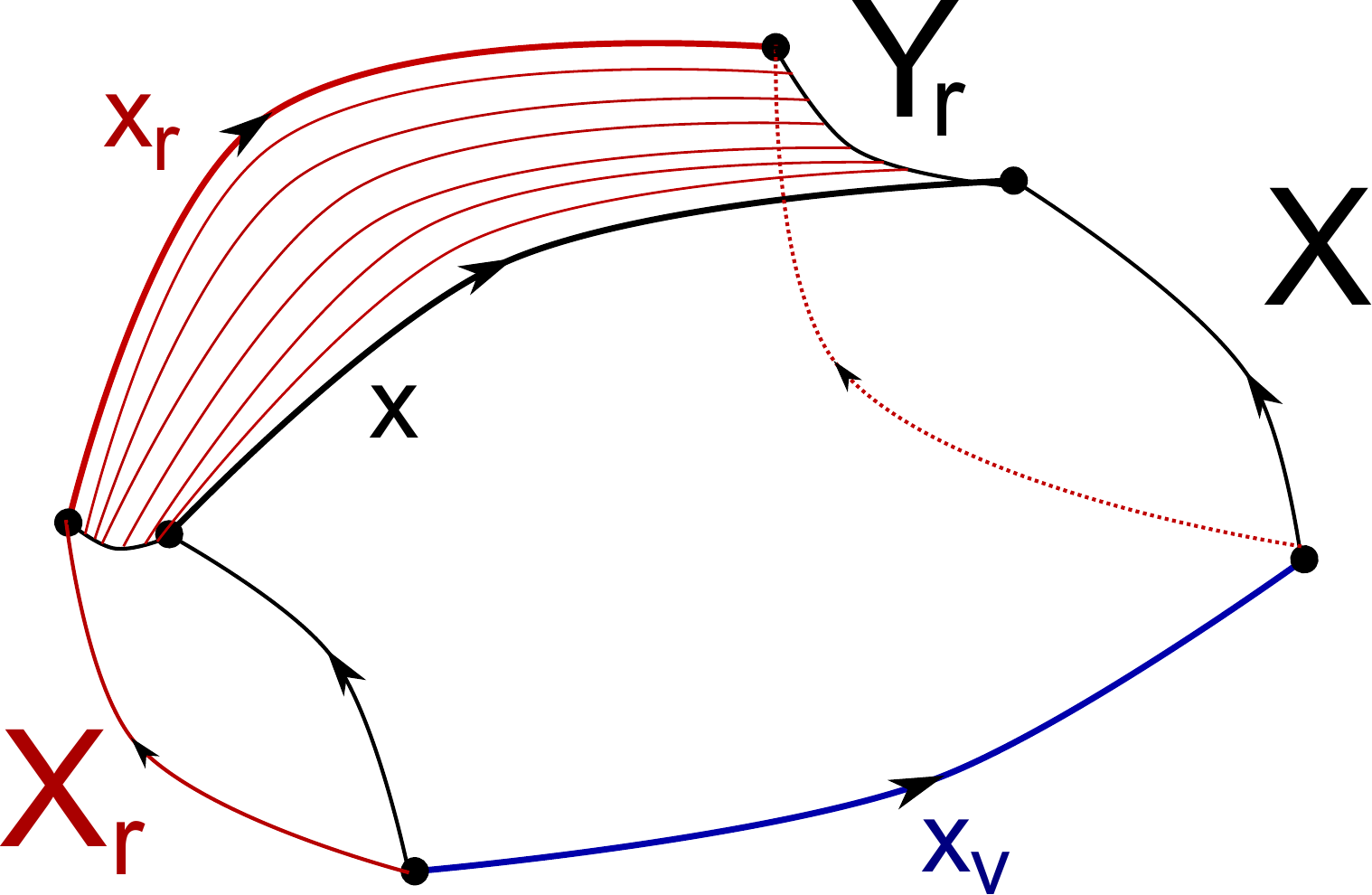}
\end{center}

\noindent Again, by the Lebesgue dominated convergence theorem we get\footnote{One easily checks that $\frac{1}{r}\int_{Y_r}\sigma \rightarrow \int_0^1 \sigma_{x(s)}(\zeta(s),x'(s))\, ds$ for $r\rightarrow 0$.}
\begin{eqnarray*} 
d_x S_k(x,T) \big [(\zeta,0)] &=& \lim_{r \rightarrow 0} \ \frac{S_k(x,T)-S_k(x_r,T)}{r} \ = \\
                                               &=&  T\int_0^1 \Big [d_qL\Big (x(s),\frac{x'(s)}{T} \Big ) \big [\zeta\big ] + d_vL\Big (x(s),\frac{x'(s)}{T} \Big ) \Big [\frac{\dot \zeta}{T}\Big ] \Big ] \, ds \ +\\
                                               &+& \int_0^1 \sigma_{x(s)} (x'(s),\zeta(s))\, ds \ + \ (\theta_0)_{x(0)}\big [\zeta(0)\big ] \ - \ (\theta_1)_{x(1)}\big [\zeta(1)\big ]\, .
\end{eqnarray*}
Therefore, if $(x,T)$ is a critical point for $S_k$, then integrating by parts we get that
\begin{eqnarray*}
0 &= &  T\int_0^1 \Big [d_qL\Big (x(s),\frac{x'(s)}{T} \Big ) - \frac{d}{ds} d_vL\Big (x(s),\frac{x'(s)}{T} \Big ) \Big ] \, \zeta \, ds \ +\\
   &+& \int_0^1 \sigma_{x(s)} (x'(s),\zeta(s))\, ds \ + \ d_vL\Big (x(0),\frac{x'(0)}{T} \Big )\, \zeta (0) + (\theta_0)_{x(0)}\big [\zeta(0)\big ] \ -\\
   &-& d_vL\Big (x(1),\frac{x'(1)}{T}\Big )\, \zeta(1) - (\theta_1)_{x(1)}\big [\zeta(1)\big ]
\end{eqnarray*}
must hold for every $\zeta \in T_x H^1_Q([0,1],M)$. Choosing $\zeta$ with compact support in $(0,1)$ we get that $\gamma(t):= x(t/T)$ is an orbit of the flow of $(L,\sigma)$. 
Choosing now any arbitrary $\zeta$ we get that $\gamma$ satisfies the boundary conditions (\ref{lagrangianformulationtheta0theta1}).
Finally, being $\gamma$ an orbit of the flow of $(L,\sigma)$,  it has constant energy and an easy inspection of the derivative of $S_k$ with respect to the variable $T$ shows that actually 
$E(\gamma,\dot \gamma) = k$. 

\begin{prop}
Suppose that the assumptions of Lemma \ref{buonadefinizionesk} are satisfied. Then the pair $(x,T)\in \mathcal M_Q$ is a critical point of the functional $S_k$ if and only if $\gamma(t)=x(t/T)$ is an orbit of the flow of $(L,\sigma)$ 
satisfying the boundary conditions (\ref{lagrangianformulationtheta0theta1}).
\label{correspondence5}
\end{prop}

We proceed now to study the relation between $S_k$ and the free-time Lagrangian action functional associated to the Lagrangian $L_{1,\lambda}$ in (\ref{LagrangianM1}) on $M_1$. 

Thus, consider a connected component $\mathcal M_\nu$ of $\mathcal M_Q$ and pick a lift $\tilde x_\nu$ of the reference path $x_\nu$. Moreover, let $(x,T)\in \mathcal M_\nu$ and let $X$ be a homotopy as in the definition of $S_k$ connecting 
the reference path $x_\nu$ to $x$. Denote by 
$$\widetilde X: [0,1]\times [0,1] \longrightarrow M_1$$ 
the unique homotopy obtained by lifting $X$ to $M_1$ and starting at $\tilde x_\nu$. Let now $\tilde x = \widetilde X(1,\cdot)$ be the lift of $x$ induced by the homotopy $\widetilde X$ and let 
$\tilde x_0$, $\tilde x_1$ the other two boundary components of $\widetilde X$. Then, by definition of $\theta_0$, $\theta_1$ we have
\begin{eqnarray*} 
\int_X \sigma \ - \int_{x_0} \theta_0 \ + \int_{x_1}\theta_1 &=& \int_{\widetilde X} \sigma_1 \ - \int_{\tilde x_0} \lambda  \ + \int_{\tilde x_1} \lambda \ = \\ 
                                                                                                 &=& \int_{\partial \widetilde X} \lambda \ - \int_{\tilde x_0} \lambda  \ + \int_{\tilde x_1} \lambda\ \ = \\
                                                                                                 &=&  \int_{\tilde x} \lambda \ - \int_{\tilde x_\nu} \lambda\, ,
\end{eqnarray*}
from which we deduce that 
$$\A_k (\tilde x,T) \ = \ S_k (x,T) \ + \int_{\tilde x_\nu} \lambda \, .$$

Observe that, in case $\mathcal M_\nu$ contains also constant paths (this may happen only if $Q_0$ and $Q_1$ intersect), we can choose $x_\nu$ to be a constant path; this choice yields
$$\A_k(\tilde x,T) \ = \ S_k(x,T)\, , \ \ \ \ \forall \ (x,T) \in \mathcal M_\nu\, .$$

\begin{lemma}
Suppose the assumptions of Lemma \ref{buonadefinizionesk} are satisfied and let $\mathcal M_\nu$ be a connected component of $\mathcal M_Q$. Then there exists a constant $a(\nu)\in \R$ depending only on $\nu$ and 
on the reference path $x_\nu$ such that 
\begin{equation}
\A_k(\tilde x,T) \ = \ S_k (x,T) \ + \ a(\nu)\, , \ \ \ \ \forall \ (x,T)\in \mathcal M_\nu\, .
\label{rapportogenerale}
\end{equation}
Furthermore, if $\mathcal M_\nu$ contains constant paths, we may assume $a(\nu)=0$.
\label{rapportoskak}
\end{lemma}

In contrast with the case of periodic orbits (cf. \cite{Mer10} and \cite{AB14}), the functional $S_k$ is, under the technical assumptions of Lemma \ref{buonadefinizionesk}, well-defined for every connected component of $\mathcal M_Q$, independently of
the fact that a suitable Ma\~{n}\'e critical value is finite or not. In fact, in the periodic case one needs the existence of a bounded primitive for the lift of $\sigma$ to the universal cover in order to show that the integral of $\sigma$ vanishes on any 2-torus. 
This property is crucial to get a well-defined functional for every free-homotopy class of loops. If this condition fails, one can in general define a functional only for contractible loops. 

However, the geometric properties of the functional $S_k$ change drastically when crossing a suitable critical energy value, which is in this case given by
\begin{equation}
c(L,\sigma;Q_0,Q_1) : = \ c(L_{1,\lambda}) \ = \inf_{u\in C^\infty (M_1)} \sup_{q\in M_1} \ H_1(q,d_qu-\lambda_q)\, ,
\label{clsigmah0h1}
\end{equation}
where $H_1$ is the lift of $H$, the Tonelli Hamiltonian given by the Fenchel dual of $L$, to the cover $M_1$. Being the lift of a Tonelli Hamiltonian, $H_1$ satisfies 
$$h_0\, \|p\|_q^2 - h_1\ \leq \ H_1(q,p) \ \leq \ h_0' \, \|p\|_q^2 + h_1'\, , \ \ \ \ \forall \ (q,p)\in T^*M_1$$
for suitable constants $h_0,h_0'>0$, $h_1,h_1'\in \R$; in particular
$$h_0\, \inf_{u\in C^\infty} \sup_{q} \|du-\lambda\|_q^2 \ - \ h_1 \ \leq \ c(L,\sigma;Q_0,Q_1)\ \leq \ h_0'\, \inf_{u\in C^\infty} \sup_{q} \|du-\lambda\|_q^2 \ + \ h_1'$$ 
and this shows that  $c(L,\sigma;Q_0,Q_1)$ is finite if and only if $\sigma_1$ admits a bounded primitive on $M_1$ of the form $\lambda - du$, for some suitable smooth function $u:M_1\rightarrow \R$.

In what follows we refer to the \textit{bounded}, respectively \textit{unbounded case} when the Ma\~n\'e critical value $c(L,\sigma;Q_0,Q_1)$ is finite, respectively infinite.
By the very definition of $c(L,\sigma;Q_0,Q_1)$ and by Lemmas \ref{unboundedness}, \ref{boundedness} and \ref{rapportoskak} we get the following 

\begin{lemma}
For every $k\geq c(L,\sigma;Q_0,Q_1)$ the functional $S_k$ is bounded from below on every connected component of $\mathcal M_Q$. The functional $S_k$ is instead 
unbounded from below on each connected component of $\mathcal M_Q$ for every $k< c(L,\sigma;Q_0,Q_1)$.
\label{boundednesssk}
\end{lemma}


\section{The Palais-Smale condition for $S_k$}
\label{thepalaissmalecondition}

In this and the next section we assume that $Q_0,Q_1,L,\sigma,\theta_0,\theta_1$ are such that the functional $S_k$ in (\ref{functionalcbc}) is well-defined and
show that $S_k$ satisfies the Palais-Smale condition on subsets of $\mathcal M_Q$ with times bounded and bounded away from zero. 
This result is the analogous of Lemma \ref{lemmalimitatezza} in this setting. Its strength relies on the fact that it does not depend on the finiteness of the Ma\~n\'e critical 
value $c(L,\sigma;Q_0,Q_1)$. 

The analogue for periodic orbits has been proven in \cite{Mer15} by extending to the unbounded case the proof in \cite{Mer10}.
In fact, the unbounded case can be reduced to the bounded one using the following 

\begin{lemma}
If $(x_h,T_h)\subseteq \mathcal M_Q$ is a Palais-Smale sequence for $S_k$ such that the $T_n$'s are bounded from above, then there exist a compact subset 
$K\subseteq M_1$ and, for every $h\in \N$, suitable lifts $x_h^1$ of $x_h$ to $M_1$ such that $x_h^1([0,1])\subseteq K$ for every $h\in \N$.
\label{reductiontobounded}
\end{lemma}

\begin{proof}
Clearly it suffices to show that the $x_h$'s have uniformly bounded length, since then one can simply choose for every $h\in \N$ a lift $x_h^1$ of $x_h$ in such a way that 
$x_h^1([0,1])\cap F\neq \emptyset$, where $F\subseteq M_1$ is any fundamental domain for $M$. Since $(x_h,T_h)$ is a Palais-Smale sequence we have in particular that 
\[\alpha_h := \ - \frac{\partial}{\partial T} S_k(x_h,T_h) \ =\ \frac{1}{T_h}\int_0^{T_h} \Big [ E\Big (\gamma_h(t),\dot \gamma_h (t) \Big ) - k\Big ]\, dt \ \longrightarrow \ 0 \, ,\]
where as usual $\gamma_h=(x_h,T_h)$. Since $L$ is a Tonelli Lagrangian, its energy function is also Tonelli and hence satisfies 
$$E(q,v) \ \geq \ a\, \|v\|_q^2 \ + \ b\, , \ \ \ \ \forall \ (q,v)\in TM\, ,$$
for suitable constants $a>0$, $b\in \R$. In particular, we have
\[\alpha_h \ \geq \ \frac{a}{T_h}\, \int_0^{T_h} \|\dot \gamma_h(t)\|^2\, dt \  - \ (b+k) \ \geq \ \frac{a}{T_h^2} \, l(x_h)^2 \ -\ (b+k)\]
by the Cauchy-Schwartz inequality and hence 
$$ l(x_h)^2 \ \leq \ \frac{T_h^2}{a} \, (\alpha_h + k+b)\, .$$
Since the $T_h$'s are bounded and the $\alpha_h$'s infinitesimal, the assertion follows.
\end{proof}

\vspace{5mm}

Lemma \ref{reductiontobounded} implies that we can treat the unbounded case exactly as the bounded one, since any primitive of $\sigma_1$ is clearly bounded on any compact subset of $M_1$.

Therefore, we might suppose that $\sigma_1$ has a bounded primitive $\lambda$ on $M_1$; it follows then that there exist constants $a>0$, $b\in \R$ such that 
\begin{equation}
L_{1,\lambda}(q,v) \ \geq \ a \, \|v\|^2_q + b\, , \ \ \ \ \forall \ (q,v) \in TM_1\, ,
\label{solitastimalagrangiana}
\end{equation}
where $L_{1,\lambda}$ is defined as in (\ref{LagrangianM1}). We show now that Palais-Smale sequences on a connected component $\mathcal M_\nu$ of $\mathcal M_Q$ not containing constant paths 
have automatically $T_h$'s bounded away from zero. 

\begin{lemma}
Let $\mathcal M_\nu$ be a connected component of $\mathcal M_Q$ not containing constant paths and let $(x_h,T_h) \subseteq \mathcal M_\nu$ be a Palais-Smale sequence for $S_k$ at level $c$. Then there exists $T_*>0$ such that 
$T_h\geq T_*$ for every $h\in \N$.
\label{lemmaboundedfromzero}
\end{lemma}

\begin{proof}
Without loss of generality we may suppose that 
$$c+1 \ \geq \ S_k(x_h,T_h) \, , \ \ \ \ \forall \ h\in \N\, .$$

This implies, using (\ref{solitastimalagrangiana}), Lemma \ref{rapportoskak} and the fact that the length of any path in $\mathcal M_\nu$ is bounded away from zero by a positive constant, say $\epsilon >0$, that 
\begin{eqnarray*} 
c+1 &\geq&  S_k(x_h,T_h) \ = \  \A_k(\tilde x_h,T_h) \ + \ a(\nu) \  \geq \\
        &\geq& \frac{a}{T_h} \int_0^1 \|\tilde x'_h(s)\|^2 \, ds \ + \ (k+b)\, T_h \ +\ a(\nu) \ \geq \\ 
        &\geq& \frac{a}{T_h}\, \epsilon^2 \ +\ (k+b)\, T_h \ +\ a(\nu)\, .
\end{eqnarray*} 
This clearly shows that the $T_h$'s are bounded away from zero. 
\end{proof}

\vspace{5mm}

The next lemma ensures that Palais-Smale sequences with $T_h\rightarrow 0$ may arise only at level $c=0$.
The proof is the same as  for Lemma \ref{yesmodification}.

\begin{lemma}
Let $\mathcal M_\nu$ be a component of $\mathcal M_Q$ that contains constant paths and let $(x_h,T_h)$ be a Palais-Smale sequence for $S_k$ such that $T_h\rightarrow 0$, then $S_k(x_h,T_h)\rightarrow 0$.
\label{palaismalealivello0}
\end{lemma}

We end this section showing that $S_k$ satisfies the Palais-Smale condition on subsets of $\mathcal M_Q$ with times bounded and bounded away from zero. 
The proof goes as in Lemma \ref{lemmalimitatezza} and is inspired by the analogous result in the periodic case (cf. \cite{Mer10}). 

\begin{prop}
Let $(x_h,T_h)\subseteq \mathcal M_\nu$ be a Palais-Smale sequence for $S_k$ such that $0<T_*\leq T_h\leq T^*<+\infty$ for every $h\in \N$. Then, passing to a subsequence if necessary, the sequence $(x_h,T_h)$ is convergent in the $H^1$-topology.
\label{completezzask}
\end{prop}

\begin{proof}
Without loss of generality we have 
$$c+1 \ \geq \ S_k(x_h,T_h) \, , \ \ \ \ \forall \ h\in \N\, .$$
Using (\ref{rapportogenerale}) and (\ref{solitastimalagrangiana}) we therefore obtain 
\begin{eqnarray*}
c+1     &\geq& S_k(x_h,T_h) \ = \ \A_k(\tilde x_h,T_h) \ + \ a(\nu) \ \geq \\
	&\geq& \frac{a}{T_h} \int_0^1 \|x_h'(s)\|^2 \, ds \ + \ (k+b)\, T_h \ + \ a(\nu)\ \geq \\
	&\geq & \frac{a}{T^*} \, \|x'_h\|_2^2 \ - \ T^*\, |k+b| \ + \ a(\nu)
\end{eqnarray*}
where $\|\cdot \|_2$ denotes the $L^2$-norm with respect to the metric on $M$. Therefore, $\|x'_h\|_2$ is uniformly bounded and hence $\{x_h\}$ is $1/2$-equi-H\"older-continuous
$$\text{dist}\, \big (x_h(s),x_h(s')\big ) \ \leq \ \int_s^{s'} |x'_h(r)|\, dr \ \leq \ |s-s'|^{1/2} \|x'_h\|_2\, .$$

By the Ascoli-Arzel\`a theorem, up to taking a subsequence, $x_h$ converges uniformly to some $x\in C([0,1],M)$. Now one proves that actually $x_h$ converges to $x$ in $H^1$ just repeating 
the proof of Lemma \ref{lemmalimitatezza} replacing $\A_k$ with $S_k$.
\end{proof}

\vspace{5mm}

Combining Proposition \ref{completezzask} with Lemma \ref{boundednesssk} we get the following corollary, which is the analogue of Corollary \ref{compattezza} in this setting

\begin{cor}
If $k>c(L,\sigma;\langle H_0,H_1\rangle )$, then any Palais-Smale sequence $(x_h,T_h)$ for $S_k$  in a given connected component of $\mathcal M_Q$ with $\, T_h\geq T_*>0$ has a converging subsequence. 
As a corollary, for every $k>c(L,\sigma;\langle H_0,H_1\rangle )$, $S_k$ satisfies the Palais-Smale condition on every connected component of $\mathcal M_Q$ not containing constant paths.
\label{compattezzask}
\end{cor}

\vspace{-8mm}


\section{Existence results in the $M_1$-weakly-exact case}
\label{existenceresults}

In this section we use the tools introduced in the previous paragraphs to extend the main results of Chapter \ref{chapter3} to the $M_1$-weakly-exact setting, under the assumption that the functional 
$S_k$ as in \eqref{functionalcbc} is well-defined. Thus, let $L:TM\rightarrow \R$ be a Tonelli Lagrangian, $Q_0,Q_1\subseteq M$ closed connected submanifolds, $M_1$ a cover of $M$ defined by \eqref{m1capitolo5}. 
Suppose moreover that $\sigma$ is an $M_1$-weakly-exact 2-form and let $\theta_0$, $\theta_1$ be primitives of $\sigma|_{Q_0},\sigma|_{Q_1}$ respectively as in (\ref{forme}). Lemma \ref{buonadefinizionesk} shows that, 
if for every cylinder $C\subseteq M$ with first part of the boundary $\partial C_0$ contained in $Q_0$ and second part $\partial C_1$ contained in $Q_1$, the pull-back of $\sigma$ to $C$ admits a primitive $\theta$ with
$$(\theta-\theta_i)|_{\partial C_i} -\theta_i \ \ \ \text{exact} \, , \ \ \ \ i=0,1\, ,$$
then the functional $S_k$ as in (\ref{functionalcbc}) is indeed well-defined. Moreover, by Proposition \ref{correspondence5} its critical points correspond to the orbits of the flow of $(L,\sigma)$ with energy $k$ satisfying 
the conormal boundary conditions \eqref{lagrangianformulationtheta0theta1}. 

We start considering supercritical energy levels and prove results analogous to Theorem \ref{teorema1} and \ref{teorema2} in this context. We just sketch the proofs, since they can be obtained from the corresponding 
ones in Chapter \ref{chapter3} with minor adjustments.

\begin{teo}
Suppose $Q_0\cap Q_1 =\emptyset$ and let $k>c(L,\sigma;Q_0,Q_1)$. Then, every connected component of $\mathcal M_Q$ carries an orbit of the flow of $(L,\sigma)$ with energy $k$ satisfying the 
boundary conditions (\ref{lagrangianformulationtheta0theta1}), which is furthermore a global minimizers of $S_k$ among its connected component.
\label{teorema1weaklyexact}
\end{teo}

\begin{proof}
The proof is the same as for Theorem \ref{teorema1}, just replacing $\A_k$ with $S_k$. Pick any connected component $\mathcal N$ of $\mathcal M_Q$. An argument as in Lemma \ref{lemmaboundedfromzero} 
shows that the sublevels of $S_k$ on $\mathcal N$ are complete. Moreover, Corollary \ref{compattezzask} implies that $S_k$ satisfies the Palais-Smale condition on $\mathcal N$. It follows that $S_k$ has a global 
minimizer on $\mathcal N$, which gives us the required orbit of the flow of $(L,\sigma)$.
\end{proof}

\vspace{3mm}

When $Q_0\cap Q_1\neq \emptyset$ we have connected components of $\mathcal M_Q$ containing constant paths and, for such components, a conclusion as in the theorem above might not hold. 
In fact, we have seen in Chapter \ref{chapter3} examples of intersecting submanifolds for which there were no Euler-Lagrange orbits satisfying the conormal boundary conditions for any supercritical energy value. 
However, an analogue of Theorem \ref{teorema3} holds also in this case; we state it directly in the general situation, thus without distinguishing between connected and not connected intersection. 
Before stating the theorem we define the energy value $k_{\mathcal N}(L,\sigma)$, which is the natural replacement  in this setting of the value $k_{\mathcal N}(L)$ defined in the third chapter. We thus set for every 
connected component $\mathcal N$ of $\mathcal M_Q$ containing constant paths
$$k_{\mathcal N}(L,\sigma) := \ \inf \ \Big \{k\in \R \ \Big |\ \inf_{\mathcal N} \ S_k \ \geq \ 0 \Big \} \ \geq \ c(L,\sigma;Q_0,Q_1)\, .$$ 

The proof of Theorem \ref{teorema2} and \ref{teorema3} goes through just replacing $\A_k$ with $S_k$, thus giving us corresponding generalizations of those results to this setting.

\begin{teo}
Suppose $Q_0\cap Q_1\neq \emptyset$. Then the following hold:

\begin{enumerate}
\item For every $k>c(L,Q_0,Q_1)$ and for every connected component of $\mathcal M_Q$ that does not contain constant paths there exists an orbit of the flow of $(L,\sigma)$ 
with energy $k$ satisfying the boundary conditions (\ref{lagrangianformulationtheta0theta1}) which is a global minimizer of $S_k$ among its connected component.

\item For every connected component $\mathcal N$ of $\mathcal M_Q$ containing constant paths and 
for every $k\in (c(L,\sigma;Q_0,Q_1), k_{\mathcal N}(L,\sigma))$ there exists an orbit of the flow of $(L,\sigma)$ with energy $k$ satisfying the boundary conditions 
(\ref{lagrangianformulationtheta0theta1}) which is a global minimizer of $S_k$ among $\mathcal N$. Moreover, if $\pi_l(\mathcal N)\neq 0$ for some $l\geq 1$, then there is such an orbit 
for every $k> k_{\mathcal N}(L,\sigma)$; if $k_{\mathcal N}(L,\sigma)>c(L,\sigma;Q_0,Q_1)$, then this holds also at level $k_{\mathcal N}(L,\sigma)$, provided that we allow the solution to be degenerate. 
\end{enumerate}
\label{teorema2weaklyexact}
\end{teo}

We now move to study the existence of orbits for subcritical energies. In virtue of Lemma \ref{reductiontobounded}  
we can treat the bounded and unbounded case at once. We enounce the main result directly in the general case, again without distinguishing between connected and not connected intersection. 

Suppose $Q_0\cap Q_1\neq \emptyset$. Notice that the primitives $\theta_0$ and $\theta_1$ of $\sigma$ on $Q_0$, $Q_1$ given by (\ref{forme}) coincide on $Q_0\cap Q_1$; this allows us to define a primitive $\theta$ of $\sigma$ on 
a small neighborhood of the intersection. Now, let $\mathcal N$ be a connected component of $\mathcal M_Q$ containing constant paths. Denote by $\Omega$ the set of constant paths contained in $\mathcal N$; 
observe that $\Omega$ need not be connected. Therefore, let $\Omega_1,\Omega_2,...$ be the connected components of $\Omega$ and for every $j=1,2, ...$ define 
$$k_{\Omega_j}:= \ \max _{q\in \Omega_j} \ E(q,0) \ + \ \max_{q\in \Omega_j} \ \frac{\|\theta_q\|^2}{4a}\,, \ \ \ \ k_{\Omega_j}^- := \ \min_{q\in \Omega_j} \ E(q,0) \ + \ \max_{q\in \Omega_j} \ \frac{\|\theta_q\|^2}{4a} \, .$$
For every $k\in (k_{\Omega_j}^-,k_{\Omega_j})$, respectively $k \in (k_{\Omega_j},c(L,\sigma;Q_0,Q_1)$, we define the minimax classes  $\Gamma_{\Omega_j}, \Gamma_j$  as
$$\Gamma_j := \Big \{ u :[0,1]\rightarrow \mathcal N \ \Big |\ u(0)=(p,T)\in \Omega_j\times (0,+\infty) \, ,\ S_k(u(1))<0 \Big \}  $$
and 
$$\Gamma_{\Omega_j} := \Big \{u:[0,1]\times \Omega_j\rightarrow \mathcal N \ \Big | \ u(0,q) = (q,T)\, , \ S_k(u(1,q))< 0 \,, \ \forall q\in \Omega_j \Big \}\, .$$
We can now define the minimax functions as in (\ref{minimaxfunction1}) and (\ref{minimaxfunction2}) just replacing $\A_k$ with $S_k$.
Exactly as done in Chapter \ref{chapter3}, one can show that the minimax functions are monotonically increasing in $k$ and strictly positive. Lemma \ref{struwe} goes through word by word replacing $\A_k$ with $S_k$ and allows to prove the following

\begin{teo}
Let $\mathcal N$ be a connected component of $\mathcal M_Q$ which contains constant paths. Then, for almost every
$$k\in \left (\min_j\, k_{\Omega_j}^-\, , \, c(L,\sigma;Q_0,Q_1 )\right )$$ 
there is an orbit of the flow of $(L,\sigma)$ with energy $k$ satisfying the conditions (\ref{lagrangianformulationtheta0theta1}).
\label{teobasseenergie2weaklyexact}
\end{teo}


\section{When $\sigma$ is only closed}
\label{whensigmaisonlyclosed}

When $\sigma$ is only closed and does not satisfy the technical assumptions at the beginning of Section \ref{existenceresults}, a functional $S_k$ whose critical points are exactly the orbits of the flow defined by the pair $(L,\sigma)$ and satisfying 
the boundary conditions (\ref{lagrangianformulationtheta0theta1}) is in general not available. However, its differential $\eta_k$ turns to be well-defined just assuming that the restriction of $\sigma$ to $Q_0$ and $Q_1$ is exact 
(observe that this assumption is the only one needed to give a sense to the boundary conditions); moreover, its zeros still correspond to the orbits of the flow defined by $(L,\sigma)$ with energy $k$ satisfying the boundary conditions 
\eqref{lagrangianformulationtheta0theta1}. We call $\eta_k\in \Omega^1(\mathcal M_Q)$ the \textit{action 1-form}. 

Though the action 1-forms considered here and in the next chapter are slightly different, 
they share the same properties. We study these properties rigorously in the next chapter (cf. Sections \ref{theaction1form} and \ref{ageneralizedpseudogradient}).

\vspace{3mm}

Throughout this section we assume that $Q_0\cap Q_1\neq \emptyset$ and that $\sigma$ is a closed 2-form with exact restriction to $Q_0$ and to $Q_1$. We fix two primitives $\theta_0$, $\theta_1$ of $\sigma|_{Q_0}$, $\sigma|_{Q_1}$
respectively and assume that they are both obtained by extending a fixed primitive $\theta$ of $\sigma$ on $Q_0\cap Q_1$ (notice that this is not always possible). We finally assume that there exists $l\geq 1$ such that $\pi_l(\mathcal M_Q)\neq 0$. 

In this setting, we look for orbits of the flow defined by $(L,\sigma)$ with energy $k$ and satisfying the boundary conditions \eqref{lagrangianformulationtheta0theta1} with respect to $\theta_0$, $\theta_1$. 
More precisely, we show that there exists an energy level above which existence of orbits for the flow of $(L,\sigma)$ satisfying the boundary conditions (\ref{lagrangianformulationtheta0theta1})
is guaranteed for almost every energy. This result generalizes the second part of Theorem \ref{teobasseenergie2weaklyexact} to this setting.

To do this we will use a method analogous to the one exploited in Chapter \ref{chapter5} to study the existence of periodic orbits for the flow of $(L,\sigma)$. Namely, we look for zeros of the action 1-form 
using variational methods which are actually similar to the ones already used in the previous chapters. The main difficulty in this setting relies precisely on the lack of the action; this will force us to replace Palais-Smale sequences by
the so-called ``critical sequences'' for $\eta_k$ (cf. Section \ref{theaction1form}) and to prove a compactness criterion for critical sequences with times bounded and bounded away from zero which generalizes Proposition \ref{completezzask}.

\vspace{3mm}

\noindent The computations in Section \ref{thefunctionalsk} show that, under the assumptions above,
\begin{equation}
\eta_k(x,T) := \ d\A_k^L(x,T) \ + \int_0^1 \sigma_{x(s)}(x'(s),\cdot ) \, ds \ + \ (\theta_0)_{x(0)}[\cdot] \ - \ (\theta_1)_{x(1)}[\cdot]
\label{etaktheta0theta1}
\end{equation}
defines a 1-form on $\mathcal M_Q$, called the \textit{action 1-form}. Here $\A_k^L:\mathcal M_Q\rightarrow \R$ denotes the free-time Lagrangian action functional associated to the Tonelli Lagrangian $L$. We sum up the properties of the 
action 1-form in the following lemma; the proof is analogous to the one for the periodic case contained in \cite{AB14} (see also Section \ref{theaction1form}).

\begin{lemma}
The action 1-form $\eta_k\in \Omega^1(\mathcal M_Q)$ is locally exact and locally Lipschitz. Moreover $\gamma=(x,T)\in \mathcal M_Q$ is a  zero of $\eta_k$ if and only if $\gamma$ is an orbit of the flow defined by $(L,\sigma)$ with energy 
$k$ satisfying the boundary conditions (\ref{lagrangianformulationtheta0theta1}).
\label{propertiesofetaktheta0theta1}
\end{lemma}

Since $\eta_k$ is locally the differential of a $C^1$-functional, its integral over closed loops in $\mathcal M_Q$ depends only on the homotopy class of the loop; in this sense we say that $\eta_k$ is closed, even though in general it is not differentiable.

In order to show the existence of zeros of $\eta_k$, we look for limiting points of critical sequences for $\eta_k$, that is sequences $(x_h,T_h)$ such that 
$$\|\eta_k(x_h,T_h)\|\longrightarrow 0\, .$$

The fact that $\eta_k$ is continuous implies that the set of zeros 
coincides with the set of limiting points of critical sequences. The following proposition provides a compactness criterion for critical sequences of $\eta_k$; the proof is the same as for Proposition \ref{prp:ps2}. Here $e(x)$ 
denotes the kinetic energy of the Sobolev-path $x:[0,1]\rightarrow M$.

\begin{prop}
Suppose $(x_h,T_h)$ is a critical sequence for $\eta_k$ in a connected component of $\mathcal M_Q$ with uniformly bounded times, then:
\begin{enumerate}
\item If $T_h\geq T_*$ for every $h\in \N$, then $(x_h,T_h)$ admits a converging subsequence.
\item If $T_h\rightarrow 0$, then $e(x_h)\rightarrow 0$.
\end{enumerate} 
\label{prp:pstheta0theta1}
\end{prop}

\vspace{-2mm}

The goal will be therefore to prove the existence of critical sequences for $\eta_k$ with times bounded and bounded away from zero; this will be done using a minimax argument analogous to the one exploited in Chapter \ref{chapter5}. 
To bypass the lack of a global action we use the varation of $\eta_k$ along a path $u:[0,1]\rightarrow \mathcal M_Q$ 
$$\Delta S_k(u):[0,1]\rightarrow \R\, ,\ \ \ \ \Delta S_k(u)(s) := \ \int_0^s u^*\eta_k\, .$$

We refer to Section \ref{ageneralizedpseudogradient} for the properties of $\Delta S_k(u)$. Suppose for the moment that $Q_0\cap Q_1$ is connected and 
denote with $\mathcal N$ the connected component of $\mathcal M_Q$ containing the constant paths. 
Consider now a sufficiently small neighborhood $\mathcal U$ of $Q_0\cap Q_1$ and extend the primitive $\theta$ of $\sigma$ on $Q_0\cap Q_1$ to be a primitive on the whole 
$\mathcal U$. By assumption, if $\epsilon>0$ is sufficiently small, then every path joining $Q_0$ to $Q_1$ with kinetic energy less than $\epsilon$ is entirely contained in $\mathcal U$.

\begin{center}
\includegraphics[height=35mm]{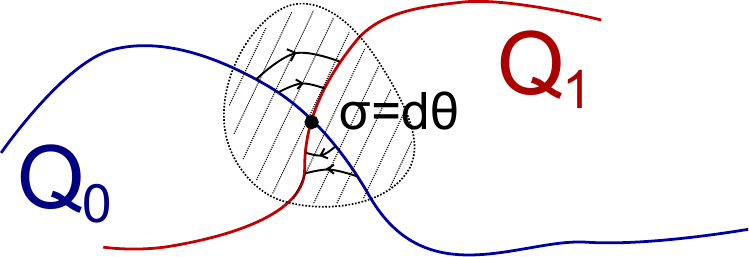}
\end{center}

\noindent It follows that $\eta_k$ is exact on $\mathcal N\cap \{e(x)\leq \epsilon\}$ with primitive $S_k$ given by
$$S_k:\mathcal N\cap \{e(x)\leq \epsilon\} \rightarrow \R\, , \ \ \ \ S_k (x,T) := \ \A_k^L(x,T) \ + \ \int_x \theta\ = \ \A_k^{L_\theta}(x,T)\, ,$$
where $L_\theta(q,v) := L(q,v)+\theta_q(v)$. As in the previous section we set 
$$k_{Q_0\cap Q_1} := \ \max_{q\in Q_0\cap Q_1} \ E(q,0) \ + \ \max_{q\in Q_0\cap Q_1} \ \frac{\|d_vL_\theta(q,0)\|^2}{4a}\, ,$$
where $a>0$ is such that (\ref{secondinequality}) is satisfied for the Lagrangian $L_\theta$. Observe that for any constant loop $(x_0,T)\in \mathcal N$ there holds 
$$S_k(x_0,T) \ = \ \Big [k-E(x_0,0)\Big ] \, T \ > \ 0$$ 
which is positive for every $k>k_{Q_0\cap Q_1}$ and tends to zero for $T\rightarrow 0$. Repeating exactly the proof of Lemma \ref{lemmaminimax1}, one gets that, for every $k>k_{Q_0\cap Q_1}$, $S_k$ is non-negative on $\mathcal N\cap \{e(x)\leq \epsilon\}$ 
and that there exists a positive constant $\alpha_k$ such that 
$$\inf_{e(x)=\epsilon} \ S_k(x,T) \ \geq \ \alpha_k\, .$$

Now fix $k^*>k_{Q_0\cap Q_1}$ and consider the smallest $l\geq 1$ such that $\pi_l(\mathcal N)\neq 0$. Pick a non-zero element $\mathfrak U\in \pi_l(\mathcal N,(x_0,T_0))$, where $(x_0,T_0)$ is a constant path such that
$S_{k^*}(x_0,T_0)<\alpha_{k*}/4$; observe that there exists an open interval $I=I(k^*)$ containing $k^*$ such that $S_k(x_0,T_0)<\alpha_{k^*}/4$ for every $k\in I$. 

For every $l\geq 2$ we interpret $S^l$ as $B^l$ with boundary $S^{l-1}$ identified to a point. With any point $\xi\in B^l$ we associate a path $a_\xi:[0,1]\rightarrow B^l$ such that $a_\xi(0)\in S^{l-1}$ and 
$a_\xi(1)=\xi$ and consider the composition $u_\xi := u\circ a_\xi:[0,1]\rightarrow \mathcal N$. By construction we have $S_k(u_\xi(0))=S_k(x_0,T_0)$ and we can define the minimax value
\begin{equation}
c^{\mathfrak u}:I\longrightarrow \R\, , \ \ \ \ c^{\mathfrak u} (k) := \ \inf_{u\in \mathfrak U} \ \max_{\xi\in B^l} \Big [S_k(x_0,T_0) \ + \Delta S_k(u_\xi)(1)\Big ]\, .
\label{minimaxfunctiontheta0theta1}
\end{equation}

Notice that this definition does not depend on $u_\xi$, since $S^{l-1}$ is path-connected and $\eta_k$ is closed. If $l=1$ then $S^{l-1}$ is not connected; in this case we set 
$a_\xi$ to be the path connecting $-1\in S^{l-1}$ with $\xi$ and define $c^{\mathfrak u}(\cdot)$ exactly as above. 

From the definition it follows that $c^{\mathfrak u}(\cdot)$ is monotonically increasing in $k$, thus almost everywhere differentiable; a proof of this is given in Lemma \ref{lem:mon}.
Moreover, $\mathfrak U$ is invariant under the normalized semi-flow of $-\sharp \eta_k$ truncated on $\mathcal N\cap \{e(x)\leq \epsilon\}$ below $\alpha_{k^*}/2$ (see Section \ref{ageneralizedpseudogradient} for the details) 
and, since $\mathfrak U$ is non-zero, for every $u\in \mathfrak U$ there exists an element $\xi \in B^l$ such that $e(u(\xi))=\epsilon$ (see \cite[Theorem 2.1.8]{Kli78}). 
In particular, $c^{\mathfrak u} (k) \geq \alpha_k$ for every $k\in I$. Hence, up to considering a smaller interval,
\begin{equation}
c^{\mathfrak u} (k) \ > \frac{\alpha_{k^*}}{2}\, , \ \ \ \ \forall \ k \in I\, .
\label{stimacualphak*}
\end{equation}

\begin{prop}
Let $k\in I$ be a point of differentiability for $c^{\mathfrak u}(\cdot)$. Then there exists a critical sequence for $\eta_k$ with times bounded and bounded away from zero.
\label{Struwetheta0theta1}
\end{prop}

The proof is exactly the same as for Proposition \ref{Struwe}. In order to exclude that the times tend to zero one shows using (\ref{stimacualphak*}) that the critical sequence can be chosen to lie in the complement of the sublevel $\{S_k\leq \alpha_{k_*}/4\}$,
namely proving that, if a point $\xi_*\in B^l$ almost realizes the maximum of 
$$\xi \ \longmapsto \ S_k(x_0,T_0) \ + \ \Delta S_k(u_\xi)(1)\, ,$$
then $u(\xi_*)\notin \{S_k\leq \alpha_{k_*}/4\}$ (cf. Lemma \ref{lem:notinwgeneral}). To ensure that the times are bounded one uses the well-known fact that subsets of $\mathcal N$ 
with bounded times are mapped by the time 1-flow into subsets with bounded times (cf. [Mer10, Lemma 5.7] or Lemma \ref{lem:ac-distper}). Propositions  \ref{prp:pstheta0theta1} and \ref{Struwetheta0theta1} imply immmediately the following

\begin{teo}
Let $(M,g)$ be a closed connected Riemannian manifold, $Q_0,Q_1\subseteq M$ be closed connected submanifolds with $Q_0\cap Q_1 \neq \emptyset$ connected, $L:TM\rightarrow \R$ be a Tonelli Lagrangian on $M$, $\sigma\in \Omega^2(M)$ be a closed 2-form with 
exact restriction to $Q_0$ and $Q_1$ and $\theta_0$,$\theta_1$ primitives of $\sigma|_{Q_0}$, $\sigma|_{Q_1}$ respectively. Assume furthermore that $\theta_0$ and $\theta_1$ coincide on $Q_0\cap Q_1$. Then, for almost every $k>k_{Q_0\cap Q_1}$ there exists an orbit for the flow defined by $(L,\sigma)$
with energy $k$ satisfying the boundary conditions (\ref{lagrangianformulationtheta0theta1}).
\label{teotheta0theta1closed}
\end{teo}

The theorem above clearly generalizes to the case $Q_0\cap Q_1$ non-connected. Consider a connected component $\mathcal N$ of $\mathcal M_Q$ containing constant paths and denote with $\Omega$ the set of constant paths contained 
in $\mathcal N$. Since $\Omega$ might not be connected we consider its connected components $\Omega_1,\Omega_2,...$ and for every $j=1,2,...$ we define $k_{\Omega_j}$, $\mathfrak U_j$, $c^{\mathfrak u}_j(\cdot)$ exactly as above 
just considering $x_0\in \Omega_j$ and replacing $Q_0\cap Q_1$ everywhere with $\Omega_j$. 

The argument above goes through word by word proving the following

\begin{teo}
Suppose the assumptions of Theorem \ref{teotheta0theta1closed} are satisfied (except the connected intersection one) and let $\mathcal N$ be a connected component of $\mathcal M_Q$ containing constant paths. Then, for almost every 
$$k\in \left (\min_j \ k_{\Omega_j}\, , \, +\infty \right )$$
there exists an orbit $\gamma\in \mathcal N$ for the flow of $(L,\sigma)$ with energy $k$ satisfying the boundary conditions (\ref{lagrangianformulationtheta0theta1}).
\label{teotheta0theta1closed2}
\end{teo}


\chapter{Periodic orbits for the flow of $(L,\sigma)$}
\label{chapter5}

In this chapter we focus on the existence of periodic orbits for the flow of the pair $(L,\sigma)$ (cf. the introduction of Chapter \ref{chapter4}), with $L:TM\rightarrow \R$ Tonelli Lagrangian and $\sigma$ closed 2-form on $M$. 
Recall that a manifold $M$ is \textit{aspherical} if all the homotopy groups $\pi_l(M)$, $l\geq 2$, vanish. Here we prove that, if $M$ is not aspherical, then for almost every $k$ larger than the maximum of the energy on 
the zero-section there exists a contractible periodic orbit with energy $k$. This result extends the celebrated Lusternik-Fet theorem \cite{FL51} about the existence of one closed contractible geodesic on every closed non-aspherical Riemannian 
manifold to this setting and it is the outcome of joint work with Gabriele Benedetti. 

The proof we give here is somehow different to the one contained in \cite{AB14}, since there the cases $l=2$ and $l>2$ were treaten separately. Here we use a slightly different construction that allows to treat both cases at once.

\vspace{3mm}

In Section \ref{theaction1form} we define the action-1 form $\eta_k$ rigorously, check its regularity property and prove a crucial compactness criterion for its critical sequences.

In Section \ref{ageneralizedpseudogradient} we discuss the properties of the negative gradient flow associated to $\eta_k$ and we exploit the interesting geometry of $\eta_k$ on the subset of contractible loops given 
by loops with small length.

In Section \ref{theminimaxclass} we use the assumption that $\pi_l(M)\neq 0$ to construct a suitable minimax class of functions $B^{l-1}\rightarrow \mathcal M_0$ mapping the boundary $S^{l-2}$ into the submanifold of constant loops and an associated minimax function.

Finally, in Section \ref{ageneralizedlyusternikfettheorem} we prove the main theorem of this chapter by generalizing the so-called \textit{Struwe monotonicity argument} to this setting.


\section{The action 1-form}
\label{theaction1form}

For a given Tonelli Lagrangian $L:TM\rightarrow\R$ we denote as usual with $H:T^*M\rightarrow\R$ the Hamiltonian given by the Fenchel dual of $L$ and with $E:TM\rightarrow\R$ the energy function associated to $L$. 
Since $L$ can be assumed, without loos of generality, to be quadratic at infinity, we also have that $E$ is quadratic at infinity; in particular
\begin{equation}
E_0\, \|v\|_q^2 - E_1 \ \leq \ E(q,v) \, , \quad \forall \ (q,v)\in TM\,,
\label{disuguaglianzeE}
\end{equation}
for suitable constants $E_0>0$, $E_1\in \R$. This estimate will be essential later on to prove the crucial Lemma \ref{lem:ps}. Let now $\sigma \in \Omega^2(M)$ be a closed 2-form. 

We aim to find closed orbits for the flow of $(L,\sigma)$ on a given energy level $E^{-1}(k)$, using variational methods on the product Hilbert manifold $\mathcal M=H^1(\T,M)\times (0,+\infty)$ of Sobolev loops with arbitrary period
of definition, even though in this generality a free-period Lagrangian action functional is not available. We denote as usual with $\mathcal M_0$ the connected component of $\mathcal M$ given by contractible loops. 

For any Sobolev loop $x\in H^1(\T,M)$  we define  
$$l(x) := \int_0^1 \|x'(s)\|\, ds \, , \ \ \ \ e(x) := \int_0^1 \|x'(s)\|^2\, ds$$
as the \textit{length}, respectively the \textit{kinetic energy} of $x$. The key ingredient of our discussion is that the periodic orbits of the flow defined by $(L,\sigma)$ on $E^{-1}(k)$ 
are in one to one correspondence with the zeros of a suitable 1-form $\eta_k\in \Omega^1(\mathcal M)$. As a first step, we consider the free-period Lagrangian action functional associated with $L$
$$\A^L_k:\mathcal M\rightarrow\R\, , \ \ \ \ \A^L_k(x,T):=\ T\,\int_0^1 \Big[L\Big (x(s),\frac{x'(s)}{T}\Big )+k\Big]\, ds\,.$$

We proceed now to define $\eta_k$. When $\sigma$ is exact with primitive $\theta$, then $\eta_k=d\A^{L_\theta}_k$ is simply the differential of the free-period Lagrangian action functional associated with the Tonelli Lagrangian $L_\theta(q,v):=L(q,v)+\theta_q(v)$. 
A computation shows that
$$d\A^{L_\theta}_k (x,T) \ = \ d\A^L_k(x,T) \ + \ \tau^\sigma_x\, ,$$
where 
\begin{equation}
\tau^\sigma_x[\xi]:=\int_0^1\sigma_{x(s)}(x'(s),\xi(s))\, ds\,
\label{tauk}
\end{equation}

In particular $\tau^\sigma$ does not depend on the particular choice of $\theta$ and it is well-defined even if $\sigma$ is not exact. Thus, in the general case we set
\begin{equation}\label{etak}
\eta_k(x,T)\ :=\ d\A^L_k(x,T) \ +\ \tau^\sigma_x \,.
\end{equation}
We easily compute for future applications
\begin{equation}\label{par-T}
\eta_k(x,T) \left[\frac{\partial}{\partial T}\right]\, =\, k-\int_0^1 E\Big(x(s),\frac{x'(s)}{T}\Big )\, ds\, =\,k-\frac{1}{T}\int_0^T E\big(\gamma(t),\dot \gamma(t)\big)\, dt\, .
\end{equation}
A proof of the following lemma can be found in \cite[Lemma 2.2]{AB14}.

\begin{lemma}\label{lem:lagloc}
Let $\psi$ be a bi-bounded time-dependent chart for $M$ and let $\Psi_{\mathcal M}$ be the associated local chart of $\mathcal M$. Then, there exists a smooth function 
$$L_{\psi,\sigma}:\T\times TB_\rho^n\times\R^+\longrightarrow\R$$
such that $\Psi_{\mathcal M}^*\eta_k=d\A^{L_{\psi,\sigma}}_k$, where $\A^{L_{\psi,\sigma}}_k:H^1(\T,B^n_\rho)\times\R^+\rightarrow\R$ is defined by
$$\A^{L_{\psi,\sigma}}_k(\xi,T):=\ T\,\int_0^1 \Big[L_{\psi,\sigma}\Big (s,\xi(s),\frac{\xi'(s)}{T},T\Big )+k\Big]\, ds\,.$$

Furthermore, for any $T_->0$ and for any $T\geq T_-$ the family of functions $L_{\psi,\sigma}(\cdot,\cdot,\cdot,T)$ is Tonelli and quadratic at infinity uniformly in $T$. 
\end{lemma}

Given the local representation of $\eta_k$ of Lemma \ref{lem:lagloc}, one just adapts the computations in \cite[Lemma 3.1]{AS09} to get the following

\begin{cor}\label{cor:reg}
The $1$-form $\eta_k$ is locally Lipschitz.
\end{cor}

Since $\eta_k$ is locally the differential of a $C^1$-functional, the integral of $\eta_k$ over a closed path $u:\T\rightarrow \mathcal M$ depends only on the free-homotopy class of $u$. In this sense, we say that $\eta_k$ is \textit{closed}, 
even if $\eta_k$ is in general not Fr\'ech\'et differentiable.

One can show that $\eta_k$ is \textit{exact} on $\mathcal M_0$ if and only if $\sigma$ is weakly-exact and that $\eta_k$ is exact on the whole $\mathcal M$ if the lift of $\sigma$ to the universal cover admits a bounded primitive, i.e.\
if $\sigma$ is a so called \textit{bounded weakly-exact 2 form} (see \cite{Mer10} or \cite{AB14}).

The following lemma states that the zeros of $\eta_k$ are in correspondence with the periodic orbits of the flow defined by $(L,\sigma)$ with energy $k$. The proof can be easily obtained from the one of Theorem \ref{criticalpoints}.

\begin{lemma}
Let $(M,g)$ be a Riemannian manifold, $\sigma$ be a closed $2$-form on $M$ and $L:T^*M\rightarrow\R$ be a Tonelli Lagrangian. Then, $\gamma=(x,T)$ is a zero of $\eta_k$ if and only if 
$\gamma$ is a periodic orbit of the flow of $(L,\sigma)$ with energy $k$.
\label{zeridietak}
\end{lemma}

In view of this result, the aim of the following sections will be to show that the set of zeros of $\eta_k$ is non-empty. The mechanism we will use to construct zeros of $\eta_k$ is to look at the limit points of \textit{critical sequences}, i.e. sequences 
$(x_h,T_h)\subset\mathcal M$ with
$$\big \|\eta_k(x_h,T_h)\big \| \ \longrightarrow \ 0\, ,$$
which are the generalization of Palais-Smale sequences to this setting. Since $\eta_k$ is a continuous $1$-form, we see that the set of limit points of critical sequences coincides with the set of zeros of $\eta_k$. Therefore, the first step is to know 
under which hypotheses a critical sequence has a limit point. Clearly, if $T_h\rightarrow 0$ or $T_h\rightarrow \infty$, the limit point set is empty. The following proposition shows that the converse is also true.

\begin{prop}\label{prp:ps2}
A critical sequence $(x_h,T_h)$ for $\eta_k$ such that $0<T_*\leq T_h\leq T^*<\infty$ has a converging subsequence.
\label{compattezzagen}
\end{prop}

\noindent For the proof we will need the following 

\begin{lemma}\label{lem:ps}
If $(x_h,T_h)$ is a critical sequence, then there exists $C>0$ such that
\begin{equation}\label{en-per}
e(x_h)\ \leq\ CT_h^2.  
\end{equation}
In particular $e(x_h)\rightarrow 0$ if $T_h$ goes to zero.
\end{lemma}

\begin{proof}
Since $(x_h,T_h)$ is a critical sequence for $\eta_k$, using (\ref{disuguaglianzeE}) and (\ref{par-T}) we obtain
\begin{eqnarray*}
o(1) &=& - \eta_k(x_h,T_h)\left [\frac{\partial}{\partial T}\right ] \ =  \int_0^1 E\Big (x_h(s),\frac{x_h'(s)}{T_h}\Big )\, ds  \ - \ k \ \geq \\ 
        &\geq& \int_0^1 \Big (E_0\, \frac{\|x_h'(s)\|^2}{T_h^2} -E_1\Big )\, ds \ - \ k \ = \ \frac{E_0}{T_h^2} \, e(x_h) \ - \ \big (k+E_1\big )\, .
\end{eqnarray*}
This  clearly implies 
$$e(x_h) \ \leq \ \frac{T_h^2}{E_0} \, \big (k+E_1 + o(1) \big )$$
and hence (\ref{en-per}) follows.
\end{proof}

\vspace{5mm}

\begin{proof} [Proposition \ref{prp:ps2}]
Since the period is bounded, we know by Lemma \ref{lem:ps}, that $e(x_h)$ is also bounded. 
Hence, the curves $x_h$ are $1/2$-equi-H\"older-continuous. By the Ascoli-Arzel\`a theorem, up to taking a subsequence they converge uniformly to a continuous curve $x$. Thus, the $x_h$'s eventually belong 
to the image of a local chart for $H^1(\T,M)$, where we know the 1-form $\eta_k$ to be exact.  Now, arguing as in the proof of Lemma \ref{lemmalimitatezza} the thesis follows.
\end{proof}

\vspace{5mm}

At this point we need a mechanism to produce critical sequences for $\eta_k$ with periods bounded and bounded away from zero. We will do this via a minimax method; for the argument we look for a vector field on $\mathcal M$ 
which generalizes the negative-gradient vectof field of the action functional $\A_k^{L_\theta}$ when $\sigma=d\theta$ is exact. In the next section, we discuss the properties of this vector field in the general case. 


\section{A truncated gradient}
\label{ageneralizedpseudogradient}

We know that when the $1$-form $\tau^\sigma$ in (\ref{tauk}) is non-exact, we cannot define a global primitive of $\eta_k$ on $\mathcal M$. However,
if $u:[0,1]\rightarrow \mathcal M$ is of class $C^1$, we can define the variation $\Delta S_k(u):[0,1]\rightarrow\R$ of $\eta_k$ along the path $u$ by
\begin{equation}
\Delta S_k(u)(s)\ :=\ \eta_k(u|_{[0,s]})\ = \ \int_0^s u^*\eta_k\,.  
\label{variazionelungou}
\end{equation}

Then, since $\eta_k$ is closed, we extend the definition of $\Delta S_k$ to any continuous path by uniform approximation with paths of class $C^1$.
Observe that $\Delta S_k(u)(0)=0$ and if $u$ takes values in a region where $\eta_k$ is exact with primitive $S_k$, then there holds
\begin{equation}\label{eq:var}
\Delta S_k(u)(s)\ =\ S_k(u(s))\ -\ S_k(u(0))\,. 
\end{equation}

The next lemma describes how $\Delta S_k$ changes under deformation of paths in $\mathcal M$ with the first endpoint fixed. The proof follows from the fact that $\eta_k$ is a closed form.

\begin{lemma}\label{lem:dec}
Let $R>0$ and let $u:[0,R]\times[0,1]\rightarrow \mathcal M$ be a homotopy of paths. Denote by $u_r:=u(r,\cdot)$ and $u^s:=u(\cdot,s)$ the paths in $\mathcal M$ obtained keeping one of the 
variables fixed. If $u^0$ is constant, then for every $s\in[0,1]$
\begin{equation}\label{eq:dec}
\Delta S_k(u_R)(s)=\Delta S_k(u_0)(s)+\Delta S_k(u^s)(R)\,.
\end{equation}
\end{lemma}

We now proceed to define a generalized pseudo-gradient. First we consider the vector field $-\sharp\eta_k$, where $\sharp$ is the duality between $T\mathcal M$ and $T^*\mathcal M$ given by the metric $g_{\mathcal M}$
in (\ref{productmetric}). By Corollary \ref{cor:reg} $-\sharp\eta_k$ is locally Lipschitz and hence we have local existence and uniqueness for solutions of the associated Cauchy problem. 

However, since $\|\eta_k\|$ is not bounded on $\mathcal M$, solutions of the Cauchy problem could escape to infinity in finite time, thus having a maximal interval of definition of finite length. To avoid this problem we consider the bounded conformal vector field 
\begin{equation}
X_k:=\frac{-\sharp\eta_k}{\sqrt{1+\|\eta_k\|^2}}\,.
\end{equation}
We define $\Phi^{k}$ as the local flow of $X_k$ on $\mathcal M$ generated by the maximal solutions 
$$u_{(x,T)}:(R^-_{(x,T)},R^+_{(x,T)})\longrightarrow\mathcal M$$
of the Cauchy problem with initial condition $u(0)=(x,T)$. Here $(x,T)$ is some element in $\mathcal M$ and $R_{(x,T)}^-,R^+_{(x,T)}$ are numbers in $\R^+\cup\{+\infty\}$. By definition, we have 
$$\Phi^{k}_r(x,T)\ =\ u_{(x,T)}(r)\ =: (x(r),T(r))\, .$$

Actually we are interested in the behavior of the local flow only in forward time, so hereafter we forget about $R^-_{(x,T)}$ and study the properties of $\Phi^k$ for positive times. As we will see in the next lemma,
the only source of incompleteness for $\Phi^{k}$ is that the map $r\mapsto T(r)$ has $0$ as a limit point.
\vspace{2mm}

\begin{prop}\label{pro:comp}
Let $u:[0,R)\rightarrow \mathcal M$ be a maximal flow-line of $\Phi^k$ and suppose that $R<+\infty$. Then necessarily 
\begin{equation}\label{eq:liminf}
\liminf_{r\rightarrow R}\ T(r)=0\,.
\end{equation}
In this case there exist a constant $C>0$ and a sequence $r_h\rightarrow R$ such that 
$$T(r_h)\longrightarrow 0\quad \text{and} \quad e(x(r_h))\ \leq \ C\, T(r_h)^2\, .$$
\end{prop}

\begin{proof}
Suppose that $R<+\infty$ and assume by contradiction that $T(r)\geq T_*$ for every $r\in[0,R)$. Observe that
\begin{equation*}
\left\|\frac{d}{dr}u\right\| \ =\ \left\|\frac{-\sharp\eta_k}{\sqrt{1+\|\eta_k\|^2}}\right\| \ =\ \frac{\|\eta_k\|}{\sqrt{1+\|\eta_k\|^2}}\ <\ 1\,.
\end{equation*}

Since the derivative of $u$ is bounded by the above inequality and $H^1(\T,M)\times[T_*,+\infty)$ is complete, there exists the limit 
$$u_*:= \ \lim_{r\rightarrow R}u(r)\, .$$

As $X_k$ is locally Lipschitz, there exists a neighbourhood $\mathcal U$ of $u_*$, such that the solutions to the Cauchy problem with initial data in $\mathcal U$ all exist in a small fixed interval $[0,r_{u_*}]$. This yields a contradiction as soon as $u(r)\in\mathcal U$ and $R-r<r_{u_*}$. Suppose now that \eqref{eq:liminf} holds. In this case there exists a sequence $r_h\rightarrow R$ such that 
$$T(r_h)\longrightarrow 0 \quad \text{and} \quad \frac{dT}{dr}(r_h)\ \leq \ 0\, .$$
Using \eqref{disuguaglianzeE} and (\ref{par-T}), we then find
\begin{equation*}
0 \ \geq \  \frac{dT}{dr}(r_h) \ =\  -(\eta_k)_{u(r_h)}\left[\frac{\partial }{\partial T}\right] \ \geq \ E_0\, \frac {e(x(r_h))}{T(r_h)^2}-E_1-k\,,
\end{equation*}
which gives the required bound for the energy. 
\end{proof}

\vspace{5mm}

The following lemma is about the $1/2$-H\"older norm of the flow-lines of $\Phi^k$ and will be used in the proof of Proposition \ref{Struwe}.

\begin{lemma}\label{lem:ac-distper}
If $u:[0,R]\rightarrow\mathcal M$ is a flow line of $\Phi^{k}$, then
\begin{equation*}
-\Delta S_k(u)(R)\ \geq \ \frac{d_{\mathcal M}(u(R),u(0))^2}{R}\,, 
\end{equation*}
where $d_{\mathcal M}$ denotes the distance induced by the metric $g_{\mathcal M}$ in (\ref{productmetric}). In particular,
\begin{equation}\label{eq:ac-per}
-\Delta S_k(u)(R)\ \geq \ \frac{(T(R)-T(0))^2}{R}\,.
\end{equation}
\end{lemma}

\begin{proof}
We just compute
\begin{eqnarray*}
-\Delta S_k(u)(R) &=& -\int_0^R\eta_k(u)\left(\frac{du}{dr}\right)dr\ \ = \ -\int_0^R \eta_k(u) \left (X_k(u)\right )\, dr \ = \\ 
                            &=& -\int_0^R \eta_k(u) \left (\frac{-\sharp \eta_k(u)}{\sqrt{1+\|\eta_k(u)\|^2}} \right ) \, dr \ = \int_0^R \frac{\|\eta_k(u)\|^2}{\sqrt{1+\|\eta_k(u)\|^2}} \, dr\ = \\
                            &=& \int_0^R \sqrt{1+\|\eta_k(u)\|^2} \cdot \|X_k(u)\|^2 \, dr \ \geq \ \int_0^R \|X_k(u)\|^2\, dr \ = \\
                            &=& \int_0^R\left\|\frac{du}{dr}\right\|^2\, dr \geq \ \frac{1}{R}\left(\int_0^R\left\|\frac{du}{dr}\right\|\, dr\right)^2\ \geq\ \frac{d_{\mathcal M}(u(R),u(0))^2}{R}\,, 
\end{eqnarray*}
where the penultimate inequality follows from the Cauchy-Schwarz inequality. To obtain \eqref{eq:ac-per} we just observe that 
$$|T(R)-T(0)|\ \leq \ d_{\mathcal M}(u(R),u(0))\, .$$
as $d_{\mathcal M}$ is a product distance.
\end{proof}

\vspace{5mm}

Lemma \ref{lem:ps} and Proposition \ref{pro:comp} show that the only source of non-completeness of the local flow $\Phi^k$ are trajectories that go closer and closer to the subset of constant loops. In particular, this yields that 
$\Phi^{k}$ is positively complete on $\mathcal M\setminus\mathcal M_0$. The case of $\mathcal M_0$ is more delicate and requires to take a deeper look to $\eta_k$ close to the set of constant loops.
We will do this for $k>e_0(L)$ obtaining two outcomes.

First, $\eta_k$ admits a positive primitive $S_k$ on the subset of loops with small kinetic energy and such a function has an interesting geometry, which will be exploited for the minimax method. 
Second, we will improve Proposition \ref{pro:comp} and show that on the flow lines with finite maximal interval of definition $S_k\rightarrow 0$.

\noindent Inside $\mathcal M_0$ we single out the submanifold of constant loops 
$$M_0 :=\ M\times \R^+\, .$$
For every $\delta>0$ consider the neighbourhood $\mathcal V_\delta$ of $M_0$ given by
\begin{equation}
\mathcal V_\delta :=\ \Big\{(x,T)\ \Big |\ e(x)<\delta\,\Big\}\,.
\end{equation}

If $\delta$ is sufficiently small, then there is a deformation retraction of $\mathcal V_\delta$ onto $M_0$ which fixes the period. Such a deformation can be obtained for example by considering the negative gradient flow of the function $e$ on $H^1(\T,M)$.
The deformation yields a capping disc $D_x$ for $x$ and, hence, an explicit primitive $S_k$ for $\eta_k$ on $\mathcal V_\delta$ defined by 
\begin{equation}
S_k(x,T) :=\ \A^L_k(x,T)\ +\ \int_{D_x}\sigma\,.
\label{primitiveonvdelta}
\end{equation}

\begin{lemma}
If $\delta$ is sufficiently small, there exists $\Theta_0>0$ such that
\begin{equation}\label{eq:area}
\left|\int_{D_x}\sigma\right|\ \leq\ \Theta_0\, l(x)^2\, , \ \ \ \ \forall \ (x,T)\in\mathcal V_\delta\, .
\end{equation}
In particular there exists $B>0$ such that
\begin{equation}\label{eq:estab}
S_k(x,T)\ \leq\ B\,\frac{e(x)}{T}\, +\, (B+k)T\, +\, \Theta_0\, l(x)^2\,.
\end{equation}
\end{lemma}

\begin{proof}
The estimate \eqref{eq:area} is exactly Lemma 7.1 in [Abb13]. Since $L$ is a Tonelli Lagrangian electromagnetic at infinity we can find $B>0$ such that 
$$L(q,v)\ \leq\ B(1+\|v\|^2)\, , \ \ \ \ \forall \ (q,v)\in TM$$
and we readily compute 
\begin{align*}
S_k(x,T)&=\ T\int_0^1\Big[L\Big (x(s),\frac{x'(s)}{T}\Big )+k\Big]\, ds\ +\ \int_{D_{x}}\sigma\\
&\leq\ T\int_0^1\Big[B\, \frac{\|x'(s)\|^2}{T^2}+B+k\Big]\, ds\ +\ \Theta_0\,l(x)^2\\
&\leq\  B\, \frac{e(x)}{T}\, +\, (B+k)T\, +\, \Theta_0\, l(x)^2
\end{align*}
as we wished to prove.
\end{proof}

\vspace{3mm}

It is worth to point out that \eqref{eq:area}, with a different $\Theta_0$, holds more generally for any closed 2-form on $M$, as this remark will be used in the proof of the next lemma.
If we want more informations on the behaviour of  $S_k$, we have to restrict the range of energies we consider. Indeed, notice that on $M_0$ the function $S_k$ reduces to
\begin{equation}\label{eq:skonm}
S_k(x_0,T)\ =\ T\big [L(x_0,0)+k\big ]\ =\ T\big [k-E(x_0,0)\big ]\,.
\end{equation}

Hence, if $k>e_0(L)$, then for every $x_0\in M$ the function $T\mapsto S_k(x_0,T)$ is increasing in $T$ and tends to zero as $T$ goes to $0$. Moreover, in the same energy range we have a positive lower bound for $S_k$ on $\partial \mathcal V_\delta$
as the following lemma shows. 

\begin{lemma}\label{lem:low}
If $k>e_0(L)$, there exists $\delta_k>0$ such that for every $\delta\in(0,\delta_k)$ there exists $\varepsilon_{k,\delta}>0$ such that
\begin{equation*}
\inf_{\partial \mathcal V_\delta}S_k\ \geq\ \varepsilon_{k,\delta}\,. 
\end{equation*}
\end{lemma}

\begin{proof}
We first prove that there exist $L_2,\Theta_1>0$ such that, for $\delta$ sufficiently small, every $(x,T)\in\mathcal V_\delta$ satisfies
\begin{equation}
\A^L_k(x,T)\ \geq\ l(x)\sqrt{L_2(k-e_0(L))}-\Theta_1l(x)^2\,,
\label{cicici}
\end{equation}
where $\A_k^L$ is the free-period Lagrangian action functional associated with $L$. Consider the smooth one-form on $M$
$$\theta(q)[v] := d_vL(q, 0)[v]\, .$$
By taking a Taylor expansion and by using the bound (\ref{secondinequality}), we get the estimate
\begin{eqnarray*}
L(q, v) &=& L(q, 0) + d_vL(q, 0)[v] +\frac12 \, d_{vv}L(q,sv)[v,v] \ \geq\\
            &\geq& - E(q,0) + \theta(q)[v] + a \, \|v\|_q^2\, .
\end{eqnarray*}
For a fixed $(x,T)\in \mathcal V_\delta$ we then compute 
\begin{eqnarray*}
\A_k(x,T) &\geq& \int_0^T \Big [-E(\gamma(t),0)+\theta(\gamma(t))[\dot \gamma(t)] + a\, \|\dot \gamma(t)\|^2 + k\Big ]\, dt \ \geq \\ 
                &\geq& \big [k-e_0(L)\big ]\, T \ + \int_0^T \gamma^*\theta \ + \ a\int_0^T \|\dot \gamma(t)\|^2\, dt \ \geq \\
                &\geq& \big [k-e_0(L)\big ]\, T \ - \ \Theta_1\, l(x)^2 + \frac aT\, l(x)^2\, ,
\end{eqnarray*}
where $\Theta_1>0$ is such that \eqref{eq:area} holds for $d\theta$ integrated over the ``canonical'' capping disc for $x$. For $l(x)$ fixed, the last expression in $T$ attains its minimum at 
$$T\ = \ l(x)\, \sqrt{\frac{a}{k-e_0(L)}}\, , $$
where it equals to
$$2\sqrt{a(k-e_0(L))} \, l(x) - \Theta_1\, l(x)^2\, .$$
The choice $L_2=a/4$ proves then \eqref{cicici}. Combining \eqref{cicici} with \eqref{eq:area},we get
\begin{equation*}
S_k(x,T)\ \geq\ l(x)\sqrt{L_2(k-e_0(L))}-(\Theta_0+\Theta_1)l(x)^2,
\end{equation*}
which is positive if $l(x)$ is small and positive. The thesis follows.
\end{proof}

\vspace{5mm}

\noindent We define the sublevel sets $\mathcal W_\delta',\mathcal W_\delta \subseteq \mathcal V_\delta$ of $S_k$ as follows
$$\mathcal W'_{\delta}:=\ \Big \{S_k<\varepsilon_{k,\delta}/4\Big \}\, , \ \  \ \ \mathcal W_{\delta}:=\ \Big \{S_k<\varepsilon_{k,\delta}/2\Big \}\, .$$

Thanks to Lemma \ref{lem:low}, we have that $\mathcal W_\delta\cap \partial\mathcal V_\delta=\emptyset$; moreover, $M_0\cap \mathcal W'_{\delta}\neq \emptyset$. 
Because of \eqref{eq:skonm}, these sets play an important role as far as critical sequences and flow lines are concerned.

\begin{lemma}\label{lem:wdelta}
If $k>e_0(L)$, then for every $\delta>0$ sufficiently small there holds:
\begin{enumerate}[\itshape i)]
 \item if $(x_h,T_h)$ is a critical sequence for $\eta_k$ in $\mathcal M_0$ such that $T_h\rightarrow 0$, then $(x_h,T_h)\subset \mathcal W'_\delta$ for every $h$ sufficiently large;
 \item If a flow line of $\Phi^k$ is not defined for all positive times, then it enters $\mathcal W'_\delta$.
\end{enumerate}
\end{lemma}

\begin{proof}
We prove the first statement. By Inequality \eqref{en-per} in Lemma \ref{lem:ps}, there exists $C>0$ such that $e(x_h)\leq C T_h^2$. Therefore, $(x_h,T_h)\in \mathcal V_\delta$ for $h$ big enough and we can use Inequality \eqref{eq:estab} to obtain
\begin{equation}
S_k(x_h,T_h)\ \leq\ B\,\frac{CT_h^2}{T_h}\, +\, (B+k)T_h\, +\, \Theta_0\, CT_h^2\,,
\end{equation}
which goes to zero as $h$ goes to infinity. We prove the second statement. Let $(x,T)\in\mathcal M_0\setminus\mathcal W'_\delta$ and let $[0,R_{(x,T)})$ be the maximal interval of definition of the flow line 
$$r\mapsto (x(r),T(r)):= \ \Phi_r^k(x,T)\, .$$
Suppose that $R_{(x,T)}<+\infty$. By Proposition \ref{pro:comp}, there exists $r_h\rightarrow R$ such that 
$$T(r_h)\longrightarrow 0\ \ \ \ \text{and} \ \ \ \ e(x(r_h))\leq C T(r_h)^2\, .$$

By the same argument as in the proof of the first statement, we have $(x_h,T_h)\in \mathcal W_\delta'$, for large $h$. This completes the proof, as $\mathcal W'_\delta$ is positively $\Phi^k$-invariant.
\end{proof}

\vspace{5mm}

Thanks to the previous lemma, in order to make $\Phi^{k}$ positively complete, we stop the flow lines entering $\mathcal W'_\delta$. 
Consider a smooth cut-off function $\kappa_\delta:\R^+\rightarrow[0,1]$ with 
$$\kappa_\delta^{-1}(0)\ =\ \big (0,\varepsilon_{k,\delta}/4\big ]\, , \quad \,\kappa_\delta^{-1}(1)\ =\ \big [ \varepsilon_{k,\delta}/2,+\infty\big ) .$$
We use this function to define $\hat\kappa_\delta:\mathcal M_0\rightarrow[0,1]$ as follows:
\[\hat\kappa_\delta\ = \begin{cases}
\ 1 & \text{on} \ \mathcal M_0\setminus\mathcal V_\delta\,,\\
\ \kappa_\delta\circ S_k & \text{on} \ \mathcal V_\delta\,. 
\end{cases}
\]
Finally, define the vector field $X_{k,\delta}:=\hat\kappa_\delta X_{k}$ and denote its semi-flow by $\Phi^{k,\delta}$.

\begin{lemma}\label{lem:com}
The time-dependent semi-flow $\Phi^{k,\delta}$ is complete on $\mathcal M_0$.
\end{lemma}

\begin{proof}
The flow $\Phi^{k,\delta}$ has the same flow lines as $\Phi^{k}$, possibly traveled at a lower speed. Hence, if a flow-line does not intersect $\mathcal W'_\delta$, it is defined for all positive times by Lemma \ref{lem:wdelta}. On the other hand, if it
intersects $\mathcal W'_\delta$, it is eventually constant since $X_{k,\delta}=0$ on $\mathcal W'_\delta$ and hence the trajectory is defined for all positive times.
\end{proof}

\vspace{5mm}

Summarizing, the non-completeness of $\Phi^k$ can be overcome by truncating it near the manifold of constant loop, namely by multiplying the vector field $X_k$ by a cut-off function $\hat \kappa_\delta$, whose role is to make 
the flow-lines constant if the kinetic energy is sufficiently small. We will see that this is not restrictive for our goals.


\section{The minimax class}
\label{theminimaxclass}

In this section we are going to see that, by Identity \eqref{eq:skonm} and Lemma \ref{lem:low}, the $1$-form $\eta_k$ exhibits a mountain pass geometry on some space $\mathfrak U$ of continuous maps 
$$(B^{l-1},S^{l-2})\longrightarrow (\mathcal M_0,M_0)\, ,$$
provided that $k>e_0(L)$. The set $\mathfrak U$ must enjoy the following two properties:
\begin{itemize}
 \item it is positively invariant under the action of $\Phi^{k,\delta}$;
 \item all its elements are based in $M_0\cap\mathcal W'_\delta$ and intersect $\partial \mathcal V_\delta$, meaning that for all $u\in \mathfrak U$ there holds $u(S^{l-2})\subseteq M_0\cap \mathcal W_\delta'$ and
 $u(\xi)\in \partial \mathcal V_\delta$ for some  $\xi \in B^{l-1}$.
\end{itemize}

We construct $\mathfrak U$ under the hypothesis that $\pi_l(M)\neq0$ for some $l\geq 2$. Thus, let $\mathfrak u\in \pi_l(M)\setminus \{0\}$  and fix $k^*>e_0(L)$. 
Let $I=I(k^*)\subseteq (e_0(L),+\infty)$ be a bounded open interval containing $k^*$; observe that there exists $T_0>0$ such that
\begin{equation}
S_k(p,T_0)\ <\ \frac{\epsilon_{k,\delta}}{4}\, , \ \ \ \ \forall \ p\in M\, , \ \forall \ k\in I\, .
\label{condizionemathfraku}
\end{equation}

In order to achieve the invariance under the action of $\Phi^{k,\delta}$, the class $\mathfrak U$ of maps $(B^{l-1},S^{l-2})\rightarrow (\mathcal M_0,M_0)$ will be therefore built in such a way that 
$$u(S^{l-2})\ \subseteq \  M\times \{T_0\}\, , \ \ \ \ \forall \ u\in \mathfrak U\, .$$

We start showing a decomposition of $S^l$ into a family of loops parametrized on $B^{l-1}$, such that to every point in $\partial B^{l-1}\cong S^{l-2}$ there corresponds a constant loop; this construction is analogous to the 
one used in the proof of Theorem 2.4.20 in \cite{Kli95}. Look at $B^{l-1}$ as the half equator in $S^l\subseteq \R^{l+1}$ given by
$$B^{l-1} \ = \ \Big \{(x_0,...,x_l)\in \R^{l+1} \ \Big |\ x_0\geq 0 \, , \ x_1=0\Big \}\, .$$

For every $p=(p_0,0,p_2,...,p_l)=(p_0,p')\in B^{l-1}$ consider the circle defined as the intersection of $S^l$ with the plane $\{x_i=p_i \ |\ i=2,...,l\}$ and parametrize it with 
$$\gamma_p :[0,1]\rightarrow S^l\, , \ \ \ \ \gamma_p(t) := \ (\sqrt{1-\|p'\|^2} \, \cos 2\pi t\, , \, \sqrt{1-\|p'\|^2}\,  \sin 2\pi t \, , p_2\, , \, ...\, , \, p_l \big )\, .$$
Notice that $\gamma_p$ is the constant loop in $p$ for every $p\in \partial B^{l-1}$.

\begin{center}
\includegraphics[height=50mm]{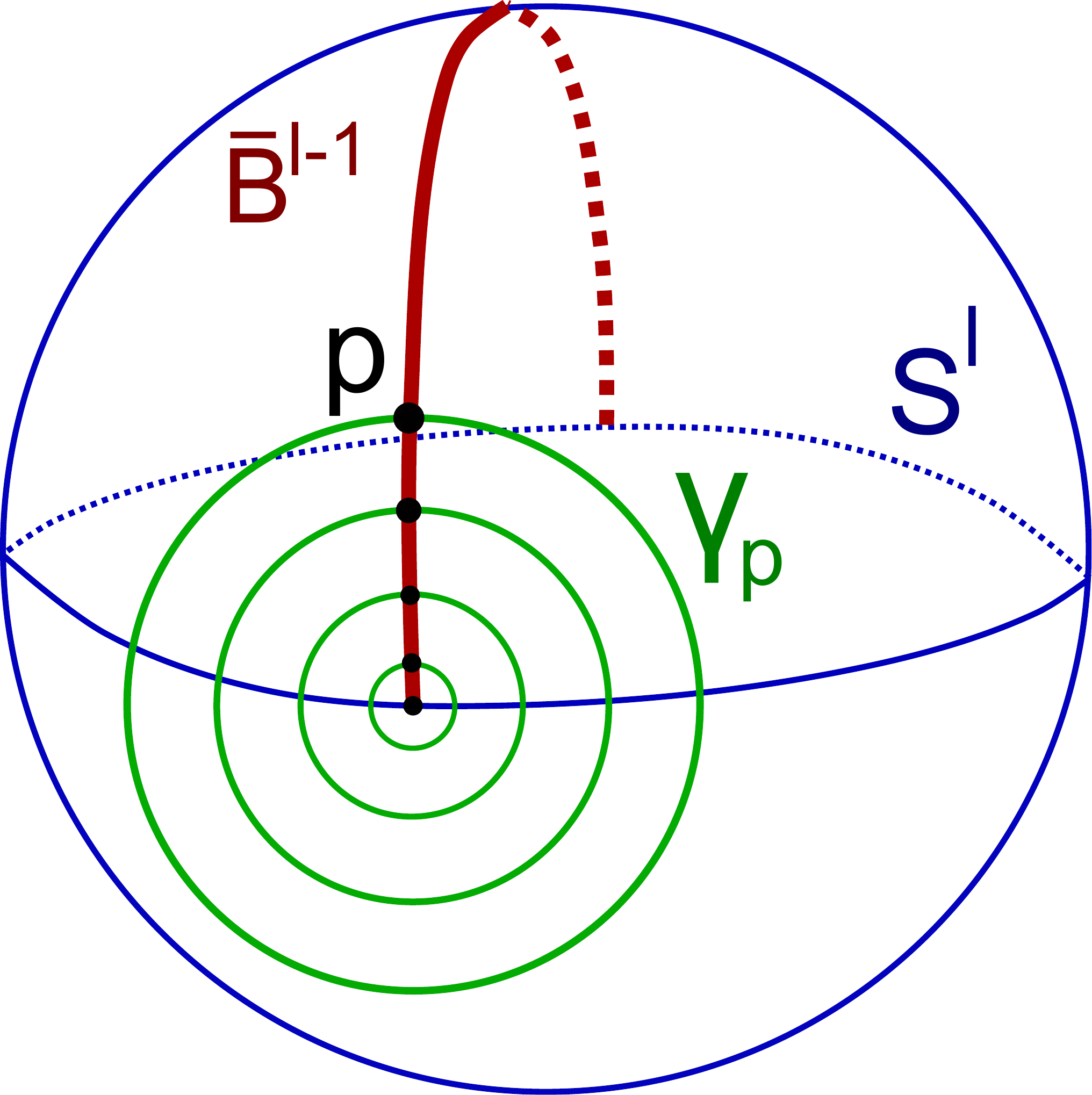}
\end{center}

\noindent Therefore, given any continuous map $f:S^l\rightarrow M$ we can define a map 
$$F\ =\ F(f):(B^{l-1},S^{l-2})\longrightarrow (C^0(\T,M),M)$$ 
just by setting $F(p) := f\circ \gamma_p(\cdot)$. The converse is also true; namely, to any continuous map $F:(B^{l-1},S^{l-2})\rightarrow (C^0(\T,M),M)$ we can associate a map $f(F):S^l\rightarrow M$ and these operations turn to be the inverse of each other
(cf. \cite[Page 180]{Kli95}). Moreover, homotopies of $f$ are mapped into homotopies of $F=F(f)$, and viceversa, through this correspondence.
Now define the class $\mathfrak U$ of maps 
$$u\ = \ (x,T):(B^{l-1},S^{l-2})\longrightarrow (\mathcal M_0,M_0)$$ 
such that $f(x)\in \mathfrak u$ and $u(S^{l-2})\subseteq M\times \{T_0\}$, with $T_0$ as in (\ref{condizionemathfraku}). 

The class $\mathfrak U$ just defined is clearly non-empty, as one readily sees taking a smooth function $f\in \mathfrak u$ and considering $u=(F(f),T)$.

Moreover, it follows from Theorem 2.1.8 in \cite{Kli78} and the fact that $\mathfrak u\in \pi_l(M)$ is non trivial that for every element $u\in \mathfrak U$ there exists $\xi \in B^{l-1}$ such that $u(\xi)\in \partial \mathcal V_\delta$.
Indeed, suppose that the image of $u$ is entirely contained in $\mathcal V_\delta$; then, there would exist a homotopy from $u$ to a map $\tilde u$ with image entirely contained in $M_0$, which would imply $[f(\tilde u)]=0$. 
Since the homotopy from $u$ to $\tilde u$ yields a homotopy from $f(u)$ to $f(\tilde u)$, this would actually imply $[f(u)]=0$ which is a contradiction.

Finally, the class $\mathfrak U$ is clearly invariant under the action of $\Phi^{k,\delta}$.

\vspace{3mm}

We are now ready to define the minimax function $c^{\mathfrak u}:I\rightarrow (0,+\infty)$. Suppose for the moment that $l\geq 3$; for every $\xi \in B^{l-1}$ consider a path $a_\xi:[0,1]\rightarrow B^{l-1}$
connecting a point on the boundary $S^{l-2}$ to $\xi$ and, for every $u\in \mathfrak U$, define the composition $u_\xi:=u\circ a_\xi$. Now set
$$\widetilde {S}_k(u,\xi) := \ S_k(u_\xi(0)) \ + \ \Delta S_k(u_\xi)(1)\, ,$$
with $S_k(\cdot)$ primitive of $\eta_k$ on $\mathcal V_\delta$ as in (\ref{primitiveonvdelta}) and $\Delta S_k(u_\xi)(1)$ variation of $\eta_k$ along the path $u_\xi$ given by (\ref{variazionelungou}), and
\begin{equation}
c^{\mathfrak u} (k) := \ \inf_{u\in \mathfrak U} \ \max_{\xi\in B^{l-1}} \ \widetilde{S}_k(u,\xi)\, .
\label{minimaxlmaggiore2}
\end{equation}

Notice that, fixed $u\in \mathfrak U$, the value $\widetilde S_k(u,\xi)$ depends only on the point $\xi$ and not on the path used to join $S^{l-2}$ with $\xi$. Indeed, let $a_\xi':[0,1]\rightarrow B^{l-1}$ be another path connecting the boundary with $\xi$; 
since $l\geq 3$, $S^{l-2}$ is path-connected and hence there exists a path $\delta:[0,1]\rightarrow S^{l-2}$ from $a_\xi(0)$ to $a_\xi'(0)$. 

\begin{center}
\includegraphics[height=45mm]{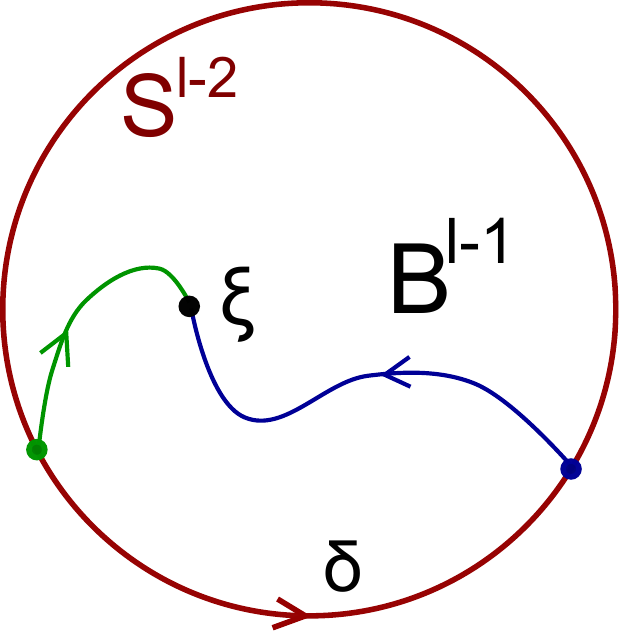}
\end{center}

\noindent Denote with $u_{\xi}':=u\circ a_\xi'$, $u_\delta := u\circ \delta$; the fact that $\eta_k$ is a closed form implies that 
\begin{eqnarray*}
\Delta S_k(u_\xi)(1) &=& \int_0^1 u_\xi^*\eta_k \  = \ \Delta S_k (u_{\xi}'\#u_\delta )(2) \ = \int_0^{2} (u_{\xi}'\#u_\delta)^* \eta_k \ = \\ 
                                     &=& \int_0^1 u_\delta^* \eta_k \ + \ \int_0^{1} (u_{\xi}')^*\eta_k \ = \ S_k(u_\xi'(0)) - S_k(u_\xi(0)) + \Delta S_k (u_{\xi}')(1)
\end{eqnarray*}
\noindent which yields 
$$S_k(u_\xi(0)) + \Delta S_k(u_\xi)(1) \, = \, S_k(u_\xi'(0)) + \Delta S_k (u_\xi')(1) \,  .$$

One could a priori repeat the same construction also for $l=2$; however, since $S^{l-2}$ is not connected anymore, the value $\widetilde S_k(u,\xi)$ might depend (and in general it does) on the path we use to connect $\xi$ with the boundary. In this case we therefore 
choose $a_\xi:[0,1]\rightarrow B^1$ such that $a_\xi(0)=-1$, $a_\xi(1)=\xi$, set $u_\xi:=u\circ a_\xi$ and define as above the minimax function $c^{\mathfrak u}(\cdot)$ by
\begin{equation*}
c^{\mathfrak u} (k) := \ \inf_{u\in \mathfrak U} \ \max_{\xi\in B^1} \ \widetilde S_k(u,\xi)\, .
\end{equation*}

With this definition we can treat both the cases $l=2$ and $l\geq 3$ at once. In the next lemma we prove a crucial monotonicity property for the function $c^{\mathfrak u}$.

\begin{lemma}\label{lem:mon}
If $k_1,k_2\in I$ are such that $k_1<k_2$, then 
\begin{equation*}
c^{\mathfrak u}(k_1)\ \leq\ c^{\mathfrak u}(k_2)\,.
\end{equation*}
\end{lemma}

\begin{proof}
We have $\eta_{k_2}-\eta_{k_1}=(k_2-k_1)\, dT$. Integrating this along ${u_\xi}$, we get
\begin{equation*}
\Delta S_{k_2}(u_\xi)(1)-\Delta S_{k_1}(u_\xi)(1)\ =\ (k_2-k_1)(T(1)-T_0)\,.
\end{equation*}
\noindent This implies that, for every $u\in\mathfrak U$, we have
\begin{equation*}
\widetilde S_{k_1}(u,\xi)\ <\ \widetilde S_{k_2}(u,\xi)\,,\quad \forall \ \xi \in B^{l-1}\, ,
\end{equation*}
\noindent since clearly there holds
 \[(k_2-k_1)(T(1)-T_0)\ >\ - (k_2-k_1)T_0\, .\]
\noindent The assertion follows then taking the inf-max over $\mathfrak U$.
\end{proof}

\vspace{5mm}

We will use the family $\mathfrak U$ in the next section to construct a critical sequence for $\eta_k$ with periods bounded and bounded away from zero whenever $k\in I$ is a point of differentiability for $c^{\mathfrak u}(\cdot)$, using a generalization of 
the \textit{Struwe monotonicity argument} \cite{Str90}. To exclude that the periods tend to zero we need the following lemma, which states that, if $\xi_*\in B^{l-1}$ almost realizes the maximum of  
$$\xi \ \longmapsto \ \widetilde S_k(u,\xi)\, ,$$
\noindent then $u(\xi_*)\notin\mathcal W_\delta$.  

\begin{lemma}
Fix $k\in I$ and $u\in\mathfrak U$. If $\xi_*\in B^{l-1}$ is such that
\begin{equation}\label{ine:notinwgeneral}
\widetilde S_k(u,\xi_*)\ >\ \max_{\xi\in B^{l-1}}\ \widetilde S_k(u,\xi)\ -\ \frac{\varepsilon_{k,\delta}}{2}\,,
\end{equation}
\noindent then $u(\xi_*)\notin \mathcal W_\delta$.
\label{lem:notinwgeneral}
\end{lemma}

\begin{proof}
By assumption  there exists $\bar \xi\in B^{l-1}$ such that $u(\bar \xi)\in \partial \mathcal V_\delta$. Without loss of generality we may suppose that the path $u_{\bar \xi}|_{[0,1)}$ is 
entirely contained in $\mathcal V_\delta$; it follows
$$\max_{\xi \in B^{l-1}} \ \widetilde S_k(u,\xi) \ \geq \ \widetilde S_k (u,\bar\xi) \ = \ S_k(u(\bar\xi))\  \geq \ \varepsilon_{k,\delta}$$
and hence in particular by \eqref{ine:notinwgeneral}
$$\widetilde S_k(u,\xi_*) \ \geq \ \frac{\varepsilon_{k,\delta}}{2}\, .$$
Suppose $l\geq 3$; if $u(\xi_*)$ would belong to $\mathcal W_\delta$ then we could choose the path $u_{\xi_*}$ to be entirely contained in $\mathcal V_\delta$; this would imply that 
$$\widetilde S_k(u,\xi_*) \ = \ S_k(u(\xi_*)) \ < \ \frac{\varepsilon_{k,\delta}}{2}$$
which is clearly a contradiction. Suppose now $l=2$ and $u(\xi_*)\in \mathcal W_\delta$; in this case we have two possibilities: either $u_{\xi_*}$ is entirely contained in $\mathcal V_\delta$, which yields a contradiction exactly as above,
or there exists $s\in (0,1]$ such that $u_{\xi_*}(s)\in \partial \mathcal V_\delta$ and $u_{\xi_*}|_{(s,1]}\subseteq \mathcal V_\delta$. We now compute using the definition
\begin{eqnarray*}
\widetilde S_k(u,u_{\xi_*}(s)) - \widetilde S_k(u,\xi_*) &=&  \Delta S_k(u_{\xi_*})(s) - \Delta S_k(u_{\xi_*})(1) \ = \\ 
         &=& S_k(u_{\xi_*}(s)) - S_k(u(\xi_*))\ > \ \frac{\varepsilon_{k,\delta}}{2}\, .
\end{eqnarray*}
It follows that 
$$\max_{\xi \in B^{1}} \ \widetilde S_k(u,\xi) \ \geq \ \widetilde S_k(u,u_{\xi_*}(s))\ > \ \widetilde S_k(u,\xi_*) + \frac{\varepsilon_{k,\delta}}{2}\, .$$
in contradiction with the assumption.
\end{proof}


\section{A generalized Lusternik-Fet theorem}
\label{ageneralizedlyusternikfettheorem}

In this section, building on the results of the previous paragraphs, we generalize the Lusternik and Fet theorem \cite{FL51} to the setting considered in this chapter.

\begin{teo}[Generalized Lusternik-Fet theorem]
Let $(M,g)$ be a closed connected Riemannian manifold, $L:TM\rightarrow \R$ be a Tonelli Lagrangian and $\sigma$ be a closed 2-form. If $\pi_l(M)\neq0$ for some $l\geq 2$, then for 
almost every $k>e_0(L)$ there exists a contractible periodic orbit for the flow of the pair $(L,\sigma)$ with energy $k$.
\label{generalizedlyusternikfettheorem}
\end{teo}

\vspace{1mm}

In virtue of Proposition \ref{prp:ps2} and of the fact that a monotone function is almost everywhere differentiable it will suffice to show the following 

\begin{prop}
Fix $k^*>e_0(L)$ and let $c^{\mathfrak u}:I\rightarrow (0,+\infty)$ as defined in (\ref{minimaxlmaggiore2}). If $k\in I$ is a point of differentiability for $c^{\mathfrak u}$, then there exists a critical sequence $(x_h,T_h)\in \mathcal M_0$ with periods
bounded and bounded away from zero.
\label{Struwe}
\end{prop}

\begin{proof} By assumption there exists a positive constant $A>0$ such that for every $k'\geq k$ sufficiently close to $k$ 
\begin{equation}
c^{\mathfrak u}(k')\ -\ c^{\mathfrak u}(k)\ \leq\ A\, (k'-k)\,.  
\label{stimalipschitz}
\end{equation}

Consider a sequence $k_m\downarrow k$ and denote by $\lambda_m:=k_m-k\downarrow 0$. Clearly we may suppose (\ref{stimalipschitz}) to hold for every $k_m$. Take a corresponding $u_m=(x_m,T_m)\in\mathfrak U$ with
\begin{equation*}
\max_{\xi \in B^{l-1}}\ \widetilde S_{k_m}(u_m,\xi)\ <\  c^{\mathfrak u}(k_m)\ +\ \lambda_m\, .
\end{equation*}
By the very definition of $c^{\mathfrak u}(k)$, the subset of all $\xi\in B^{l-1}$ such that
\begin{equation}
\widetilde S_{k}(u_m,\xi)\ >\ c^{\mathfrak u}(k)\ -\ \lambda_m\,
\label{maggiorecu-lambdam1}
\end{equation}
\noindent is non-empty. As in the previous section denote by $(u_m)_\xi:=u_m\circ a_\xi$ the composition of  $u_m$ with a path $a_\xi:[0,1]\rightarrow B^{l-1}$ connecting $S^{l-2}$ to $\xi$. We now readily compute 
\begin{eqnarray*}
\widetilde S_k(u_m,\xi) &=& S_k((u_m)_\xi(0)) + \Delta S_k((u_m)_\xi)(1)\ = \ S_k ((u_m)_\xi(0)) \ +  \int_0^1 (u_m)_\xi^*\eta_k \ = \\ 
                                 &=& S_k((u_m)_\xi(0)) \ + \int_0^1 (u_m)_\xi^*d\A_k^L \ + \int_0^1 (u_m)_\xi^*\tau^\sigma \ = \\
                                 &=&  S_k((u_m)_\xi(0)) \ + \ \A_k^L((u_m)_\xi(1)) \ - \ \A_k^L((u_m)_\xi(0)) \ + \int_0^1 (u_m)_\xi^*\tau^\sigma \ = \\
                                 &=& \A_k^L(u_{m}(\xi)) + \int_0^1 (u_m)_\xi^*\tau^\sigma\, ,
\end{eqnarray*}
where we have used the fact that $S_k\equiv \A_k^L$ on $M_0$ and that $(u_m)_\xi(1)=u_m(\xi)$. Here $\A_k^L$ is the free-period Lagrangian action functional associated to $L$. 
Notice that the integral in the last expression is independent on $k$; hence, it follows that 
$$\widetilde S_{k_m}(u_{m},\xi)-\widetilde S_k(u_{m},\xi) \ = \  \A_{k_m}^L(u_{m}(\xi))-  \A_k^L(u_{m}(\xi))\ = \ \lambda_m\, T_{m}(\xi)\, .$$
\noindent Therefore, if $\xi\in B^{l-1}$ satisfies (\ref{maggiorecu-lambdam1}), we get
\begin{equation*}
T_{m}(\xi)\, =\, \frac{\widetilde S_{k_m}(u_{m},\xi)-\widetilde S_{k}(u_{m},\xi)}{\lambda_m}\, \leq \, \frac{c^{\mathfrak u}(k_m)+\lambda_m-\big (c^{\mathfrak u}(k)-\lambda_m\big )}{\lambda_m}\, \leq \, A+2\, , 
\end{equation*}
\noindent and at the same time, by (\ref{stimalipschitz}),
$$\widetilde S_k(u_{m},\xi) \ \leq\ \widetilde S_{k_m}(u_m,\xi) \ <\ c^{\mathfrak u}(k_m) + \lambda_m \ < \ c^{\mathfrak u} (k) + (A+1)\, \lambda_m\, .$$
\noindent Summing up, for every $m\in \N$ and every $\xi \in B^{l-1}$ either
\begin{equation}
\widetilde S_{k}(u_m,\xi)\ \leq\ c^{\mathfrak u}(k)-\lambda_m\, ,
\label{primaalternativa}
\end{equation}
\noindent or 
\begin{equation}
\widetilde S_k(u_m,\xi) \in \Big(c^{\mathfrak u}(k)-\lambda_m\,,\ c^{\mathfrak u}(k)+(A+1)\lambda_m\Big)\ \ \ \text{and} \ \ \ T_{m}(\xi)\ <\ A+2\, .
\label{secondaalternativa}
\end{equation}
\noindent Consider now for every $r\in[0,1]$ the element $u_m^r\in\mathfrak U$ given by
\begin{equation*}
u_{m}^r(\xi):=\ \Phi^{k,\delta}_r(u_{m}(\xi))\, , \ \ \ \ \forall \ \xi\in B^{l-1}\, .
\end{equation*}
\noindent Equation \eqref{eq:dec} in Lemma \ref{lem:dec} implies that the map 
$$r\ \longmapsto \ \widetilde S_{k}(u^r_{m},\xi)$$
\noindent is decreasing. Combining this fact with (\ref{primaalternativa}) and (\ref{secondaalternativa}) we obtain that 
\begin{equation}
\max_{(s,\zeta)}\ \widetilde S_{k}(u^r_{m},\xi)\ <\ c^{\mathfrak u}(k)+(A+1)\lambda_m\,, \quad \forall\, r\in[0,1]
\label{fofofo}
\end{equation}
\noindent and the following dichotomy holds: either,
\begin{itemize}
 \item[\itshape (a)] $\widetilde S_{k}(u_{m}^1,\xi)\ \leq\  c^{\mathfrak u}(k)-\lambda_m$,
\end{itemize}
\noindent or 
\begin{itemize}
 \item[\itshape (b)] $\widetilde S_{k}(u_m^r,\xi)\ \in\ \Big(c^{\mathfrak u}(k)-\lambda_m,\, c^{\mathfrak u}(k)+(A+1)\lambda_m\Big)$, for every $r\in[0,1]$.
\end{itemize}
\noindent Suppose the second alternative holds. Then (\ref{fofofo}) implies 
\begin{equation}
\widetilde S_{k}(u_{m}^r,\xi)\ >\ c^{\mathfrak u}(k) - \lambda_m\ >\ \max_{\xi\in B^{l-1}}\ \widetilde S_{k}(u_{m}^r,\xi) - (A+2)\, \lambda_m\,.
\end{equation}
\noindent By Lemma \ref{lem:notinwgeneral}, $u^r_{m}(\xi)\notin \mathcal W_\delta$, provided that 
$$(A+2)\, \lambda_m\ \leq\ \frac{\varepsilon_{k,\delta}}{2}\, ,$$
which is true for $m$ big enough. This implies that $r\mapsto u^r_{m}(\xi)$ is a genuine flow line of the untruncated flow $\Phi^k$. Applying \eqref{eq:ac-per} we get that 
$$T^r_{m}(\xi)\ \leq\ \big |T^r_{m}(\xi)-T_{m}(\xi)\big |\, +\, T_{m}(\xi)\ \leq\ \sqrt{r(A+2)\lambda_m}+A+2\ <\ A+3\,,$$
where the last inequality is true for $m$ sufficiently large. After this preparation, we claim that there exists a critical sequence $(x_h,T_h)$ contained in $\{T<A+3\}\setminus \mathcal W_\delta$. 
The claim readily implies the statement of the proposition.

To prove the claim we argue by contradiction and suppose that such a critical sequence does not exist. Then we can find $\varepsilon_*>0$ such that on $\{T<A+3\}\setminus \mathcal W_\delta$
$$\|\eta_k\|\ \geq\ \|X_k\| \ \geq \ \sqrt{\varepsilon_*}\, .$$
\noindent If $\xi\in B^{l-1}$ satisfies the alternative \textit{(b)} above, by \eqref{eq:var} we have
\begin{equation*}
\widetilde S_k(u_{m}^1,\xi)-\widetilde S_k(u_{m},\xi) \ = -\int_0^1\big | \eta_k(X_k)\big |^2\, dr\ \leq \ - \varepsilon_*\,,
\end{equation*}
where we have used the fact that $r\mapsto u^r_{m}(\xi)$ is a flow line contained in $\{T<A+3\}\setminus \mathcal W_\delta$. Therefore, it follows that 
\begin{equation}
\widetilde S_k(u_{m}^1,\xi)\ \leq\ \widetilde S_k(u_{m},\xi)-\varepsilon_*\ \leq\ c^{\mathfrak u}(k)+(A+1)\lambda_m-\varepsilon_*\,.
\end{equation}
\noindent On the other hand, we have by assumption 
$$\widetilde S_k(u_{m}^1,\xi)\ > \ c^{\mathfrak u}(k)-\lambda_m\, .$$
\noindent Hence, the set of $\xi\in B^{l-1}$ satisfying the alternative \textit{(b)} is empty as soon as 
$$(A+1)\lambda_m-\varepsilon_*\ <\ -\lambda_m\, .$$

Therefore, for $m$ big enough, all the $\xi\in B^{l-1}$ satisfy \textit{(a)}. Since $u^1_m$ belongs to $\mathfrak U$, this contradicts the definition of $c^{\mathfrak u}(k)$ and finishes the proof.
\end{proof}


\chapter{Oscillating magnetic fields on $\T^2$}
\label{chapter6}

In this chapter we deal with a particular kind of magnetic flows on $T\T^2$. Recall that, in the setting of Chapter \ref{chapter5}, the magnetic flow of the pair $(g,\sigma)$ on $T\T^2$
is the flow associated with the pair $(E_{kin},\sigma)$, where $E_{kin}$ is the kinetic energy associated with the Riemannian metric $g$ and $\sigma$ is a closed 2-form on $\T^2$.

Hereafter we assume the 2-form $\sigma$ to be non exact and \textit{oscillating}; by this we mean that the density of $\sigma$
with respect to the area form takes both positive and negative values. Up to changing the orientation of $\T^2$, we may clearly assume that $\sigma$ has positive integral over $\T^2$.
Under these assumptions we generalize the main result in \cite{AMMP14} (for $M=\T^2$) to this setting, thus proving the following:

\begin{teo*}
Let $g$ be a Riemannian metric on $\T^2$ and let $\sigma$ be a non-exact oscillating 2-form. Then there exists $\tau_+(g,\sigma)>0$ such that, for almost every $k\in (0,\tau_+(g,\sigma))$, 
$E_{kin}^{-1}(k)$ carries infinitely many geometrically distinct closed magnetic geodesics.
\end{teo*}

This result is the outcome of joint work with Gabriele Benedetti and complements our main theorem in \cite{AB15a}, where the same problem on surfaces with genus larger than one is considered. Extending
this to the case $M=S^2$ represents a challenging open problem. The key ingredient of our discussion is that the closed magnetic geodesics on $E_{kin}^{-1}(k)$ correspond to the zeros of the action 1-form 
\begin{equation}\label{etak}
\eta_k(x,T) :=\ d\A^{E_{kin}}_k(x,T)\ +\int_0^1\sigma_{x(s)}(x'(s),\cdot)\, ds\, ,
\end{equation}
where $\A_k^{E_{kin}}$ is the free-period Lagrangian action functional associated with the kinetic energy $E_{kin}$. The first goal is therefore to prove the existence of infinitely many zeros of $\eta_k$; we will do this by using variational
methods similar to the one used in the proof of the generalized Lusternik-Fet theorem. The second step is then to show that the corresponding closed orbits are not the iterates of finitely many closed orbits. 

\vspace{3mm}

In Section \ref{localminimizersonsurfaces} we recall the existence of local minimizers of the action on surfaces for sufficiently low energy levels. This has been proven by Taimanov in \cite{Tai92b,Tai92a,Tai93} and, 
indipendently, by Contreras, Macarini and Paternain in \cite{CMP04}. We also observe that all the local properties that hold for the free-period Lagrangian action functional continue to hold in this setting; in particular,
local minimizers of the action continue to be local minimizers of the action when iterated and the zeros of $\eta_k$ cease to be of mountain pass nature if iterated sufficiently many times.

In Section \ref{theminimaxclasses} we use the existence of the Taimanov local minimizer and its iterates to construct minimax classes for the action 1-form $\eta_k$. Since the Taimanov's local minimizer might not 
depend continuously on $k$, we will have to consider slightly different minimax classes than the natural ones to achieve the crucial monotonicity property of the associated minimax functions.

In Section \ref{astruwetypemonotonicityargument} we prove the main theorem of this chapter, by suitably extending the \textit{Struwe monotonicity argument} and the main idea of \cite{AMMP14} to this setting.


\section{Local minimizers on $\T^2$} 
\label{localminimizersonsurfaces}

The mechanism we will use to construct zeros of $\eta_k$ is, exactly as in Chapter \ref{chapter5}, to look at the limit points of critical sequences for $\eta_k$. We will do this via a minimax method. The first step in this direction
is to  explain what we mean by \textit{local minimizer of the action}. Observe preliminarly that, being $\pi_2(\T^2)=0$, the form $\sigma$ is weakly-exact and hence the action 1-form $\eta_k$ is exact on the connected component 
$\mathcal M_0$  of $\mathcal M$ given by contractible loops (cf. \cite{Mer10}) with primitive given by 
\begin{equation}
\A_k^{E_{kin}}(x,T) \ + \int_{C(x)} \sigma\, ,
\label{primitiveonm0}
\end{equation} 
where $C(x)$ is any capping disc for $x$. It turns out that $\eta_k$ is exact on $\mathcal M\setminus \mathcal M_0$  if and only if $\sigma$ is exact (see again \cite{Mer10}).
However, $\eta_k$ is exact on small neighborhoods of a non-contractible closed magnetic geodesic; this allows to talk about local minimizers of the action also when $\eta_k$ is not globally exact, as we now show.

Therefore, suppose that $\gamma=(x,T)\in \mathcal M$ is a non-contractible closed magnetic geodesic with energy $k$ or, equivalently, a zero of $\eta_k$. Since $\eta_k$ is invariant under the 
$\T$-action on $\mathcal M$ given by changing the base point of a loop, we have that the whole circle $\T\cdot \gamma = \{ (x(\tau+\cdot,T)) |\tau \in \T\}$ is contained in the set of zeros of $\eta_k$. 

Thus, consider a sufficiently small open neighborhood $\, \mathcal U_\gamma$ of $\T\cdot \gamma$ and observe that for any loop $(y,S) \in \mathcal U_\gamma$ 
we can join $x$ to $y$ with a path $Z(y)$ entirely contained in $\mathcal U_\gamma$. This yields a well-defined primitive of $\eta_k$ on $\mathcal U_\gamma$
\begin{equation}
S_k^\gamma:\mathcal U_\gamma \longrightarrow \R\, , \ \ \ \ S_k^\gamma(y,S) := \ \A_k^{E_{kin}}(y,S) \ + \int_{Z(y)} \sigma\, .
\label{localprimitiveofetak}
\end{equation}

Whenever a closed magnetic geodesic $\gamma$ is given, we fix once for all an open neighborhood $\mathcal U_\gamma\supseteq T\cdot \gamma$ on which $\eta_k$ is exact with primitive $S_k^\gamma$ as in \eqref{localprimitiveofetak}.

\begin{defn}
We say that a non-contractible $\gamma=(x,T)\in \mathcal M$ is a $\mathsf{local\ minimizer}$ of the action if there exist an open neighborhood $\mathcal V_\gamma\subseteq \mathcal U_\gamma$ of $\T\cdot \gamma$ 
such that 
$$S_k^\gamma(y,S) \ \geq \ S_k^\gamma(x,T) \, , \ \ \ \ \forall \ (y,S) \in \mathcal V_\gamma\, .$$
\noindent Moreover, we say that the local minimizer $\gamma=(x,T)$ is $\mathsf{strict}$ if
$$S_k^\gamma(y,S) \ > \ S_k^\gamma(x,T) \, , \ \ \ \ \forall \ (y,S) \in \mathcal V_\gamma \setminus \T\cdot \gamma\, .$$
\label{localminimizers/strict}
\end{defn}

Observe that the functional $S_k^\gamma$ in (\ref{localprimitiveofetak}) can be extended to a well-defined $\N$-equivariant functional on the (disjoint) union of open neighborhoods of $\T\cdot \gamma^n$ 
\begin{equation}
S_k^\gamma: \bigcup_{n\in \N} \ \mathcal U_{\gamma^n} \ \longrightarrow \ \R\, .
\label{localfunctional}
\end{equation}

For all $n\in \N$, the open neighborhood $\mathcal U_{\gamma^n}$ of $\T\cdot \gamma^n$ is chosen to contain the set $(\mathcal U_\gamma)^n$, given by iterating $n$-times each loop in $\mathcal U_\gamma$, 
and to be sufficiently small in such a way that $\eta_k$ is exact on it. For the rest of the chapter, whenever a closed magnetic geodesic $\gamma$ is given, we suppose the $\mathcal U_{\gamma^n}$'s to be fixed.
Recall that $\N$-equivariance for $S_k^\gamma$ means that the following property holds
$$S_k^\gamma(x^m,mT)\ = \ m\, S_k^\gamma(x,T)\, , \ \ \ \ \forall \ (x,T)\in \bigcup_n \ \mathcal U_{\gamma_n}\, , \ \forall \ m\in \N\, .$$

The functional $S_k^\gamma$ actually coincides, up to a constant, with the free-period Lagrangian action functional $\A_ k^{L_\theta}$ associated with  
$$L_\theta(q,v):=\ E_{kin}(q,v)+\theta_q(v)\, ,$$
where $\theta$ is a local primitive of $\sigma$ on a small open neighborhood of the image of $\gamma$. In particular, all the local properties that hold for the free-period Lagrangian action functional continue to hold for 
the functional $S_k^\gamma$; here we recall briefly the ones that are relevant to our discussion (for the details we refer to \cite{AB15a}). We start with the so-called \textit{persistence of local minimizers}, which ensures that a local minimizer of the action continues
to be a local minimizer of the action also when iterated. 

\begin{prop}
If $\gamma$ is a (strict) local minimizer of the action, then for every $n\geq 1$ the $n$-th iterate $\gamma^n$ is also a (strict) local minimizer of the action.
\label{persistenceoflocalminimizers}
\end{prop}

The proof in \cite[Lemma 3.1]{AMP13} goes through without any change. It is worth to point out that this result holds only in dimension 2 and in the orientable case. Counterexamples to this for the free-period Lagrangian action functional associated
with the kinetic energy  in dimension bigger than two or on non-orientable surfaces are described in \cite{Hed32} and \cite[Example 9.7.1]{KH95}, respectively. 

The second property that will be needed later on is the fact that, roughly speaking, critical points of the action cease to be of mountain-pass type 
if iterated sufficiently many times. In other words, zeros of $\eta_k$ are not of mountain-pass nature if iterated sufficiently many times. This fact is proven in \cite{AMMP14} (see also \cite{AB15a}) and, in contrast with the persistence of 
local minimizers, holds in any dimension.

\begin{prop}
Let $\T \cdot \gamma$ be contained in the set of zeros of $\eta_k$ and denote with $S_k^\gamma$ the local primitive of $\eta_k$ as in \eqref{localfunctional}. 
Assume moreover that, for every $n\in \N$, $\T\cdot \gamma^n$ is an isolated circle in the set of zeros of $\eta_k$. Then, for all $n\in\N$ sufficiently large there exists a neighborhood $\mathcal W\subseteq \mathcal U_{\gamma^n}$ of $\T \cdot \gamma^n$ such that  
the following holds: if $\gamma_0, \gamma_1 \in \{S_k^\gamma<S_k^\gamma(\gamma^n)\}\subseteq \mathcal U_{\gamma^n}$ are contained in the same connected component of 
$$\{S_k^\gamma<S_k^\gamma(\gamma^n)\} \ \cup \ \mathcal W\, ,$$
\noindent then they are contained in the same connected component of $\{S_k^\gamma<S_k^\gamma(\gamma^n)\}$.
\label{iterationofmountainpasses}
\end{prop}

Finally, if $\gamma$ is a strict local minimizer of the action with energy $k$, then we might find neighborhoods of $\T\cdot \gamma$ on whose boundary the infimum of $S_k^\gamma$ is strictly larger than $S_k^\gamma(\gamma)$. 
We refer to \cite[Lemma 4.3]{AMP13} for the easy proof.

\begin{prop}
Let $\gamma$ be a strict local minimizer of the action with energy $k$ and let $S_k^\gamma$ be the local primitive of $\eta_k$ as in \eqref{localfunctional}. Then, there exist an open neighborhood $\mathcal V$ of $\T\cdot \gamma$ such that 
the following holds
\begin{equation}
\inf_{\partial \mathcal V} \ S_k^\gamma \ > \ S_k^\gamma(\gamma)\, .
\label{inequalitystrictlocalminimizer}
\end{equation}
\label{strictlocalminimizer}
\end{prop}

\vspace{-6mm}

As already mentioned, we will prove the existence of zeros for $\eta_k$ via a minimax method. The starting point will be the existence of local minimizers of the action for low energies, which we now recall.
Consider the family of Taimanov functionals
$$\mathcal T_k:\mathcal F_+\longrightarrow \R\, ,$$
where $k\in(0,+\infty)$ and $\mathcal F_+$ is the space of positively oriented (possibly with boundary and not necessarily connected) embedded surfaces in $\T^2$ (in
\cite{Tai92b,Tai92a,Tai93} Taimanov considers the so-called \textit{films}):
\begin{equation}
\mathcal T_k(\Pi) := \ \sqrt{2k}\cdot l(\partial \Pi) + \int_\Pi \sigma\,,
\label{Taimanovfunctional}
\end{equation}
\noindent where $l(\partial \Pi)$ denotes the length of the boundary of $\Pi$. Observe that $\emptyset \in \mathcal F_+$ and
\begin{equation}
\mathcal T_k (\emptyset) \ = \ 0 \, , \ \ \ \ \mathcal T_k(\T^2) \ = \ \int_{\T^2} \sigma \ >\ 0\, ;
\label{positivi}
\end{equation}
\noindent moreover the family $\big \{\mathcal T_k\big \}$ is increasing in $k$ and each $\mathcal T_k$ is bounded from below since
$$\mathcal T_k(\Pi)\ \geq \ -\Vert \sigma\Vert_\infty\cdot \operatorname{area}_g(\T^2)\, .$$
\noindent Define now the value
$$\tau_+(M,g,\sigma):= \ \inf\big \{ k\ \big |\ \inf \mathcal T_k\geq 0\ \big\} \ = \ \sup \big \{ k\ \big  |\ \inf \mathcal T_k< 0\ \big\}\, .$$

The functionals $\mathcal T_k$ can be lifted to any finite cover $p':M'\rightarrow \T^2$, thus giving rise to the set of values $\tau_+(M',g,\sigma)$. We can then define the \textit{Taimanov critical value} as
\begin{equation}
\tau_+(g,\sigma):=\ \sup\Big\{\, \tau_+(M',g,\sigma)\ \Big |\ p':M'\rightarrow \T^2 \mbox{ finite cover }\Big\}\, .
\end{equation}

In \cite{CMP04} it was shown that, when $\sigma=d\theta$ is exact, the Taimanov critical value coincides with the Ma\~n\'e critical value of the abelian cover $c_0(L_\theta)$.
When $\sigma$ is not exact, $\tau_+(g,\sigma)$ is positive exactly when $\sigma$ is oscillating (cf. \cite[Lemma 6.2]{AB15a}).

\begin{lemma}
Let $\sigma$ be a non-exact oscillating 2 form on $(\T^2,g)$, then 
$$\tau_+(g,\sigma)\ >\ 0\, .$$
\end{lemma}

We can now state the main theorem about the existence of local minimizers for the action on  $\T^2$. It is worth to point out that this result holds also for any closed connected orientable surface. For the proof we refer to \cite[Lemma 6.4]{AB15a}. 

\begin{teo}
Let $g$ be a Riemannian metric on $\T^2$, $\sigma \in \Omega^2(\T^2)$ be a non-exact oscillating form. Then, for every $k<\tau_+(g,\sigma)$ there exists a closed magnetic geodesic $\alpha_k$ on $\T^2$ with energy $k$ which is a local minimizer of the action. 
\label{Taimanov92}
\end{teo}

This result follows directly from Taimanov's theorem about the existence of global minimizers for $\mathcal T_k$, which states that for every $k<\tau_+(g,\sigma)$ there is a smooth positively oriented embedded surface $\Pi$, which is a global minimizer of 
$\mathcal T_k$ on the space of positively oriented surfaces on a finite cover $M'$ and with $\mathcal T_k(\Pi)<0$. Each boundary component of $\Pi$ is then a closed magnetic geodesic for the magnetic flow lifted to $M'$; notice that the boundary of $\Pi$ 
is non-empty by (\ref{positivi}). The proofs are contained in \cite{Tai92a} (case $M=S^2$) and in \cite{Tai93} (general case). We also refer to \cite{CMP04} for a new proof using methods coming from geometric measure theory.

\vspace{3mm}

The local minimizer of the action $\alpha_k$ need \textit{not} (and in general will not) be contractible. However, if it happens to be contractible, then the existence of a global primitive of $\eta_k$ on $\mathcal M_0$ allows to prove the main theorem of this chapter
by using literally the same argument as in \cite{AB15a}.

Hereafter we will therefore assume that the Taimanov's local minimizer is non-contractible. We explain now briefly how the existence of a local minimizer (with its iterates) yields minimax classes for $\eta_k$. Observe that, for all $n\in \N$ 
$$\pi_1(\mathcal M,\alpha_k^n)\ \cong \ \Z\times \Z\, ,$$
as every connected component of $H^1(\T,\T^2)$ is homeomorphic to 
$$\T^2 \times \Big  \{\gamma \in H^1(\T,\R^2)\ \Big | \int_0^1 \gamma(t)\, dt = 0\Big \}\, .$$
If we write $\alpha_k^n=(x_k^n,n\,T_k)$, one generator of $\pi_1(\mathcal M,\alpha_k^n)$ is given by 
\begin{equation}
\mathfrak a_n:[0,1]\rightarrow \ \mathcal M\, , \ \ \ \ \mathfrak a_n(s)(\cdot) := \big (x_k^n(s+\cdot),nT_k\big )
\label{beta0n}
\end{equation}
and corresponds to the ``change of base-point'' in $\alpha_k^n$. Observe that 
$$\int_{\mathfrak a_n} \sigma \ = \ 0$$
and hence, since $\eta_k$ is non exact on $\mathcal M\setminus \mathcal M_0$, there exists another generator of $\pi_1(\mathcal M,\alpha_k^n)$, say $\mathfrak b_n$, such that
$$\int_{\mathfrak b_n} \sigma \ \neq \ 0\, .$$

One then considers, for all $n\in \N$, the class of loops in $\mathcal M$ based at $\alpha_k^n$ and homotopic to $\mathfrak b_n$ and defines a corresponding minimax function using the integration of $\eta_k$ along paths as in section \ref{theminimaxclass}. 
The problem is that this construction could yield minimax functions which do not depend monotonically on $k$, since the Taimanov's local minimizer could a priori depend on $k$ in a non-continuous fashion.

Therefore, we shall modify the minimax classes to retrieve the desired monotonicity. This will be the goal of the next section.


\section{The minimax classes}
\label{theminimaxclasses}

For any $k\in (0,\tau_+(g,\sigma))$, let $\alpha_k\in \mathcal M$ be a local minimizer of the action with energy $k$. As already pointed out, we may suppose $\alpha_k$  to be 
non-contractible, as otherwise we could prove the main theorem of this chapter exactly as done in \cite{AB15a}.

Fix now $k^*\in (0,\tau_+(g,\sigma))$; if the Taimanov's local minimizer $\alpha_{k^*}$ is not strict, then there exists a sequence of local minimizers of the action approaching $\alpha_{k^*}$, which are all 
closed magnetic geodesics with energy $k^*$. Thus, hereafter we may suppose without loss of generality the local minimizer $\alpha_{k^*}$ to be strict.

For every $n\in \N$ let $\mathcal U_{\alpha_{k^*}^n}$ be an open neighborhood of $\T\cdot \alpha_{k^*}^n$ as in the definition of the $\N$-equivariant local primitive $S_k^{\alpha_{k^*}}$ 
of $\eta_k$ given by \eqref{localfunctional}. By Proposition \ref{strictlocalminimizer} we may find an open neighborhood $\mathcal V\subseteq \mathcal U_{\alpha_{k^*}}$ of $\T\cdot \alpha_{k^*}$ such that 
$$\inf_{\partial \mathcal V}\ S_{k^*}^{\alpha_{k^*}}\ > \ S_{k *}^{\alpha_{k *}}(\alpha_{k^*})\, .$$

Lemma 3.1 in \cite{AMMP14} implies now that there exists an open intervall $I$ containing $k^*$ such that for all $k\in I$ the set 
\begin{equation}
M_k:= \ \overline{\Big \{ \text{local minimizers of} \ S_k^{\alpha_{k^*}} \ \text{in} \ \mathcal V\Big \}}
\label{Mk}
\end{equation}
is non-empty and compact (see also \cite[Lemma 8.1]{AB15a}). With $M_k^n\subseteq \mathcal U_{\alpha_{k^*}^n}$ we denote the set given by iterating $n$-times every element in $M_k$.

\begin{lemma}
There exists an open interval $I=I(k^*)\subseteq (0,\tau_+(g,\sigma))$ containing $k^*$ and which has the following properties:
\begin{enumerate}
\item For every $k\in I$ the set $M_k$ in \eqref{Mk} is a non-empty compact set.
\item For every $k\in I$ there holds
$$\sup_{k'\in I} \ \max_{M_{k'}} \ S_{k'}^{\alpha_{k^*}} \ < \ \inf_{\partial \mathcal V} \ S_k^{\alpha_{k^*}}\, .$$
\end{enumerate}	
\label{rococo}
\end{lemma}

\vspace{-4mm}

For every $n\in \N$ let $\mathfrak a_n$ be the generator of $\pi_1(\mathcal M,\alpha_{k^*}^n)$ given by \eqref{beta0n} and let $\mathfrak b_n$ be another generator, over which the integral of $\sigma$ does not vanish. 
Observe that there is no loop $u:[0,1]\rightarrow \mathcal U_{\alpha_{k^*}^n}$ based at $\alpha_{k^*}^n$ which is homotopic to $\mathfrak b_n$, as otherwise $\eta_k$ would be non exact on $\mathcal U_{\alpha_{k^*}^n}$.
We would like to define the minimax class $\mathcal P_n(k)$ as the set of paths in $\mathcal M$ starting in and ending at $M_k^n$ which are ``homotopic'' to $\mathfrak b_n$. To do this we have to close the considered 
paths within $\mathcal U_{\alpha_{k^*}^n}$, as we now show.

For every element $\gamma\in M_k$ we consider a path $\delta_k^\gamma:[0,1]\rightarrow \mathcal V$ such that $\delta_k^\gamma(0)=\gamma$ and $\delta_k^\gamma(1)=\alpha_{k^*}$; namely 
$\delta_k^\gamma$ is connecting $\gamma$ with the Taimanov's local minimizer $\alpha_{k^*}$ and is entirely contained in $\mathcal V$. For every $n\in \N$ we denote with $\delta_k^{\gamma^n}$ the path
in $\mathcal U_{\alpha_{k^*}^n}$ connecting $\gamma^n$ with $\alpha_{k^*}^n$ which is given by iterating $n$-times every loop of $\delta_k^\gamma$. 

Consider now a path $u:[0,1]\rightarrow \mathcal M$ starting and ending in $M_k^n$; then
$$u(0) \ = \ \gamma_0^n\, , \ \ \ \ u(1)\ = \ \gamma_1^n\, ,$$
for some $\gamma_0,\gamma_1\in M_k$ and hence we have that the juxtaposition
$$\mathcal J(u):= \ \delta_k^{\gamma_0^n}\# u \# \left (\delta_k^{\gamma_1^n}\right )^{-1}$$
is a loop based at $\alpha_{k^*}^n$. We thus set
\begin{equation}
\mathcal P_n(k) := \ \Big \{u:[0,1]\rightarrow \mathcal M \ \Big |\ u(0),u(1)\in M_k^n\, , \ \big [\mathcal J(u)\big ] = [\mathfrak b_n]\Big \}\, .
\label{minimaxclass}
\end{equation}
As already done  in Chapter \ref{chapter5}, for every $u\in \mathcal P_n(k)$ we set 
$$\widetilde S_k(u,s) := \ S_k^{\alpha_{k^*}}(u(0)) \ + \int_0^s u^*\eta_k\, , \ \ \ \ \forall s\in [0,1]$$
and define $c_n:I\rightarrow \R$ by
\begin{equation}
c_n(k) := \inf_{u\in \mathcal P_n(k)} \max_{s\in [0,1]} \ \widetilde S_k(u,s)\, .
\label{minimaxfunctionetaknotexact}
\end{equation}

We show now, using an argument similar to the one in \cite{AMMP14}, that the minimax functions $c_n$ are monotonically increasing in $k$. The first step to prove this is to show that, given $k_0<k_1 \in I$, there always exist 
paths $w$ entirely contained in $\mathcal V$ which connect $M_{k_0}$ to a given element $\gamma\in M_{k_1}$ and such that $S_{k_0}^{\alpha_{k^*}}\circ w \leq S_{k_0}(\gamma)$.

\begin{lemma} Let $k_0<k_1 \in I$. For every $\gamma\in M_{k_1}$ there exists a continuous path $w:[0,1]\rightarrow \mathcal V$ such that $w(0)\in M_{k_0}$, $w(1)=\gamma$ and 
$$S_{k_0}^{\alpha_{k^*}}\circ w \ \leq \ S_{k_0}^{\alpha_{k^*}}(\gamma)\, .$$
\label{circolo}
\end{lemma}

\vspace{-8mm}

\begin{proof}
The element $\gamma$ is a periodic orbit of energy $k_1$ and in particular is not a critical point of $S_{k_0}^{\alpha_{k^*}}$. Set $a:=S_{k_0}^{\alpha_{k^*}}(\gamma)$; being a regular point of the hypersurface $(S_{k_0}^{\alpha_{k^*}})^{-1}(a)$, 
$\gamma$ can be connected to a point $\beta\in \mathcal V\cap \{S_{k_0}^{\alpha_{k^*}}<a\}$ by a continuous path which is entirely contained in the sublevel set $\{S_{k_0}^{\alpha_{k^*}}\leq a\}$. By Lemma \ref{rococo} above
$$a\ = \ S_{k_0}^{\alpha_{k^*}}(\gamma) \ < \ \inf_{\partial \mathcal V} \ S_{k_0}^{\alpha_{k^*}}$$
and hence the connected component of $\{S_{k_0}^{\alpha_{k^*}}\leq a\}$ which contains $\gamma$ is contained in $\mathcal V$. Since $S_{k_0}^{\alpha_{k^*}}$ satisfies the Palais-Smale condition on $\mathcal V$, the above 
fact ensures the existence of a global minimizer $\delta$ of the restriction of $S_{k_0}^{\alpha_{k^*}}$ to the connected component of $\{S_{k_0}^{\alpha_{k^*}}<a\}$ that contains $\beta$. Such a $\delta$ belongs to $M_{k_0}$ 
and can be connected to $\beta$ by a continuous path in $\{S_{k_0}^{\alpha_{k^*}}<a\}$. We conclude that there exists a continuous path 
$w:[0,1]\longrightarrow \{S_{k_0}^{\alpha_{k^*}}\leq a\}$ such that $w(0)=\delta \in M_{k_0}$ and $w(1)=\gamma$.
\end{proof}

\vspace{2mm}

\begin{lemma}
For every $n\in \N$, the minimax function $k\longmapsto c_n(k)$ in \eqref{minimaxfunctionetaknotexact} is monotonically increasing. 
\label{monotonicityetakexact}
\end{lemma}

\begin{proof}
Let $k_0<k_1 \in I$. Consider $u\in \mathcal P_n(k_1)$; by the definition of the minimax class, $u(0)$ is the $n$-th iterate of an element in $M_{k_1}$
and by Lemma \ref{circolo} can be joined, within $\mathcal U_{\alpha_{k^*}^n}$, with the $n$-th iterate of some element of $M_{k_0}$ by a path entirely contained in $\{S_{k_0}^{\alpha_{k^*}}\leq S_{k_0}^{\alpha_{k^*}}(u(0))\}$.
The same holds also for $u(1)$. By concatenation we obtain a path $v:[0,1]\rightarrow \mathcal M$ such that $v(0),v(1)\in M_{k_0}^n$, $v_{[1/3,2/3]} = u\big (3\, (\cdot -1/3)\big )$ and 
\begin{equation}
v\big ([0,1/3]\big ) \ \subseteq \ \big \{ S_{k_0}^{\alpha_{k^*}}\leq S_{k_0}^{\alpha_{k^*}}(u(0))\}\, , \ \ v\big ([2/3,1]\big ) \ \subseteq \ \big \{ S_{k_0}^{\alpha_{k^*}}\leq S_{k_0}^{\alpha_{k^*}}(u(1))\}\, .
\label{grugrugru}
\end{equation}
Thus, $v\in \mathcal P_n(k_0)$ and, by \eqref{grugrugru}, it satisfies for $s\in [0,1/3]$
$$\widetilde S_{k_0}(v,s) \ = \ S_{k_0}^{\alpha_{k^*}}(v(0)) \ + \int_0^s v^*\eta_{k_0}\ = \ S_{k_0}^{\alpha_{k^*}}(v(s))\ \leq \ S_{k_0}^{\alpha_{k^*}}(u(0)) \ \leq \ S_{k_1}^{\alpha_{k^*}}(u(0))\, ;$$
for $s\in [1/3,2/3]$ there holds 
\begin{eqnarray*}
\widetilde S_{k_0}(v,s) &=& S_{k_0}^{\alpha_{k^*}}(v(0)) \ + \int_0^s u^*\eta_{k_0}\ = \ S_{k_0}^{\alpha_{k^*}}(v(1/3))\ + \int_{1/3}^s v^*\eta_{k_0} \ \leq  \\
                                      &\leq& S_{k_1}^{\alpha_{k^*}}(u(0)) \ + \int_0^{3(s-1/3)} \!\!\!\!\!\!\!\!\!\!\! u^*\eta_{k_0} \ \leq \ \widetilde S_{k_1}\big (u,3(s-1/3)\big )\, .
\end{eqnarray*}
Finally, for $s\in [2/3,1]$ we have
\begin{eqnarray*}
\widetilde S_{k_0}(v,s) &=& S_{k_0}^{\alpha_{k^*}}(v(0)) \ + \int_0^s u^*\eta_{k_0}\ \leq \ S_{k_1}^{\alpha_{k^*}}(u(0)) \ + \int_0^1 u^*\eta_{k_0} \ + \int_{2/3}^{3(s-2/3)} \!\!\!\!\!\!\!\!\!\!\! v^*\eta_{k_0}\ \leq \\ 
                                      &\leq& \widetilde S_{k_1}(u,1) \ + \ S_{k_0}^{\alpha_{k^*}}(v(3(s-2/3))) \ - \ S_{k_0}^{\alpha_{k^*}}(v(2/3)) \ \leq \ \widetilde S_{k_1}(u,1)\, ,
\end{eqnarray*}
as it follows from \eqref{grugrugru} (observe that $v(2/3)=u(1))$. Summarizing, we have that 
$$\max_{s\in [0,1]} \ \widetilde S_{k_0}(v,s) \ \leq \ \max_{s\in [0,1]} \ \widetilde S_{k_1}(u,s)$$ 
and hence taking the infimum over all $v\in \mathcal P_n(k_0)$ we get 
$$c_n(k_0) \ \leq \  \max_{s\in [0,1]} \ \widetilde S_{k_1}(u,s)\, .$$
By taking the infimum over all $u\in \mathcal P_n(k_1)$ we conclude that $c_n(k_0)\leq c_n(k_1)$.
\end{proof}


\section{A Struwe-type monotonicity argument}
\label{astruwetypemonotonicityargument}

In this section, building on the results of the previous ones, we prove the main theorem of this chapter. Namely, we show that the following holds

\begin{teo}
Let $g$ be a Riemannian metric on $\T^2$ and let $\sigma$ be a non-exact oscillating 2-form on $\T^2$. Then, for almost every $k\in (0,\tau_+(g,\sigma))$ 
the energy level $E^{-1}(k)$ carries infinitely many geometrically distinct closed magnetic geodesics.
\label{teo6}
\end{teo}

The theorem is a trivial consequence of Proposition \ref{Struwesurface} and Theorem \ref{teoremafinale} below.  In the previous section we introduced a sequence of minimax functions $c_n:I\rightarrow \R$, where $I$ is an 
open interval around a fixed $k^*\in (0,\tau_+(g,\sigma))$, and showed a crucial monotonicity property for the functions $c_n$. This will allow us to prove the existence of zeros for the action 1-form $\eta_k$ for every 
$k\in I$ at which all the minimax functions are differentiable. It is a well-known fact that monotone functions are almost everywhere differentiable; in particular, the set of points in $I$ at which every function $c_n$ is 
differentiable is a full measure set in $I$. 

Observe preliminarly that, if the set $M_k$ in \eqref{Mk} consists of infinitely many circles, then we immediately get the existence of infinitely many geometrically distinct closed magnetic geodesics with energy $k$. 
Therefore, we may also assume that $k\in I$ is such that $\# M_k<\infty$, meaning that the set $M_k$ consists of only finitely many circles; clearly, all the elements in $M_k$ (hence, all their iterates) are strict $k$-local minimizers of the action. 
Thus, we define the set 
\begin{equation}
J\, =\, J(k^*) := \, \Big \{k\in I\ \Big | \ \# M_k <\infty\, , \ c_n \ \text{differentiable at} \ k, \, \forall \, n\in \N\Big \}
\label{definizioneJ}
\end{equation}
and prove that, for every $k\in J$, the minimax functions $c_n$ yield zeros of $\eta_k$ by showing the existence of critical sequences for $\eta_k$ with periods bounded and bounded away from zero. The assertion follows then 
from Proposition \ref{prp:ps2}.

To exclude that the zeros of $\eta_k$ detected by the $c_n$'s are contained in $M_k$ we use an argument analogous to the one used in Proposition \ref{Struwe} to exclude that the periods of the critical sequence tend to zero.
More precisely, since $\# M_k<\infty$, all the elements in $M_k$ are strict local minimizers of the action and so are also their iterates. Let now $n\in \N$ be fixed and consider the set $M_k^n$; from Proposition 
\ref{strictlocalminimizer} it follows that for every $\gamma\in M_k^n$ there exists a small open neighborhood $\mathcal V_n(\gamma)\subseteq \mathcal U_{\alpha_{k^*}^n}$ of $\T\cdot \gamma$ such that
$$S_k^{\alpha_{k^*}}(\gamma) \ = \ \inf_{\partial \mathcal V_n(\gamma)} \ S_k \ - \ \epsilon_n(\gamma,k)\, ,$$
for some positive $\epsilon_n(\gamma,k)>0$. By the very definition of the minimax class, every element $u\in \mathcal P_n(k)$ with starting point in $\T\cdot \gamma$ has to intersect $\partial \mathcal V_n(\gamma)$.

Now arguing as in the proof of Lemma \ref{lem:notinwgeneral} one shows that there exists an open neighborhood $\mathcal W_n(\gamma)\subseteq \mathcal V_n(\gamma)$ of $\T\cdot \gamma$ such that, if $s_*\in [0,1]$ satisfies
$$\widetilde S_k(u,s_*)\ \geq \ \max_{s\in [0,1]} \ \widetilde S_k(u,s) \ - \ \frac{\epsilon_n(\gamma,k)}{2}\, ,$$
then necessarily $u(s)\notin \mathcal W_n(\gamma)$. Since by assumption $M_k^n$ consists of finitely many circles, repeating the same procedure for every one of them we end up with a small open neighborhood $\mathcal W_n$ 
of $M_k^n$ such that, if $s_*\in [0,1]$ almost realizes the maximum of $s\longmapsto \widetilde S_k(u,s)$, then necessarily $u(s_*)\notin \mathcal W_n$.

Finally, we show in Proposition \ref{Struwesurface} that the periods of the critical sequences are bounded from above by suitably generalizing the \textit{Struwe monotonicity argument} \cite{Str90} to this setting. The ideas behind the 
proof of the proposition are exactly the same as for Proposition \ref{Struwe} and are based on the well-known fact that the time-1 flow of $-\sharp \eta_k$ maps subsets of $\mathcal M$ with bounded periods into subsets with bounded periods (cf.
\cite[Lemma 5.7]{Mer10} or Lemma \ref{lem:ac-distper}).

Notice that critical sequences for $\eta_k$ have clearly periods bounded away from zero, since we are working only with non-contractible loops (cf. Lemma \ref{lem:ps}). 

\begin{prop}
For every $k\in J$ and for every $n\in \N$ there exists a zero for $\eta_k$ which is not contained in $M_k^n$.
\label{Struwesurface}
\end{prop}

\begin{proof}
Fix $k\in J$ and $n\in \N$. In virtue of Proposition \ref{prp:ps2} it suffices to show the existence of a critical sequence for $\eta_k$ with periods bounded and bounded away from zero. Any limit point of such a critical sequence will be then  a zero 
of $\eta_k$; this concludes the proof since by the observation above such a zero does not lie in $M_k^n$. 

Thus, choose a strictly decreasing sequence $k_h\downarrow k$ and set $\lambda_h:=k_h-k$. Since $k\in J$, without loss of generality we may suppose that for all $h\in \N$ there holds
\begin{equation}
 c_n(k_h) - c_n(k) \ \leq \ M\, \lambda_h\, .
\label{lipschitzcontinuityk}
\end{equation}
\noindent For every $h\in \N$ choose $u_h=(x_h,T_h)\in \mathcal P_n(k_h)$ such that 
$$\max_{s\in [0,1]} \ \widetilde S_{k_h}(u_h,s) \ < \ c_n(k_h) \ + \ \lambda_h\, .$$
\noindent Suppose that for a certain $s\in[0,1]$ there holds
\begin{equation}
\widetilde S_{k}(u_h,s)\ >\ c_n(k)\ -\ \lambda_h\, ;
\label{maggiorecu-lambdam}
\end{equation}
\noindent then it follows 
$$T_h(s) \ = \ \frac{\widetilde S_{k_h}(u_h,s) - \widetilde S_k(u_h,s)}{\lambda_h} \ \leq \ \frac{c_n(k_h)+\lambda_h-c_n(k)+\lambda_h}{\lambda_h} \ \leq \ M + 2$$
\noindent and at the same time, using (\ref{lipschitzcontinuityk}),
\begin{eqnarray*}
\widetilde S_k(u_h,s) &\leq& \widetilde S_{k_h}(u_h,s) \ < \ c_n (k) + (M+1)\, \lambda_h\, .
\end{eqnarray*}
\noindent Summing up, for every $h\in \N$ and every $s \in[0,1]$ either
\begin{equation*}
\widetilde S_{k}(u_h,s)\ \leq\ c_n(k)-\lambda_h\, ,
\end{equation*}
\noindent or 
\begin{equation*}
\widetilde S_k(u_h,s) \in \Big(c_n(k)-\lambda_h\,,\ c_n(k)+(M+1)\lambda_h\Big)\ \ \ \text{and} \ \ \ T_h(s)\ <\ M+2\, .
\end{equation*}

By the very definition of the minimax class $\mathcal P_n(k_h)$ we have that $u_h(0),u_h(1)\in M_{k_h}^n$ for every $h\in \N$. Lemma \ref{circolo} implies now that $u_h(0)$ and $u_h(1)$ can be joined to elements in $M_k^n$ 
with paths entirely contained in $\mathcal U_{\alpha_{k^*}^n}$ and without increasing the local action in \eqref{localfunctional}, thus giving rise to paths $v_h\in \mathcal P_n(k)$ that also satisfy the above dichotomy, meaning that either
\begin{equation}
\widetilde S_{k}(v_h,s)\ \leq\ c_n(k)-\lambda_h\, ,
\label{primaalternativak}
\end{equation}
\noindent or 
\begin{equation}
\widetilde S_k(v_h,s) \in \Big(c_n(k)-\lambda_h\,,\ c_n(k)+(M+1)\lambda_h\Big)\ \ \ \text{and} \ \ \ T_h(s)\ <\ M+2\, .
\label{secondaalternativak}
\end{equation}

Up to taking a subsequence if necessary, we may also suppose that all the paths $v_h$ start from the same circle in $M_k^n$ and end in the same circle in $M_k^n$, meaning that there exist $\alpha_0,\alpha_1\in M_k^n$ such that 
$$v_h(i)\ \in \ \T\cdot \alpha_i \, ,\ \ \ \ \text{for} \  \ i =0,1\, .$$
\noindent Consider now for every $r\in[0,1]$ the element $u_h^r\in\mathcal P_n(k)$ given by
\begin{equation*}
v_h^r(s):=\ \Phi^k_r(v_h(s))\, , \ \ \ \ \forall \ s\in[0,1]\, ,
\end{equation*}
where $\Phi^k$ denotes the local flow of the vector field $X_k$ conformally equivalent to $-\sharp \eta_k$ (cf. Section \ref{ageneralizedpseudogradient}). Equation \eqref{eq:dec} in Lemma \ref{lem:dec} implies now that the map
$$r\ \longmapsto \ \widetilde S_{k}(v^r_h,s)$$
\noindent is decreasing. Combining this fact with (\ref{primaalternativak}) and (\ref{secondaalternativak}) we obtain that 
\begin{equation}
\max_{s\in [0,1]}\ \widetilde S_{k}(v^r_h,s)\ <\ c_n(k)+(M+1)\lambda_h\,, \quad \forall\, r\in[0,1]
\label{rururu}
\end{equation}
\noindent and the following dichotomy holds: either,
\begin{itemize}
 \item[\itshape (a)] $\widetilde S_{k}(v_h^1,s)\ \leq\  c_n(k)-\lambda_h$,
\end{itemize}
\noindent or 
\begin{itemize}
 \item[\itshape (b)] $\widetilde S_{k}(v_h^r,s)\ \in\ \Big(c_n(k)-\lambda_h,\, c_n(k)+(M+1)\lambda_h\Big)$, for every $r\in[0,1]$.
\end{itemize}
\noindent Suppose the second alternative holds. Then from (\ref{rururu}) we get that
\begin{equation}
\widetilde S_{k}(v_h^r,s)\ >\ c_n(k) - \lambda_h\ >\ \max_{s\in [0,1]}\ \widetilde S_{k}(v_h^r,s) - (M+1)\, \lambda_h\,,
\end{equation}
which implies that $v^r_h(s)\notin \mathcal W_n$ for every $h$ large enough, where $\mathcal W_n$ is a suitable neighborhood of $M_n^k$. Furthermore, applying \eqref{eq:ac-per} we get that 
$$T^r_h(s)\ \leq\ \big |T^r_h(s)-T_h(s)\big |\, +\, T_h(s)\ \leq\ \sqrt{r(M+2)\lambda_h}+M+2\ <\ M+3\,,$$
where the last inequality is true for $h$ sufficiently large. After this preparation, we claim that there exists a critical sequence $(x_h,T_h)$ contained in $\{T<M+3\}\setminus \mathcal W_n$. 

To prove the claim we argue by contradiction and suppose that such a critical sequence does not exist. Then we can find $\varepsilon_*>0$ such that on $\{T<M+3\}\setminus \mathcal W_n$
$$\|\eta_k\|\ \geq\ \|X_k\| \ \geq \ \sqrt{\varepsilon_*}\, .$$
\noindent If $s\in [0,1]$ satisfies the alternative \textit{(b)} above, by \eqref{eq:var} we have
\begin{equation*}
\widetilde S_k(v_h^1,s)- \widetilde S_k(v_h,s) \ = -\int_0^1\big | \eta_k(X_k)\big |^2\, dr\ \leq \ - \varepsilon_*\,,
\end{equation*}
\noindent and hence
\begin{equation}
\widetilde S_k(v_h^1,s)\ \leq\ \widetilde S_k(v_h,s)-\varepsilon_*\ \leq\ c_n(k)+(M+1)\lambda_h-\varepsilon_*\,.
\end{equation}
\noindent On the other hand, we have by assumption 
$$\widetilde S_k(v_h^1,s)\ > \ c_n(k)-\lambda_h\, .$$
\noindent Hence, the set of $s\in [0,1]$ satisfying the alternative \textit{(b)} is empty as soon as 
$$(M+1)\lambda_h-\varepsilon_*\ <\ -\lambda_h\, .$$

Therefore, for $h$ big enough, all the $s\in[0,1]$ satisfy the alternative \textit{(a)}. Since $v^1_h$ belongs to $\mathcal P_n(k)$, this contradicts the definition of $c_n(k)$ and finishes the proof.
\end{proof}

\vspace{5mm} 

With the next result we conclude the proof of Theorem \ref{teo6}. The argument is identical to the one in \cite{AMMP14} (see also \cite{AB15a} for the non-exact oscillating case on high genus surfaces) and it is
based on Proposition \ref{iterationofmountainpasses}, which states that zeros of $\eta_k$ cease to be of mountain-pass type if iterated sufficiently many times.

\begin{teo}
In the hypotheses of Theorem \ref{teo6}, let $k^*\in  (0,\tau_+(g,\sigma))$ be such that the Taimanov's local minimizer $\alpha_{k^*}$ is strict. 
Moreover, for every $n\in \N$, let $c_n:I\rightarrow \R$ be the minimax function as defined in  \eqref{minimaxfunctionetaknotexact}.
Then, for almost every $k\in I$ the energy level $E^{-1}(k)$ carries infinitely closed magnetic geodesics.
\label{teoremafinale}
\end{teo}

\begin{proof}
If $k\in I$ is such that $\# M_k =\infty$ then there is nothing to prove. Therefore, it is enough to show that for all $k\in J$ the energy level $E^{-1}(k)$ carries infinitely many geometrically distinct closed magnetic geodesics,
with $J\subseteq I$ as in (\ref{definizioneJ}).

Pick $k\in J$ and assume by contradiction that $E^{-1}(k)$ has only finitely many closed magnetic geodesics. Then, the zero set of $\eta_k$ consists of finitely many circles 
$$\T \cdot \gamma_1\, , \ \T\cdot \gamma_2\, , \ldots \, , \ \T\cdot \gamma_l\, ,$$
\noindent together with their iterates $\T\cdot \gamma_j^n$. By Proposition \ref{iterationofmountainpasses} we can find a natural number $n_*$ such that 
$\gamma_j^n$ is not of mountain-pass type for any $1\leq j\leq l$ and for any $n\geq n_*$. 

Since $\alpha_{k_*}$ is non-contractible there exists $n\in\N$ such that $\alpha_{k_*}^n$ is not freely homotopic to $\gamma_j^m$, for every $1\leq j\leq l$ and $1\leq m\leq n_*-1$.  
Then Proposition \ref{Struwesurface} implies that, for every $n\geq n_*$, there exists a zero $\beta_n$ for $\eta_k$ in the connected component $\mathcal M_{\alpha_{k^*}^n}$ 
of $\mathcal M$ containing $\alpha_{k^*}^{n}$ which lies outside $M_k^{n}$. Thanks to our finiteness assumption 
$$\beta_n \ = \ \gamma_{j_n}^{m_n}$$
for some $j_n\in \{1,...,l\}$ and some integer $m_n\geq n_*$. This yields an obvious contradiction, since $\beta_n$ is of mountain-pass type while $\gamma_{j_n}^{m_n}$  is not.
\end{proof}


\appendix

\chapter{Reminders}


\section{Homotopy theory.}
\label{homotopytheory}

In this section, following \cite{Hat02}, we recall the basic notions on homotopy theory needed throughout the thesis. 
For any $n\in \N$ let $I^n=[0,1]^n$ be the $n$-dimensional cube (i.e. the product of $n$ copies of the interval); its boundary $\partial I^n$ is the subspace consisting of points 
with at least one coordinate equal to 0 or 1. For a pointed space $(X,x_0)$ we define $\pi_n(X,x_0)$ as the set of homotopy classes of maps 
$$f:(I^n,\partial I^n)\longrightarrow (X,x_0)\, ,$$
where homotopies are required to satisfy $ f_t(\partial I^n)=x_0$ for all $t\in [0,1]$. This definition extends to the case $n=0$ by taking $I^0$ to be a point and $\partial I^0$ to be the empty-set, 
so that $\pi_0(X,x_0)$ is just the set of path-components of $X$. When $n\geq 2$, a sum operation in $\pi_n(X,x_0)$ (generalizing the composition on $\pi_1$) is defined by 
$$(f+g)(s_1,s_2,...,s_n):= \left \{ \begin{array}{l} f(2s_1,s_2,...,s_n)\ \ \ \ \ \ \ \ \ s_1\in [0,\frac{1}{2}]\, ;\\ \\ g(2s_1-1,s_2,...,s_n)\ \ \ \ s_1\in [\frac{1}{2},1]\, ;\end{array}\right.$$

\begin{lemma}
$\pi_n(X,x_0)$ is a group for any $n\in \N$, abelian for $n\geq 2$.
\end{lemma}

\noindent We may also view $\pi_n(X,x_0)$ as homotopy classes of maps 
$$f:(S^n,s_0)\longrightarrow (X,x_0)\, ,$$
where homotopies are through maps of the same form; in this interpretation of $\pi_n(X,x_0)$ the sum $f+g$ is the composition 
$$S^n \ \stackrel{c}{\longrightarrow}\ S^n \vee S^n \ \stackrel{f\vee g}{\longrightarrow} \ X$$
where $c$ collapses the equator $S^{n-1}$ in $S^n$ to a point and we choose the basepoint $s_0$ to lie in this $S^{n-1}$. If $X$ is path-connected, then different choices of the base point $x_0$ 
always produce isomorphic groups $\pi_n(X,x_0)$, just as for $\pi_1(X,x_0)$, so one is justified to write $\pi_n(X)$ instead of $\pi_n(X,x_0)$. Given a path $\gamma:[0,1]\rightarrow X$ from $x_0$ 
to another point $x_1$, we may associate to each map $f:(I^n,\partial I^n)\rightarrow (X,x_0)$ a new map 
$$\gamma f :(I^n,\partial I^n)\longrightarrow (X,x_1)$$ 
by shrinking the domain of $f$ to a smaller concentric cube in $I^n$, then inserting the path $\gamma$ on each radial segment in the shell between this smaller cube and $\partial I^n$ (this definition
needs some notational adjustments in the case $n=1$). A homotopy of $\gamma$ or $f$ through maps fixing $\partial I$ or $\partial I^n$, respectively, yields an homotopy of $\gamma f$ through maps 
$(I^n,\partial I^n)\rightarrow (X,x_0)$. Here are three other basic properties: 
\begin{enumerate}
\item $\gamma (f+g) \sim \gamma f + \gamma g$.
\item $(\gamma \eta) f \sim \gamma (\eta f)$.
\item $1f\sim f$.
\end{enumerate} 
\noindent In virtue of the properties above the change-of-basepoint transformation 
$$\beta_\gamma : \pi_n(X,x_1)\longrightarrow \pi_n(X,x_0)\, ,\ \ \ \beta([f]):= [\gamma f]$$
is an isomorphism with inverse $\beta_{\gamma^{-1}}$, where $\gamma^{-1}$ denotes the inverse path of $\gamma$. If $\gamma$ is a loop at the base point $x_0$, then the map 
$$\pi_1(X,x_0)\longrightarrow \text{Aut}(\pi_n(X,x_0))\, ,\ \ \ \ [\gamma]\longmapsto \beta_\gamma$$
is a group homomorphism, called the \textit{action} of $\pi_1(X,x_0)$ on $\pi_n(X,x_0)$; in the case $n=1$ this is the action of $\pi_1$ onto itself by inner automorphisms.

\begin{prop}
The following hold:

\begin{enumerate}
\item A covering projection $p:(\tilde{X},\tilde{x}_0)\rightarrow (X,x_0)$ induces isomorphisms 
$$p_*:\pi_n(\tilde{X},\tilde{x}_0)\rightarrow \pi_n(X,x_0)$$
for any $n\geq 2$.
\item Let $\{X_\alpha\}_{\alpha\, \in\, \mathcal A}$ be any collection of path-connected spaces, then 
$$\pi_n \left ( \prod_{\alpha\, \in \, \mathcal A} \ X_\alpha \right )\ \cong \ \prod_{\alpha\, \in \, \mathcal A} \pi_n (X_\alpha )$$
for all $n\in \N$.
\end{enumerate}
\end{prop}

Statement 1 in the above proposition implies obviously that $\pi_n(X,x_0)=0$ for $n\geq 2$, whenever $X$ has a contractible universal cover; this applies for example to $S^1$ and more generally to the 
$n$-dimensional torus $\T^n$. Very useful generalizations of the homotopy groups $\pi_n(X,x_0)$ are the \textit{relative homotopy groups} $\pi_n(X,A,x_0)$ for a pair $(X,A)$ with basepoint $x_0\in A$. 
To define these, regard $I^{n-1}$ as the face of $I^n$ with last coordinate equal to zero, say $s_n=0$, and let $J^{n-1}$ be defined as 
\begin{equation}
J^{n-1}:=\ \overline{\partial I^n \setminus I^{n-1}}\, ,
\label{jn-1}
\end{equation}
that is the closure of the union of the remaining faces of $I^n$. Then $\pi_n(X,A,x_0)$ is defined to be the set of homotopy classes of maps 
$$(I^n,\partial I^n,J^{n-1}) \ \longrightarrow \ (X,A,x_0)\, ,$$
with homotopies through maps of the same form. It is worth noticing that this definition does not extend naturally to the case $n=0$; we leave $\pi_0(X,A,x_0)$ undefined, since we are not interesting in it. Observe that 
$$\pi_n(X,x_0,x_0)\ =\ \pi_n(X,x_0)$$
so that absolute homotopy groups are a special case of relative homotopy groups. A sum operation is defined in $\pi_n(X,A,x_0)$ by the same formulas as for $\pi_n(X,x_0)$, except that the coordinate $s_n$  now plays a special
role and is no longer available for the sum operation. Thus $\pi_n(X,A,x_0)$ is a group for any $n\geq 2$, abelian for $n\geq 3$; for $n=1$ we have $I^1= [0,1],\, I^0= \{0\},\, J^0 = \{1\}$, so that $\pi_1(X,A,x_0)$ is the set of homotopy classes of maps 
$$([0,1], \{0\},\{1\}) \ \longrightarrow (X,A,x_0)\, ,$$
i.e. of  paths in $X$ from a varying point in $A$ to the fixed basepoint $x_0\in A$, and hence in general this is not a group in any natural way.  Just as elements of $\pi_n(X,x_0)$ can be regarded as homotopy classes of maps 
$(S^n,s_0)\rightarrow (X,x_0)$, there is an alternative definition of $\pi_n(X,A,x_0)$ as the set of homotopy classes of maps 
$$(D^n,S^{n-1},s_0)\ \longrightarrow \ (X,A,x_0)\, ,$$
since collapsing $J^{n-1}$ to a point converts $(I^n,\partial I^n,J^{n-1})$ into $(D^n,S^{n-1},s_0)$. From this viewpoint, addition is done via the map $c:D^n\rightarrow D^n\vee D^n$ collapsing $D^{n-1}\subseteq D^n$ to a point.
A useful and conceptually enlightening reformulation of what it means for element of $\pi_n(X,A,x_0)$ to be trivial is the following 

\vspace{3mm}

\begin{lemma}[Compression criterion]
A map $f:(D^n,S^{n-1},s_0)\rightarrow (X,A,x_0)$ represents zero in $\pi_n(X,A,x_0)$ if and only if $f$ is homotopic relative to $S^{n-1}$ to a map $g$ with image contained in $A$.
\end{lemma}

\begin{proof}
If $f$ is homotopic to such a map $g$, then $[f]=[g]$ in $\pi_n(X,A,x_0)$, and $[g]=0$ via the homotopy obtained by composing $g$ with a deformation retraction of $D^n$ onto $s_0$. Conversely, 
if $[f]=0$ via a homotopy $F:D^n\times I \longrightarrow X$, by restricting $F$ to a family of $n$-disks in $D^n\times I$ starting with $D^n\times \{0\}$ and ending with 
$$\big (D^n\times \{1\} \big ) \ \cup \ \big ( S^{n-1}\times I \big )\, ,$$
all the disks in the family having the same boundary, then we get an homotopy from $f$ to a map with image contained in $A$ and stationary on $S^{n-1}$.
\end{proof}

\vspace{5mm}

Denote by $i:(A,x_0)\hookrightarrow (X,x_0)$, $j:(X,x_0,x_0)\hookrightarrow (X,A,x_0)$ the canonical inclusions and by $i_*,j_*$ the induced maps on $\pi_n$; then the following holds

\begin{teo}[Exact sequence in relative homotopy] We have an exact sequence on relative homotopy groups given by
$$ ... \rightarrow \pi_n(A,x_0) \stackrel{i_*}{\longrightarrow} \pi_n(X,x_0) \stackrel{j_*}{\longrightarrow} \pi_n(X,A,x_0) \stackrel{\partial}{\longrightarrow} \pi_{n-1}(A,x_0) \rightarrow ... \rightarrow \pi_0(X,x_0)\, ,$$
where $\partial$ is defined by restricting maps $(I^n,\partial I^n,J^{n-1})\rightarrow (X,A,x_0)$ to $I^{n-1}$, or equivalently by restricting maps $(D^n,S^{n-1},s_0)\rightarrow (X,A,x_0)$ to $S^{n-1}$.
\end{teo}

\begin{oss}
Near the end of the sequence, where group structures are not defined, exactness still makes sense: the image of one map is the kernel of the next, those elements mapping to the homotopy class of the constant map.
\end{oss}

A map $p:E\rightarrow B$ is said to have the \textit{homotopy lifting property} with respect to a space $X$ if given an homotopy $g_t:X\rightarrow B$ and a map $\tilde{g}_0:X\rightarrow E$ lifting $g_0$ (that is $p\circ \tilde{g}_0=g_0$), then 
there exists a homotopy 
$$\tilde{g}_t:X\longrightarrow E$$
lifting $g_t$. From a formal viewpoint, this can be regarded as a special case of the lift extension property for a pair $(Z,A)$, which asserts that every map $Z\rightarrow B$ 
has a lift $Z\rightarrow E$ extending a given lift defined on the subspace $A\subseteq Z$; the case 
$$(Z,A)\ =\ (X\times I,X\times \{0\})$$
is exactly the homotopy lifting property. A \textit{fibration} is a map $p:E\rightarrow B$ having the homotopy lifting property with respect to all spaces $X$; for example the projection 
$$p:B\times F \longrightarrow B$$
is a fibration, since we can choose lifts of the form $\tilde{g}_t(x)=(g_t(x),h(x))$, where $\tilde{g}_0=(g_0(x),h(x))$, but is not a covering map unless $F$ is a discrete space. 

\begin{teo}
Suppose that $p:E\rightarrow B$ has the homotopy lifting property with respect to any disk $D^k$, $k\geq 0$. Choose $b_0\in B$ and $e_0\in F:=p^{-1}(b_0)$; then 
$$p_*:\pi_n (E,F,e_0) \longrightarrow \pi_n(B,b_0)$$
is an isomorphism for any $n\geq 1$. If $B$ is in addtion path-connected, then there is an exact sequence 
$$... \rightarrow \pi_n(F,e_0)\rightarrow \pi_n(E,e_0)\stackrel{p_*}{\longrightarrow} \pi_n(B,b_0) \rightarrow \pi_{n-1}(F,e_0) \rightarrow ... \rightarrow \pi_0(E,e_0)\rightarrow 0\, .$$
\end{teo}

The proof will use a relative form of the homotopy lifting property above; the map $p:E\rightarrow B$ is said to have the \textit{homotopy lifting property for a pair} $(X,A)$ if each homotopy 
$f_t:X\rightarrow B$ lifts to a homotopy $\tilde{g}_t:X\rightarrow E$ starting with a given $\tilde{g}_0$ and extending a given lift $\tilde{g}_t:A\rightarrow E$. In other words, the 
homotopy lifting property for a pair $(X,A)$ is the lift extension property for 
$$(X\times I, X\times \{0\} \ \cup \ A\times I )\, .$$

The homotopy lifting property for $D^k$ is equivalent to the homotopy lifting property for the pair $(D^k,\partial D^k)$, since the pairs 
$$(D^k\times I,D^k\times \{0\})\, , \ \ \ \ (D^k\times I,D^k\times \{0\}\ \cup \ \partial D^k \times I)$$
are homeomorphic. This implies that the homotopy lifting property for disks is equivalent to the homotopy lifting property for all CW-pairs $(X,A)$.

\vspace{6mm}

\begin{proof}
First we show that $p_*$ is onto; represent an element of $\pi_n(B,b_0)$ by a map 
$$f:(I^n,\partial I^n)\longrightarrow (B,b_0)\, .$$

The costant map to $e_0$ provides a lift of $f$ to $E$ over the subspace $J^{n-1}\subseteq I^n$, so the relative homotopy lifting property for $(I^{n-1},\partial I^{n-1})$ 
extends this to a lift 
$$\tilde{f}:I^n \rightarrow E$$ 
which satisfies $\tilde{f}(\partial I^n)\subseteq F$, since $f(\partial I^n)=b_0$. Then $\tilde{f}$ represents an element of $\pi_n(E,F,e_0)$ with $p_*([\tilde{f}])=[f]$ since $p\tilde{f}=f$.
Given now two maps 
$$\tilde{f}_0,\tilde{f}_1 : (I^n,\partial I^n,J^{n-1}) \longrightarrow (E,F,e_0)$$
\noindent such that $p_*([\tilde{f}_0])=p_*([\tilde{f}_1])$, let 
$$G:(I^n\times I,\partial I^n\times I) \longrightarrow (B,b_0)$$
be an homotopy from $p\circ \tilde{f}_0$ to $p\circ \tilde{f}_1$. We have a partial lift $\tilde{G}$ given by $\tilde{f}_0$ on $I^n\times \{0\}$, $\tilde{f}_1$ on $I^n\times \{1\}$ and the constant map $e_0$ on $J^{n-1}\times I$; 
after permuting the last two coordinates of $I^n\times I$, the relative homotopy lifting property gives an extension of this partial lift to a full lift $\tilde{G}:I^n\times I \rightarrow E$ which gives a homotopy 
$$\tilde{f}_t :(I^n,\partial I^n,J^{n-1})\longrightarrow (E,F,e_0)$$
from $\tilde{f}_0$ to $\tilde{f}_1$. For the last statement of the theorem we plug $\pi_n(B,b_0)$ in for $\pi_n(E,F,e_0)$ 
in the relative homotopy exact sequence for the pair $(E,F)$; the map $\pi_n(E,e_0)\rightarrow \pi_n(E,F,e_0)$ in the exact sequence then becomes the composition 
$$\pi_n(E,e_0)\longrightarrow \pi_n(E,F,e_0) \stackrel{p_*}{\longrightarrow} \pi_n(B,b_0)\, ,$$
which is just $p_*:\pi_n(E,e_0)\rightarrow \pi_n(B,b_0)$. The surjectivity of $\pi_0(F,e_0)\rightarrow \pi_0(E,e_0)$ comes from the hypothesis that $B$ is path-connected since 
a path in $E$ from an arbitrary point $e\in E$ to $F$ can be obtained by lifting a path in $B$ from $p(e)$ to $b_0$.
\end{proof}


\section{Conormal bundles.}
\label{conormalbundles}

In this section we recall the definition and the basic properties of the conormal bundle to a submanifold $Q$ of a given manifold $M$. For the details we refer to \cite{AS09}.

\begin{defn}
Let $Q\subseteq M$ be a submanifold. The $\mathsf{conormal\ bundle}$ of $Q$ is 
\begin{eqnarray*}
N^*Q &:=& \Big \{x\in T^*M \ \Big |\ \pi^*(x)\in Q,\ x[\zeta] = 0 \ \ \forall \zeta \in T_{\pi^*(x)}Q\Big \}\ =\\
&=& \Big \{ (q,p)\in T^*M\ \Big |\ q\in Q,\ T_qQ\subseteq \ker p\Big \}\, .
\end{eqnarray*}
\end{defn}

The conormal bundle has a natural structure of vector bundle over $Q$ of dimension $\text{codim}\, Q$; in particular, $N^*Q$ is a smooth $n$-dimensional submanifold of $T^*M$. Observe that, 
if $Q=M$, the conormal bundle is the zero section $\mathbb O_{T^*M} = M\times \{0\}$ while, if $Q=\{q\}$, the conormal bundle coincides with $T_q^*M$.

\begin{lemma}
If $Q\subseteq M$ is a smooth submanifold then the Liouville 1-form $\lambda$ vanishes identically on $N^*Q$; in particular $N^*Q$ is a Lagrangian submanifold of $T^*M$. Furthermore, the Liouville vector field $\eta$ is tangent to $N^*Q$.
\label{conlemma}
\end{lemma}

\begin{proof}
The first statement follows obviously from the definitions of $\lambda$ and $N^*Q$; now, if $x\in N^*Q$, then for every $\zeta\in T_xN^*Q$ we have
$$\omega(x)[\eta(x),\zeta]\ =\ \lambda(x)[\zeta]\ =\ 0\, .$$

Then $\eta(x)$ belongs to the (symplectic) orthogonal space of $T_xN^*Q$; but, since $N^*Q$ is Lagrangian, the orthogonal of $T_xN^*Q$ is itself and hence $\eta(x)\in T_x N^*Q$.
\end{proof}

\vspace{4mm}

Lemma \ref{conlemma} states that every conormal bundle is Lagrangian; in fact, the Liouville 1-form $\lambda$ vanishes identically on $N^*Q$. The converse, under an additional completeness hypothesis, is also true, as the 
theorem below states (cf. [AS09]). 


\begin{teo}
Let $L\subseteq T^*M$ be a n-dimensional smooth submanifold such that the Liouville 1-form $\lambda$ vanishes identically on $L$; then the intersection $R = L\cap \mathbb O_{T^*M}$  is a smooth 
submanifold. Moreover, if $L$ is a closed subset of $T^*M$, then $L=N^*R$.
\end{teo}

\backmatter 

\addcontentsline{toc}{chapter}{Bibliography}

\bibliography{tesidottoratobibliografia}
\bibliographystyle{amsalpha}

\end{document}